\documentclass[12pt]{article}
\usepackage{fullpage}
\usepackage{amsmath,amsfonts,amsthm,amssymb,epsfig}
\usepackage[
backend=biber,style=alphabetic,maxbibnames=99
]{biblatex}
\usepackage[all]{xy}
\usepackage{enumitem}
\usepackage[usenames,dvipsnames]{xcolor}
\usepackage[bookmarks=true,colorlinks=true, linkcolor=Mulberry, citecolor=Periwinkle]{hyperref}
\usepackage{graphicx}
\usepackage{tikz-cd}
\usepackage{float}
\usepackage{tipa}
\usepackage{MnSymbol}
\newtheorem{theorem}{Theorem}[section]
\newtheorem{lemma}[theorem]{Lemma}
\newtheorem{proposition}[theorem]{Proposition}
\newtheorem{corollary}[theorem]{Corollary}

\newtheorem{definition}[theorem]{Definition}

\theoremstyle{definition}

\newtheorem{example}[theorem]{Example}
\newtheorem{remark}[theorem]{Remark}

\newcommand\bcdot{\ensuremath{%
  \mathchoice%
   {\mskip\thinmuskip\lower0.2ex\hbox{\scalebox{1.5}{$\cdot$}}\mskip\thinmuskip}}%
   {\mskip\thinmuskip\lower0.2ex\hbox{\scalebox{1.5}{$\cdot$}}\mskip\thinmuskip}%   
   {\lower0.3ex\hbox{\scalebox{1.2}{$\cdot$}}}%
   {\lower0.3ex\hbox{\scalebox{1.2}{$\cdot$}}}%
   }

\newcommand{\on}{\operatorname}
\newcommand{\mc}{\mathcal}
\newcommand{\mbb}{\mathbb}

\newcommand{\wh}{\widehat}

\newcommand{\ul}{\underline}

\makeatletter
\def\legendre@dash#1#2{\hb@xt@#1{%
  \kern-#2\p@
  \cleaders\hbox{\kern.5\p@
    \vrule\@height.2\p@\@depth.2\p@\@width\p@
    \kern.5\p@}\hfil
  \kern-#2\p@
  }}
\def\@legendre#1#2#3#4#5{\mathopen{}\left(
  \sbox\z@{$\genfrac{}{}{0pt}{#1}{#3#4}{#3#5}$}%
  \dimen@=\wd\z@
  \kern-\p@\vcenter{\box0}\kern-\dimen@\vcenter{\legendre@dash\dimen@{#2}}\kern-\p@
  \right)\mathclose{}}
\newcommand\legendre[2]{\mathchoice
  {\@legendre{0}{1}{}{#1}{#2}}
  {\@legendre{1}{.5}{\vphantom{1}}{#1}{#2}}
  {\@legendre{2}{0}{\vphantom{1}}{#1}{#2}}
  {\@legendre{3}{0}{\vphantom{1}}{#1}{#2}}
}
\def\dlegendre{\@legendre{0}{1}{}}
\def\tlegendre{\@legendre{1}{0.5}{\vphantom{1}}}
\makeatother

\begin{document}

\title{Weil-Moore anima}
\author{Dustin Clausen}

\maketitle

\tableofcontents

\section{Introduction}

The Weil group of the rational number field $\mbb{Q}$ is a locally compact Hausdorff group $W_{\mbb{Q}}$ which refines the absolute Galois group $G_\mbb{Q}$, via a surjective homomorphism
$$W_{\mbb{Q}} \twoheadrightarrow G_\mbb{Q}$$
with connected kernel.  It determines an analogous Weil group $W_K$ for an arbitrary number field $K$, by taking the pullback along the inclusion $G_K\to G_\mbb{Q}$.

This Weil group was constructed by Weil in \cite{weil}.  The motivation came from L-functions, or more specifically, from the desire to unite two different classes of L-functions: the Artin L-functions attached to complex Galois representations, and the Hecke L-functions attached to Hecke characters $\mbb{A}_K^\times/K^\times \to \mbb{C}^\times$.  These classes were known to agree on a common locus, for a deep reason: a one-dimensional representation of $G_K$, which corresponds to a homomorphism
$$G_K^{ab}\to \mbb{C}^\times,$$
can be composed with the Artin map $\mbb{A}_K^\times/K^\times\twoheadrightarrow G_K^{ab}$ to get a Hecke character.  This sets up a 1-1 correspondence between one-dimensional Galois representations and Hecke characters with finite image, and Artin reciprocity says that the corresponding Artin and Hecke L-functions are equal, even factor-by-factor in their Euler products.

Not all Artin representations are one dimensional (or built from such via direct sums and inductions), and, on the other hand, not all Hecke characters have finite image.  So we have two separate worlds, though with substantial common overlap.  Then it makes sense to look for an even larger world containing both of them.  One such is provided by the notion of complex representation of the Weil group.  The key point is that the Artin map lifts to an \emph{isomorphism}

$$W_K^{ab} \simeq \mbb{A}_K^\times/K^\times$$

\noindent for any number field $K$.  Weil was able to define Artin-style L-functions attached to Weil group representations, and he showed that Hecke L-functions exactly correspond to the case of one-dimensional representations via the above lifted Artin map.

A beautiful aspect of Weil groups is that although the motivation for them comes from analytic number theory and representation theory, their construction involves a completely different set of ideas: namely it requires, and indeed encapsulates (see \cite{tatentb}), the \emph{cohomological} approach to class field theory.  The key idea is to view the absolute Galois group $G_\mbb{Q}$ not in the naive manner, as the inverse limit of its finite quotients (indexed by number fields $K$ inside a fixed $\overline{\mbb{Q}}$, as usual), but as the inverse limit of coarser terms
$$G_{\mbb{Q}} = \varprojlim_K \on{Gal}(K^{ab}/\mbb{Q}),$$
which themselves sit in short exact sequences of groups
$$1 \to G_K^{ab} \to \on{Gal}(K^{ab}/\mbb{Q}) \to \on{Gal}(K/\mbb{Q}) \to 1$$
with finite quotient and abelian kernel.  To build the Weil group we want to replace $G_K^{ab}$ in this short exact sequence by its refinement $\mbb{A}_K^\times/K^\times$.

Here comes the miraculous ``coincidence'': short exact sequences of the desired type are classified, up to isomorphism, by group cohomology classes
$$u \in H^2(\on{Gal}(K/\mbb{Q});\mathbb{A}_K^\times/K^\times),$$
and cohomological class field theory exactly produces canonical such classes $u=u_{K/\mbb{Q}}$.  One has to take some care to put the resulting group extensions into an inverse system, but in the end it can be done, see \cite{artintate}.

Our perspective in this article is that one should not stop with the Weil group.  Other cohomological considerations, originating in work of Tate, suggest the need to further refine the Weil group.  What we will see is that, contrary to naive hypothesis, a number field should not be a ``$K(\pi,1)$'': besides the ``fundamental group'' (the Weil group), there should also be higher homotopy groups.

Formally speaking, the first step has to be saying exactly what kind of object we mean to construct.  We want something that has homotopy groups, but where these homotopy groups are not abstract groups, but rather topological groups (in fact locally compact Hausdorff groups, like $W_{\mbb{Q}}$).  Thankfully, in joint work with Peter Scholze we defined and proved some basic properties of a class of objects, called \emph{condensed anima}, which can play such a role.

\begin{remark}
I often get asked to explain the terminology ``condensed" and the terminology ``anima", which were decided on by myself and Peter Scholze.  My explanations are usually something like the following.

\begin{enumerate}
\item The idea of condensed sets is that one encodes topological structures as being built from \emph{profinite sets} via colimits.  Profinite sets tend to look like Cantor sets.

Why the name ``condensed''? If one starts with an abstract set, then instead of specifying a collection of open subsets, as is required to describe a topology in the traditional sense, one specifies which families of points are allowed to ``condense'' together to form a compact Hausdorff space.  These compact Hausdorff spaces can in turn be exploded into profinite sets, because every compact Hausdorff space is a quotient of a profinite set by a profinite equivalence relation.

\item An \emph{anima} is just a homotopy type, but viewed as an object in its natural ambient $\infty$-category rather than its 1-categorical shadow (the classical homotopy category).

Why the name ``anima''?  There is a two-fold motivation.  First, one has the following intuition described by Beilinson in \cite{beilinson}.  A set is \emph{static} in the sense that its points are stuck where they are: they cannot move around within the set.  But if we have an anima $A$, with underlying set $\pi_0A$, then a point of $A$ is free to move around within its connected component: as Beilinson says, the points of $A$ are naturally ``animated''.  Second, an anima is something like a soul.  If we have, say, a CW complex, then its associated anima is what remains of the CW complex when it is stripped of its worldly (point-set) incarnation.
\end{enumerate}

A \emph{condensed anima} is an object which has both aspects, but in an independent way.  Thus it has both a ``topological'' aspect and a ``homotopy-theoretic'' aspect, and the topology is \emph{not} collapsed to homotopy theory.  Thus for example the Eilenberg-Maclane space $K(\mbb{R},n)$ on the real numbers with its usual structure of topological abelian group exists as a condensed anima, and it is different from the ordinary anima $K(\mbb{R}^\delta,n)$ and $\ast$.  The former is its underlying ordinary anima, and the latter is its homotopy type.

\end{remark}

Our problem is therefore well-posed: we want to find a condensed anima with fundamental group $W_{\mbb{Q}}$ and nontrivial higher homotopy groups.  The main thing we want from this condensed anima is that its cohomological properties should be ``better" than those of $W_{\mbb{Q}}$ or $G_{\mbb{Q}}$ (or rather of $BW_{\mbb{Q}}=K(W_\mbb{Q},1)$ or $BG_{\mbb{Q}}=K(G_\mbb{Q},1)$, corresponding to the fact that we're talking about group cohomology).

We state the theorem now, and give more precise motivation afterwards.
 
\begin{theorem}
There exists a condensed anima $X_{\mbb{Q}}$ with the following properties:
\begin{enumerate}
\item The condensed set $\pi_0 X_{\mbb{Q}}$ is $\ast$.
\item The condensed group $\pi_1 X_{\mbb{Q}}$ is isomorphic to the Weil group $W_\mbb{Q}$ of the rational numbers.
\item For $n>1$, the condensed abelian group $\pi_n X_{\mbb{Q}}$ is compact Hausdorff.
\item For any number field $K$, if $X_K \to X_{\mbb{Q}}$ denotes the corresponding finite etale cover, then $H^i(X_K;\mbb{R}/\mbb{Z})=0$ for $i>1$.
\end{enumerate}
\end{theorem}

To motivate the crucial property 4, recall that the essence of Weil groups is the existence of canonical isomorphisms $W_K^{ab}\simeq \mbb{A}_K^\ast/K^\ast$ refining the Artin map.  Equivalently, by Pontryagin duality, we can phrase this as a cohomology calculation:
$$H^1(BW_K;\mbb{R}/\mbb{Z})= (\mbb{A}_K^\ast/K^\ast)^{\vee},$$
where we use $(-)^\vee$ to denote Pontryagin duality.

However, the higher cohomology groups $H^{>1}(BW_K;\mbb{R}/\mbb{Z})$ do not all vanish.  This can be interpreted as saying that the isomorphism $W_K^{ab}\simeq \mathbb{A}_K^\ast/K^\times$ is not ``robust".  The purpose of the $X_K$ is to fix this.

\begin{remark}
   There are Weil groups associated not just to number fields, but also function fields over finite fields, and local fields.  In non-archimedean cases (function fields and nonarchimedean local fields), the higher cohomology above does vanish by a theorem of Tate (written up by Serre in \cite{serrewt1}), so we will not be changing anything in non-archimedean cases.  In other words, for nonarchimedean $K$ we will have $X_K=BW_K$.
\end{remark}

\begin{remark}
One can think of the relationship between $X_K$ and $BW_K = K(W_K,1)$ as being analogous to the relationship between $S^n$ and $K(\mbb{Z},n)$ in algebraic topology: the spheres $S^n$ have simple cohomology (they are ``Moore spaces''), while the $K(\mbb{Z},n)$ have simple homotopy (they are ``Eilenberg-Maclane spaces'').

Serre in \cite{serrethesis} explained how one can build spheres from Eilenberg-Maclane spaces, up to homotopy equivalence, by an inductive procedure, building up the Postnikov tower of $S^n$.  We apply this same procedure to build $X_{\mbb{Q}}$ from $BW_\mbb{Q}=K(W_\mbb{Q},1)$.  However, because of the nontrivial fundamental group, it is not at all formal that such a procedure works, and we need to prove a rather delicate Ext vanishing statement (Proposition \ref{extvanishing}) to see that there are no obstructions.  The argument uses the framework of \emph{class formations} defined by Artin-Tate in \cite{artintate}.
\end{remark}

\begin{remark}
The question of how (and in what sense) to put a geometric structure on $X_{\mbb{Q}}$ is tantalizing.  Presumably, $X_{\mbb{Q}}$ should be, in some vague sense, the condensed homotopy type underlying the space of closed points in the global analog of the Fargues-Fontaine curve (whose existence has been hypothesized by Scholze), analogously to how, in \cite{farguesscholze}, $BW_{\mbb{Q}_p}$ appears (in some sense) as the homotopy type of $\on{Div}^1$, which is the space of closed points in the Fargues-Fontaine curve.

Another interesting question is that of how to calculate the higher homotopy groups of $X_\mbb{Q}$.  In principle our construction gives a method analogous to Serre's from his thesis, but we do not investigate this in detail.  For $\pi_2$, see Lemma \ref{pi2}.
\end{remark}

Our main point is that the cohomology of these $X_K$ somehow ``improves'' the usual Galois cohomology.  Notably, we have the following, which for the moment we state imprecisely:

\begin{enumerate}
\item The cohomological dimension  (with respect to discrete coefficients) of $X_K$, and its variants $X_{K,S}$ with restricted ramification at a nonempty set of places $S$, is always $2$, independently of whether $K$ has a real place or not and independently of Leopoldt's conjecture.
\item There are local variants of the $X_K$, and there are versions of global and local Poitou-Tate duality for these Weil-Moore anima in which the archimedean and non-archimedean places contribute in a uniform way. (At nonarchimedean places we have $X_{K_\nu}=BW_{K_\nu}$, so there is nothing new, but at archimedean places there is a change.)
\item A simpler version of Tate's global Euler characteristic formula holds for these $X_{K,S}$ (with $S$ finite): the Euler characteristic is always $0$, just as in the function field case.
\end{enumerate}

Let us expand on this by describing some idiosyncrasies of Galois cohomology, which we aim to ``repair'' by considering the cohomology of the $X_K$ instead.  We can frame the discussion in terms of the well-known analogy with 3-manifolds with boundary (\cite{morishita}). The main point is that this analogy is slightly broken, for reasons coming from the archimedean places.

More explicitly, consider Poitou-Tate duality (\cite{tateduality}, \cite{milneduality}), which is supposed to be the analog of Poincar\'{e} duality.  Let $K$ be a number field, let $p$ be a prime, let $S$ be a finite set of places of $K$ containing all archimedean places and all places above $p$, and let $M$ be a finite $G_{K,S}$-module of $p$-power order.  Then the continuous group cohomology $R\Gamma(BG_{K,S};M)$ is finite of $p$-power order in each degree, and its Pontryagin dual is described as follows:

$$R\Gamma(BG_{K,S};M)^\vee\simeq \widetilde{R\Gamma}_c(BG_{F,S};M^\vee(1)[3]).$$
Here
$$\widetilde{R\Gamma}_c(BG_{K,S};M) := \on{Fib}\left(R\Gamma(BG_{K,S};M)\to \oplus_{\nu\in S} \widetilde{R\Gamma}(BG_{K_\nu};M)\right),$$
where $\widetilde{R\Gamma}(BG_{K_\nu};M)$ denotes the usual continuous group cohomology if $\nu$ is nonarchimedean, and Tate cohomology (of the finite group $G_{K_\nu}$) if $\nu$ is archimedean.

This looks \emph{almost} like Poincaré duality for the interior of a compact 3-manifold with boundary, where we interpret $S$ as parametrizing the boundary components, and we interpret the Tate twist as the orientation local system.  But only ``almost", because at the archimedean places we use Tate cohomology instead of usual cohomology.  For example at complex places the Tate cohomology vanishes, so somehow we ignore those components, and at real places we just get some 2-torsion groups, but generally in infinitely many degrees, positive and negative.  On the other hand, at nonarchimedean places we use the usual cohomology, and there local Tate duality looks just like Poincaré duality for a 2-dimensional closed manifold.  Thus there is an asymmetry between the nonarchimedean and archimedean places, where the nonarchimedean places fit into the philosophy of a 2-dimensional boundary component but the archimedean places do not.

For a more elementary perspective, recall that the ``fundamental class'' for Poitou-Tate duality comes from the famous calculation of the Brauer group of a number field, namely that it is the kernel of the sum of the invariant maps from the local Brauer groups.  At nonarchimedean places the invariant map is an isomorphism $Br(F_\nu)\simeq \mbb{Q}/\mbb{Z}$, but at archimedean places it is only an injection, with image of order $\leq 2$.  Somehow, at the archimedean places the Brauer group is too small.

At the root of all this is the fact that while the absolute Galois group $G_{\mbb{Q}_p}$ of $\mbb{Q}_p$ is a big and powerful object (according to Fontaine its $p$-adic representations can encode a lot of $p$-adic Hodge theoretic information), the absolute Galois groups of $\mbb{R}$ and $\mbb{C}$ are too small to reveal any interesting information, say of Hodge-theoretic nature.

This is partially fixed by passing from Galois groups to Weil groups.  The Weil group of the local field $\mbb{C}$ is $\mbb{C}^\times$ and the Weil group of $\mathbb{R}$ is an extension of $G_\mbb{R}$ by $\mbb{C}^\times$.  So we do get something bigger, and indeed complex representations of archimedean Weil groups can be used to encode pure Hodge structures with complex coeficients, see \cite{tatentb}.  From a group cohomology perspective, Weil groups also fix the problem with the Brauer group: if you calculate $H^2(BW_\mbb{R};\mbb{Q}/\mbb{Z}(1))$ or $H^2(BW_\mbb{C};\mbb{Q}/\mbb{Z}(1))$, you get the full $\mbb{Q}/\mbb{Z}$; correspondingly, the Weil version of the global Brauer group is also the ``correct'' one for an arbitrary number field $K$, namely
$$H^2(BW_K;\mbb{Q}/\mbb{Z}(1))=\on{ker}\left(
\oplus_\nu\mbb{Q}/\mbb{Z}\overset{\Sigma}{
\rightarrow}\mbb{Q}/\mbb{Z}\right),$$
where the direct sum is over all places of $K$, as was shown by Deligne in \cite{deligne}.

However, there is still a problem with the Weil group cohomology, in higher degrees: the cohomological dimension of $W_\mbb{C}$ and $W_\mbb{R}$ is not $2$ like in the nonarchimedean case, but infinite.  For example $H^i(BW_\mbb{C};\mbb{Q}/\mbb{Z})$ gives $\mathbb{Q}/\mbb{Z}$ not just in degree 2, but in all nonnegative even degrees.  Indeed the classifying stack $BW_\mbb{C}$ has the homotopy type of $BS^1 \sim \mbb{C}\mbb{P}^\infty$.  So we still don't have the ``correct'' 2-dimensional object, which would have the same kind of Poincaré duality as the Galois group (or Weil group) of a nonarchimedean local field.

There is a straightforward fix to this, which is to take just the 2-skeleton of $\mbb{C}\mbb{P}^\infty$, or in other words the 2-sphere $S^2$.  This of course has 2-dimensional Poincaré duality.  Moreover we can make $G_\mbb{R}$ act on this by the antipodal map on the geometric 2-sphere and this gives a reasonable 2-dimensional object for $\mbb{R}$, namely $\mbb{R}\mbb{P}^2$.

In fact this passage to $\mbb{R}\mbb{P}^2$ is also natural from the perspective of Hodge theory, or more precisely Simpson's twistor theory (\cite{simpson}).  Indeed $S^2$ is $\mbb{P}^1(\mbb{C})$, and vector bundles on $\mbb{P}^1_\mbb{C}$ equipped with semilinear structure relative to the antipodal involution are at the heart of Simpson's reinterpretation of what is a Hodge structure (with $\mbb{R}$-coefficients over $\mbb{C}$, this time).  In particular, antipodal twistor structures have Chern classes living in the cohomology of this 2-skeleton of the real Weil group. Following the discussion of Fargues in \cite{farguescft}, this is analogous to how vector bundles on the Fargues-Fontaine curve have Chern classes with values in the cohomology of the $p$-adic Weil group.

Although it has the correct cohomology, the homotopy type (anima) $S^2$ of course has trivial fundamental group and certainly cannot recover the Weil group $\mbb{C}^\times$ as a topological group, so we lost something here.  But this is exactly the sort of thing condensed anima are good for.  There is an action of $W_\mbb{C}=\mbb{C}^\times$ on the unit quaternions $\mbb{H}^\times$ by multiplication, hence an action of the condensed group $W_\mbb{C}$ on the homotopy type $|\mbb{H}^\times|\sim S^3$, and when we take the quotient (in condensed anima) we get something with fundamental group $W_\mbb{C}$ but whose cohomology (with discrete coefficients) is the same as that of $S^2$ (think of the Hopf fibration).  The same construction works for the reals as well, using that $W_\mbb{R}$ can be realized as a subgroup of $\mbb{H}^\times$.

Thus we get a fairly satisfying picture at the archimedean places.  The Weil group admits a reasonable refinement: a condensed anima whose $\pi_1$ is the Weil group, and whose cohomology with discrete coefficients has two dimensional duality exactly like a non-archimedean local field.  (Actually the refinement we give in the body of this paper will be slightly different than the one mentioned above: we will modify this quotient $S^3/W_\mbb{R}$ to make the higher homotopy groups compact.)

The issue, then is to make a similar refinement for the Weil group of a number field, compatible with the above story at the archimedean places (and with the usual Weil groups at non-archimedean places, which are in no need of refinement), such that we get a reasonable version of Poitou-Tate duality.  That is what we accomplish in this article.

To finish this introduction, let's give a brief tour of some areas of number theory where Galois (or Weil group) cohomology arises, and then indicate how the refinement we suggest somehow improves the picture.  We will not go into any more details about these points in the body of the paper; we discuss them here just for motivation.

\begin{enumerate}
\item First and most immediately, there is the ``Weil-etale'' perspective on special values of zeta functions developed in \cite{lichtenbaum}, where the exact same issues as discussed above arise.  Lichtenbaum considers a Weil group version of the compactly supported cohomology of the ring of integers $\mathcal{O}_F$ of a number field $F$, with coefficents in $\mbb{Z}$, and shows that \emph{if we truncate to degrees $\leq 3$}, then this gives an object in $\on{Perf}(\mbb{Z})$, and that moreover, its base-change to $\on{Perf}(\mbb{R})$ admits an expression of the form $\on{fib}(V\overset{\Phi}{\rightarrow}V)$ for some canonical $V\in\on{Perf}(\mbb{R})$ with canonical endomorphism $\Phi$.  On the level of K-theory, it follows that the point in $K(\mbb{Z})$ defined by this cohomology promotes to the homotopy fiber of $K(\mbb{Z})\to K(\mbb{R})$, whence we get a natural element of the group $K_1(\mbb{R}) / K_1(\mbb{Z}) = \mbb{R}^\times/\mbb{Z}^\times$ of nonzero real numbers up to sign.

The result of \cite{lichtenbaum} is that this nonzero real number (up to sign) is equal to the the leading coefficient of $\zeta_F$ at $s=0$, with the order of vanishing being given by the Euler characteristic of $V$.  This is a compelling reformulation of the Dirichlet class number formula.  However, in high degrees the compactly supported cohomology with $\mbb{Z}$-coefficients is nonzero and huge, see \cite{flach}.  Thus the ad hoc truncation used above was necessary, and the picture is somewhat spoiled.  The trouble again is that the Weil group of a number field or of an archimedean local field has unwanted cohomology in degrees $>3$.  This is fixed by the object we propose, so Lichtenbaum's idea goes through directly and without truncation.

\item Similarly, turning towards $p$-adic coefficients, we can consider Kato's generalized Iwasawa main conjecture (\cite{kato}).  Let $p$ be a prime, $\Lambda$ a commutative pro-$p$-finite ring and let $M$ be a lisse $\Lambda$-sheaf on $\on{Spec}(\mbb{Z}[1/S])_{et}$, where $S$ is a finite set of prime numbers with $p\in S$.  Kato conjectured that associated to this data is a canonical \emph{zeta element}, which is a particular isomorphism
$$\zeta_M: \on{det}_{\Lambda}(R\Gamma_c(\on{Spec}(\mbb{Z}[1/S])_{et};M))\simeq \Lambda.$$
Besides natural formal properties, these zeta elements are, crucially, supposed to satisfy a more nuanced axiom which says that when $M$ comes from a motive, there is a precise relation with special values of (complex) L-functions.

Notably, here $R\Gamma_c$ is defined in the completely naive way:
$$R\Gamma_c(\on{Spec}(\mbb{Z}[1/S])_{et};M) := \on{Fib}(R\Gamma (\on{Spec}(\mbb{Z}[1/S])_{et};M)\to \oplus_{\nu \in S\cup \{\infty\}} R\Gamma(\on{Spec}(\mbb{Q}_\nu)_{et};M)).$$
We use usual etale (or Galois) cohomology everywhere, not Tate cohomology or Weil group cohomology.  We do have
$$R\Gamma_c(\on{Spec}(\mbb{Z}[1/S])_{et};M)\in \on{Perf}(\Lambda)$$
for all $p,S,\Lambda,M$ as above, which shows that the determinant in the statement is well-defined.  Moreover, Tate's global Euler characteristic formula has a reformulation which says that the above perfect complex always has Euler characteristic $0$ (a ``$K_0$'' companion to Kato's conjecture, which implicitly has to do with $K_1$, or rather $\tau_{\leq 1}K$).

It seems that here, in contrast to the the situations described above, everything behaves perfectly logically and there is no need for a refinement.  Yet, we are proposing a refinement, so we should check what it does.  Thankfully, everything makes sense: in this setting, we will show that there is a natural isomorphism
$$R\Gamma_c(\on{Spec}(\mbb{Z}[1/S])_{et};M) \overset{\sim}{\rightarrow} R\Gamma_c(X_{\mbb{Q},S};M).$$
That is, with these kinds of coefficients, compactly supported etale (or Galois) cohomology is the same as compactly supported Weil-Moore cohomology, so Kato's conjecture can equivalently be formulated in the Weil-Moore context.

Moreover, once we are in the Weil-Moore world, we gain something: by Poitou-Tate duality for Weil-Moore anima, we can also reformulate Kato's conjecture as positing trivializations
$$\on{det}_\Lambda R\Gamma(X_{\mbb{Q},S};M) \simeq \Lambda,$$
for any $M$ and $\Lambda$ as above, here using the non-compactly supported cohomology of the Weil-Moore anima. More precisely, the two conjectures are equivalent under switching $M\leftrightarrow M^\vee(1)$.

This reformulation should constitute a modest improvement.  Actually, Kato remarks that for the purposes of producing such zeta isomorphisms it is helpful to pass to the Poitou-Tate dual perspective of usual (not compactly supported) cohomology, because that is where ``special elements'' like cyclotomic units naturally live.  However, since as we mentioned above, classical Poitou-Tate duality has some strange aspects at the archimedean places, this means that Kato has to introduce an archimedean correction factor to his conjecture, and moreover the prime $p=2$ requires extra attention.  When we pass to the Weil-Moore setting, this becomes much simpler, as stated above.

\item The previous points had to do with direct uses for the \emph{cohomology} of the $X_{K,S}$.  One can say with reasonable justification that this on this linear level, $X_{K,S}$ just provides a reorganization of known aspects of Galois cohomology.  However, the fact that it arises from the non-linear object $X_{K,S}$ gives extra information, and we will now explain that this information is relevant in the Langlands program.

More precisely, we mean the Langlands program with mod $p$ or $p$-adic coefficients, for some fixed prime $p$; thus we avoid having to contend with the conjectural \emph{Langlands group} (\cite{langlands}, \cite{arthur}), which appears when working with the classical $\mathbb{C}$-coefficients due to the fact that even the Weil group $W_\mbb{Q}$ does not have nearly enough complex representations to match automorphic phenomena.  But for mod $p$ coefficients this is not a concern.

Our basic claim is that on the ``Galois side''/``spectral side'' of the Langlands correspondence for number fields, one should replace the notion of Galois representation by the notion of a \emph{Weil-Moore representation}: that is, a finite rank local system on $X_K$.  At the naive level this gives nothing new: as
$$\tau_{\leq 1}X_K = BW_K,$$
a Weil-Moore representation is the same thing as a Weil group representation.  If the coefficient ring $R$ for our representations is discrete or $p$-adic, it's also the same as a Galois representation, because $W_K/W_K^\circ = G_K$.

More precisely, the above reasoning holds when $R$ is a \emph{classical} commutative ring, i.e.\ when $R=\pi_0R$.  If $R$ has nontrivial derived structure, then the notion of Weil-Moore representation differs from the notion of Weil group representation which in turn differs from the notion of Galois group representation.

Let us be a bit more precise about this.  For a group scheme $\wh{G}$ over $\Lambda=\mbb{Z}/p^n\mbb{Z}$, we can consider, for any animated commutative $\Lambda$-algebra $R$, the anima of maps of condensed anima
$$X_K \to (B\wh{G})(R),$$
where the target is discrete (as a condensed anima, meaning it is just an ordinary anima).  This defines a formal derived stack $\mathfrak{X}_{K}=\mathfrak{X}_{K,\wh{G}}$ over $\on{Spec}(\Lambda)$ organizing ``deformations of Weil-Moore representations''.  Then these $\mathfrak{X}_{K,\wh{G}}$ provide a derived structure on the usual Galois deformation functors (\cite{mazur}), or Wang-Erickson stacks, \cite{carlwe}, which is \emph{different} than the naive derived structure considered by Galatius-Venkatesh in \cite{galatius}.

To see what this might give in practice, let's turn to the remarkable conjectures of Emerton-Gee-Hellmann and Zhu, see \cite{egh} and \cite{zhu}.  These posit a certain ``spectral decomposition'' of the compactly supported cohomology of Shimura varieties.  In this context, to give a \emph{spectral decomposition} of a vector space means to find stack of ``parameters'' and an Ind-coherent sheaf on that stack whose global sections identify with the vector space.  (This rather refined notion of what it means to give a spectral decomposition arose in the geometric Langlands program, see \cite{gaitsgoryshtukas} and \cite{locsysres}, as an extension and reinterpretation of the work of Lafforgue \cite{lafforgue}.  We refer to \cite{scholzebourbaki}, especially Remark 5.4, for a particularly cogent explanation.)

In the case of the above-mentioned conjectures, the stack of parameters is quite literally a moduli stack of L-parameters, and the Ind-coherent sheaf is directly described (conjecturally) in terms of \emph{categorical} local Langlands at the non-archimedean places!

Actually, this is a bit of an oversimplification, because one needs to somehow take care of the archimedean places as well.  We refer to \cite{egh} and upcoming papers of Emerton-Zhu and Davis-Emerton-Vilonen for more details.  But this is precisely where the Weil-Moore versions can potentially help, in making the archimedean places more directly analogous to the non-archimedean places.  As a bonus, one can formulate a more general conjecture describing the compactly supported cohomology with finite coefficients of an arbitrary adelic quotient, at least modulo standard technicalities having to do with the Hasse priniciple.

As a preliminary to this, one needs to understand categorical local Langlands at the archimedean places, with discrete coefficients, analogously to the categorical local Langlands conjectures at non-archimedean places discussed in \cite{farguesscholze}, \cite{zhu}, \cite{egh}.  In fact, in conversations with Peter Scholze several years ago we came up with a proposal for what this archimedean analogue should be: namely, following Ben-Zvi--Nadler (\cite{bznbetti}), it should simply be Betti geometric Langlands for $\mbb{P}^1_{\mbb{C}}$ at complex places, and a version of Betti geometric Langlands for the twistor line $\on{Tw}_{\mbb{R}}$ at real places.  See also \cite{scholzerll}, which is a more involved version of the same idea aimed at real local Langlands with complex coefficients instead of discrete coefficients.

The Betti geometric Langlands of Ben-Zvi--Nadler (\cite{bznbetti}) for real curves should look \emph{roughly} as follows (we will explain the subtleties later).  Let $X$ be a smooth proper curve over $\mbb{R}$ such that $X(\mbb{R})=\emptyset$; on the dual side, the same data is encoded in the closed conformal (not necessarily oriented or connected) surface
$$|X| = X(\mbb{C})/\on{Gal}(\mbb{C}/\mbb{R}).$$
This $|X|$ is equivalently the ``moduli space of closed points in $X$'' and is the analog of $\on{Div}^1$ from the Fargues-Scholze setting, \cite{farguesscholze}.

Let $G$ be a reductive group over $\mbb{R}$.  On the automorphic side, one takes sheaves (of $\Lambda$-modules) with nilpotent singular support on the \emph{real points} of the moduli stack of $G$ bundles on $X$, and on the spectral side one takes Ind-coherent sheaves with nilpotent singular support on the stack of $\wh{G}$-local systems on $|X|$.
Note that if $X$ is a curve over $\mbb{C}$ viewed as a (geometrically disconnected) curve over $\mbb{R}$, this reduces to the usual Betti geometric Langlands from \cite{bznbetti}.

If $\on{Tw}_{\mbb{R}}$ denotes the \emph{twistor line}, the nontrivial real form of $\mbb{P}^1$, then $|\on{Tw}_{\mbb{R}}|=\mbb{R}\mbb{P}^2$ and this notion of local system exactly matches the notion of a Weil-Moore representation with $\Lambda$-coefficients over $\on{Spec}(\mbb{R})$.  Thus what we are proposing is that our suggested modifications of the Galois/Weil group are exactly what is required to make Betti geometric Langlands \emph{directly} relevant to arithmetic Langlands for number fields (not just relevant by analogy).

However, it's important to note some technical subtleties in this situation.  The closed surface $|X|$ can be non-orientable, and indeed it is non-orientable for our primary example $X=\on{Tw}_{\mbb{R}}$.  This means we must use the C-group (\cite{buzzardgee}) instead of the L-group when talking about local systems for the dual group on the curve.  More precisely, the usual action of $\on{Gal}(\mbb{C}/\mbb{R})=\pi_1|X|$ on $\wh{G}$ coming from Langlands duality gives rise to a local system of reductive groups on $|X|$, allowing to define a twisted notion of $\wh{G}$-local system.  This is the twist which corresponds to the usual notion of L-parameters, but it needs to be twisted even more (or differently): there is a canonical map $\mu_2 \to Z(\wh{G})$ (see \cite{bernsteinsign}) which means we can further twist the notion of $\wh{G}$-local system by a $\mu_2$-gerbe on $|X|$, and the gerbe we take is the gerbe of square roots of the Tate twist. This extra-twisted (C-group) version is exactly the version of local system which arises from considering geometric Satake in families over $|X|$ (viewed as the space parametrizing closed points of $X$), compare \cite{farguesscholze}.
 
Even more annoyingly (though somehow also equivalently, but on the dual side), the canonical bundle $\Omega^1$ on $\on{Tw}_{\mbb{R}}$ does not admit a square root, and this precludes, at least for certain $G$ such as $G=SL_2$, the existence of a Whittaker sheaf, which is used to pin down the geometric Langlands equivalence.  In geometric Langlands over $\mbb{C}$, this is just an annoying technicality which means that geometric Langlands is not canonical (one needs to fix Whittaker data to make it canonical).  For geometric Langlands over $\mbb{R}$, however, it is a more serious issue, because when there is no Whittaker data, geometric Langlands simply cannot hold in the usual formulation.  To get around this, one can (and must) make a further twist.  We could either twist the automorphic category as described in \cite{bzsv} C.7  to get a canonical Whittaker sheaf, or else twist the spectral side to lose the structure sheaf.  The former solution is more symmetric and aesthetically pleasing, but in the global situation with number fields it is less clear what form the automorphic Whittaker twist should take, so let us explain the spectral Whittaker twist instead.

More precisely we twist the category of Ind-coherent sheaves by a certain canonical $\mu_2$-gerbe on the stack of $\wh{G}^C$-local systems.  I believe this $\mu_2$-gerbe can be described as follows (compare \cite{gaitsgoryraskin}): any reductive group admits a canonical $\mu_2$-central extension, see again \cite{bernsteinsign}, which implies that a $\wh{G}^C$-local system on $|X|$ induces a $\mu_2$-gerbe on $|X|$.  Cupping with the gerbe of square roots of the canonical bundle of $X$ and integrating over $|X|$ we get a $\mu_2$-gerbe on the stack of $\wh{G}^C$-local systems. This should be the gerbe we use to twist the notion of Ind-coherent sheaf in order to get a geometric Langlands correspondence which is independent of a choice of Whittaker sheaf.

The above subtleties with Whittaker data should also exist for the categorical local Langlands with mod $p^n$ coefficients at nonarchimedean places.  \cite{egh} avoids them by restricting to groups of adjoint type, where the Whittaker twist is trivial (so only the C-group twist remains).

Now, returning to the global situation, let $K$ be a number field and $G$ a reductive group over $K$, and let us say that our goal is to give a spectral decomposition of the compactly supported cohomology
$$R\Gamma_c(G(\mbb{A}_K)/G(K);\Lambda)$$
 of the standard adelic quotient, where we fix our coefficient ring to be $\Lambda=\mbb{Z}/p^n\mbb{Z}$ for some prime $p$ and $n\geq 1$.  Note that compactly supported cohomology is the most elementary possible variant of cohomology to consider, because it is the only one which naturally produces a \emph{discrete} $\Lambda$-module on an arbitrary locally compact Hausdorff space (in the case at hand, $G(\mbb{A}_K)/G(K)$); even better the above compactly supported cohomology gives an admissible representation of $G(\mbb{A}_K)$.  Thus the above space appears as a reasonable finite-coefficient analog of the familiar space of automorphic forms with $\mbb{C}$-coefficients.

Actually, we need to switch to a slight variant of $G(\mbb{A}_K)/G(K)$, namely
$$(BG)_c(K) := \on{Fib}((BG)(K) \to (BG)(\mbb{A}_K)).$$
In more standard adelic terms this can be described as
$$(BG)_c(K) = \bigsqcup_\xi G^\xi(K)\backslash G(\mbb{A}_K)$$
where $\xi$ runs over a set of representatives for $\on{ker}^1(K,G)$ with $G^\xi$ the corresponding pure inner twist.  Thus, passing from the standard adelic quotient to this $(BG)_c(K)$ is a rather mild modification, and if $G$ satisfies the Hasse principle it is not a modificiation at all.

We take as known the (semi-)local version of our question, namely the compactly supported cohomology
$$R\Gamma_c(G(\mbb{A}_K);\Lambda),$$
as a $G(\mbb{A}_K)$-representation, and seek to describe $R\Gamma_c((BG)_c(K);\Lambda)$ in terms of it, but over on the spectral side.

An adelic version of categorical local Langlands should conjecturally produce from $R\Gamma_c(G(\mbb{A}_K);\Lambda)$ a Whittker-twisted Ind-coherent sheaf $\mathfrak{U}_{\mbb{A}_K}$ (which one could call the \emph{Plancherel sheaf}) on the stack
$\mathfrak{X}_{\mbb{A}_K}$ of adelic Langlands parameters in the C-group sense, defined in terms of the stacks of local C-parameters as follows:
$$\mathfrak{X}_{\mbb{A}_K} := \varinjlim_S \prod_{\nu\in S}\mathfrak{X}_{K_\nu} \times \prod_{\nu\not\in S} \mathfrak{X}_{\mc{O}_{K_\nu}}.$$
We recall that at the nonarchimedean places, $X_{K_\nu}=BW_{K_\nu}$, so there is nothing novel here compared to the story decribed in \cite{egh}, but at archimedean places we use Betti geometric Langlands as described above.  Also, as in \cite{egh}, the theory is likely to become cleaner if, when $\nu$ lies above $p$, we take $\mathfrak{X}_{K_\nu}$ to be the Emerton-Gee stack (\cite{emertongee}) parametrizing etale $(\varphi,\Gamma)$-modules instead of the stack of Weil group representations.

There is a global-to-local map
$$f:\mathfrak{X}_{K} \to \mathfrak{X}_{\mbb{A}_K}$$
from the moduli stack of global L-parameters again in the Weil-Moore sense and with the C-group twist, and an expectation in line with the formalism of \cite{egh}, \cite{zhu} and \cite{locsysres}, and consistent with suggestions of Scholze concerning a hypothetical ``global $\on{Bun}_G$" (c.f.\ \cite{slh} Conjecture F in the function field case) and suggestions of \cite{bzsv} concerning the relevance of TQFT-like structures to number-theoretic Langlands, would be that we should have an isomorphism

$$R\Gamma_c((BG)_c(K);\Lambda) \simeq R\Gamma(\mathfrak{X}_K; f^!\mathfrak{U}_{\mbb{A}_K}).$$

  In other words, the Ind-coherent sheaf $f^!\mathfrak{U}_{\mbb{A}_K}$ should provide the desired refined spectral decomposition of the given space of (mod $p^n$) automorphic forms for $G$.  Note that the Plancherel sheaf $\mathfrak{U}_{\mbb{A}_K}$ is built from purely local data, so in this conjecture (as in the similar conjectures in \cite{egh}, \cite{zhu}, \cite{locsysres}) all the global information about automorphic forms comes on the dual side from applying the functor $f^!$.  An important technical remark is that the global sections on the right makes sense because the Whittaker gerbe trivializes on pullback to $\mathfrak{X}_K$, essentially because of Hasse reciprocity.  We also note that the version of $R\Gamma(\mathfrak{X}_K;-)$ in play here is the one which by construction preserves colimits, and on coherent sheaves (which are always pushed forward from some algebraic closed substack of $\mathfrak{X}_K$) is given by the usual global sections; see \cite{cautis} for relevant generalities on Ind-coherent sheaves.

It would be interesting to gain enough understanding of Betti Langlands for $\on{Tw}_{\mbb{R}}$ to see that this formulation is compatible with the one in \cite{egh} in the case of Shimura varieties.  In particular it would be nice to see the familiar ``oddness'' of the Galois representations one gets from known spaces of cohomological automorphic forms, which is put in by hand in \cite{egh}, being (at least conjecturally) a consequence of archimedean categorical local Langlands.

One can similarly translate Zhu's conjectures on the compactly-supported cohomology of moduli stacks of shtukas (\cite{zhu}) from the function field case to the number field case using these ideas.  Rather, we can translate the Galois side of Zhu's conjectural isomorphisms, without yet knowing what should go on the automorphic side.  In any case this supports Scholze's conviction (\cite{scholzeicm},\cite{scholzebourbaki}) that shtukas (and their moduli) should somehow make sense in the number field case as well.

\item Speaking of number field shtukas, they tie into another motivation for the present work.  For such a notion to exist, there needs to be an object ``$\Omega^1_{\mbb{Q}/\mbb{F}_1}$'', in some (exotic) sense a line bundle on $\on{Spec}(\mbb{Q})$, which provides the correct twists when studying behaviour of shtukas along their locus of modification, analogous to the existing object $\Omega^1_{F/\mbb{F}_q}$ in the function field case.

What can we say about this (undefined) object? Only one thing is currently known: if we pull back along $\on{Spa}(\mbb{Q}_p)\to \on{Spec}(\mbb{Q})$, then the resulting object ``$\Omega^1_{\mbb{Q}_p/\mbb{F}_1}$'' actually has a meaning, as discussed by Faltings (\cite{faltings}).  It can be analysed like this: first, if $F/\mbb{Q}_p$ is perfectoid and contains all $p$-power roots of unity, then Fontaine's calculations (\cite{fontaine}) show that the cotangent complex $\mbb{L}_{F/\mbb{Q}_p}$ (or rather its natural completion) functorially identifies with the shifted Tate twist $F(1)[1]$.  On the other hand it's reasonable to posit that $\mbb{L}_{F/\mbb{F}_1}=0$, because the analog in characteristic $p$ holds.  Combining, we see that
$$``\mbb{L}_{\mbb{Q}_p/\mbb{F}_1} = \mbb{Q}_p(1)."$$
Here, more formally, $\mbb{Q}_p(1)$ is a pro-etale line bundle on $\on{Spa}(\mbb{Q}_p)$ whose sections over any large enough pro-etale extension $F$ of $\mbb{Q}_p$ canonically identifies with the rank one $F$-module $F(1)$.  This is indeed the object giving the relevant twists in the picture of local shtukas with one leg, \cite{scholzeberkeley}.

Keeping with our cohomological theme, although we don't know in which category the global ``$\Omega^1_{\mbb{Q}/\mbb{F}_1}$'' should live, we can try to ask what its first Chern class should be, say with finite coefficients.  Again we only know what happens on base-change to $\mbb{Q}_p$.  Since etale cohomology has pro-etale descent, pro-etale vector bundles over $\mbb{Q}_p$ have Chern classes in Galois cohomology just like ordinary vector bundles.  For $\mbb{Q}_p(1)$ and with (mod $p^n$) coefficients, the answer was given by Petrov when discussing joint work with Pan: we have that
$$c_1(\mbb{Q}_p(1)) \in H^2(G_{\mbb{Q}_p};\mbb{Z}/p^n\mbb{Z}(1)) \mapsto 1 \in \mbb{Z}/p^n\mbb{Z}$$
under the invariant map $\on{inv}: H^2(G_{\mbb{Q}_p};\mbb{Z}/p^n\mbb{Z}(1))\simeq \mbb{Z}/p^n\mbb{Z}$ (the fundamental class for local Tate duality).  On the other hand one can also calculate that with mod $\ell^n$ coefficients for $\ell\neq p$, we have
$$c_1(\mbb{Q}_p(1)) = 0 \in H^2(G_{\mbb{Q}_p};\mbb{Z}/\ell^n\mbb{Z}(1)).$$
What does this mean for $``c_1(\Omega^1_{\mbb{Q}/\mbb{F}_1})"$ in (mod $p^n$) cohomology?  Its local invariant at the prime $p$ is as nontrivial as possible (generator of a cyclic group of order $p^n$), while its local invariants at primes $\ell\neq p$ are trivial.  By reciprocity, this means that its local invariant at the archimedean place must be as nontrivial as possible, for all $p$.

Again, the Brauer group of $\mbb{R}$ is too small to accommodate such a Chern class.  This means that $c_1(\Omega^1_{\mbb{Q}/\mbb{F}_1})$ cannot live in usual Galois cohomology $H^2(G_\mbb{Q};\mbb{Z}/n\mbb{Z}(1))$, because this involves the too-small Brauer group at the infinite place.  But there's also a separate issue: the base-change of $``\Omega^1_{\mbb{Q}/\mbb{F}_1}"$ to the infinite place should not be construed to be $\Omega^1_{\mbb{R}/\mbb{F}_1}$, because that object should vanish ($\mbb{R}$ behaves like a perfectoid).

Both problems are fixed if we make the above-suggested replacement of $\mbb{R}$ by the twistor line $\on{Tw}_{\mbb{R}}$.  In fact, something very pleasant happens: the first Chern class
$$c_1(\Omega^1_{\on{Tw}_{\mbb{R}}/\mbb{R}}) \in H^2(\on{Tw}_{\mbb{R}};\mbb{Z}/n\mbb{Z}(1))$$
of the cotangent bundle of $\on{Tw}_{\mbb{R}}$ maps to the generator $-1\in\mbb{Z}/p^n\mbb{Z}$ under the invariant map for all $p$ and $n$, exactly as required by reciprocity!  (Note also that since $\mbb{R}$ counts as perfectoid, at the heuristic level we have $\Omega^1_{\on{Tw}_{\mbb{R}}/\mbb{R}}=\Omega^1_{\on{Tw}_{\mbb{R}}/\mbb{F}_1}$.)  Thus, cohomologically speaking, it makes sense for $\Omega^1_{\on{Tw}_{\mbb{R}}/\mbb{R}}$ to be the base-change to the infinite place of our fictional object $\Omega^1_{\mbb{Q}/\mbb{F}_1}$.

This gives the impression that $c_1(\Omega^1_{\mbb{Q}/\mbb{F}_1})$ really does deserve to exist, and that it should live in $H^2(X_\mbb{Q};\mbb{Z}/n\mbb{Z}(1))$ where $X_\mbb{Q}$ is some kind of ``space'' replacing $\on{Spec}(\mbb{Q})$ for which Poitou-Tate duality takes the optimal form as discussed above.  In this article we at least produce $X_\mbb{Q}$ as a condensed anima.  This is not enough to make sense of $\Omega^1$, but it is enough sense to make sense of its first Chern class.
\end{enumerate}

\subsection*{Dedication}

I dedicate this article to the memory of John Tate, my grandfather and mentor.  His kindness and his excellent taste were crucial for my development, and his work continues to shine an unusually bright light on the mathematical landscape.

\subsection*{Acknowledgements}

I thank Ahmed Abbes, Johannes Ansch\"utz, Robert Burklund, Matt Emerton, Ofer Gabber, Dennis Gaitsgory, Toby Gee, Jesper Grodal, Maxim Kontsevich, Jacob Lurie, Adrien Morin, Alexander Petrov, Vincent Pilloni, Sam Raskin, and Peter Scholze for helpful remarks and conversations.

I also thank the referee for their comments, which were very helpful for improving the text.  Finally, I thank Semon Rezchikov for supplying me with several LLM-generated ``referee reports'' which revealed some typos and minor errors.

\section{Background on condensed anima}\label{condback}

In this paper we use the formalism of \emph{condensed anima}, developed by the author and Scholze.  Here we give a rapid overview of some basics of this theory.

The quickest version of the definition is the following.

\begin{definition}
Let $\on{Extr}$ denote the category of extremally disconnected compact Hausdorff spaces.  Note that $\on{Extr}$ has finite coproducts given by disjoint unions.

The $\infty$-category of \emph{condensed anima} is the $\infty$-category
$$\on{CondAn} = \mc{P}_{\Sigma}(\on{Extr})$$
freely generated by $\on{Extr}$ under sifted colimits, as formally encoded in the $\mc{P}_{\Sigma}$-construction of \cite{htt} 5.5.8.

In other words, $\on{CondAn}$ is the $\infty$-category generated by compact projective objects, such that the full subcategory of compact projective objects in $\on{CondAn}$ identifies with $\on{Extr}$.
\end{definition}

By construction, $\on{CondAn}$ is a full subcategory of presheaves of anima on $\on{Extr}$, namely it is spanned by those $X:\on{Extr}^{op} \to \on{An}$ satisfying the following two conditions:
\begin{enumerate}
\item If $I$ is a finite set and $\{S_i\}_{i\in I}$ is an $I$-indexed collection of objects in $\on{Extr}$, then
$$X(\sqcup_{i\in I} S_i) \overset{\sim}{\rightarrow} \prod_{i\in I} X(S_i).$$
\item $X$ is a small colimit of representable presheaves.
\end{enumerate}

Note that finite coproducts in $\on{Extr}$ are disjoint and universal.  It follows that $\on{CondAn}$ can also be construed as an $\infty$-category of sheaves of anima on $\on{Extr}$ with respect to a Grothendieck topology, satisfying the technical smallness condition 2.  The only thing preventing it from being an $\infty$-topos in the technical sense is that $\on{CondAn}$ is not presentable (because $\on{Extr}$ is not essentially small).  But this does not cause problems in practice, and one generally has access to the whole toolkit of an $\infty$-topos.

An important example for us is the Postnikov tower of a condensed anima $X$, or more generally the relative Postnikov tower of a map of condensed anima $X\to Y$.  Viewing a consensed anima $X$ as a presheaf of anima on $\on{Extr}$, the Postnikov tower of $X$ is perfectly naive: it is given by section-wise Postnikov trunctation (as this clearly preserves the sheaf condition).  The relative Postnikov tower of $X\to Y$ is hardly more complicated: its formation commutes with base-change, so by descent one can reduce to the case of $Y\in \on{Extr}$ and in that case the relative Postnikov tower agrees with the absolute one past the first stage.  In any case one easily sees that all Postnikov towers converge by reduction to the case of $\on{An}$.

The most basic condensed anima are in principle the objects of $\on{Extr}$ coming via the Yoneda embedding, but these are completely inexplicit (the basic example is the Stone-Cech compactification of an infinite discrete set).  The more general \emph{compact Hausdorff spaces} give a more practical starting point.  Explicitly, there is a fully faithful embedding

$$\on{CHaus} \to \on{CondAn}$$
sending $K\in \on{CHaus}$ to the presheaf
$$S \mapsto \on{Cont}(S,K)$$
on $\on{Extr}$.

For $K\in\on{CHaus}$ we can always find an $S_0\in \on{Extr}$ with a surjective continuous map $S_0 \twoheadrightarrow K$, then an $S_1\in\on{Extr}$ with a surjective continuous map $S_1 \twoheadrightarrow S_0\times_K S_0$, etc., and in the end we can find a hypercover of $K$ in $\on{CondAn}$ such that the terms lie in $\on{Extr}$.  In joint work with Scholze (see \cite{condensed}), it is shown that this kind of hypercover, though completely inexplicit, can be used to do basic calculations.  In particular:

\begin{theorem}\label{cohchaus}
Suppose $K\in\on{CHaus}\subset \on{CondAn}$.

\begin{enumerate}
\item If $A$ is a discrete abelian group, then $R\Gamma(K;A)$ (in the sense of cohomology of the ``topos" $\on{CondAn}_{/K}$) agrees with the sheaf cohomology $R\Gamma_{shf}(K;A)$ of the topological space $K$; more precisely the comparison map
$$R\Gamma_{shf}(K;A)\overset{\sim}{\rightarrow} R\Gamma_{cond}(K;A)$$
is an isomorphism, where this comparison map comes from the fact that an open cover is also a cover in the condensed Grothendieck topology.
\item If $V$ is a Frechet space over $\mbb{R}$, viewed as a condensed abelian group, then
$$H^i(K;V)=0$$
for $i>0$, and (of course) $H^0(K;V) = \on{Cont}(K,V)$.
\end{enumerate}
\end{theorem}
\begin{proof}
For 1, see \cite{condensed} Theorem 3.2.  For 2, in the case $V=\mbb{R}$ see \cite{condensed} Theorem 3.3.  The same proof works for any $\mbb{R}$-Banach space.  The passage to Frechet spaces follows by a standard Mittag-Leffler argument (see \cite{cscx} Lemma 4.8).
\end{proof}

\begin{remark}
Implicitly, in the above $H^i(X;-)$ is viewed as a functor from condensed abelian groups to abelian groups.  But it naturally promotes to a functor $\ul{H}^i(X;-)$ from condensed abelian groups to condensed abelian groups, where the $S$-valued points (for $S\in\on{Extr})$ are given by $H^i(S\times X;-)$.  In other words, the cohomology naturally has condensed structure.  Since $S\times K\in\on{CHaus}$ whenever $K\in\on{CHaus}$, the above results immediately allow to read off the condensed structure as well.  For example, in 1 above the cohomology $\ul{H}^i(K;A)$ is discrete (as a condensed abelian group) and in 2 the cohomology $\ul{H}^i(K;V)$ vanishes as a condensed abelian group for $i>0$ and for $i=0$ identifies with the (condensed abelian group associated to the) Frechet space of continuous functions $\on{Cont}(K;V)$.
\end{remark}

One can often extend these results to more general classes of topological spaces by writing a topological space $X$ as a union of compact Hausdorff subspaces.  For example, by this argument Claim 1 of Theorem \ref{cohchaus} directly applies to arbitrary \emph{locally} compact Hausdorff spaces as well.  Similarly, we can extend in a different direction, to objects like $BG$ (viewed as a condensed groupoid, meaning a $1$-truncated condensed anima) for a locally compact Hausdorff group $G$, by the bar resolution of $BG$.  For example:

\begin{lemma}\label{cohliegp}
Suppose $G$ is a real Lie group.

\begin{enumerate}
\item For $A$ a discrete abelian group, $H^i(BG;A)$ for $i\geq 0$ identifies with the usual cohomology of the classifying space of the Lie group $G$ with $A$-coefficients in the sense of algebraic topology.
\item If $V$ is a Frechet space with continuous $G$-action, then $\ul{R\Gamma}(BG;V)$ is calculated by the usual complex of continuous cochains from the theory of continuous group cohomology.  (If $\pi_0G$ is countable, this is a complex of Frechet spaces; in general it is a complex of products of Frechet spaces.)
\end{enumerate}
\end{lemma}
\begin{proof}
For 1, this follows by the usual bar resolution and the above (functorial) comparison with sheaf cohomology, which (functorially) compares to usual cohomology.  For 2, by the same bar resolution it suffices to show that the cohomology of $G^n$ with Frechet coefficients vanishes in positive degrees for $n\geq 0$.  This follows from Theorem \ref{cohchaus}, the fact that the connected components of Lie groups are $\sigma$-compact, and another Mittag-Leffler argument.\end{proof}

\begin{remark}\label{vanest}
If $K$ is a maximal compact subgroup of $G$ and $G$ is connected, then  the famous Van Est theorem shows that the cohomology in 2 is also calculated by the relative Lie algebra cohomology of $(\on{Lie}(G),\on{Lie}(K))$ acting on the (Frechet) subspace of $C^\infty$-vectors in $V$; see \cite{hochmost}, \cite{kyed}.
\end{remark}

In \cite{condensed} it was also explained how to leverage these kinds of cohomology calculations into Ext calculations between various condensed abelian groups.  Formally, these Ext's are defined to be Hom's between shifts in $\on{D}(\on{CondAb})$, the derived $\infty$-category of condensed abelian groups (which we in turn define to be the $\infty$-category of small $\on{D}(\mbb{Z})$-valued sheaves on $\on{Extr}$).  In particular, we proved the following result (\cite{condensed} Corollary 4.9).

\begin{theorem}\label{extlchaus}
Let $A$ and $B$ be locally compact Hausdorff abelian groups, viewed as condensed abelian groups.  Then:
\begin{enumerate}
\item $\on{Hom}(A,B)$ identifies with the usual group of continuous homomorphisms $A\to B$, and more generally the internal $\ul{\on{Hom}}(A,B)$ identifies with the group of continuous homomorphisms $A \to B$ equipped with the compact-open topology, viewed as a condensed abelian group.
\item $\on{Ext}^1(A,B)$ identifies with the group of extensions of $A$ by $B$ in the exact category of locally compact Hausdorff abelian groups.
\item $\ul{\on{Ext}}^i(A,B)=0$ for all $i>1$, and if $B=\mbb{R}/\mbb{Z}$ or $B=\mbb{R}$ then also $\ul{\on{Ext}}^1(A,B)=0$.
\end{enumerate}
\end{theorem}

\begin{remark}\label{openquotient}
Implicit in 2 is the claim that the functor $\on{LCA} \to \on{CondAb}$ from locally compact Hausdorff abelian groups to condensed abelian groups is exact.  The preservation of kernels is obvious, and the only other claim to check is that a quotient map in $\on{LCA}$ goes to a surjective map of condensed abelian groups (or just condensed sets).  But in fact this is true for any open surjective continuous map of locally compact Hausdorff spaces $f:X\twoheadrightarrow Y$, since such a map is ``compact-covering'': for every compact subset $K\subset Y$ there is a compact subset $L\subset X$ with $f(L)=K$.

In particular, we also have the following non-abelian version: if $G$ is a locally compact Hausdorff group and $H\subset G$ is a closed subgroup, then the usual topological quotient map $G \to G/H$ of locally compact Hausdorff spaces is surjective as a map of condensed sets.  We conclude that the topological quotient $G/H$ matches the quotient in condensed sets.
\end{remark}

\begin{remark}
Similarly, the $\on{Ext}^1$ comparison above implies the following: if $A$ and $B$ are locally compact Hausdorff abelian groups and $X$ is a condensed abelian group which is an extension of $B$ by $A$, then $X$ is locally compact Hausdorff.

Again, the nonabelian generalization holds: if we have an extension of condensed groups
$$1 \to N \to G \to G/N \to 1$$
such that $N$ and $G/N$ are locally compact Hausdorff, then so is $G$.  This follows from an even-more-nonabelian result: suppose $X\to Y$ is a map of condensed sets with $Y$ locally compact Hausdorff, such that for all $S\in\on{Extr}$ with a map $S\to Y$ the pullback $S\times_YX$ is locally compact Hausdorff.  Then $X$ is locally compact Hausdorff.

To prove the above result, working locally on $Y$ in the topological sense we can assume that $Y$ is compact Hausdorff.  Choose a surjection $S_0\twoheadrightarrow Y$ with $S_0\in\on{Extr}$, so that $Y$ is the quotient of $S_0$ by the closed equivalence relation $S_0\times_Y S_0$.  Then setting $T_0=S_0\times_YX$, we get that $X$ is the quotient of $T_0$ by $T_0\times_X T_0$.  Now by assumption $T_0$ is locally compact Hausdorff; moreover by pullback we see that $T_0\times_XT_0\subset T_0\times T_0$ is represented by a closed subset and even gives a \emph{proper} equivalence relation on the locally compact Hausdorff space $T_0$.  It follows that the topological quotient is also locally compact Hausdorff and the quotient map is proper; in particular, it is also surjective as a map of condensed sets and hence a fortiori the topological quotient equals the condensed quotient, proving the claim.
\end{remark}

A corollary of the above theorem is the following ``derived Pontryagin duality''.

\begin{corollary}\label{dpd}
Let $\on{D}(\on{LCA})$ denote the full subcategory of $\on{D}(\on{CondAb})$ spanned by those $M\in \on{D}(\on{CondAb})$ such that $H_iM$ is isomorphic to the cokernel of an injective map of locally compact Hausdorff abelian groups, for all $i\in \mbb{Z}$.  View $\on{LCA}\subset \on{D}(\on{LCA})$ sitting in degree $0$.

Then:
\begin{enumerate}
\item $\on{D}(\on{LCA})$ is closed under shifts, cones, and retracts.
\item the functor
$$(-)^\vee:=\ul{R\on{Hom}}(-;\mbb{R}/\mbb{Z}): \on{D}(\on{CondAb})^{op}\to \on{D}(\on{CondAb})$$
sends $\on{D}(\on{LCA})^{op}$ inside $\on{D}(\on{LCA})$;
\item for $M\in \on{D}(\on{LCA})$ we have
$$M \overset{\sim}{\rightarrow} (M^\vee)^\vee;$$
\item If $M\in\on{LCA}$ then $M^\vee\in\on{LCA}$, and this identifies with the usual Pontryagin dual.
\end{enumerate}
\end{corollary}
\begin{proof}
For part 1, it suffices to see that the full subactegory of $\on{CondAb}$ spanned by the quotients of $\on{LCA}$ groups is an abelian subcategory of $\on{CondAb}$ closed under extensions.  But this is formal from the claim about $\on{Ext}^i$ for $i=1,2$ in Theorem REF, plus the fact that $\on{LCA}\to \on{CondAb}$ is exact, fully faithful, and closed under kernels.

Let $M\in \on{D}(\on{LCA})$. For parts 2 and 3, we need to show that $M^\vee\in \on{D}(\on{LCA})$ and that $M\overset{\sim}{\rightarrow} (M^\vee)^\vee$. First assume $M$ is concentrated in degree $0$ and corresponds to a locally compact abelian group.  Then by Theorem \ref{extlchaus}, $M^\vee$ identifies with the Pontryagin dual of $M$ in degree $0$.  Thus the claims hold by usual Pontryagin duality.  By devissage, claims 2 and 3 therefore also hold for $M$ concentrated in only finitely  many degrees.  Then because of the t-exactness on $\on{LCA}$ there is no issue in passing to the limit/colimit along the Postnikov tower/cotower to deduce the claim for arbitrary $M$.
\end{proof}

Importantly, this corollary automatically passes to a more general situation.  Suppose $G$ is a condensed group.  We can consider the equivariant derived category
$$\on{D}_G(\on{CondAb})$$
of derived condensed abelian groups with $G$-action.  Just as $\on{D}(\on{CondAb})$ can be viewed as the $\infty$-category of small $\on{D}(\mbb{Z})$-valued sheaves on $\on{Extr}$, this $\on{D}_G(\on{CondAb})$ can be viewed as the $\infty$-category of small $\on{D}(\mbb{Z})$-valued sheaves on the slice category $\on{Extr}_{/BG}$ (inside $\on{CondAn}_{/BG}$, where $BG$ is the usual ``delooping'' of $G$ in $\on{CondAn}$).  The base-point of $BG$ induces a conservative forgetful functor
$$\on{D}_G(\on{CondAb}) \to \on{D}(\on{CondAb}),$$
and it follows formally from slice topos nonsense that this functor commutes with internal Hom.  In other words, if $M,N \in \on{D}_G(\on{CondAb})$, then the internal Hom $\ul{R\on{Hom}}(M,N)$
in $\on{D}(\on{CondAb})$ acquires a canonical $G$-action and serves as the internal Hom in $\on{D}_G(\on{CondAb})$ as well.  We deduce the following:

\begin{corollary}\label{equidpd}
Let $G$ be a condensed group.  Consider $\mbb{R}/\mbb{Z}$ as a condensed $G$-module with trivial action, and the resulting functor
$$(-)^\vee:= \ul{R\on{Hom}}(-;\mbb{R}/\mbb{Z}): \on{D}_G(\on{CondAb})^{op}\to \on{D}_G(\on{CondAb}).$$
Let $\on{D}_G(\on{LCA})$ denote the full subcategory of $\on{D}_G(\on{CondAb})$ spanned by those $M$ such that the underling condensed group $H_iM$ is a quotient of objects in $\on{LCA}$ for all $i\in \mbb{Z}$.

\begin{enumerate}
\item For all $M\in \on{D}_G(\on{LCA})$, we have $M\overset{\sim}{\rightarrow} (M^\vee)^\vee$ and $M^\vee\in \on{D}_G(\on{LCA})$;
\item For $M,N\in\on{D}_G(\on{LCA})$ we have $\on{RHom}_G(M,N) \simeq \on{RHom}_G(N^\vee,M^\vee)$ induced by $(-)^\vee$.
\end{enumerate}
\end{corollary}
\begin{proof}
By the fact that the forgetful functor $\on{D}_G(\on{CondAb})\to\on{D}(\on{CondAb})$ is conservative and preserves internal Hom, part 1 follows directly from Corollary \ref{dpd}.  Since by part 1, $(-)^\vee$ is an (anti)-equivalence of categories from $\on{D}_G(\on{LCA})$ to itself, part 2 follows immediately.
\end{proof}

Besides this result, another useful technique for calculating in $\on{D}_G(\on{CondAb})$ is the following. Suppose $M$ is a condensed abelian group with $G$-action.  Then there is a $G$-equivariant natural injection
$$M \hookrightarrow \ul{C}(G;M)$$
where we use $\ul{C}$ to denote the internal Hom in condensed sets, defined by $m\mapsto (g\mapsto g\cdot m)$.  Note that $\ul{C}(G;M)$ is the co-free $G$-module on the condensed abelian group $M$ (viewed without $G$-action).  If $\ul{H}^i(G;M)=0$ for all $i>0$ then it is also co-free in the derived sense, and hence 
$$\on{RHom}_G(N,\ul{C}(G;M)) = \on{RHom}(N,M)$$
for any condensed $G$-module $N$.  For example, this kind of argument gives the following.

\begin{lemma}\label{extagainstrvs}
Suppose $G$ is a compact Hausdorff group.  Then for any locally compact Hausdorff $G$-module $M$ and any finite-dimensional real vector space $V$ with continuous $G$-action, we have
$$\on{Ext}_G^i(M,V)=0$$
and 
$$\on{Ext}_G^i(V,M)=0$$
for all $i>0$.

Here we use $\on{Ext}_G^i$ to denote Hom's between shifts in $\on{D}_G(\on{CondAb})$.
\end{lemma}
\begin{proof}
By derived Pontryagin duality (\ref{equidpd}), it suffices to prove the first claim.  Integration against a Haar measure shows that
$$V\hookrightarrow \ul{C}(G;V)$$
is a \emph{split} injection of $G$-modules.  Thus we reduce to proving Ext vanishing against $\ul{C}(G;V)$.  The higher cohomology $\ul{H}^i(G;V)$ vanishes by Theorem \ref{cohchaus}, so this reduces to the Ext vanishing from $M$ to $V$ in $\on{D}(\on{CondAb})$, which is given by Theorem \ref{extlchaus}.
\end{proof}

\begin{lemma}\label{compareexts}
Suppose $G$ is a profinite group, and $M,N$ are discrete abelian groups with $G$-action.  The natural map
$$\on{Ext}^i_{\on{Mod}_G}(M,N) \to \on{Ext}^i_G(M,N)$$
is an isomorphism, where on the source we have the Ext groups defined using the abelian category of discrete abelian groups with $G$-action, and in the target we have the group of maps $M \to N[i]$ in $\on{D}_G(\on{CondAb})$.
\end{lemma}

\begin{proof}
The claim is clear for $i=0$.  Choosing an injective resolution of $N$, it thus suffices to show that $\on{Ext}_G^i(M,N)=0$ for $i>0$ when $N$ is injective.  By injectivity, the map
$$N \hookrightarrow \ul{C}(G;N)$$
is split injective (note that the target is discrete).  Thus it suffices to show that
$$\on{Ext}^i_G(M,\ul{C}(G;N))=0.$$
By Theorem \ref{cohchaus} there is no higher cohomology on $G$ with $N$-coefficients, so it suffices to show $\on{Ext}^i(M,N)=0$ for $i>0$ where these Ext are calculated in $\on{D}(\on{CondAb})$.  A special case of Theorem \ref{extlchaus} shows that those are the same as in usual abelian groups; thus it suffices to show that $N$ is an injective abelian group.  But this is clear using $N=
\cup_H N^H$, the union being over open subgroups of $G$: each $H$-fixed point $N^H$ is obviously injective as $N$ is injective as a $G$-module, but on the other hand the class of injective abelian groups is closed under filtered colimits (e.g.\ because it matches the class of divisible abelian groups).
\end{proof}

Another corollary of derived Pontryagin duality is the following fact about vanishing of higher derived limits.

\begin{corollary}\label{rlim}
Suppose $(A_i)_{i\in I}$ is a filtered inverse system of compact Hausdorff abelian groups, viewed as sitting in degree $0$ inside of $\on{D}(\on{CondAb})$.  Then
$$H_n \varprojlim A_i=0$$
for $n\neq 0$.
\end{corollary}
\begin{proof}
This follows by applying $(-)^\vee$ to the obvious fact that a filtered colimit of (discrete, in this case) objects in degree $0$ still lives in degree $0$.
\end{proof}

We also have the following nonabelian analog, which must be proved by different means.

\begin{lemma}\label{nonabrlim}
The full subcategory of $\on{CondAn}$ spanned by the condensed anima of the form $BG$ for some compact Hausdorff group $G$ is closed under filtered inverse limits.
\end{lemma}
\begin{proof}
Suppose $(X_i)_{i\in I}$ is our filtered inverse system of condensed anima.  By \cite{htt} 5.3.1.18 it is no loss of generality to assume the indexing ($\infty$-)category $I$ is a poset.

Let us first show that there exists a map from $\ast$ to the inverse limit, or equivalently, there exist compatible basepoints in all the terms in our limit diagram.  To start with, fix arbitrarily a basepoint $x_i \in X_i$ for all $i\in I$.  For any $j\geq i$ in $I$, let $f_{ji}:X_j \to X_i$ denote the transition map in the inverse system.  Then the condensed set of homotopies $\gamma: f_{ji}(x_j)\simeq x_i$, denote it $T_{ji}$, is a torsor for $\pi_1(X_i,x_i)$, hence lies in $\on{CHaus}\subset\on{CondAn}$.  Moreover there are natural continuous maps
$$\alpha_{kji}: T_{kj} \times T_{ji} \to T_{ki}$$
for $k\geq j\geq i$ induced by the given homotopy $f_{ji}\circ f_{kj} \sim f_{ki}$.

For $a\geq b\geq c$ in $I$, consider the subset
$$T_{abc} \subset \prod_{j\geq i \in I} T_{ji}$$ consisting of those $(\gamma_{ji})$ which are compatible under $\alpha_{abc}$.  (This is just a condition on three of the coordinates of the product.)  Then this is a closed subspace of a compact Hausdorff space.  Moreover, for any finite collection of such triples $a,b,c$, there is an $i_0\in I$ bigger than all of the $a,b,c$ which occur; then for any $i\leq i_0$ we can choose a $\gamma_i: f_{i_0i}(x_{i_0})\sim x_i$, and then for $i\geq i'$ both $\leq i_0$ we can uniquely choose a $\gamma_{ii'}$ compatible with $\gamma_i$ and $\gamma_{i'}$, and then such $\gamma_{ii'}$ will automatically be compatible under all $\alpha_{abc}$ with $a,b,c\leq i_0$.  In this way we see that the intersection of any finite number of these $T_{abc}$ is nonempty. Thus by the finite intersection property for closed subspaces of compact Hausdorff spaces, we conclude that $\cap T_{abc}\neq \emptyset$.  But a point in $\cap T_{abc}$ in particular gives a point in $\varprojlim X_i$, by construction.

Via these compatible basepoints, it follows that our diagram promotes to diagram of compact Hausdorff groups.  Thus it suffices to see that if $(G_i)_{i\in I}$ is an inverse system of compact Hausdorff groups with inverse limit $G$, then
$$BG\overset{\sim}{\rightarrow} \varprojlim_i BG_i,$$
or in other words a compatible family of $G_i$-torsors is the same as a $G$-torsor.  Given such a compatible family of $G_i$-torsors (over a profinite set) we can pass to the inverse limit, and as a filtered inverse limit of nonempty compact Hausdorff spaces is nonempty, we get a $G$-torsor, and this produces the required inverse to the above map.\end{proof}

One last thing worth mentioning about condensed anima is that unlike in many toposes, the terminal object $\ast$ actually does behave like a point.  More precisely, the functor
$$X\mapsto X(\ast)$$
from $\on{CondAn}$ to $\on{An}$ preserves all limits and all colimits, and (hence) all Postnikov towers.  In particular, if $X$ is a condensed anima with $\tau_{\leq 0}X=\ast$, then there is only one homotopy class of maps $\ast \to X$, and in particular the homotopy groups of $X$ are well-defined (as condensed groups) up to isomorphism in the usual way.  Also, again for $X$ satisfying $\tau_{\leq 0}X=\ast$, if we fix such a basepoint $x:\ast \to X$, yielding the condensed group $G=\pi_1(X,x)$, then we get a natural isomorphism
$$\tau_{\leq 1}X \simeq BG.$$
It follows that ``coefficient systems on $X$ in degree $0$'', that is, $0$-truncated abelian group objects in $\on{CondAn}_{/X}$, are identified with $0$-truncated abelian group objects in $\on{CondAn}_{/BG}$, which in turn are the same as condensed abelian groups with $G$-action.

\section{``Almost compact'' condensed anima}

Weil groups are not usually compact, but they are ``almost compact'' in the following technical sense.

\begin{definition}\label{acomp}
Say that a locally compact Hausdorff group $G$ is \emph{almost compact} if it admits a compact normal subgroup $K$ such that either:
\begin{enumerate}
\item $G=K$;
\item $G/K\simeq (\mbb{Z},+)$; or
\item $G/K\simeq (\mbb{R},+)$.
\end{enumerate}
For brevity, by an ``almost compact group'' we will mean an almost compact locally compact Hausdorff group.
\end{definition}

\begin{lemma}\label{maxcompact}
Let $G$ be an almost compact group.  Then $K$ as in Definition \ref{acomp} is unique and functorial (i.e.\ a continuous homomorphism $G\to G'$ of almost compact groups restricts to $K\to K'$).  Indeed $K$ is the largest compact subgroup of $G$.
\end{lemma}
\begin{proof}
By a ``largest'' compact subgroup, we mean a compact subgroup which contains all other compact subgroups.  Then this follows from the fact that every compact subgroup of $\mbb{Z}$ or $\mbb{R}$ is trivial.
\end{proof}

\begin{lemma}\label{acompstructure}
Let $G$ be an almost compact group, with largest compact subgroup $K$.

\begin{enumerate}
\item There exists an action of $G/K$ on $K$ such that $G$ is isomorphic to the semidirect product $G/K \ltimes K$.
\item If $G/K\simeq \mbb{R}$, then there exists an isomorphism of $G$ with the direct product $\mbb{R} \times K$.
\end{enumerate}
\end{lemma}
\begin{proof}
By \cite{hofmann}, continuous homomorphisms from $\mbb{R}$ lift along quotient maps (by closed subgroups) in the category of locally compact Hausdorff groups.  The same is obviously true for $(\mbb{Z},+)$, this being the free group on one generator.  This means that $G \to G/N$ admits a section, which as usual identifies $G$ with the semidirect product of the resulting conjugation action of $G/N$ on $N$.

For the second claim, start from an isomorphism $G\simeq G/K\ltimes K$ as in 1, and view the action of $G/K$ on $K$ as a continuous homomorphism $G/K \to \on{Aut}(K)$ where the target (the group of continuous automorpishms of $K$) has the compact-open topology.  There is a continuous homomorphism $c:K\to \on{Aut}(K)$ recording the inner automorphisms; let $\on{Inn}(K)$ denote the (closed, as $K$ is compact) image of this map, so that alternatively we have
$$\on{Inn}(K) = K/Z(K)$$
where $Z(K)=\on{ker}(c)$ is the center of $K$ (a closed subgroup).  Then $\on{Inn}(K)$ is a closed normal subgroup of $\on{Aut}(K)$, and (this being the main point) a theorem of Iwasawa (\cite{iwasawa} Theorem 1) says that the quotient $\on{Out}(K) = \on{Aut}(K)/\on{Inn}(K)$ is totally disconnected.

Since $G/K\simeq \mbb{R}$ is connected, it follows that our homomorphism $G/K\to \on{Aut}(K)$ lands in $\on{Inn}(K)$.  Since again $\mbb{R}$ is projective in locally compact Hausdorff groups by \cite{hofmann}, it follows that we can lift $G/K \to \on{Aut}(K)$ along $c$ to a homomorphism $G/K \to K$.

This translates into the desired statement.  Let us explain this in an abstract way, without formulas, by passing to condensed anima.  An action of a condensed group $H$ on another condensed group $H'$ induces an action of $H$ on the condensed anima $BH'$, and on general grounds the quotient $(BH')/H$ identifies with $B( H\ltimes H')$.  The fact that the action lifts along $c:H'\to \on{Aut}(H')$ implies that the induced action on $BH'$ is homotopic to the constant action, hence the quotient $(BH')/H$ identifies with the quotient of the constant action, which is $BH' \times BH$.  Thus we get an isomorphism $H\ltimes H' \simeq H'\times H$.
\end{proof}

The following permanence property of almost compact groups is useful.

\begin{lemma}\label{acomppermanence}
Let $G$ be an almost compact group.  Then any closed subgroup $H\subset G$ is almost compact, and for every closed normal subgroup $N\subset G$, the quotient $G/N$ is almost compact.
\end{lemma}
\begin{proof}
Let $K$ be the largest compact subgroup of $G$.  Since $\pi: G \to G/K$ is an open surjective map of locally compact Hausdorff spaces, it is compact-covering, i.e.\ for any compact subset $C \subset G/K$ there is a compact subset $C'\subset G$ such that $\pi(C') = C$.  We deduce that $\pi^{-1}C = C'\cdot K$ is also compact, so that $\pi$ is proper.  It follows that $\pi$ is closed.  Thus $\pi(H)\subset G/K$ is closed.  As a closed subgroup of either $0$, $(\mbb{Z},+)$, or $\mbb{R}$, we get that $\pi(H)$ is also isomorphic to either $0$, $(\mbb{Z},+)$ or $\mbb{R}$.  Moreover, as $\pi$ is proper, so is $\pi|_H: H \to \pi(H)$, and therefore $\pi|_H$ is a quotient map.  Thus $H\cap K \subset H$ is a compact normal subgroup for which $H/H\cap K = \pi(H)$ is isomorphic to either $0$, $\mbb{Z}$, or $\mbb{R}$, and hence $H$ is almost compact, as claimed.

Now suppose $N$ is a closed normal subgroup.  The image of $K$ in $G/N$ is a compact normal subgroup, and the resulting quotient $G/NK$ is also the quotient of $G/K$ by a closed normal subgroup, hence is isomorphic to either $0$, $\mbb{Z}$, $\mbb{R}$, $\mbb{R}/\mbb{Z}$, or $\mbb{Z}/n\mbb{Z}$ for some positive integer $n$.  In the first 3 cases we're done, and in the last two cases $G/N$ is an extension of a compact group by a compact group, hence compact, so we're also done.
\end{proof}

Our refinements of Weil groups will be condensed anima with fundamental group the Weil group, and with compact Hausdorff higher homotopy.  We now give some preliminaries on the cohomology of such condensed anima, starting from the case of compact fundamental group.  From here on, when discussing cohomology we don't distinguish notationally between the condensed group $\ul{H}^i(X;A)$ and the underlying abstract group $H^i(X;A)$.  Rather, we will always just write $H^i(X;A)$, and make clear with words whether we mean to consider the condensed structure or not.

\begin{lemma}\label{cohchausx}
Let $X$ be a condensed anima equipped with an isomorphism $\tau_{\leq 1}X \simeq BG$ for some compact Hausdorff group $G$ and such that $\pi_iX$ is compact Hausdorff for all $i\geq 2$. Then:
\begin{enumerate}
\item For all $i\geq 0$, the functor $H^i(X;-)$, from condensed abelian groups with $G$-action to condensed abelian groups, commutes with filtered colimits.
\item For every Frechet space $V$ over $\mbb{R}$ with continuous $G$-action, $H^i(X;V)=0$ for all $i>0$ and $H^0(X;V)=V^G$ (a closed subspace of $V$, hence also Frechet).
\item For every discrete abelian group $D$ with continuous $G$-action, $H^i(X;D)$ is discrete for all $i\geq 0$.
\end{enumerate}
\end{lemma}
\begin{proof}
First consider the case $X=BG$.  The bar resolution shows that the cohomology of $X$ is the (derived) inverse limit of the cohomologies of the $G^n$, but in any range of degrees it's just a finite inverse limit.  To prove 1, since filtered colimits are exact it therefore suffices to prove it for the $G^n$.  But these   $G^n$ are compact Hausdorff spaces and hence can be resolved by profinite sets, for which the commutation is clear on general grounds (the Grothendieck topology defining condensed anima is finitary).  Claim 2 (again for $X=BG$) follows from a Haar measure argument exactly as in the proof of \ref{extagainstrvs}.  For 3, again the bar resolution reduces us to showing the analogous claim for $H^i(S;D)$ when $S$ is compact Hausdorff --- or just profinite, by a further resolution.  (We freely use that discrete abelian groups form a Serre subcategory of condensed abelian groups.) But by \ref{cohchaus} $H^i(S;D)=0$ for $i>0$ when $S$ is profinite, and when $i=0$ it's the internal Hom from $S$ to $D$, which is discrete.

Next consider the case $X=K(A,n)$ for $n\geq 1$, where $A$ is a compact Hausdorff abelian group.  The case $n=1$ has already been handled, so we can proceed inductively.  Again we have a bar resolution of $K(A,n+1)= BK(A,n)$ which shows fairly directly that 1,2, and 3 for $K(A,n)$ imply 1,2,  and 3 for $K(A,n+1)$.

Now look at the general case.  By looking degree-wise we can assume $X$ is truncated, and then we proceed by induction on $n$ for which $X$ is $n$-truncated.  The case $n=1$ is already handled, so suppose $n>1$, $X$ is $n$-truncated, and we know the claim for $\tau_{\leq n-1}X$.  Consider the fiber sequence
$$K(\pi_nX,n) \to X \to \tau_{\leq n-1}X.$$
This implies
$$R\Gamma(X;A) = R\Gamma(\tau_{\leq n-1}X;R\Gamma(K(\pi_nX,n),A))$$
for any condensed abelian group $A$ with $G$-action.  Thus 1 for $X$ follows from 1 for $K(\pi_nX,n)$ and 1 for $\tau_{\leq n-1}X$, and similarly for 2 and 3.\end{proof}

Now we extend this slightly, requiring only that $G$ be \emph{almost} compact.

\begin{proposition}\label{cohachaus}
Suppose $X$ is a condensed anima equipped with $\tau_{\leq 1}X\simeq BG$ for some almost compact group $G$ and such that $\pi_iX$ is compact Hausdorff for $i>1$.  Let $K\subset G$ denote the largest compact subgroup (Lemma \ref{maxcompact}).  Then:
\begin{enumerate}
\item For all $i\geq 0$, the functor $H^i(X;-)$, from condensed abelian groups with $G$-action to condensed abelian groups, commutes with filtered colimits.
\item Let $V$ be a Frechet space over $\mbb{R}$ with continuous $G$-action.
\begin{enumerate}
\item If $G$ is compact we have $R\Gamma(X;V) \simeq V^G[0]$.
\item If $G/K\simeq \mbb{Z}$, we have
$$R\Gamma(X;V) \simeq [V^K \overset{1-\gamma}{\longrightarrow} V^K],$$
where $\gamma$ is any chosen generator of $G/K$ ($\simeq \mbb{Z}$).
\item If $G/K\simeq \mbb{R}$, we have
$$R\Gamma(X;V) \simeq [(V^K)_\infty \overset{\xi}{\longrightarrow}(V^K)_\infty]$$
where $(V^K)_\infty$ is the Frechet space of smooth vectors in the Frechet $G/K$-representation $V^K$ and $\xi$ is any chosen generator of $\on{Lie}(G/K)\simeq \mbb{R}$ (acting on $(V^K)_\infty$ by differentiation).
\end{enumerate}

\item If $D$ is a discrete abelian group with $G$-action, then $H^i(X;D)$ is discrete for all $i\geq 0$.
\end{enumerate}

\end{proposition}
\begin{proof}
If $G$ is compact this follows from Lemma \ref{cohchausx}.  Otherwise, $G/K$ is either $\mbb{R}$ or $\mbb{Z}$.  Consider the fiber sequence
$$F \to X \to B(G/K).$$
Then Lemma \ref{cohchausx} applies to $F$.  To deduce 1 for $X$ from 1 for $F$ it suffices to show that cohomology on $B\mbb{Z}$ and $B\mbb{R}$ commutes with filtered colimits.  The $\mbb{Z}$-case is clear because
$$R\Gamma(B\mbb{Z},M) = [M \overset{1-\gamma}{\longrightarrow} M].$$
For the $\mbb{R}$-case, consider the fiber sequence
$$B\mbb{Z} \to B\mbb{R} \to B\mbb{R}/\mbb{Z}.$$
This reduces us to the $\mbb{Z}$-case and a case covered by Lemma \ref{cohchausx}, finishing the proof of 1.

For 2 we proceed similarly.  It suffices to show that if $V$ is a Frechet space with $G/K$-action, then $R\Gamma(B(G/K);V) = [V \overset{1-\gamma}{\longrightarrow} V]$
in the $\mbb{Z}$-case, and $R\Gamma(B(G/K);V) = [V_\infty\overset{\xi}{\longrightarrow} V_\infty]$ in the $\mbb{R}$-case.  The $\mbb{Z}$-case is clear since it holds for any coefficient system, and for the $\mbb{R}$-case it follows from Van Est's comparison with Lie algebra cohomology, see Remark \ref{vanest}.

Finally, for 3, again by the compact case it suffices to show that $H^i(B(G/K);D)$ is discrete for all $i\geq 0$ and all discrete $D$.  In the $\mbb{Z}$-case this is obvious because
$$R\Gamma(B(G/K);D) = [D \overset{1-\gamma}{\longrightarrow}D].$$
In the $\mbb{R}$-case we claim more precisely that $H^i(B(G/N);D)=0$ for $i>0$ and $H^0(B(G/K);D)=D$.  Because $G/K\simeq \mbb{R}$ is contractible, this follows from Lemma \ref{cohliegp}.
\end{proof}

Starting from Proposition \ref{cohachaus} we can also analyze the cohomology of many other coefficient systems.  Because of derived Pontryagin duality, the coefficient system $\mbb{R}/\mbb{Z}$ is of particular importance.  To analyze it in terms of coefficient systems to which the above applies, consider the short exact sequence
$$0 \to \mbb{Z} \to \mbb{R} \to \mbb{R}/\mbb{Z} \to 0$$
with trivial $\pi_1X$-action.  Obviously, $\mbb{Z}$ is discrete and $\mbb{R}$ is Frechet.  Then Proposition \ref{cohachaus} together with the long exact sequence in cohomology implies the following:

\begin{corollary}\label{cohrmodz}
Let $X$ be a condensed anima such that $\tau_{\leq 1}X=BG$ for some  almost compact $G$ and $\pi_iX$ is compact Hausdorff for all $i>1$.  Then we have:

\begin{enumerate}
\item $H^0(X;\mbb{R}/\mbb{Z})=\mbb{R}/\mbb{Z}$;
\item $H^1(X;\mbb{R}/\mbb{Z})= \on{Hom}_{cont}(G,\mbb{R}/\mbb{Z})$ (with the compact-open topology; this is a locally compact Hausdorff abelian group);
\item for $i\geq 2$, the condensed abelian group $H^i(X;\mbb{R}/\mbb{Z})$ identifies with $H^{i+1}(X;\mbb{Z})$ and is discrete.
\end{enumerate}
\end{corollary}
\begin{proof}
Claims 1 and 2 are immediate.  For Claim  3, by Proposition \ref{cohachaus} part 2 applied to $V=\mbb{R}$ with trivial $G$-action, we get in particular that $H^i(X;\mbb{R})=0$ for all $i\geq 2$.  Thus 3 follows from the long exact sequence associated to $\mbb{Z} \to \mbb{R} \to \mbb{R}/\mbb{Z}$.
\end{proof}

We can also describe the cohomology of $X$ with coefficients in any locally compact Hausdorff group $A$ (with trivial $\pi_1X$-action) in terms of the above answer for $A=\mbb{R}/\mbb{Z}$.  This is rather formal.

\begin{lemma}\label{cohlchauscoeff}
Let $X$ be a condensed anima, let $A$ be a locally compact Hausdorff abelian group, and let $A^\vee$ denote its Pontryagin dual.  The natural map
$$R\Gamma(X;A)\overset{\sim}{\longrightarrow} \ul{R\on{Hom}}(A^\vee, R\Gamma(X;\mbb{R}/\mbb{Z}))$$
in $\on{D}(\on{CondAb})$ is an isomorphism.
\end{lemma}
\begin{proof}
For formal reasons of adjunction we have
$$\ul{R\on{Hom}}(A^\vee,R\Gamma(X;\mbb{R}/\mbb{Z}))\simeq R\Gamma(X;\ul{R\on{Hom}}(A^\vee;\mbb{R}/\mbb{Z})).$$
The natural map in question arises from $A \to \ul{R\on{Hom}}(A^\vee;\mbb{R}/\mbb{Z})$, and it is an isomorphism as a special case of derived Pontryagin duality, \ref{dpd}.
\end{proof}

We will also want to study a generalization of this where the locally compact Hausdorff group $A$ has nontrivial $G$-action, where again we assume $X$ comes with $\tau_{\leq 1}X\simeq BG$.  The usual trick from algebraic topology of passing to the universal cover yields, in this context, the following:

$$R\Gamma(X;A) \simeq R\on{Hom}_G(A^\vee; R\Gamma(\widetilde{X};\mbb{R}/\mbb{Z})),$$
where $\widetilde{X}=X\times_{BG}\ast$.  Indeed, we have $X=\widetilde{X}/G$, and then above isomorphism follows from descent. Thus it suffices to understand the $\mbb{R}/\mbb{Z}$-cohomology of $\widetilde{X}$ as a $G$-object.

Our main case of interest is that where $A$ is compact.  Then there is a slight simplification, because the $G$-action on $A$ necessarily factors through the quotient of $G$ by its identity component $G^\circ$.  This means we don't have to go all the way to the universal cover; we can get away with the following.

\begin{definition}\label{weirdunivcov}
For a condensed anima $X$ equipped with a map $X\to BG$ for some locally compact Hausdorff group $G$, define
$$X^\circ = X\times_{BG}BG^\circ,$$
where $G^\circ$ denotes the connected component of the identity in $G$.
\end{definition}

We recall that $G^\circ$ is a closed normal subgroup of $G$, and $G/G^\circ$ is a totally disconnected locally compact Hausdorff group.  Note that there is a natural action of $G/G^\circ$ on $BG^\circ$ (as an unpointed condensed anima), and
$$(BG^\circ)/(G/G^\circ) = BG.$$
This follows by descent, since $BG^\circ = \on{fib}(BG\to B(G/G^\circ))$.  For the same reason, with notation as above we have a natural $G/G^\circ$ action on $X^\circ$ and
$$X^\circ/(G/G^\circ)=X.$$
In this way, the data of the map of condensed anima $X \to BG$ is equivalent to the data of the map of $G/G^\circ$-equivariant condensed anima $X^\circ \to BG^\circ$.

\begin{lemma}\label{cohchauscoeff}
Let $X$ be a condensed anima equipped with a map $f:X\to BG$ for some locally compact Hausdorff group $G$.  Suppose given a compact abelian group $A$ with $G$-action, which we can also view as an abelian group object over $BG$.  Then:
\begin{enumerate}
\item The $G$-action on $A$ factors through to a $G/G^\circ$-action, thereby inducing a $G/G^\circ$-action on the (discrete) Pontryagin dual $A^\vee$ as well;
\item We have $R\Gamma(X;f^\ast A) \overset{\sim}{\rightarrow} \on{RHom}_{G/G^\circ}(A^\vee,R\Gamma(X^\circ;\mbb{R}/\mbb{Z})).$
\end{enumerate}
\end{lemma}
\begin{proof}
Part 1 follows from the fact that the automorphism group of a compact Hausdorff abelian group (with the compact-open topology) is totally disconnected.  Part 2 then follows by descent from Lemma \ref{cohlchauscoeff}.
\end{proof}

Using this, we have the following interpretation of the usual procedure of building anima step-by-step via their Postnikov tower.

\begin{lemma}\label{postchaus}
Let $G$ be a locally compact Hausdorff group.  For $n\geq 1$, let $\mathcal{X}_n(BG)$ denote the $\infty$-groupoid of $n$-truncated condensed anima $X$ equipped with an isomorphism $\tau_{\leq 1}X\simeq BG$ such that $\pi_iX$ is compact Hausdorff for all $i\geq 2$.  Note that $\tau_{\leq n}$ induces a natural map
$$\mathcal{X}_{n+1}(BG) \to \mathcal{X}_n(BG).$$

For $x\in \mathcal{X}_n(BG)$ corresponding to $X\to BG$, the $\infty$-groupoid of lifts of $x$ to $\mathcal{X}_{n+1}(BG)$ identifies with the $\infty$-groupoid of pairs $(A,\kappa)$ where $A$ is a compact Hausdorff abelian group with continuous $G/G^\circ$-action and
$$\kappa: X^\circ \to K(A,n+2)$$
is a $G/G^\circ$-equivariant map of condensed anima, recalling that $X^\circ := X\times_{BG} BG^\circ$.  More precisely, this identification is induced by sending a pair $(A,\kappa)$ as above to the condensed anima $Y = \on{fib}(\kappa)/(G/G^\circ)$, equipped with its tautological map to $X^\circ/(G/G^\circ) = X$, which induces the required $\tau_{\leq n}Y \simeq X$.
\end{lemma}
\begin{proof}
We described how to go from $(A,\kappa)$ to a lift $Y\to BG$ of $x$ to $\mc{X}_{n+1}(BG)$.  To go the other direction, we use the Postnikov tower of $Y\to BG$ in condensed anima over $BG$ (see \cite{htt} 6.5 and 7.2), or equivalently by descent, the Postnikov tower of $Y^\circ \to BG^\circ$ in $G/G^\circ$-equivariant condensed anima over $BG^\circ$.

The $G/G^\circ$-equivariant map $Y^\circ\to \tau_{\leq n}Y^\circ =X^\circ$ is an $n+1$-gerbe.  Thus there is a unique $G/G^\circ$-equivariant abelian group object $\mc{A}$ in (relatively $0$-truncated) condensed anima over $X^\circ$ such that $Y^\circ \to X^\circ$ promotes to an $n+1$-gerbe banded by $\mc{A}$.  A relatively $0$-truncated anima over $X^\circ$ uniquely descends to $\tau_{\leq 1}X^\circ = BG^\circ$, so we can equivalently treat $\mc{A}$ as a $G/G^\circ$-equivariant abelian group object over $BG^\circ$.  Next we claim it is uniquely pulled back from a $G/G^\circ$-equivariant abelian group object over $\ast$, i.e.\ a condensed abelian group $A$ with $G/G^\circ$-action, and that this condensed abelian group $A$ is compact Hausdorff.

For this, note that on pullback along any point $x$ of $X$ (unique up to homotopy, as $\pi_0X=\ast$ and $\ast$ is projective in condensed anima), $\mathcal{A}$ gives $\pi_{n+1}(X,x)$, which is compact Hausdorff by assumption.  A $G/G^\circ$-equivariant abelian group object over $\tau_{\leq 1}X = BG^\circ$ is the same as an abelian group object over $(BG^\circ)/(G/G^\circ)\simeq BG$.  Thus $\mathcal{A}$ corresponds to a compact Hausdorff group $A$ with continuous $G$-action.  Since the automorphism group of a compact abelian Hausdorff group is totally disconnected, the action must be trivial on $G^\circ$.  Therefore, the $G$-action descends (uniquely) to a $G/G^\circ$-action, and this proves the claim (that $\mathcal{A}$ is uniquely pulled back from a $G/G^\circ$-equivariant compact Hausdorff abelian group $A$).

We deduce that $Y^\circ \to X^\circ$ has the structure of an $(n+1)$-gerbe banded by this $A$.  Thus by \cite{htt} Theorem 7.2.2.6 it is uniquely classified by a map
$$X^\circ \to K(A,n+2).$$
The universal $(n+1)$-gerbe banded by $A$ corresponds to the canonical base-point of $K(A,n+2)$, so this translates to the desired statement.
\end{proof}

To finish this section, we discuss phenomena related to filtered inverse limits of these ``almost compact'' condensed anima, again starting with the compact case.  For motivation, we recall the following well-known fact (for which we provide a proof since we couldn't find one in the literature):

\begin{lemma}
The functor from the pro-category of compact Lie groups to the category of compact Hausdorff groups, given by sending a pro-system to its inverse limit, is an equivalence.
\end{lemma}
\begin{proof}
Let $G$ be a compact Hausdorff group.  By the Peter-Weyl theorem, $L^2(G)$ is the Hilbert space direct sum of finite-dimensional unitary representations of $G$.  It follows that $G$ maps by an injective continuous homomorphism to a product of unitary groups.  Thus $G$ is isomorphic to a closed subgroup of a product $\prod_{i\in I}H_i$ of compact Lie groups.  For a finite subset $I_0\subset I$, let $G_{I_0}$ be the image of $G$ under the projection $\prod_{i\in I}H_i \to \prod_{i\in I_0}H_i$.  Then $G$ is the filtered inverse limit of the $G_{I_0}$, and each $G_{I_0}$ is a compact Lie group because a closed subgroup of a Lie group is a Lie group.

Thus the functor from the pro-category of compact Lie groups to compact Hausdorff groups is essentially surjective.  To show fully faithfulness, we need to see that if $(G_i)_{i\in I}$ is a pro-system of compact Lie groups (indexed by a filtered poset $I$) with inverse limit $G$, then for any compact Lie group $H$ we have
$$\varinjlim_{i\in I}\on{Hom}_{cont}(G_i,H)\overset{\sim}{\rightarrow}\on{Hom}_{cont}(G,H).$$
In fact, we will see that this holds for any Lie group $H$.  Let $\mc{N}$ denote the poset of closed normal subgroups $N$ of $G$ for which $G/N$ is a Lie group.  Note that $\mc{N}$ is closed under pairwaise intersection, as $G/(N\cap N')\hookrightarrow G/N\times G/N'$ and a closed subgroup of a Lie group is a Lie group.  It is also nonempty, because for any $i\in I$ we have $\on{ker}(G\to G_i)\in\mc{N}$.   Thus $N\in\mc{N} \mapsto G/N$ gives a pro-system of compact Lie groups.  Then we claim that $(G_i)_{i\in I}$ is pro-isomorphic to $(G/N)_{N\in\mc{N}}$.  Admitting this for now, to complete the proof it then suffices to show that
$$\varinjlim_{N\in\mc{N}}\on{Hom}_{cont}(G/N,H)\overset{\sim}{\rightarrow} \on{Hom}_{cont}(G,H).$$
Now the map is clearly injective.  To show surjectivity, let $f:G\to H$ be a continuous homomorphism.  Then $N=\on{ker}(f)$ is a closed normal subgroup such that $G/N\simeq\on{im}(f)$ is a closed subgroup of $H$, hence a Lie group, and moreover $f$ factors through $G/N$, as required.

To prove the claim that $(G_i)_{i\in I}$ is pro-isomorphic to $(G/N)_{N\in\mc{N}}$, it suffices to combine the following two facts:
\begin{enumerate}
\item The map $I\to \mc{N}$ sending $i\in I$ to $N_i:=\on{ker}(G\to G_i)$ is cofinal (so that the pro-system $(G/N_i)_{i\in I}$ maps isomorphically to the pro-system $(G/N)_{N\in\mc{N}}$);
\item The map of inverse systems $(G/N_i)_{i\in I} \to (G_i)_{i\in I}$ induced by the evident termwise inclusions is a pro-isomorphism.
\end{enumerate}
For 1, we need to see that for all $N\in\mc{N}$, there is an $i\in I$ with $N_i\subset N$.  More generally we claim this holds for any filtered system $(N_i)_{i\in I}$ of closed normal subgroups of a compact Hausdorff group $G$ such that $\cap_i N_i=\{1\}$.  First suppose $G$ is a Lie group and $N=\{1\}$.  We need to show $N_i=\{1\}$ for some $i$.  Note that $(\on{Lie}(N_i))_{i\in I}$ gives a filtered system of subspaces of the finite-dimensional $\on{Lie}(G)$, hence it must stabilize, say at $i_0\in I$, so $\on{Lie}(N_i)=\on{Lie}(N_{i_0})$ for all $i\geq i_0$.  As $N_i$ is a closed subgroup of $N_{i_0}$ and these are Lie groups, this implies that $(N_i)^\circ=(N_{i_0})^\circ$ for all $i\geq i_0$.  Then again the $\pi_0(N_i)$, being finite, must also stabilize after some point, and we deduce that the $N_i$ themselves stabilize after some point.  Since $\{1\}=\cap_i N_i$, we find that $N_i=\{1\}$ after the stabilization point, as desired.  Returning to the general case, let $N'_i$ be the image of $N_i$ in $G/N$.  A filtered inverse limit of surjective maps of compact Hausdorff spaces is surjective, so $\cap_i N_i\twoheadrightarrow \cap_i N'_i$, hence $\cap_i N'_i=\{1\}$ and we conclude from the Lie case that $N'_i=\{1\}$ for some $i$, which says $N_i\subset N$ for some $i$, as required for 1.

For 2, we need to show that for all $i_0\in I$, there is an $i\geq i_0$ such that $G_i\to G_{i_0}$ lands inside $G/N_{i_0} = \on{im}(G\to G_{i_0})$.  Let $H_i=\on{im}(G_i\to G_{i_0})$.  As in the proof of 1 above we see that the $H_i$ stabilize and that $\on{im}(G\to G_{i_0})=\cap_i H_i$.  The claim follows.
\end{proof}

We would like to extend this to ``compact anima'' with trivial $\pi_0$.

\begin{lemma}\label{invlimchaus}
Let $(X_i)_{i\in I}$ be a filtered inverse system of condensed anima, each of which satisfies $\pi_0(X_i)=\ast$ and $\pi_n(X_i)\in\on{CHaus}$ for all $n\geq 1$.  Then:
\begin{enumerate}
\item $X=\varprojlim_i X_i$ satisfies the same conditions ($\pi_0X=\ast$, $\pi_nX\in\on{CHaus}$).
\item We have $\pi_n X\overset{\sim}{\rightarrow} \varprojlim_i \pi_n(X_i)$ for all $n\geq 0$ (all basepoints of $X$ are homotopic by 1).
\item Suppose $Y$ is a condensed anima such that $\pi_n(Y,y)$ is a Lie group for all $n\geq 1$ and all $y:\ast\to Y$, and such that $Y$ is $d$-truncated for some $d<\infty$.  Then 
$$\varinjlim_i \on{Map}(X_i,Y)\overset{\sim}{\rightarrow}\on{Map}(X,Y).$$
\end{enumerate}

\end{lemma}
\begin{proof}
Parts 1 and 2 follow from induction on the Postnikov tower using Lemmas \ref{rlim} and \ref{nonabrlim}.  For 3, by 1 we have that maps to a $0$-truncated $Y$ are the same for $X$ and any $X_i$, and are the same as maps from a point.  Thus pulling back along any $y:\ast \to Y$ we can assume $\pi_0Y=\ast$.  Again by induction on the Postnikov tower of $Y$ we reduce to the case $Y=BH$ for a Lie group $H$ plus the claim that cohomology with abelian Lie group coefficients on $X$ is the colimit of cohomology on the $X_i$ (for a coefficient system coming from some $X_{i_0}$).  The case $Y=BH$ follows from the classical remarks about inverse limits of compact Lie groups made above.  For the claim about cohomology, suppose $N$ is an abelian Lie group with $\pi_1X_{i_0}$-action for some $i_0\in I$; we want that
$$\varinjlim_{i>i_0}H^n(X_i;N)\overset{\sim}{\rightarrow} H^n(X;N)$$
for all $n\in\mbb{Z}$.
We have the natural short exact sequence
$$0 \to N^\circ \to N \to N/N^\circ \to 0$$
which means we can reduce to two separate cases: the discrete case and the connected case.

First suppose $N$ is discrete.  Arguing using bar resolutions and the Postnikov tower of the $X_i$ as in the proof of Lemma \ref{cohchausx} we reduce to showing that $N$-cohomology on compact Hausdorff spaces sends inverse limits of such spaces to filtered colimits; but this follows by comparison with sheaf cohomology, \ref{cohchaus}.

Now suppose $N$ is connected.  Then we have a natural short exact sequence
$$0 \to \Lambda_N \to \on{Lie}(N) \to N \to 0$$
with $\Lambda_N$ discrete, so we reduce to the case where $N$ is a finite dimensional $\mbb{R}$-vector space.  But in that case by Lemma \ref{cohchausx} we reduce to the claim that
$$N^{\pi_1X} = \cup_i N^{\pi_1X_i},$$
which again follows from the facts recalled above about inverse limits of compact Hausdorff groups.
\end{proof}

\begin{theorem}\label{proclie}
Let $\on{CLieAn}_{<\infty}$ denote the full subcategory of $\on{CondAn}$ spanned by the $Y$ such that:
\begin{enumerate}
\item $\pi_0Y=\ast$;
\item $\pi_nY$ is a compact Lie group for all $n\geq 1$;
\item $Y$ is $d$-truncated for some $d<\infty$.
\end{enumerate}
Then the functor
$$\on{Pro}(\on{CLieAn}_{<\infty})\to \on{CondAn}$$
sending a pro-object in $\on{CLieAn}_{<\infty}$ to its inverse limit in $\on{CondAn}$ is fully faithful, and the essential image consists of those condensed anima $X$ such that $\pi_0X=\ast$ and $\pi_nX$ is compact Hausdorff for all $n\geq 1$.
\end{theorem}
\begin{proof}
Fully faithfulness is immediate from Lemma \ref{invlimchaus}, as is the fact that every object $X$ in the essential image satisfies the stated properties.  Thus we need to show that any $X$ satisfying the stated properties lies in the essential image.  The essential image is closed under filtered inverse limits by the fully faithfulness, so we can reduce to the case of $d$-truncated $X$, which in turn we will handle by induction.  When $d=1$ this follows from the fact (recalled above) that every compact Hausdorff group is a filtered inverse limit of compact Lie groups.  Assume now $X$ is $d+1$-truncated and that $\tau_{\leq d}X = \varprojlim_i C_i$ with $C_i\in \on{CLieAn}_{<\infty}$.  By Lemma \ref{invlimchaus} we can assume each $C_i$ is $d$-truncated, and if $G$ denotes the fundamental group of $X$ and $G_i$ for $i\in I$ denotes the fundamental group of $X_i$, then
$$G = \varprojlim_i G_i.$$
Now, let
$$(\tau_{\leq d}X)^\circ \to K(A,n+2)$$
denote the k-invariant for $X$, from Lemma \ref{postchaus}, so this is a $G/G^\circ$-equivariant map with $A$ a compact Hausdorff abelian group with $G/G^\circ$-action.  The Pontryagin dual $A^\vee$ is a discrete abelian group with $G/G^\circ$-action, hence it can be written as a filtered union
$$A^\vee = \cup_j D_j$$
where each $D_j$ is a finitely generated abelian group with continuous $G/G^\circ$-action that factors through to a $G_i/G_i^\circ$-action for some $i$ (depending on $j$).  Applying Pontryagin duality, we get that
$$A = \varprojlim_j A_j$$
where $A_j$ is a compact abelian Lie group with $G_i/G_i^\circ$-action.  For each $j$, the composition
$$(\tau_{\leq d}X)^\circ \to K(A,n+2) \to K(A_j,n+2),$$
which is equivalently an $(n+2)$-cocycle on $X$ for the twisted coefficient system determined by $A_j$, factors through a $G_{i'}/G_{i'}^\circ$-equivariant map $C_{i'}^\circ \to K(A_j,n+2)$ for some $i'>i$ by Lemma \ref{invlimchaus}.  Varying over $j$, we deduce that $X$ is an inverse limit over extensions of $C_{i'}$ by $K(A_j,n+1)$, finishing the proof.
\end{proof}

One advantage of reducing to this sort of ``finite type'' situation is that we get a clearer handle on discrete coefficient systems.

\begin{lemma}\label{lietoan}
Suppose $C\in \on{CondAn}$ is such that $\pi_0C=\ast$ and $\pi_nC$ is a Lie group for all $n\geq 1$.  Then there is an initial anima $|C|$ equipped with a map
$$C\to |C|.$$

It is uniquely characterized by the following properties:
\begin{enumerate}
\item $\pi_0|C|=\ast$;
\item $\pi_1C \to \pi_1|C|$ identifies $\pi_1|C|$ with the component group of the Lie group $\pi_1C$;
\item for all discrete abelian groups $M$ with $\pi_1|C|$-action, the map
$$H^i(|C|;M) \to H^i(C;M)$$
is an isomorphism for all $i\geq 0$.
\end{enumerate}

We moreover have that discrete coefficient systems over $C$ (that is, discrete anima with an action of the condensed $\Omega C$) are the same as over $|C|$.

\end{lemma}
\begin{proof}
It is clear from a Postnikov tower argument that any $C\to |C|$ satisfying 1-3 will be the initial anima with a map from $C$.  Thus it suffices to show existence of $C\to |C|$ satisfying 1,2,3 and the ``moreover".

Let us show that any such $C$ lifts (uniquely!) to the $\infty$-topos of sheaves on real topological manifolds with respect to the open cover topology.  This easily implies the desired claim using that manifolds are locally contractible, c.f.\ \cite{clausen} Section 5.  For the lifting claim, working inductively on the Postnikov tower, it suffices to show that cohomology with abelian Lie group coefficient systems is the same for a sheaf on manifolds as for its underlying condensed anima.  But this is follows from Lemma \ref{cohchaus} (and its more obvious analog for sheaves on manifolds), using the usual devissage to reduce to discrete coefficients and real vector space coefficients.\end{proof}

 \begin{remark}\label{prodiscrete}
 Combining, we deduce that if $X$ is a condensed anima with $\pi_0X=\ast$ and $\pi_nX$ compact Hausdorff for $n\geq 1$, then we get a pro-system of anima $(A_i)_{i\in I}$ with a map
$$X \to \varprojlim_i A_i$$
which is pro-universal with respect to maps to discrete anima; moreover we can choose the $A_i$ such that each $\pi_1 A_i$ is a quotient of $\pi_1X$ by an open normal subgroup (in particular, $\pi_1A_i$ is finite) and each higher $\pi_nA_i$ is a finitely generated abelian group.  The same conclusion applies if $\pi_1X$ is only almost compact with quotient $\mbb{R}$, as the extra $\mbb{R}$ does not affect maps to anima, c.f.\ the proof of Lemma \ref{cohachaus}.

Namely, we can take $A_i = |C_i|$ where $X=\varprojlim_i C_i$ with $C_i\in\on{CLieAn}_{<\infty}$ as in Theorem \ref{proclie}.
\end{remark}

\begin{remark}\label{homotopyofanima}
If $X,Y$ both satisfy the hypotheses of Lemma \ref{lietoan} ($\pi_0=\ast$ and $\pi_n$ is Lie for $n\geq 1$), then so does $X\times Y$, and
$$|X\times Y|\overset{\sim}{\rightarrow} |X|\times |Y|.$$
This follows easily from the proof of Lemma \ref{lietoan}, compare with the discussion in \cite{clausen} Section 5.

It follows from this that if $G$ is a group object satisfying the hypothesis in Lemma \ref{lietoan} acting on $X$, then
$$|X/G|=|X|/|G|.$$
Fixing a basepoint of $X$ and letting $G=\pi_1X$, we deduce
$$|X| = |\tau_{\geq 2}X|/|G|,$$
so that applying $|\cdot|$ to the first stage of the Postnikov tower of $X$ still yields a fiber sequence
$$|\tau_{\geq 2}X| \to |X| \to |BG|.$$
This reduces the study of $|X|$ to that of $|BG|$ and that of $X$ with $\pi_1X=\{1\}$.

This $|BG|$ is just the underlying homotopy type of the classical classifying space of the Lie group $G$ from algebraic topology.  As for $|X|$ when $X$ satisfies $\pi_1X=\{1\}$, it can be understood in terms of the Postnikov tower, based on the following: if $A$ is an abelian Lie group, then there is a canonical exact sequence
$$ 0 \to \Lambda_A \to \on{Lie}(A) \to A \to A/A^\circ \to 0$$
with $\Lambda_A$ a free $\mbb{Z}$-module of finite rank, $A/A^\circ$ a discrete abelian group, and $\on{Lie}(A)$ a finite-dimensional $\mbb{R}$-vector space.  Moreover
$$0\to \on{Lie}(A)/\Lambda_A \to A\to A/A^\circ \to  0$$
is (noncanonically) split.  Comparing cohomology using Serre spectral sequences, it follows that we have a canonical fiber sequence
$$K(\Lambda_A,n+1) \to |K(A,n)| \to K(A/A^\circ,n)$$
and that noncanonically $|K(A,n)|\simeq K(\Lambda_A,n+1) \times K(A/A^\circ,n)$.  Returning to general $X$ such that $\tau_{\leq 1}X=\ast$ and $\pi_iX$ is a Lie group for all $i\geq 2$, then using the Postnikov tower (which still yields fiber sequences on $|\cdot|$ by a similar argument to the above) we deduce natural short exact sequences
$$0 \to \pi_{n+1}X/(\pi_{n+1}X)^\circ \to \pi_{n+1}|X| \to \Lambda_{\pi_nX} \to 0$$
for $n\geq 1$ which are noncanonically split.
\end{remark}

Continuing the theme, we can furthermore pass to the profinite completion.

\begin{proposition}\label{profcomp}
Let $X\in\on{CondAn}$ with $\pi_0X=\ast$ (for simplicity).  Then $X$ admits a profinite completion, namely a map
$$X\to \wh{X}$$
in $\on{CondAn}$ such that:
\begin{enumerate}
\item $\pi_0\wh{X}=\ast$, and $\pi_n\wh{X}$ is a profinite group for all $n\geq 1$;
\item If $Y\in\on{CondAn}$ is such that $\pi_0Y=\ast$ and $\pi_nY$ is a profinite group for all $n\geq 1$, then
$$\on{Map}(\wh{X},Y)\overset{\sim}{\rightarrow}\on{Map}(X,Y).$$
\end{enumerate}
\end{proposition}
\begin{proof}
The $\infty$-category of all connected $\pi$-finite anima $A$ equipped with a map $X\to A$ is essentially small and co-filtered (indeed, if we take all $\pi$-finite anima $A$ with a map $X\to A$, then pullbacks imply this is co-filtered; but the connected $A$ form a cofinal sub-collection, as we can replace $f:X\to A$ by $f:X\to\on{im}(f)$).  Then we define $\wh{X}$ to be the resulting filtered inverse limit
$$\wh{X} := \varprojlim_{X\to A} A.$$
By Lemma \ref{invlimchaus}, if $B$ is any connected $\pi$-finite anima we have
$$\varinjlim_{X\to A} \on{Map}(A,B)\overset{\sim}{\rightarrow} \on{Map}(\wh{X},B).$$
On the other hand the composition
$$\varinjlim_{X\to A}\on{Map}(A,B) \to \on{Map}(X,B)$$
is also an isomorphism, as an instance of the following general categorical fact: if $\mc{C}$ is an $\infty$-category and $F:\mc{C}\to\on{An}$ is a functor, then for all $B\in\mc{C}$ the natural map
$$\varinjlim_{A\in\mc{C},\xi\in F(A)}\on{Map}(A,B)\to F(B)$$
is an isomorphism.  (Writing $F$ as a colimit of functors of the form $\on{Map}(A_0,-)$, we reduce to that case, when the indexing category for the colimit has terminal object given by $A=A_0$ and $\xi=\on{id}$.)  Combining, we deduce that
$$\on{Map}(\wh{X},B)\overset{\sim}{\rightarrow}\on{Map}(X,B)$$
for any connected $\pi$-finite anima $B$.  The same necessarily holds when $B$ is any inverse limit of connected $\pi$-finite anima.  But the proof of Theorem \ref{proclie} shows that any $Y$ as in 2 is an inverse limit of connected $\pi$-finite anima.  Thus we deduce that 2 holds for $X\to\wh{X}$.  By Lemma \ref{invlimchaus} we also have that $\wh{X}$ satisfies 1, so we conclude the proof.
\end{proof}

In particular, if $X$ satisfies $\pi_0X=\ast$ and $\pi_nX\in\on{CHaus}$ for $n\geq 1$, then if we write
$$X = \varprojlim_i C_i$$
with $C_i\in \on{CLieAn}_{<\infty}$ (Theorem \ref{proclie}), then
$$\wh{X} = \varprojlim_i \wh{|C_i|}$$
where $\wh{|C_i|}$ is the profinite completion of the usual anima $|C_i|$.  Since $\pi_1|C_i|$ is finite and $\pi_n|C_i|$ is finitely generated for all $i\geq 2$, we have that $\pi_1\wh{|C_i|}=\pi_1|C_i|$ and
$$\pi_n\wh{|C_i|} = (\pi_n|C_i|)\otimes\wh{\mbb{Z}}$$
for $n\geq 2$.

We can view an object $C\in\on{CondAn}$ with $\pi_0C=\ast$ and $\pi_nC$ a compact Lie group for all $n\geq 1$ as some kind of differential refinement of its homotopy type $|C|$ and ask, what kinds of more standard geometric structures on a connected topological space will give rise to such a refined homotopy type?

More precisely, we recall (c.f.\ \cite{scholzeanalytic} Lemma 11.9) that if $X$ is, e.g.\ a topological space homotopy equivalent to a CW-complex, then its associated condensed set $\ul{X}$ admits an initial map to an anima $|X|$, and this anima is the usual homotopy type of $X$.  One can ask what structure on a connected $X$ will naturally produce a $C$ as above and a map
$$\ul{X} \to C $$
which induces an isomorphism on $|\cdot|$.

\begin{example}\label{riemannian}
Let us consider the simplest nontrivial example, that of an $n$-sphere $S^n$ with $n\geq 3$ odd, viewed as a condensed set. The homotopy groups of the anima $
|S^n|$ are $\ast$ in degrees $<n$, $\mbb{Z}$ in degree $n$, and finite in degrees $>n$.  Comparing with Remark \ref{homotopyofanima}, we deduce that our desired $C$ should have the same homotopy groups as $|S^n|$ except in degrees $n-1$ and $n$, where we should have $\pi_{n-1}C = \mbb{R}/\mbb{Z}$ and $\pi_nC=0$.

I claim we can build a canonical such $C$ plus a map $S^n \to C$ inducing an isomorphism on $|\cdot|$ given any choice of smooth structure and Riemannian metric on $S^n$ (or really, just the associated volume form).  Let $\on{or}$ denote the orientation local system on $S^n$ (it is trivial, but never mind that).  As is well-known, the Riemannian metric gives rise to a canonical volume form, which is a global section of the sheaf
$$\Omega^n\otimes\on{or}$$
on $S^n$.  This sheaf canonically maps to the shifted twisted de Rham complex $\Omega^\bullet\otimes\on{or}[n]$, and even to
$$\on{Fib}(\Omega^\bullet\otimes\on{or}[n] \to \Omega^0\otimes\on{or}[n]).$$
Let use denote by $\mbb{R}^\delta$ (resp.\ $\mbb{R}$) the sheaf of locally constant (resp.\ continuous) real-valued functions on $S^n$; these sheaves are represented in condensed sets by the indicated condensed abelian groups.  Then up to canonical quasi-isomorphism, $\Omega^\bullet\otimes\on{or}[n]$ identifies with $\mbb{R}^\delta\otimes\on{or}[n]$ (by the Poincar\'{e} Lemma), and on the other hand $\Omega^0\otimes\on{or}[n]$ maps to $\mbb{R}\otimes\on{or}[n]$ because smooth functions yield continuous functions.     Combining with our global section of $\Omega^n\otimes \on{or}$, we deduce a global section of $\mbb{R}^\delta\otimes \on{or}[n]$ together with a trivialization of the induced global section of $\mbb{R}\otimes\on{or}[n]$.

Thus, passing to condensed anima, we have a section of the family of Eilenberg-Maclane spaces $K(\mbb{R}^\delta,n)$ over $S^n$ twisted by the orientation sheaf, together with a nullhomotopy of the induced section of $K(\mbb{R},n)$.  This gives rise to the desired data.  Indeed, if we forget the twists for simplicity, then we get
$$S^n \to K(\mbb{R}^\delta,n) \to K(\mbb{R},n)$$
with null composition; as $\mbb{R}^\delta$ is a discrete abelian group, $K(\mbb{R}^\delta,n)$ is a discrete anima, so the first map uniquely factors through $|S^n|$, and then we can set
$$C = \on{fib}(|S^n| \to K(\mbb{R},n)),$$
which receives a map from $S^n$ by the nullhomotopy of $S^n\to K(\mbb{R},n)$.

If $G$ is a compact Lie group acting freely on $S^n$ by isometries (or, just preserving the volume form), then since the above was canonical we get an induced
$$S^n/G \to C/G$$
which provides a solution to our problem for $X=S^n/G$ as well.  For example, when $n=3$ we can take
$$G = \on{Pin}^-(2) = U(1)\cup j\cdot U(1)\subset\mbb{H}^\times$$
(which is the norm-1 subgroup of $W_{\mbb{R}}$) acting by left multiplication on the unit quaterions $S^3$ with its canonical invariant metric and in this way we get a canonical ``compact homotopy type'' for $X=\mbb{R}\mbb{P}^2$.
\end{example}

\section{Moore anima}\label{mooresection}

We set ourselves the following problem: for an almost compact group $G$ (Definition \ref{acomp}), we want to modify the condensed anima $BG$ by adding higher homotopy in order to simplify the cohomology.  More precisely, we make the following definition.

\begin{definition}\label{moore}
Let $G$ be a locally compact Hausdorff group.  A \emph{Moore anima} for $G$ is a condensed anima $X$ equipped with a map $t: X \to BG$ satisfying the following properties:
\begin{enumerate}
\item The map $t$ induces an isomorphism $\tau_{\leq 1}X \simeq BG$.
\item For all $i\geq 2$, the homotopy (condensed) abelian group $\pi_i X$ is compact Hausdorff.
\item We have
$$H^i(X^\circ;\mbb{R}/\mbb{Z})=0$$
as a condensed abelian group for all $i\geq 2$, where (see Definition \ref{weirdunivcov}) $X^\circ=X\times_{BG}BG^\circ$.
\end{enumerate}
\end{definition}

\begin{remark}
Our main case of interest is when $G$ is almost compact; then the condensed abelian group $H^i(X^\circ;\mbb{R}/\mbb{Z})$ is discrete for $i\geq 2$ by Corollary \ref{cohrmodz}, so in 3 above, vanishing in the condensed sense and vanishing in the ordinary sense are equivalent.  This discreteness is also what motivates the requirement that the $\pi_iX$ be compact Hausdorff, as we will eventually see that the homotopy groups are determined as Pontryagin dual to such cohomology groups, e.g.\
$$\pi_2 X = H^3(BG^\circ;\mbb{R}/\mbb{Z})^\vee.$$
\end{remark}

\begin{example}
Suppose $G$ is totally disconnected.  Then $BG^\circ=\ast$, and it follows that $BG$ itself (or, more precisely, the identity  map $BG \to BG$) is already a Moore anima for $G$ in this sense.  In fact the $\infty$-groupoid of Moore anima for $G$ is contractible, as we will see from the general analysis.  Thus the theory in this section is trivial in the totally disconnected case.
\end{example}

The method for producing Moore anima (or more precisely, for analyzing the $\infty$-groupoid of Moore anima) will be via the familiar procedure of \emph{obstruction theory} from homotopy theory, based on the Postnikov tower.  In this we follow \cite{smith} on equivariant Moore spaces, but transported to the condensed context.

We start with a truncated variant of Definition \ref{moore}.

\begin{definition}\label{tmoore}
Let $G$ be a locally compact Hausdorff group and $n\geq 1$.  An \emph{$n$-stage Moore anima} for $G$ is a condensed anima $X$ equipped with a map $t: X \to BG$ satisfying the following properties:
\begin{enumerate}
\item $X$ is $n$-truncated: $X\overset{\sim}{\rightarrow} \tau_{\leq n}X$.
\item The map $t$ induces an isomorphism $\tau_{\leq 1}X \simeq BG$.
\item For all $i\geq 2$, the homotopy (condensed) abelian group $\pi_i X$ is compact Hausdorff.
\item Letting $X^\circ = X\times_{BG}BG^\circ$, we have $H^i(X^\circ;\mbb{R}/\mbb{Z})=0$ when $n+1\geq i>1$.
\end{enumerate}
\end{definition}

Formally, we can allow $n=\infty$: an $\infty$-stage Moore anima for $G$ is just a Moore anima for $G$ in the sense of Definition \ref{moore}.

\begin{remark}
Naively, it would seem more natural to take the weaker version of 4 where you only ask cohomology vanishing up to degree $n$, not $n+1$.  But the above form of the definition makes the obstruction theory much simpler.
\end{remark}

\begin{proposition}\label{moorevstmoore}
Let $G$ be a locally compact Hausdorff group.

\begin{enumerate}
\item Let $1 \leq n <\infty$, and let $Y\to BG$ be an $n+1$-stage Moore anima for $G$.  Then $X:=\tau_{\leq n}Y \to BG$ is an $n$-stage Moore anima for $G$.
\item Let $t:X \to BG$ be a map of condensed anima.  Then $t$ is a Moore anima for $G$ if and only if for all $1\leq n<\infty$, the induced map $\tau_{\leq n}X \to BG$ is an $n$-stage Moore anima for G.
\end{enumerate}
\end{proposition}
\begin{proof}
For 1, by construction $X$ is $n$-truncated, $\tau_{\leq 1}X \overset{\sim}{\rightarrow} BG$, and $\pi_iX$ is compact Hausdorff for all $i\geq 2$.  Note that, since $BG$ and $BG^\circ$ are 1-truncated, we have
$$\tau_{\leq n}(Y^\circ) = X^\circ.$$
Since $K(\mbb{R}/\mbb{Z},i)$ is $i$-truncated, it follows immediately from the corresponding fact for $Y^\circ$ that $H^i(X^\circ;\mbb{R}/\mbb{Z})=0$ for $n\geq i>1$.  The only nontrivial point is to show that $H^{n+1}(X^\circ;\mbb{R}/\mbb{Z})=0$.  For this, note by Lemma \ref{postchaus} that there is a fiber sequence
$$Y^\circ \to X^\circ \to K(A,n+2)$$
where $A$ is a compact Hausdorff abelian group (isomorphic to $\pi_{n+1}Y^\circ$).  Consider the associated Serre spectral sequence for cohomology with $\mbb{R}/\mbb{Z}$-coefficients.  The terms on the $E_2$-page which contribute to $H^{n+1}(X^\circ;\mbb{R}/\mbb{Z})$ are the
$$H^i(K(A,n+2);H^j(Y^\circ;\mbb{R}/\mbb{Z}))$$
for $i+j=n+1$.  Since $\tau_{\leq n+1}K(A,n+2)=\ast$, the degree-$i$ cohomology of $K(A,n+2)$ with any coefficients vanishes for $0<i\leq n+1$, so the only possibly nonzero term is $H^0(K(A,n+2);H^{n+1}(Y^\circ;\mbb{R}/\mbb{Z}))$. But this is also $0$ because $H^{n+1}(Y^\circ;\mbb{R}/\mbb{Z})=0$ by the hypothesis on $Y$.

For 2, the implication that $X\to BG$ is a Moore anima $\Rightarrow$  $\tau_{\leq n}X \to BG$ is an $n$-stage Moore anima follows exactly as in the proof of 1 above, and the converse follows from the fact that $H^i(X^\circ;\mbb{R}/\mbb{Z}) = H^i(\tau_{\leq n}X^\circ;\mbb{R}/\mbb{Z})$ for large enough $n$ (namely $n\geq i$).
\end{proof}

We can reinterpret this in light of the following definition.

\begin{definition}
Let $G$ be a locally compact Hausdorff group.  Let $\mathcal{M}(BG)$ denote the $\infty$-groupoid (anima) of Moore anima for $G$, and for $n\geq 1$, let $\mathcal{M}_n(BG)$ denote the $\infty$-groupoid of $n$-stage Moore anima for $G$.
\end{definition}

\begin{corollary}\label{moorecaninduct}
Let $G$ be a locally compact Hausdorff group.  Then we have
$$\mathcal{M}(BG) = \varprojlim_n \mathcal{M}_n(BG),$$
where the implicit maps are induced by Postnikov truncation.
\end{corollary}
\begin{proof}
The maps are well-defined by Proposition \ref{moorevstmoore} part 1 and the forward direction of Proposition \ref{moorevstmoore} part 2.  That we get an isomorphism results from the Postnikov completeness of condensed anima and the backward direction of Proposition \ref{moorevstmoore} part 2.
\end{proof}

This reveals the strategy for analyzing $\mathcal{M}(BG)$: first study $\mathcal{M}_1(BG)$, then study the homotopy fibers of each map $\mathcal{M}_{n+1}(BG)\to \mc{M}_n(BG)$.

\begin{lemma}\label{moore1}
Let $G$ be a locally compact Hausdorff group.  Then the following conditions are equivalent:
\begin{enumerate}
\item $\mathcal{M}_1(BG)\neq\emptyset$;
\item $\mathcal{M}_1(BG)\simeq \ast$;
\item $\on{id}:BG \to BG$ is a 1-stage Moore anima for $G$;
\item $H^2(BG^\circ;\mbb{R}/\mbb{Z})=0$.
\end{enumerate}
If $G$ is almost compact, these conditions are furthermore equivalent to the following: $\pi_1(K^\circ)_{tors}=0$, where $K^\circ \subset G^\circ$ denotes the maximal compact subgroup of $G^\circ$, and see Lemma \ref{compactliegrouplemma} for the meaning of $\pi_1(K^\circ)_{tors}$.
\end{lemma}
\begin{proof}
Axioms 1 and 2 in Definition \ref{tmoore} imply that if $t:X\to BG$ is a 1-stage Moore anima for $G$, then $t$ is an isomorphism.  This uniquely determines $t$ and proves that 1, 2 and 3 are equivalent.  To see that 3 is equivalent to 4, note that the only nontrivial condition is axiom 4 in Definition \ref{tmoore}, which for $t=\on{id}:BG\to BG$ exactly says 4.  Finally, suppose $G$ is almost compact.  Then $G^\circ$ is also almost compact (being a closed subgroup of $G$, Lemma \ref{acomppermanence}).  But it is also connected, hence either $K^\circ=G^\circ$ or $G^\circ = \mbb{R}\times K^\circ$ by Lemma \ref{acompstructure}, and in either case we have
$$H^2(BG^\circ;\mbb{R}/\mbb{Z})=H^3(BG^\circ;\mbb{Z}) = H^3(BK^\circ;\mbb{Z}),$$
and $K^\circ$ is connected.  Thus the conclusion follows from Lemma \ref{compactliegrouplemma}.
\end{proof}

\begin{lemma}\label{compactliegrouplemma}
Let $K$ be a connected compact Hausdorff group.  Define a profinite abelian group $\pi_1(K)_{tors}$ as follows: Let $N$ run over all closed normal subgroups of $K$ such that $K/N$ is a Lie group, and set
$$\pi_1(K)_{tors} = \varprojlim_N \pi_1(K/N)_{tors},$$
where $\pi_1(K/N)$ stands for the usual fundamental group of the compact connected Lie group $K/N$ (or equivalently the quotient of the co-weight lattice by the co-root lattice), and $(-)_{tors}$ means the torsion subgroup.

Then:
\begin{enumerate}
\item We have a canonical functorial isomorphism $$H^3(BK;\mbb{Z})\simeq (\pi_1(K)_{tors})^\vee,$$
where $(-)^\vee$ means Pontryagin duality.
\item $H^4(BK;\mbb{Z})$ is torsionfree.
\end{enumerate}
\end{lemma}
\begin{proof}

Invoking Lemma \ref{invlimchaus}, we reduce to showing that if $K$ is a connected compact Lie group, then there is a natural isomorphism $H^3(BK;\mbb{Z})\simeq (\pi_1(K)_{tors})^\vee$
and that
$H^4(BK;\mbb{Z})$
is torsionfree.

We recall that the integral homology of $BK$ in low degres is as follows:
$$H_0(BK)=\mbb{Z}, H_1(BK)=0, H_2(BK)=\pi_1K, H_3(BK)=0.$$
Indeed, up to degree 2 this follows from Hurewicz, and to get the answer in degree 3 one can use the Serre spectral sequence for
$$K/T \to BT \to BK$$
where $T\subset K$ is a maximal torus, plus the fact that the homology of $K/T$ is torsionfree and lives in even degrees (e.g.\ by the Bruhat decomposition).  From the universal coefficient theorem we deduce that $H^3(BK; \mbb{Z})= \on{Ext}(\pi_1K,\mbb{Z})=((\pi_1K)_{tors})^\vee$, and also that $H^4(BK;\mbb{Z})= \on{Hom}(H_4(BK);\mbb{Z})$, proving the claims.
\end{proof}

In particular, we deduce the following, which suffices for the application to Weil groups.  We remark that the hypothetical Langlands group from \cite{arthur} will also satisfy that $\pi_1(K^\circ)_{tors}=0$, so the same conclusion would apply to it as well.

\begin{proposition}\label{moore1ab}
Let $G$ be an almost compact group, and suppose that $G^\circ$ is abelian.  Then $\mc{M}_1(BG)=\ast$.
\end{proposition}
\begin{proof}
By the above, it suffices to show that if $K$ is a connected compact abelian group, then $\pi_1(K)_{tors}=0$.  But such a $K$ is Pontryagin dual to a torsionfree discrete group, hence $K$ is a filtered inverse limit of finite products of copies of $\mbb{R}/\mbb{Z}$, giving the claim.
\end{proof}

Having analyzed $\mc{M}_1(BG)$, we now turn to the homotopy fibers of
$$\mc{M}_{n+1}(BG) \to \mc{M}_n(BG)$$
for $n\geq 1$.  Recall the anima $\mc{X}_n(BG)$ from Lemma \ref{postchaus}: they classify the $n$-truncated condensed anima $X\to BG$ such that $\tau_{\leq 1}X\overset{\sim}{\rightarrow}BG$ and $\pi_iX$ is compact Hausdorff for $i\geq 2$.  By definition, $\mc{M}_n(BG)$ is the full subanima of $\mc{X}_n(BG)$ spanned by those $X\to BG$ such that $H^i(X^\circ;\mbb{R}/\mbb{Z})=0$ for $2\leq i\leq n+1$.

\begin{lemma}\label{mooreinduct}
Let $G$ be a locally compact Hausdorff group, and let $n\geq 1$.  Suppose given $t\in \mc{M}_n(BG)$, corresponding to an $n$-stage Moore anima $t:X\to BG$.  Let $X^\circ= X\times_{BG} BG^\circ$.  Then the anima of lifts of $t$ to $\mc{M}_{n+1}(BG)$ identifies with the anima of pairs $(A,\kappa)$ where $A$ is a compact Hausdorff abelian group with continuous $G/G^\circ$-action and $\kappa$ is a $G/G^\circ$-equiviariant map of condensed anima
$$\kappa: X^\circ \to K(A,n+2)$$
such that the following condition is satisfied: the natural homomorphism
$$A^\vee \simeq H^{n+2}(K(A,n+2);\mbb{R}/\mbb{Z}) \overset{\kappa^\ast}{\longrightarrow} H^{n+2}(X^\circ;\mbb{R}/\mbb{Z})$$
is an isomorphism (of condensed abelian groups).

More precisely, this identification sends $(A,\kappa)$ to $\on{fib}(\kappa)/(G/G^\circ)$ equipped with its natural identification $\tau_{\leq n}\on{fib}(\kappa)/(G/G^\circ) \simeq X^\circ/(G/G^\circ) = X$ over $BG$.
\end{lemma}
\begin{proof}
By Lemma \ref{postchaus}, the anima of lifts of $t$ to $\mc{X}_{n+1}(BG)$ identifies with the anima of pairs $(A,\kappa)$ without the condition on $\kappa^\ast$.  Thus it suffices to show that for such $(A,\kappa)$, the above $\kappa^\ast$ is an isomorphism if and only if $H^i(Y^\circ;\mbb{R}/\mbb{Z})=0$ for all $2\leq i\leq n+2$, where $Y^\circ=\on{fib}(\kappa)$.

For this, consider the Serre spectral sequence for cohomology with $\mbb{R}/\mbb{Z}$-coefficients associated to the fiber sequence
$$K(A,n+1) \to Y^\circ \to X^\circ.$$
With standard indexing, the $E_2$-page is made up of the condensed abelian groups
$$E_2^{i,j}=H^i(X^\circ; H^j(K(A,n+1);\mbb{R}/\mbb{Z})),$$
with the $(i,j)$-term contributing to $H^{i+j}(Y^\circ;\mbb{R}/\mbb{Z})$ at the end of the spectral sequence.

When $j=0$, we have $H^0(K(A,n+1);\mbb{R}/\mbb{Z})=\mbb{R}/\mbb{Z}$, so by assumption on $X$ we have $E_2^{i,0}=0$ for $2\leq i\leq n+1$.  On the other hand for $0<j<n+1$ we have $E_2^{i,j}=0$ for all $i$ because $\tau_{\leq n}K(A,n+1)=\ast$.  In the edge case $j=n+1$, we have
$$E_2^{0,n+1}=H^{n+1}(K(A,n+1);\mbb{R}/\mbb{Z})=A^\vee$$
and
$$E_2^{1,n+1} = H^1(X^\circ;A^\vee) = \on{Hom}(G^\circ,A^\vee)=0$$
because $G^\circ$ is connected and $A^\vee$ is discrete.  Furthermore,
$$E_2^{0,n+2} = H^{n+2}(K(A,n+1);\mbb{R}/\mbb{Z}) = \on{Ext}^1(A;\mbb{R}/\mbb{Z})=0,$$
where the comparison with Ext follows because we are in the stable range, and the Ext vanishing follows from Proposition \ref{extlchaus}.

It follows that:
\begin{enumerate}
\item for $2\leq p<n+1$, all terms contributing to $H^p(Y^\circ;\mbb{R}/\mbb{Z})$ are $0$;
\item the only possibly nonzero term contributing to $H^{n+1}(Y^\circ;\mbb{R}/\mbb{Z})$ is $E_2^{0,n+1} = A^\vee$;
\item the only possibly nonzero term contributing to $H^{n+2}(Y^\circ;\mbb{R}/\mbb{Z})$ is $E_2^{n+2,0} = H^{n+2}(X^\circ;\mbb{R}/\mbb{Z})$;
\item the only possibly nonzero differential in this range is
$$d_{n+2}: E_2^{0,n+1} \to E_2^{n+2,0}.$$
\end{enumerate}

We deduce that $H^{n+1}(Y^\circ;\mbb{R}/\mbb{Z})= \on{ker}(d_{n+2})$ and $H^{n+2}(Y^\circ;\mbb{R}/\mbb{Z})=\on{coker}(d_{n+2})$, so that $Y\in \mc{M}_{n+1}(BG)$ if and only if $d_{n+2}$ is an isomorphism.

This differential is transgressive, so we can understand it explicitly. Note that our fiber sequence $K(A,n+1) \to Y^\circ \to X^\circ$ is the base-change via $\kappa$ of the fiber sequence $K(A,n+1) \to \ast \to K(A,n+2)$.   There the analogous transgressive differential in the Serre spectral sequence is necessarily an isomorphism because $\ast$ has vanishing cohomology in positive degrees.  By functoriality of the Serre spectral sequence, we deduce that $d_{n+2}$ is an isomorphism if and only if $\kappa^\ast$ is an isomorphism, proving our claim.
\end{proof}

In particular, $A\simeq \pi_{n+1}(Y^\circ)\simeq \pi_{n+1}(Y)$, which is the ``homotopy group to add'' in going from our $n$-stage Moore anima to the $n+1$-stage, is canonically determined by the given $n$-stage Moore anima $X$: namely, as a $G/G^\circ$-equivariant compact Hausdorff group it identifies with the Pontryagin dual to $H^{n+2}(X^\circ;\mbb{R}/\mbb{Z})$, which must necessarily be discrete if we are to be able to pass to the $n+1$-stage (though this discreteness is automatic when $G$ is almost compact, Lemma \ref{cohrmodz}).

So we have no choice in the $A$, and the only choice is in $\kappa$.  More precisely, the anima of lifts of $t\in\mc{M}_{n}(BG)$ to $\mc{M}_{n+1}(BG)$ identifies with the anima of $G/G^\circ$-equivariant maps
$$\kappa: X^\circ \to K(H^{n+2}(X^\circ;\mbb{R}/\mbb{Z})^\vee;n+2)$$
such that the pullback on $\mbb{R}/\mbb{Z}$-cohomology in degree $n+2$ induces the canonical identification of $H^{n+2}(X^\circ;\mbb{R}/\mbb{Z})$ with its double Pontryagin dual.

We now analyze the choice of $\kappa$ further.  It turns out that the following object is fundamental.

\begin{definition}
Let $G$ be a locally compact Hausdorff group, and let $G^\circ$ be the connected component of the identity in $G$.  Define
$$c_G := \tau_{\geq -1} R\Gamma(BG^\circ;\mbb{R}/\mbb{Z}),$$
which we view as an object in $\on{D}_{G/G^\circ}(\on{CondAb})$, the derived $\infty$-category of condensed abelian groups with $G/G^\circ$-action.
\end{definition}

Thus $c_G$ lives in (homological) degrees $0$ and $-1$. We have
$$H_0(c_G) = \mbb{R}/\mbb{Z}$$
and
$$H_{-1}(c_G) = \on{Hom}_{cont}(G^\circ;\mbb{R}/\mbb{Z}),$$
the usual unitary character group of the locally compact Hausdorff group $G^\circ$.

Now we can give the following translation of Lemma \ref{mooreinduct}, which describes the obstruction theory for Moore anima.

\begin{theorem}\label{mooreobstruct}
Let $G$ be an almost compact group, let $n\geq 1$, and let $t \in \mc{M}_n(BG)$ be an $n$-stage Moore anima for $G$, namely $t:X\to BG$ is a condensed anima over $BG$ satisfying the axioms in Definition \ref{tmoore}.  Let $H:= H^{n+2}(X^\circ;\mbb{R}/\mbb{Z})$ be the indicated cohomology group, with its natural $G/G^\circ$-action inherited from $X^\circ = X\times_{BG} BG^\circ$.

Then there is a canonical map
$$\sigma: H \to c_G[n+3]$$
in $\on{D}_{G/G^\circ}(\on{CondAb})$ (with $H$ in degree $0$) such that the anima of lifts of $t$ to $\mc{M}_{n+1}(BG)$ identifies with the anima of null-homotopies of $\sigma$.
\end{theorem}
\begin{proof}
Consider $M:=\tau_{\geq 0} R\Gamma(X^\circ;\mbb{R}/\mbb{Z}[n+2])$ as an object in $\on{D}_{G/G^\circ}(\on{CondAb})$ concentrated in degrees $0$ through $n+2$.  By definition of an $n$-stage Moore anima, we have
$$\tau_{\geq 1} M = c_G[n+2].$$
On the other hand, by definition of $H$ we have
$$\tau_{\leq 0}M \simeq H.$$

Thus the boundary map for the cofiber sequence
$$\tau_{\geq 1}M \to M \to \tau_{\leq 0}M$$
gives a map
$$\sigma: H \to c_G[n+3]$$
of the correct form.

To show it has the indicated property, note that a null-homotopy of $\sigma$ is the same thing as a splitting of the above cofiber sequence.  But this is the same as a map $H \to M$ such that the composition with $M\to H_0M = H$ is the identity map.  If we define the compact Hausdorff $G/G^\circ$-module $A$ by
$$A := H^\vee,$$
then by Lemma \ref{cohchauscoeff} this is also the same thing as a map
$$\kappa: X^\circ \to K(A,n+2)$$
for which $\kappa^\ast$ as in Lemma \ref{mooreinduct} is the tautological identification $A^\vee = H$.  Thus the statement follows from Lemma \ref{mooreinduct}.
\end{proof}

\begin{corollary}\label{mooreconclusion}
Suppose $G$ is an almost compact group satisfying the following two assumptions:
\begin{enumerate}
\item $H^2(BG^\circ;\mbb{R}/\mbb{Z})=0$;
\item For all discrete $G/G^\circ$-modules $M$ and all $i\geq 3$ we have
$$\on{Ext}^i_{G/G^\circ}(M,c_G)=0$$
where
$$c_G = \tau_{\geq -1} R\Gamma(BG^\circ;\mbb{R}/\mbb{Z}).$$
\end{enumerate}
Then $G$ admits a Moore anima, unique up to isomorphism.
\end{corollary}
\begin{proof}
By Theorem \ref{mooreobstruct}, condition 2 implies that the homotopy fibers of the maps
$$\mc{M}_{n+1}(BG) \to \mc{M}_n(BG)$$
have $\pi_i=\ast$ for $i<n$.  When we pass to the inverse limit we deduce that the homotopy fibers of 
$$\varprojlim_n \mc{M}_n(BG) \to \mc{M}_1(BG)$$
have $\pi_0=\ast$.  But by Lemma \ref{moore1}, condition 1 implies that $\mc{M}_1(BG)=\ast$, and by Corollary \ref{moorecaninduct} the source is $\mc{M}(BG)$; the conclusion follows.
\end{proof}

\begin{example}\label{tdmoore}
Suppose $G$ is a totally disconnected almost compact group.  Then $G^\circ=\{1\}$, so
$$R\Gamma(BG^\circ;\mbb{R}/\mbb{Z})=\mbb{R}/\mbb{Z}.$$
We see inductively that every $H$ appearing in Corollary \ref{mooreobstruct} is $0$, and consequently $\mc{M}_{n+1}(BG)\to\mc{M}_n(BG)$ is an isomorphism for all $n\geq 1$.  Moreover $\mc{M}_1(BG)=\ast$.  We deduce that $\mc{M}(BG)=\ast$ and that $BG$ is the (truly unique) Moore anima for $G$.
\end{example}

\section{Background on class formations}\label{cfsection}

The notion of class formation arose as an axiomatization of the cohomological approach to class field theory.  We recall one form of the classical definition from \cite{artintate}:

\begin{definition}\label{cfprofgp}
A \emph{class formation} is a pair $(\Gamma,A)$ where:
\begin{enumerate}
\item $\Gamma$ is a profinite group;

\item $A$ is a discrete abelian group with continuous $\Gamma$-action,
\end{enumerate}
such that the following conditions on group cohomology are satisfied, for every open subgroup $H\subset \Gamma$ and every open normal subgroup $N\subset H$:
\begin{enumerate}
\item $H^1(H/N;A^N)=0$;
\item $H^2(H/N;A^N)$ is cyclic of order $[H:N]$.
\end{enumerate}
\end{definition}

The point of the notion of class formation is that if we give ourselves a bit of extra structure, namely a collection of generators $(u_{H,N})_{H,N}$ for these cyclic groups $H^2(H/N;A^N)$, then via the Tate-Nakayama lemma (\cite{corpslocaux} IX) we deduce that cup product with $u_{H,N}$ induces an isomorphism in Tate cohomology
$$\widehat{H}^i(H/N;\mbb{Z})\overset{\sim}{\rightarrow} \wh{H}^{i+2}(H/N;A^N)$$
for all $i\in\mbb{Z}$.  In particular, when $i=-2$ we get
$$(H/N)^{ab}\simeq A^H/\on{tr}(A^N).$$
If we assume suitable compatibility properties for these $u_{H,N}$ under inflation and restrction (see \cite{corpslocaux} XI), then we can pass to the limit over $N$ to get the associated \emph{Artin maps}
$$A^H \to H^{ab}$$
which have dense image in the profinite group $H^{ab}$, and satisfy various natural compatibilities as $H$ varies over open subgroups of $\Gamma$.  In this way we can think of a class formation as providing abelian groups which naturally refine the $H^{ab}$ as $H$ varies over open subgroups of $\Gamma$.

In practice, $\Gamma$ will be an absolute Galois group.  In particular, it will not actually be canonically defined as a profinite group, but only as a connected profinite groupoid, or equivalently a \emph{Galois category} in the sense of \cite{SGA1}.  To make our discussions more canonical, we therefore switch to that language.

\begin{definition}\label{cfgal}
A \emph{class formation} (reformulated) is a pair $(\mc{C},A)$ where:
\begin{enumerate}
\item $\mc{C}$ is a Galois category;
\item $A$ is a sheaf of abelian groups on $\mc{C}$,
\end{enumerate}
such that the following conditions are satisfied, for every connected object $X\in\mc{C}$ and every subgroup $G\subset \on{Aut}(X)$:
\begin{enumerate}
\item $H^1(G;A(X))=0$;
\item $H^2(G;A(X))$ is cyclic of order $\#G$.
\end{enumerate}
\end{definition}

It is clear from Grothendieck Galois theory that, if we fix a fiber functor on the Galois category $\mc{C}$, or equivalently, if we fix a profinite group $\Gamma$ and an identification of $\mc{C}$ with the category of finite continuous $\Gamma$-sets, then the notions of class formation in Definitions \ref{cfprofgp} and \ref{cfgal} are one-to-one equivalent: the data and axioms line up.

In the following section we aim to explain how to get Weil groups, or rather (to say things more canonically) Weil groupoids, from class formations, essentially retelling part of \cite{artintate}.  But for that it's not enough to just have a class formation: one needs extra data.  At a first level, one needs chosen compatible generators for the groups $H^2(G;A(X))$, exactly as is required to get Artin maps following Tate-Nakayama as above.  But actually one needs more, namely some cocycle-level version of such data, in order to canonically build the corresponding ``group extensions", not just up to isomorphism.

In fact, such cocycle-level lifts also show up in another place in the theory, at least implicitly, namely in the proof of Tate-Nakayama.  Given a finite group $G$, a $G$-module $A$, and a class $u\in H^2(G;A)$, then one way to describe the Tate-Nakayama argument is that one interprets $u\in \on{Ext}^2_G(\mbb{Z},A)$ instead; then one represents $u$ by a map 
$$\mbb{Z} \to A[2]$$
in $\on{D}(\mbb{Z}[G])$ and passes to the homotopy fiber.  For this fiber one has the hypothesis that its Tate cohomology vanishes in two consecutive degrees for all subgroups $H\subset G$, and one wants to conclude it vanishes identically.  The point in the end (after some standard devissage to reduce to $p$-groups, then to cyclic groups of order $p$) is that the Tate cohomology of cyclic groups is 2--periodic, so that vanishing in two consecutive degrees is equivalent to vanishing identically.

Now, the datum of a cocycle-level lift of the group cohomology class $u$ is the same as the datum of the above map in the derived $\infty$-category of $G$-modules, and it is also the same as the datum of the fiber object, equipped with isomorphisms $H_0\simeq\mbb{Z}$ and $H_1\simeq A$.  All of these (and more) are just different ways of talking about gerbes over $BG$ banded by $A$.  As usual in mathematics, the interpretation which is closest to linear algebra is the easiest one to work with.  In the end we are led to the following.

\begin{definition}\label{rcf}
A \emph{refined class formation} is a triple $(\mc{C},F,\alpha)$ where:
\begin{enumerate}
\item $\mc{C}$ is a Galois category;
\item $F:\mc{C}\to \on{D}(\mbb{Z})$ is a \emph{co-sheaf} such that $F(X)$ is concentrated in the range of degrees $[0,1]$ for all $X\in\mc{C}$;
\item $\alpha$ is a map $H_0F\to \mbb{Z}$ of functors $\mc{C}\to \on{Ab}$ to the constant functor with value $\mbb{Z}$, such that $\alpha(X)$ is an isomorphism for all connected $X\in\mc{C}$.
\end{enumerate}
\end{definition}

\begin{remark}\label{whatisaderivedsheaf}
We recall that if $\mc{C}$ is a Galois category and $\mc{D}$ is an $\infty$-category with finite products and fixed points for all finite groups, then a functor
$$F:\mc{C}^{op}\to\mc{D}$$
is a sheaf if and only if:
\begin{enumerate}
\item $F(\emptyset) = \ast$;
\item $F(X\sqcup Y)\overset{\sim}{\rightarrow} F(X)\times F(Y)$ for all $X,Y\in\mc{C}$;
\item $F(X)\overset{\sim}{\rightarrow} F(Y)^G$ for all Galois maps $Y\to X$ with Galois group $G$. (And here we can restrict to $X,Y$ connected if desired.)
\end{enumerate}
Indeed, for $\mc{D}=\on{An}$ see \cite{cm1} Section 4, and the general case follows by Yoneda.

By passing to opposite categories, we see that the dual description applies to cosheaves; thus $F:\mc{C}\to\on{D}(\mbb{Z})$ is a cosheaf if and only if it sends finite disjoint unions to direct sums, and $F(Y)_G\overset{\sim}{\rightarrow} F(X)$ when $Y\to X$ is $G$-Galois.  Note that here $(-)_G$ is the \emph{homotopy} orbits, which on objects in degree $0$ corresponds to the derived functors of $G$-orbits.
\end{remark}

We will show that if $(\mc{C},F,\alpha)$ is a refined class formation, then $H_1F$, which is a co-presheaf on $\mc{C}$, also carries the canonical structure of a sheaf via the formalism of \emph{transfer maps}, and that $(\mc{C},H_1F)$ is a class formation. Conversely, we will explain what it takes to promote a class formation to a refined class formation.

For this we use three fairly standard lemmas, for which our basic references will be \cite{cm1} and \cite{cm2}.  The first explicates sheaves/cosheaves on Galois categories attached to finite groups.  Every Galois category is canonically the filtered union of Galois categories each of which is equivalent to the Galois category of a finite group, so this is also useful for general Galois categories.

\begin{lemma}\label{sheafbg}
Let $G$ be a finite group, and let $\mc{C}(G)$ denote the Galois category of finite $G$-sets.  Let $\mc{D}$ be an $\infty$-category which admits all finite coproducts and all $H$-orbit objects for all subgroups $H\subset G$.  Then restriction along the inclusion $BG\subset \mc{C}(G)$ of the $G$-torsors induces an equivalence
$$\on{CoSh}(\mc{C}(G);\mc{D})\overset{\sim}{\rightarrow}\on{Fun}(BG;\mc{D})$$
with inverse given by left Kan extension.

Dually, if $\mc{D}$ admits all finite products and all $H$-fixed point objects, then restriction along $BG\subset \mc{C}(G)^{op}$ induces an equivalence
$$\on{Sh}(\mc{C};\mc{D})\overset{\sim}{\rightarrow}\on{Fun}(BG;\mc{D})$$
with inverse given by right Kan extension.
\end{lemma}
\begin{proof}
It suffices to treat the dual assertion about sheaves and right Kan extensions.  Then for $d\in \mc{D}$, consider the composition $h_d\circ F$, where $h_d = \on{Map}(d,-):\mc{C}\to\on{An}$.  By definition $F$ is a sheaf if and only if $h_d\circ F$ is a sheaf for all $d\in\mc{D}$.  On the other hand the pointwise criterion for right Kan extension shows that $F$ is the right Kan extension of its restriction to a given full subcategory if and only if the same is true for $h_d\circ F$ for all $d\in\mc{D}$.  Thus we reduce to the case $\mc{D}=\on{An}$, where the assertion is given by \cite{cm1} Example 4.4.
\end{proof}

The second result explains the mechanism of ``transfer maps'', which provide a canonical presheaf structure on any cosheaf and dually a canonical pre-cosheaf structure on any sheaf.

\begin{lemma}\label{transferlemma}
Let $\mc{C}$ be a Galois category, and let $\mc{D}$ be a semi-additive $\infty$-category which admits $G$-orbit and $G$-fixed point objects for all finite groups $G$.  Denote by $\on{Span}(\mc{C})$ the span category ($(2,1)$-category, actually) of $\mc{C}$, where the objects are the objects of $\mc{C}$, the maps from $X$ to $Y$ are the spans
$$X \leftarrow M \to Y$$
in $\mc{C}$, and composition is given by pullback.  Note that there are obvious functors
$$\mc{C} \to \on{Span}(\mc{C})$$
and
$$\mc{C}^{op} \to \on{Span}(\mc{C}),$$
sending each object to itself and sending a map $X\to Y$ to the span $X=X\to Y$ (resp.\ $Y\leftarrow X=X$). Then:

\begin{enumerate}
\item Let $\on{Fun}^{coshf}(\on{Span}(\mc{C}),\mc{D})$ denote the full subcategory of functors $F:\on{Span}(\mc{C})\to \mc{D}$ such that the restriction along $\mc{C}\to \on{Span}(\mc{C})$ is a cosheaf.  Then we have
$$\on{Fun}^{coshf}(\on{Span}(\mc{C}),\mc{D})\overset{\sim}{\rightarrow}\on{coSh}(\mc{C};\mc{D})$$
via restriction along $\mc{C} \to \on{Span}(\mc{C})$.
\item Let $\on{Fun}^{shf}(\on{Span}(\mc{C}),\mc{D})$ denote the full subcategory of functors $F:\on{Span}(\mc{C})\to \mc{D}$ such that the restriction along $\mc{C}^{op}\to \on{Span}(\mc{C})$ is a sheaf.  Then we have
$$\on{Fun}^{shf}(\on{Span}(\mc{C}),\mc{D})\overset{\sim}{\rightarrow}\on{Sh}(\mc{C};\mc{D})$$
via restriction along $\mc{C}^{op} \to \on{Span}(\mc{C})$.
\item Let $F \in \on{Fun}(\on{Span}(\mc{C}),\mc{D})$.  Any two of the following three properties imply the third:
\begin{enumerate}
\item $F\in \on{Fun}^{coshf}(\on{Span}(\mc{C}),\mc{D})$;
\item $F\in \on{Fun}^{shf}(\on{Span}(\mc{C}),\mc{D})$;
\item for all connected $X\in\mc{C}$ and all subgroups $G\subset \on{Aut}(X)$, the natural ``transfer'' map $F(X)_G \to F(X)^G$, associated to the induced $G$-action on $F(X)$, is an isomorphism.
\end{enumerate}
\end{enumerate}
\end{lemma}
\begin{proof}
Claims 1 and 2 are dual, so it suffices to treat 2.  Let $\mc{N}(\mc{C})$ denote the poset of isomorphism classes of connected Galois objects in $\mc{C}$, and for $N\in\mc{N}(\mc{C})$ consider the full subcategory $\mc{C}_N\subset \mc{C}$ consisting of objects which can be covered by a finite disjoint union of objects in the isomorphism class $N$.  The poset structure is such that
$$N\leq N' \Leftrightarrow \mc{C}_{N'}\subseteq \mc{C}_N.$$
Note that if $\mc{C}=\mc{C}(\Gamma)$ for a profinite group $\Gamma$, then $\mc{N}(\mc{C})$ is the poset of open normal subgroups of $\Gamma$ with poset structure given by inclusion, and $\mc{C}_N$ identifies with the Galois category of finite $\Gamma/N$-sets.

We have a filtered colimit
$$\mc{C} = \varinjlim_N \mc{C}_N$$
of finitary sites (c.f. \cite{cm1} 3.1) and $\mc{C}_N$ is equivalent to the Galois category of finite $G$-sets for some finite group $G$.  It follows readily that the claim in 2 reduces to the case when $\mc{C}=\mc{C}(G)$ for some finite group $G$.  Then by Lemma \ref{sheafbg} it suffices to show that restriction along $BG \subset \mc{C}^{op} \to \on{Span}(\mc{C})$ induces an equivalence
$$\on{Fun}^{shf}(\on{Span}(\mc{C}),\mc{D})\overset{\sim}{\rightarrow}\on{Fun}(BG;\mc{D}).$$
Note by \cite{cm1} Prop.\ 4.2 that the sheaf property for a presheaf $\mc{F}:\mc{C}^{op}\to\mc{D}$ breaks into two parts:
\begin{enumerate}
\item $\mc{F}$ preserves finite products;
\item $\mc{F}(X/H)\overset{\sim}{\rightarrow} \mc{F}(X)^H$ for all connected $X\in\mc{C}$ and all subgroups $H\subset\on{Aut}(X)$.
\end{enumerate}
Assuming 1 ($\mc{F}$ preserves finite products), we can (following \cite{cm1} Prop.\ 4.11) rephrase 2 as the condition that $\mc{F}(X)\overset{\sim}{\rightarrow} \mc{F}(X\times G)^G$ for all $X\in\mc{C}(G)$.  It follows that our desired claim $\on{Fun}^{shf}(\on{Span}(\mc{C}),\mc{D})\overset{\sim}{\rightarrow}\on{Fun}(BG,\mc{D})$, in the case $\mc{D}=\on{CMon}(\on{An})$, is given by \cite{cm2} Prop.\ 2.9.  To finish the proof of 2 (and 1, by duality) let us explain in the following paragraph why the general case reduces to that one.

Note that since $\mc{D}$ is semi-additive, by \cite{gepner} we have that for every $d\in\mc{D}$, the functor $h_d = \on{Map}(d,-):\mc{D} \to \on{An}$ canonically refines to a functor with values in $\on{CMon}(\on{An})$.  Thus by the case $\mc{D}=\on{CMon}(\on{An})$ we can uniquely extend $h_d\circ F: BG \to \on{CMon}(\on{An})$ to an object $\widetilde{h_d\circ F}\in \on{Fun}^{shf}(\on{Span}(\mc{C}),\on{CMon}(\on{An}))$.   Then it suffices to show that for every $X\in\on{Span}(\mc{C})$ the additive functor $\mc{D}^{op}\to \on{An}$ sending
$$d\mapsto \widetilde{h_d\circ F}(X)$$
is representable.  But we can write $X$ as a finite disjoint union of orbits and then this follows from the assumption that $\mc{D}$ has all fixed points for finite group actions.

Thus we have 1 and 2.  Now consider 3.  If $F$ satisfies two of the three conditions (a),(b),(c), then in particular $F$ is either a sheaf or a cosheaf, and therefore $F$ is additive.  Thus we can restrict our attention to additive functors $F\in\on{Fun}^{\oplus}(\on{Span}(\mc{C}),\mc{D})$.  Now, for connected $X\in\mc{C}$ and $G\subset \on{Aut}(X)$ a subgroup, consider the composition
$$F(X)_G \to F(X/G) \to F(X)^G,$$
where the first map uses the co-presheaf structure and the second map uses the presheaf structure.  Then $F$ is a cosheaf if and only if the first map is an isomorphism for all $(X,G)$, and $F$ is a sheaf if and only if the second map is an isomorphism for all $(X,G)$.  Thus by the 2 of 3 property for isomorphisms it suffices to show that the composition of the above two maps identifies with the transfer as in (c), for any $F\in \on{Fun}^{\oplus}(\on{Span}(\mc{C}),\mc{D})$.  In other words, we need to see that the $G\times G$-equivariant map
$$F(X) \to F(X/G) \to F(X)$$
identifies with $\sum_{g\in G} g$.  But this follows by definition of composition in the span category.
\end{proof}

In particular, given a cosheaf on $\mc{C}$, we get a canonical presheaf structure with the same values on objects (uniquely extend the cosheaf to the span category by 1, then restrict along $\mc{C}^{op}\to\on{Span}(\mc{C})$), and this presheaf will be a sheaf if and only if the transfer-isomorphism condition 3(c) holds.  We formalize this as follows.

\begin{definition}\label{trfunctor}
Let $\mc{C}$ be a Galois category, $\mc{D}$ a semi-additive $\infty$-category which admits $G$-orbit and $G$-fixed point objects for all finite groups $G$, and let $F\in \on{coSh}(\mc{C};\mc{D})$.  Define the presheaf
$$F^{tr} \in \on{Fun}(\mc{C}^{op},\mc{D})$$
as the image of $F$ under the composition
$$\on{coSh}(\mc{C};\mc{D})\simeq\on{Fun}^{coshf}(\on{Span}(\mc{C}),\mc{D})\rightarrow \on{Fun}(\mc{C}^{op},\mc{D}),$$
where the first map is the inverse to the isomorphism of 1 of Lemma \ref{transferlemma} and the second map is composition with the natural map $\mc{C}^{op}\to\on{Span}(\mc{C})$ described in Lemma \ref{transferlemma}.

Dually, if $F$ is a sheaf we get a co-presheaf $F^{tr}$.  Note that by construction $F$ and $F^{tr}$ agree on objects; moreover the association $F\mapsto F^{tr}$ is functorial in $\mc{C}$ (for exact functors) and $\mc{D}$ (for additive functors), via composition.
\end{definition}

\begin{example}
Let $\mc{C}$ be a Galois category and $A$ a sheaf of abelian groups on $\mc{C}$.  Then $U\mapsto R\Gamma(U;A)$ describes the corresponding sheaf on $\mc{C}$ with values in $\on{D}(\mbb{Z})$ (via the identification of the heart of $\on{Sh}(\mc{C};\on{D}(\mbb{Z}))$ with $\on{Sh}(\mc{C};\on{Ab})$).  For $U\to V$ a map of connected objects of $\mc{C}$, consider the induced transfer map
$$R\Gamma(U;A)\to R\Gamma(V;A)$$
in $\on{D}(\mbb{Z})$, from the above definition.  We claim that after applying $H_{-i}$, this corresopnds to the classical transfer (co-restriction) map in profinite group cohomology
$$H^i(U;A)\to H^i(V;A)$$
as defined via the mechanism of $\delta$-functors.  Indeed, by dimension shift with injective resolutions we reduce to the case $i=0$; thus it suffices to show the claim on the non-derived level, i.e.\ that the transfer map $A(U) \to A(V)$ from Definition \ref{trfunctor} (with $\mc{D}=\on{Ab}$) identifies with the classical transfer map.  But $A(V)\subset A(U)$ via restriction, so it suffices to see that the composition $A(U) \to A(V) \subset A(U)$ is the same in both cases.  This follows from the definition of composition in the span category.
\end{example}

\begin{example}\label{ztr}
Let $\mc{C}$ be a Galois category, and consider the cosheaf of abelian groups $\ul{\mbb{Z}}$ which sends $X\in\mc{C}$ to the free abelian group on the set of components of $X$ (so that by definition a refined class formation has $H_0=\ul{\mbb{Z}}$).  For a map $X \to Y$ of connected objects, the associated transfer map
$$\ul{\mbb{Z}}(Y)\rightarrow \ul{\mbb{Z}}(X),$$
which is then just a map $\mbb{Z} \to \mbb{Z}$,
identifies with multiplication by the degree of $X\to Y$.  The resulting presheaf $\mbb{Z}^{tr}:\mc{C}^{op}\to\on{Ab}$ can be alternatively be embedded as a sub-presheaf of the constant sheaf with value $\mbb{Q}$, as follows: for connected $X\in\mc{C}$, we map
$$\mbb{Z}^{tr}(X) \to \mbb{Q}$$
by the multiplication by $\frac{1}{\#X}$ map $\mbb{Z}\to \mbb{Q}$, thereby identifying $\mbb{Z}^{tr}(X)$ with $\frac{1}{\#X}\mbb{Z}\subset\mbb{Q}$.  Here $\#X$ denotes the degree of $X$ over $\ast$ (the cardinality of $X$, if we identify $\mc{C}$ with the category of finite $\Gamma$-sets for some profinite group $\Gamma$).

This map $\mbb{Z}^{tr}\to\mbb{Q}$ induces an isomorphism of sheafification of $\mbb{Z}^{tr}$ with the constant sheaf with value
$$\frac{1}{\#\mc{C}}\mbb{Z}:= \bigcup_{X} \frac{1}{\#X}\mbb{Z}\subset\mbb{Q},$$
the union being over all connected objects in $\mc{C}$.
\end{example}

The third standard lemma is a version of the Tate-Nakayama lemma.

\begin{lemma}\label{uniformbound}
Let $G$ be a finite group, and let $\mc{D}$ a $\mbb{Z}$-linear stable $\infty$-category equipped with a t-structure and admitting $H$-fixed point and $H$-orbit objects for all subgroups $H\subset G$.  Suppose $M\in\on{Fun}(BG;\mc{D})$ is an object of $\mc{D}$ equipped with $G$-action such that $M$ is concentrated in a bounded range of degrees (with respect to the t-structure).  Then the following properties are equivalent:
\begin{enumerate}
\item For all subgroups $H\subset G$, the object $M_H$ (homotopy orbits for the $H$-action, or $H$-hyperhomology) is concentrated in a bounded range of degrees.
\item For all subgroups $H\subset G$, the object $M^H$ (homotopy fixed points for the $H$-action, or $H$-hypercohomology) is concentrated in a bounded range of degrees.
\item For all subgroups $H\subset G$, the natural transfer map $M_H \to M^H$ is an isomorphism.
\end{enumerate}
Under these conditions, if $M$ is in degrees $[a,b]$, then so are $M^H$ and $M_H$ for all $H\subset G$.
\end{lemma}
\begin{proof}
Because 1 and 2 are dual to each other and 3 is self-dual (with respect to $\mc{D}\leftrightarrow\mc{D}^{op}$), it suffices to show that 2 and 3 are equivalent.  Suppose 3 holds.  If $M$ is in degrees $[a,b]$, then $M^H$ is in degrees $\leq b$ and $M_H$ is in degrees $\geq a$.  As $M_H\simeq M^H$ by hypothesis, this implies $M^H$ is concentrated in degrees $[a,b]$, i.e.\ condition 2 holds (moreover we get the final claim).  For the converse 2 $\Rightarrow 3$, we follow the standard Tate-Nakayama argument.

We want to show 2 $\Rightarrow$ 3.  Replacing $\mc{D}$ by $\on{Ind}(\mc{D}_0)$, where $\mc{D}_0$ is the smallest stable subcategory of $\mc{D}$ containing all $M^H$ for $H\subset G$ and closed under t-truncation, we can assume that $\mc{D}$ is presentable.  It follows by the multiplicative transfer machinery in \cite{cm2} that
$$M^{tH}:= \on{cofib}(M_H \to M^H)$$
is a module over $\mbb{Z}^{tH}\in \on{CAlg}(\on{D}(\mbb{Z}))$.  In fact the assignment $G/H \mapsto M^{tH}$ organizes into a functor on the span category of finite $G$-sets as in the previous lemma, and even in that context it is a module over $G/H\mapsto \mbb{Z}^{tH}$ which is a ring in a natural way.

Now, first suppose $H$ is cyclic.  Then the standard resolution implies that $\mbb{Z}^{tH}$ is periodic.  But assuming 2, $M^{tH}$ vanishes in low enough degrees; therefore, being a module over a periodic ring spectrum, it must be $0$.  This proves that 2 $\Rightarrow$ 3 when $G$ is cyclic.  By the standard induction with a nontrivial central subgroup, we deduce from this that 2 $\Rightarrow$ 3 also holds if $G$ is nilpotent.  Next, note that for a general finite group $G$ and an object $M\in\mc{D}$ with $G$-action and subgroup $H\subset G$, there are maps $M^{tG} \to M^{tH}$ and $M^{tH} \to M^{tG}$ such that the composition $M^{tG} \to M^{tG}$ is multiplication by $\#[G:H]$.  This follows from the multiplicative span category formalism referred to above, using the corresponding (obvious) fact for $G/H \mapsto \mbb{Z}^H$.

In total, we deduce that $M^{tG}$ is killed by multiplication by $[G:H]$ for all nilpotent subgroups $H\subset G$.  By the Sylow theorems there is no prime number which divides all such indices, and it follows that $M^{tG}=0$.  Replacing $G$ by a subgroup we deduce the claim.
\end{proof}

Combining these lemmas we get the following interesting principle, which is closely related to the duality formalism explained by Tate in \cite{galcoh} Appendix 1.

\begin{proposition}\label{trissheaf}
Let $\mc{C}$ be a Galois category and $\mc{D}$ a $\mbb{Z}$-linear stable $\infty$-category equipped with a t-structure and admitting $H$-fixed point and $H$-orbit objects for all finite groups $H$.  Suppose $\mc{F}:\mc{C}^{op}\to\mc{D}$ is a sheaf such that $\mc{F}(X)$ is t-bounded for all $X\in\mc{C}$.  Then the co-presheaf $\mc{F}^{tr}:\mc{C}\to\mc{D}$ is a cosheaf.  Similarly, if $F:\mc{C}\to\mc{D}$ is a cosheaf which takes t-bounded values, then $F^{tr}$ is a sheaf.  Moreover, in such situations we have $\mc{F}=(\mc{F}^{tr})^{tr}$ and $F=(F^{tr})^{tr}$, i.e.\ $(-)^{tr}$ is involutive.
\end{proposition}
\begin{proof}
By Lemma \ref{transferlemma}, we have that $\mc{F}^{tr}$ is a cosheaf if and only if for all $X\in\mc{C}$ and all $H\subset \on{Aut}(X)$, the transfer map
$$\mc{F}(X)_H\to \mc{F}(X)^H$$
is an isomorphism. By Lemma \ref{uniformbound}, this holds if the $\mc{F}(X)^H = F(X/H)$ are t-bounded, but they are by assumption.  The argument that $F^{tr}$ is a sheaf if $F$ is a sectionwise t-bounded cosheaf is the same.  The involutivity follows immediately from the definition of $(-)^{tr}$ (Definition \ref{trfunctor}).
\end{proof}

Now we will finally explain the relationship between class formations and refined class formations.

\begin{theorem}\label{cfvsrcf}
For a Galois category $\mc{C}$, recall the associated subgroup $\frac{1}{\#\mc{C}}\mbb{Z}\subset\mbb{Q}$ from Example \ref{ztr}.

Giving a refined class formation $(\mc{C},F,\alpha)$ is equivalent to giving a class formation $(\mc{C},A)$ together with a map
$$\kappa:\frac{1}{\#\mc{C}}\mbb{Z} \to R\Gamma(\mc{C};A[2])$$
in $\on{D}(\mbb{Z})$ which, on $H_0$, induces an isomorphism
$$\left(\frac{1}{\#\mc{C}}\mbb{Z}\right)/\mbb{Z}\simeq H^2(\mc{C},A).$$

More precisely, given the latter data $(\mc{C},A,\kappa)$ we can build a refined class formation as follows: by adjunction we can interpret $\kappa$ as a map from the constant sheaf $\ul{\frac{1}{\#\mc{C}}\mbb{Z}}$ to $A[2]$ in the derived category of abelian sheaves on $\mc{C}$; if $\mc{F}$ denotes the $\on{D}(\mbb{Z})$-valued sheaf represented by the fiber of this map, then $F:= \mc{F}^{tr}$ (see Definition \ref{trfunctor}) satisfies that for $X\in\mc{C}$ connected the map
$$H_0(F(X)) = H_0(\mc{F}(X)) \to H_0(\ul{\frac{1}{\#\mc{C}}\mbb{Z}}(X))\overset{\cdot \#X}{\rightarrow} \mbb{Q}$$
induces an isomorphism $\alpha: H_0(F(X))\simeq \mbb{Z}$ such that
$$(\mc{C},F,\alpha)$$
is a refined class formation, and this describes an equivalence from the 2-groupoid of data $(\mc{C},A,\kappa)$ as above to the 2-groupoid of refined class formations.  (We explicitly describe the inverse equivalence in the proof.)
\end{theorem}

\begin{proof}
Suppose $(\mc{C},F,\alpha)$ is a refined class formation; we aim to produce the data $(A,\kappa)$ of a class formation on $\mc{C}$ with a $\kappa$ as in the statement.  By Proposition \ref{trissheaf}, $F^{tr}:\mc{C}^{op}\to\on{D}(\mbb{Z})$ is a sheaf.  Its top nonvanishing homology group $A:=H_1(F^{tr})$ is therefore a sheaf of abelian groups.  Now we show that $(\mc{C},A)$ is a class formation, i.e.\ that it satisfies the cohomology axioms 1 and 2 from Definition \ref{cfgal}.  Let $X\in\mc{C}$ be connected and $G\subset\on{Aut}(X)$ a subgroup.  Since $F(X)_H=F(X/H)$ lives in degrees $[0,1]$ for all subgroups $H\subset G$, we deduce from Lemma \ref{uniformbound} that $F(X)_G\overset{\sim}{\rightarrow}F(X)^G$, or in other words the Tate (hyper)cohomology of $F(X)$ vanishes.  As $F(X)$ lives in degrees $[0,1]$ with $H_1F(X)=A(X)$ and $H_0F(X)=\mbb{Z}$, we deduce that the Tate cohomology of $\mbb{Z}$ and $A(X)$ differ by a shift of 2.  The conclusion follows because $\wh{H}^0(G,\mbb{Z})=\mbb{Z}/\# G\mbb{Z}$ and $\wh{H}^{-1}(G,\mbb{Z})=0$.

Next we show that one also gets a natural map
$$\kappa:\frac{1}{\#\mc{C}}\mbb{Z} \to R\Gamma(\mc{C};A[2])$$
from the refined class formation $(\mc{C},F,\alpha)$.  For this, consider the sheafified Postnikov tower of the derived sheaf $F^{tr}$.  On the presheaf level we have $H_1F^{tr}=A$ which is already a sheaf, whereas $H_0F^{tr}=\mbb{Z}^{tr}$ because the formation of $(-)^{tr}$ commutes with additive functors.  By Example \ref{ztr} we can identify the sheafification of this with the constant sheaf on $\frac{1}{\#\mc{C}}\mbb{Z}$.  Thus, by adjunction, the Postnikov tower of $F^{tr}$ as a derived sheaf is classified by a map $\kappa:\frac{1}{\#\mc{C}}\mbb{Z} \to R\Gamma(\mc{C};A[2])$ of the desired form.

Next we should check that $\kappa$ has the announced property that on $H_0$ it induces $\left(\frac{1}{\#\mc{C}}\mbb{Z}\right)/\mbb{Z}\simeq H^2(\mc{C},A)$.  By construction there is a fiber sequence
$$F(\ast) \to R\Gamma(\mc{C};\frac{1}{\#\mc{C}}\mbb{Z})\to R\Gamma(\mc{C};A[2])$$
where the second map is induced by $\kappa$.  As $H_0F(\ast)=\mbb{Z}$ maps to $H_0(\mc{C};\frac{1}{\#\mc{C}}\mbb{Z})=\frac{1}{\#\mc{C}}\mbb{Z}$ via inclusion (by construction), and $H_{-1}F(\ast)=0$ by hypothesis, we deduce the claim from the long exact sequence in homology.

Now we explain the converse: supposing given a class formation $(\mc{C},A)$ and a map 
$$\kappa:\frac{1}{\#\mc{C}}\mbb{Z} \to R\Gamma(\mc{C};A[2])$$
inducing an isomorphism $\left(\frac{1}{\#\mc{C}}\mbb{Z}\right)/\mbb{Z}\simeq H^2(\mc{C},A)$, let us build a refined class formation.  The given map $\kappa$ classifies the (sheafified) Postnikov tower of a $\on{D}(\mbb{Z})$ sheaf $\mc{F}$ on $\mc{C}$ with sheafified $H_0$ fixed as the constant sheaf $\frac{1}{\#\mc{C}}\mbb{Z}$, with $H_1$ fixed as $A$, and all other sheafified homology groups $0$.  More precisely, we have a fiber sequence
$$\mc{F} \to R\Gamma(-;\frac{1}{\#\mc{C}}\mbb{Z}) \to R\Gamma(-;A[2])$$
in $\on{Fun}(\mc{C}^{op};\on{D}(\mbb{Z}))$.  We claim:
\begin{enumerate}
\item $\mc{F} \to R\Gamma(-;\frac{1}{\#\mc{C}}\mbb{Z})$ is injective on presheaf-level $H_0$ and identifies $H_0\mc{F}$ with $\mbb{Z}^{tr}$ (as a sub-presheaf of the constant sheaf on $\frac{1}{\#\mc{C}}\mbb{Z}$ via the association of Example \ref{ztr});
\item $H_i\mc{F}=0$ on the presheaf level for all $i\neq 0,1$.
\end{enumerate}
Assuming this, we deduce from Proposition \ref{trissheaf} that $F:=\mc{F}^{tr}$ is a cosheaf living in degrees $[0,1]$.  Moreover its $H_0$ is a cosheaf of abelian groups such that applying $(-)^{tr}$ gives $\mbb{Z}^{tr}$, and the only possibility for that is the constant cosheaf $\mbb{Z}$.  Thus we get a refined class formation.  Moreover, the two procedures (going from refined class formations to class formations plus the extra map, and going backwards) are inverse because $(-)^{tr}$ is involutive, see Proposition \ref{trissheaf}.

Thus it suffices to show claims 1 and 2 above.  For this, we note the following simple consequence of the class formation axioms: for any connected $X\in\mc{C}$, we have:
\begin{enumerate}
\item $H^1(X;A)=0$;
\item the restriction map
$$H^2(\mc{C};A)\to H^2(X;A)$$
is surjective, and its kernel is the unique subgroup of order $\#X$ in $H^2(\mc{C};A)$.
\end{enumerate}
Granting this, from the long exact sequence associated to $\mc{F}(X) \to R\Gamma(X;\frac{1}{\#\mc{C}}\mbb{Z}) \to R\Gamma(X;A[2])$ we deduce that $H_0\mc{F}(X)$ identifies with the kernel of the map from $\frac{1}{\#\mc{C}}\mbb{Z}$ to the quotient of $\left(\frac{1}{\#\mc{C}}\mbb{Z}\right)/\mbb{Z}$ by its $\# X$-torsion subgroup $\frac{1}{\#X}\mbb{Z}/\mbb{Z}$, which gives desired claim 1 above.

The remaining claim is that $\mc{F}$ lives in degrees $[0,1]$ also on the presheaf level, or equivalently, that the presheaf level $\tau_{\geq 0}\mc{F}(-)$ is a sheaf.  Thus we need to see that if $X\in{C}$ is connected and $G\subset\on{Aut}(X)$ is a subgroup, then
$$\tau_{\geq 0}\mc{F}(X/G) \overset{\sim}{\rightarrow} \left(\tau_{\geq 0}\mc{F}(X)\right)^G.$$
Comparing with the known result $\mc{F}(X/G)\overset{\sim}{\rightarrow} \mc{F}(X)^G$, we see that it suffices to show that $\left(\tau_{\geq 0}\mc{F}(X)\right)^G$ lives in degrees $\geq 0$.  For that, by Lemma \ref{uniformbound}, it suffices to show that the $G$-Tate hypercohomology of $\tau_{\geq 0}\mc{F}(X)$ vanishes, for all $X\in\mc{C}$ connected and all subgroups $G\subset \on{Aut}(X)$.

However, note that by the above analysis, the class $u_{G,X}\in H^2(G;A(X))$ which classifies Postnikov tower of the $G$-object $\tau_{\geq 0}\mc{F}(X)$ is uniquely characterized as follows: its inflation
$$\on{infl}(u_{G,X}) \in H^2(X/G;A)$$
is equal to the image of $\frac{1}{\#G}\mbb{Z}$ via the composition
$$\frac{1}{\#\mc{C}}\mbb{Z} \to H^2(\mc{C};A) \to H^2(X/G;A)$$
where the first map is induced by our given data and the second map is restriction.

In particular, $u_{G,X}$ generates the cyclic group $H^2(G;A(X))$ of order $\#G$.  By the class field axioms it follows that the Tate hypercohomology we're interested vanishes in two consecutive degrees (for all $G$), therefore it vanishes in all degrees by Tate-Nakayama, finishing the proof.
\end{proof}

\begin{corollary}\label{howtopromote}
Let $(\mc{C},A)$ be a class formation.  Denote by $\mc{G}$ the 2-groupoid of lifts of $(\mc{C},A)$ to a refined class formation.  Then for the homotopy sets of $\mc{G}$ we have the following:
\begin{enumerate}
\item $\pi_0\mc{G}=\on{InjHom}(H^2(\mc{C};A),\mbb{Q}/\mbb{Z})$;
\item $\pi_1\mc{G}\simeq \on{Ext}(\frac{1}{\#\mc{C}}\mbb{Z},\Gamma(\mc{C};A))$;
\item $\pi_2\mc{G}\simeq \on{Hom}(\frac{1}{\#\mc{C}}\mbb{Z},\Gamma(\mc{C};A))$.
\end{enumerate}
Here $\on{InjHom},\on{Ext},\on{Hom}$ stand for respectively the sets of injective homomorphisms, isomorphism classes of extensions, and homomorphisms, all in the category of abelian groups.
\end{corollary}
\begin{proof}
By Theorem \ref{cfvsrcf}, $\mc{G}$ is a union of components of $$\mc{G}' = \on{Map}_{\on{D}(\mbb{Z})}(\frac{1}{\#\mc{C}}\mbb{Z},R\Gamma(\mc{C};A[2])),$$
namely it is the full sub-2-groupoid on those maps which are surjective on $H_0$ with kernel equal to $\mbb{Z}$. Because $\on{Ext}^i=0$ for $i>1$ in the category of abelian groups, and $H^1(\mc{C};A)=0$, we get that
$$\pi_0\mc{G}' = \on{Hom}(\frac{1}{\#\mc{C}}\mbb{Z},H^2(\mc{C};A)), \pi_1\mc{G}' = \on{Ext}(\frac{1}{\#\mc{C}}\mbb{Z},\Gamma(\mc{C};A)), \pi_2\mc{G}'=\on{Hom}(\frac{1}{\#\mc{C}}\mbb{Z},\Gamma(\mc{C};A)).$$
To get the homotopy groups of $\mc{G}$ we just impose the condition on $\pi_0$ that the homomorphism $\frac{1}{\#\mc{C}}\mbb{Z}\to H^2(\mc{C};A)$ be surjective with kernel $\mbb{Z}$.  Thus
$$\pi_0\mc{G} = \on{Iso}(\left(\frac{1}{\#\mc{C}}\mbb{Z}\right)/\mbb{Z},H^2(\mc{C};A)).$$
Given such an isomorphism we obviously get an injective homomorphism
$$H^2(\mc{C};A)\to\mbb{Q}/\mbb{Z}$$
by taking the inverse; conversely, using that
$$H^2(\mc{C};A) = \cup_X H^2(G;A(X))$$
via the inflation maps, where $X$ runs over all Galois objects in $\mc{C}$ with group $G$, we deduce that the image any injection $H^2(\mc{C};A)\hookrightarrow\mbb{Q}/\mbb{Z}$ must be $\left(\frac{1}{\#\mc{C}}\mbb{Z}\right)/\mbb{Z}$, thereby establishing that 
$$\on{Iso}(\left(\frac{1}{\#\mc{C}}\mbb{Z}\right)/\mbb{Z},H^2(\mc{C};A))\simeq \on{InjHom}(H^2(\mc{C};A),\mbb{Q}/\mbb{Z})$$
and hence finishing the proof.
\end{proof}

Note that the $\pi_0$-level datum here, namely the datum of an isomorphism $H^2(\mc{C};A)\simeq\frac{1}{\#\mc{C}}\mbb{Z}/\mbb{Z}$, exactly corresponds to the usual data of chosen compatible generators for the $H^2(G;A(X))$ in the class formation, or equivalently compatible systems of $\on{inv}$ maps as in \cite{corpslocaux} XI.  In particular, any class formation equipped with such a system of canonical classes (or inv maps) promotes, uniquely up to isomorphism, to a refined class formation.

\begin{example}\label{newrcf}
Suppose $(\mc{C},F,\alpha)$ is a refined class formation, with corresponding class formation $(\mc{C},A)$ plus map
$$\frac{1}{\#\mc{C}}\mbb{Z} \overset{\kappa}{\rightarrow} R\Gamma(\mc{C};A[2]),$$
from Theorem \ref{cfvsrcf}.

We have the following canonical ways to build new refined class formations from $(\mc{C},F,\alpha)$, and their interpretations via Theorem \ref{cfvsrcf}:
\begin{enumerate}
\item For any connected object $X\in\mc{C}$, by restriction we get a refined class formation
$$(\mc{C}_{/X},F|_{\mc{C}_{/X}},\alpha|_{\mc{C}_{/X}}).$$
(Note that if $\mc{C}$ corresponds to a profinite group $\Gamma$, then $\mc{C}_{/X}$ corresponds to an open subgroup $H\subset\Gamma$ up to conjugation.)

This corresponds to the class formation $(\mc{C}_{/X},A|_{\mc{C}_{/X}})$ together with the map
$$\kappa':\frac{1}{\#\mc{C}_{/X}}\mbb{Z}\to R\Gamma(\mc{C}_{/X};A|_{\mc{C}_{/X}}[2])$$
given as follows:
$$\kappa'(\alpha) = \on{res}(\kappa\left(\frac{1}{\#X}\alpha\right)),$$
where $\on{res}:R\Gamma(\mc{C};A[2])\rightarrow R\Gamma(\mc{C}_{/X};A|_{\mc{C}_{/X}}[2])$ is the usual restriction map. 

\item For any full subcategory $\mc{C}'\subset\mc{C}$ closed under finite limits and colimits (which is then a Galois category in its own right), we get a refined class formation
$$(\mc{C}',F|_{\mc{C}'},\alpha|_{\mc{C}'}).$$
(Note that if $\mc{C}$ corresponds to a profinite group $\Gamma$, then $\mc{C}'$ corresponds to a quotient $\Gamma\twoheadrightarrow\Gamma'$.)

This corresponds to the class formation $(\mc{C}',A|_{\mc{C}'})$ together with the map
$$\kappa': \frac{1}{\#\mc{C}'}\mbb{Z} \to R\Gamma(\mc{C}';A|_{\mc{C}'}[2])$$
which is uniquely determined as follows: the inflation map
$$\on{infl}: R\Gamma(\mc{C}';A|_{\mc{C}'}[2])\to R\Gamma(\mc{C};A[2])$$
is an isomorphism on $\pi_1$ and $\pi_2$ and injective on $\pi_0$, hence on maps from any abelian group in degree $0$ it gives a fully faithful functor of 2-groupoids, so it suffices to characterize $\on{infl}\circ \kappa'$, and this is given by
$$\on{infl}\circ\kappa' = \kappa|_{\frac{1}{\#\mc{C}'}\mbb{Z}}$$
using that $\frac{1}{\#\mc{C}'}\mbb{Z}\subset\frac{1}{\#\mc{C}}\mbb{Z}$.

\item Suppose given a subsheaf $K\subset A$ of abelian groups which is \emph{cohomologically trivial}, i.e.\ $H^i(X;K)=0$ for all $X\in\mc{C}$ connected and all $i>0$, or equivalently (see \cite{galcoh}, appendix of Tate) the $G$-Tate cohomology of $K(X)$ vanishes for all connected $X\in\mc{C}$ and all subgroups $G\subset\on{Aut}(X)$.

Then the quotient $A/K$ on the presheaf level is a sheaf and $(\mc{C},A/K)$ is a class formation.  The composition
$$\frac{1}{\#\mc{C}}\mbb{Z} \overset{\kappa}{\rightarrow} R\Gamma(\mc{C};A[2]) \to R\Gamma(\mc{C};A/K[2])$$
corresponds to the refined class formation given by $(\mc{C},F',\alpha')$ where $F'$ is defined by

$$F'(X) := \on{cofib}(K^{tr}[1]\to F)$$
where the map $K^{tr}[1] \to F$ is uniquely characterized on $H_1$ where it is given by applying $(-)^{tr}$ to the inclusion $K\to A$, and $\alpha'$ is uniquely induced by $\alpha$ via the map $F\to F'$ which is an isomorphism on $H_0$.  
\end{enumerate} 
\end{example}

\begin{remark}
We can summarize points 1 and 2 as follows: a (refined) class formation on a profinite group canonically induces (refined) class formations on any open subgroup or quotient group.  Note that, in fact, the data of a (refined) class formation on a profinite group is \emph{equivalent} to the data of compatible (refined) class formations on each of its finite quotients; this is obvious from the definitions.
\end{remark}

For the class formations $(\mc{C},A)$ arising in number theory, $A$ will not just be a sheaf of abstract abelian groups, rather it will be a sheaf of topological (or condensed) abelian groups.  This extra structure is important to remember for discussing the Weil groups.  We have the following obvious condensed analog of Definition \ref{rcf}  .

\begin{definition}\label{crcf}
A \emph{condensed refined class formation} is a triple $(\mc{C},F,\alpha)$ consisting of a Galois category $\mc{C}$, a cosheaf
$$F:\mc{C} \to \on{D}(\on{CondAb})$$
which is sectionwise concentrated in degrees $[0,1]$, and a natural transformation
$$\alpha:H_0F\to \mbb{Z}$$
such that for all connected $X\in\mc{C}$, the map $\alpha:H_0F(X)\overset{\sim}{\rightarrow}\mbb{Z}$ is an isomorphism.
\end{definition}

Note that if $(\mc{C},F,\alpha)$ is a condensed refined class formation, then we get an underlying refined class formation, together with the structure of condensed sheaf on the ``class group'' sheaf $A$, produced by the formalism of transfers exactly as in the non-condensed case.

Conversely:

\begin{lemma}\label{condautomatic}
Suppose given:
\begin{enumerate}
\item a refined class formation, corresponding to data $(\mc{C},A,\kappa)$ as in Theorem \ref{cfvsrcf};
\item a promotion $A^{cond}$ of $A$ to a sheaf of condensed abelian groups.
\end{enumerate}

Assume that $A^{cond}(X)$ is almost compact (Definition \ref{acomp}) for all connected $X\in\mc{C}$.  Then there is a unique condensed refined class formation inducing the data 1,2 above via the obvious forgetful functor.
\end{lemma}
\begin{proof}
View the datum $\kappa$ as a map
$$\mbb{Z}^{tr} \to A[2]$$
in $\on{Fun}(\mc{C}^{op},\on{D}(\mbb{Z}))$ (this classifies the Postnikov tower of $F^{tr}$ viewed as a presheaf).  Because $\mbb{Z}$ is discrete, by the datum 2 this is equivalent to giving a map
$$\mbb{Z}^{tr} \to A^{cond}[2]$$
in $\on{Fun}(\mc{C}^{op},\on{D}(\on{CondAb}))$.  Let $\mathcal{F}$ denote the fiber of this map.  Then it suffices to show that $\mc{F}$ is a sheaf (as then by Proposition \ref{trissheaf} $\mc{F}^{tr}$ will be a cosheaf and hence will give the required condensed refined class formation).  Thus, let $X\in\mc{C}$ be connected and let $G\subset\on{Aut}(X)$ be a subgroup, with $Y:=X/G$.  Consider the map
$$\mc{F}(Y) \to \mc{F}(X)^G;$$
we need to show it's an isomorphism.  Note that we know this to be true on underlying objects of $\on{D}(\mbb{Z})$, i.e.\ after applying $(-)(\ast)$.

Let $K(X)$ denote the largest compact subgroup of $A^{cond}(X)$, so that we have a short exact sequence of $G$-objects of $\on{CondAb}$

$$0 \to K(X) \to A^{cond}(X) \to C(X)\to 0$$

\noindent where $C(X)$ is either $0$, $\mbb{Z}$, or $\mbb{R}$ with some undetermined action of $G$.  The induced homomorphism $A^{cond}(X)^G \to C(X)^G$ (of locally compact Hausdorff abelian groups, with target isomorphic to either $0$,$\mbb{Z}$, or $\mbb{R}$) has image containing $\#G\cdot C(X)^G$; it follows that the cokernel is a finite (discrete) abelian group.  Then from the long exact sequence in cohomology and the fact that compact Hausdorff abelian groups are an abelian subcategory of $\on{CondAb}$ closed under extensions, we deduce that $H^i(G;A^{cond}(X))$ is compact Hausdorff for all $i>0$.

In particular, we deduce that $H^1(G;A^{cond}(X))=0$ as a condensed group and $H^2(G;A^{cond}(X))$ is cyclic of order $\#G$ as a condensed group, from the corresponding facts about the underlying groups.  Then using the long exact sequence in $G$-cohomology associated to
$$A^{cond}(X)[1] \to \mc{F}(X) \to \mbb{Z}[0]$$
we see that $H_0(\mc{F}(X)^G)\simeq \mbb{Z}$ as a condensed abelian group and so our claim in degree $0$ follows from the claim on underlying abelian groups.  Similarly we deduce $H_i(\mc{F}(X)^G)$ is compact Hausdorff for $i<0$ hence vanishes by the corresponding fact on underlying abelian groups.  All that remains is the claim in degree $1$, but that amounts to the assumption that $A^{cond}$ is a sheaf.
\end{proof}

Now let us discuss the usual class formations from number theory, paying special attention to the question of how canonically we can promote them to refined class formations.

\begin{example}\label{exrcf}
\begin{enumerate}
\item Let us start with the most important example, the class formation for the rational number field $\mbb{Q}$.  The Galois category is $\mc{C}_\mbb{Q}$, the finite etale site of $\on{Spec}(\mbb{Q})$.  The sheaf of condensed abelian groups $A$ on $\mc{C}_\mbb{Q}$ is provided by the idele class groups, so 
$$A(K) = \mbb{A}_K^\times /K^\times$$
for $K$ a number field.  Note that $A(K)$ is almost compact: the usual norm map
$$N:A(K) \to \mbb{R}_{>0}$$
is surjective with compact kernel $A(K)^{1}$.  The class field axioms, plus the sheaf axiom for $A$, are satisfied by global class field theory for number fields.  Thus, by Lemmas \ref{howtopromote} and \ref{condautomatic}, the 2-groupoid of lifts of $(\mc{C}_\mbb{Q},A)$ to a condensed refined class formation has:
\begin{enumerate}
\item $\pi_0 = \on{Inj}(H^2(\mc{C}_\mbb{Q};A),\mbb{Q}/\mbb{Z})$;
\item $\pi_1 = \on{Ext}(\mbb{Q},\mbb{A}^\times/\mbb{Q}^\times)=0$;
\item $\pi_2 = \on{Hom}(\mbb{Q},\mbb{A}^\times/\mbb{Q}^\times)=\mbb{R}_{>0}$.
\end{enumerate}
Again by global class field theory, there is a canonical point in $\pi_0$, corresponding to the usual isomorphism
$$\on{inv}: H^2(\mc{C}_{\mbb{Q}};A)\simeq \mbb{Q}/\mbb{Z},$$
or, equivalently, the usual system of fundamental classes, or again equivalently (see \cite{tatentb}), the usual system of Artin maps. (Actually there are two canonical choices, differing by sign; let us say that we take the convention that uniformizers go to arithmetic Frobenius elements via the associated Artin maps.)

Thus, any two refined (condensed) class formations, each inducing the usual global class field theory for number fields, are isomorphic by an isomorphism which is unique up to homotopy; but that homotopy is not unique: it has an $\mbb{R}_{>0}$-ambiguity.

We can get rid of this ambiguity as follows: we replace $A$ by its subsheaf $A^1$ (the kernel of the norm map).  Note that
$$\mathbb{R}_{>0}\oplus A^1\overset{\sim}{\longrightarrow} A$$
via the map given by the inclusion on the second factor and the map induced by the inclusion of the positive reals into the archimedean factor of $\mbb{A}_{\mbb{Q}}^\times/\mbb{Q}^\times$ on the first factor.  Repeating the above analysis for the class formation $(\mc{C}_{\mbb{Q}},A^1)$ we find that $\pi_1=\pi_2=0$ so there is a completely unique condensed refined class formation $(\mc{C}_{\mbb{Q}},F^1,\alpha)$ inducing $(\mc{C}_{\mbb{Q}},A^1)$ and the usual system of Artin maps for number fields.  Then setting
$$F = \mbb{R}_{>0}[1]\oplus F^1,$$
we find that 
$$(\mc{C}_{\mbb{Q}},F,\alpha)$$
gives a \emph{canonical} condensed refined class formation refining the usual class formation and Artin maps of class field theory for number fields.

\item Let $K$ be a number field.  Then applying Example \ref{newrcf} to $\on{Spec}(K)\in\mc{C}_{\mbb{Q}}$, we deduce that the restriction of $(\mc{C}_{\mbb{Q}}, F,\alpha)$ to $\mc{C}_K = (\mc{C}_{\mbb{Q}})_{/\on{Spec}(K)}$ gives a canonical condensed refined class formation (refining class field theory over $K$).

\item More generally, let $K$ be a number field and $S$ a set of places of $K$, which we will eventually assume to be nonempty.  Define a finite extension $L/K$ to be \emph{unramified outside $S$} if for all places $\nu$ of $K$ not in $S$ and all places $\mu$ of $L$ above $\nu$, the extension of local fields $L_{\mu}/K_\nu$ is:
\begin{enumerate}
\item unramified in the usual sense, if $\nu$ is nonarchimedean;
\item trivial, if $\nu$ is archimedean.
\end{enumerate}
Then the full subcategory $\mc{C}_{K,S}\subset \mc{C}_K$ of finite etale $K$-algebras which are products of field extensions $L/K$ unramified outside $S$ is a Galois category.  (If $S$ contains all archimedean places, then this is the finite etale site of $\on{Spec}(\mc{O}_{K,S})$, and the corresponding Galois group is the traditional Galois group $G_{K,S}$ with restricted ramification.)

Consider the sheaf $U$ on $\mc{C}_{K,S}$ given by
$$U(L) = \prod_{\mu\not\in S} (L_{\mu}^\times)^1$$
for any finite extension $L/K$ unramified outside $S$, where we abusively let $S$ also denote the set of places of $L$ above $S$, and $(L_{\mu}^\times)^1\subset L_{\mu}^\times$ denotes the maximal compact subgroup (group of elements of norm 1).  Then $U$ is cohomologically trivial.  Moreover, if $S\neq\emptyset$, then the map
$$U(L) \subset \mbb{A}_L^\times\to A(L)=\mbb{A}_L^\times/L^\times$$
is injective for all $L$ as above. Thus, by Example \ref{newrcf} we deduce that when $S\neq \emptyset$,

$$(\mc{C}_{K,S},A/U)$$
is a class formation, and
$$(\mc{C}_{K,S},\on{cofib}(U[1]\to F),\alpha)$$
is a canonical condensed refined class formation refining it.

Note that although the Galois category $\mc{C}_{K,S}$ does not change if we add or subtract complex places from $S$, the class formation does change, and so will the associated Weil group.

More explicitly, the class groups in this class formation are as follows: for $L/K$ unramified outside $S$, we have
$$A_{L,S} = L^\times\backslash\mathbb{A}_L^\times/\prod_{\mu\not\in S}(L_\mu^\times)^1.$$
Suppose $S$ is finite and contains all archimedean places, so that $L/K$ corresponds to the finite etale $\mc{O}_{K,S}$-algebra $\mc{O}_{L,S}$.  By the usual adelic descent, we can reinterpret the above class group as being the ``compactly supported Picard group'' of $\mc{O}_{L,S}$, that is the abelian group of line bundles on $\on{Spec}(\mc{O}_{L,S})$ equipped with a trivialization on base-change to $\sqcup_{\mu\in S} \on{Spec}(L_\mu)$.  (Such objects have no automorphisms, so this $\on{Pic}_c(\mc{O}_{L,S})$ is legitimately an abelian group.)  This interpretation gives us an exact sequence
$$0 \to \mc{O}_{L,S}^\times \to \prod_{\mu\in S} L_\mu^\times \to A_{L,S} \to \on{Pic}(\mc{O}_{L,S}) \to 0$$
which is often convenient for analyzing $A_{L,S}$.

\item In the above examples with number fields, we managed to find \emph{canonical} refined class formations refining the usual class formations.  But now we consider the ostensibly much simpler setting of \emph{finite fields}, where, strangely, there is seemingly no canonical choice.

Let $p$ be a prime with associated finite field $\mbb{F}_p$, and let $\mc{C}_{\mbb{F}_p}$ denote the Galois category of finite etale $\mbb{F}_p$-schemes.  Then
$$(\mc{C}_{\mbb{F}_p},\mbb{Z})$$
is a class formation, where we mean the constant sheaf on $\mbb{Z}$.  Moreover, there is a unique isomorphism
$$\on{inv}:H^2(\mc{C}_{\mbb{F}_p};\mbb{Z})\simeq \mbb{Q}/\mbb{Z}$$
such that the induced Artin map
$$\mbb{Z}\to G_{\mbb{F}_p}^{ab}$$
sends $1$ to the element given by the (arithmetic) Frobenius automorphism $x\mapsto x^p$.  By \ref{howtopromote} the 2-groupoid of lifts to a refined class formation, inducing the above $\on{inv}$, has:
\begin{enumerate}
\item $\pi_0=\ast$;
\item $\pi_1\simeq\on{Ext}(\mbb{Q},\mbb{Z})=\wh{\mbb{Z}}/\mbb{Z}$;
\item $\pi_2\simeq\on{Hom}(\mbb{Q},\mbb{Z})=0.$
\end{enumerate}
Thus  any two choices of refined class formation are isomorphic, but not uniquely: there is a $\wh{\mbb{Z}}/\mbb{Z}$-ambiguity in the choice of isomorphism (but no higher ambiguity).

Explicitly, suppose we choose an algebraic closure $\overline{\mbb{F}_p}/\mbb{F}_p$.  This determines an equivalence of $\mc{C}_{\mbb{F}_p}$ with the Galois category of finite $\on{Gal}(\overline{\mbb{F}_p}/\mbb{F}_p)$-sets, or an isomorphism of the Galois groupoid of $\mbb{F}_p$ with $B\on{Gal}(\overline{\mbb{F}_p}/\mbb{F}_p)$.  But we also have a canonical homomorphism $\mbb{Z} \to \on{Gal}(\overline{\mbb{F}_p}/\mbb{F}_p)$ given by the powers of Frobenius, which gives an associated map of groupoids
$$B\mbb{Z} \to \mc{G}_{\mbb{F}_p}.$$
Under the equivalence of refined class formations with Weil groupoids described in the following section, this gives rise to a refined class formation which refines $(\mc{C}_{\mbb{F}_p},\mbb{Z})$.

From the perspective of the definition of a refined class formation, this same situation can be described as follows.  Suppose again that we fix an $\overline{\mbb{F}_p}$, and let $\varphi$ denote Frobenius.  Then we get a refined class formation given by a cosheaf $F$ on the finite etale site of $\on{Spec}(\mbb{F}_p)$ defined as follows: we send $X\in \on{fEt}_{\mbb{F}_p}$ to the object in $\on{D}(\mbb{Z})$ represented by the complex (in homological degrees $1$ and $0$) given by
$$\mbb{Z}[X(\overline{\mbb{F}_p})] \overset{\varphi - 1}{\longrightarrow} \mbb{Z}[X(\overline{\mbb{F}_p})],$$
with the augmentation $\alpha:F\to\mbb{Z}$ being induced by the obvious augmentation on the $0^{th}$ term of the above complex.  The isomorphism $H_1F\simeq \mbb{Z}^{tr}$ (denoting by $\mbb{Z}$ the constant sheaf) is given by the diagonal inclusion of $\mbb{Z}$ in the first term in the above complex.

An isomorphism between two algebraic closures induces an isomorphism of refined class formations, inducing the identity on underlying class formations; this map from isomorphisms of algebraic closures to isomorphisms of refined class formations is surjective, and two isomorphisms induce (canonically) equivalent isomorphisms of refined class formations if and only if they differ by a power of Frobenius.  This explains the $\wh{\mbb{Z}}/\mbb{Z}$-ambiguity.

It is unclear to the author whether one can expect a canonical choice of refined class formation in this setting, or equivalently, if one can expect a canonical choice of algebraic closure of $\mbb{F}_p$ ``up to powers of Frobenius".  (It is also unclear to the author how to make the word ``canonical" precise, so as to turn this into a well-defined mathematical question.)
\item Suppose now $k$ is a general finite field, of characteristic $p$.  Then $k/\mbb{F}_p$ in a unique manner.  Thus if we fix a refined class formation for $\mbb{F}_p$ (e.g.\ by fixing an algebraic closure, as above), then from Example \ref{newrcf} we get in a canonical way an induced refined class formation for $k$ (even if we don't embed $k$ into our choice of $\overline{\mbb{F}_p}$), simply by restriction on the level of cosheaves.
\item Now let $K$ be a nonarchimedean local field, with residue field $k$, a finite field.  By local class field theory, there is a canonical condensed class formation given by $(\mc{C}_K,\mbb{G}_m)$ where $\mc{C}_K$ is the Galois category of finite etale $K$-schemes and $\mbb{G}_m$ is defined as usual by
$$\mbb{G}_m(L) = L^\times$$
for a finite separable extension $L/K$.  Moreover, there is a unique isomorphism
$$\on{inv}:H^2(\mc{C}_K;\mbb{G}_m)\simeq\mbb{Q}/\mbb{Z}$$
such that the associated Artin map
$$K^\times \to G_K^{ab}$$
sends every uniformizer to some Galois element inducing (arithmetic) Frobenius on each unramified extension.  The 2-groupoid of lifts to a condensed refined class formation inducing the above $\on{inv}$ has:
\begin{enumerate}
\item $\pi_0=\ast$;
\item $\pi_1=\on{Ext}(\mbb{Q},K^\times)=\wh{\mbb{Z}}/\mbb{Z}$;
\item $\pi_2=\on{Hom}(\mbb{Q},K^\times)=0$.
\end{enumerate}
 Note that this is exactly the same ambiguity as in the finite field case.  This can be explained as follows.  Suppose we choose a refined class formation inducing $(\mc{C}_K,\mbb{G}_m,\on{inv})$.  Note that we have a fully faithful embedding $\mc{C}_k\subset \mc{C}_K$ giving the unramified extensions, and the norm-1 subgroup $\mbb{G}_m^1$ is cohomologically trivial on $\mc{C}_k$, with quotient $\mbb{G}_m/\mbb{G}_m^1\simeq\mbb{Z}$ via the normalized valuation (uniformizer $\mapsto 1$).  Applying Example \ref{newrcf}, we get an associated refined class formation on $\mc{C}_k$ with underlying sheaf $\mbb{Z}$.  This induces an equivalence of 2-groupoids between the refinements of class field theory for $K$ and for $k$.  In other words, if we fix our refined class formation for $k$ as above, we canonically get an induced refined class formation for $K$ as well.
 
 This is more clear from the perspective of Weil groupoids: a choice of Weil groupoid $\mc{W}_k\to \mc{G}_k$ for $k$ canonically induces one for $K$ by taking $\mc{W}_K:= \mc{W}_k\times_{\mc{G}_k}\mc{G}_K$.  But we can also see it from the perspective of refined class formations.  To take the most basic example of $K=\mbb{Q}_p$, suppose we fix an $\overline{\mbb{F}_p}$ as in the discussion of finite fields.  Then we can write down a refined class formation for $\mbb{Q}_p$ as follows:  It sends a finite extension $K$ of $\mbb{Q}_p$ to the two-term complex of abelian groups
 $$(K\otimes_{\mbb{Z}_p}W(\overline{\mbb{F}_p}))^\times \overset{\varphi -1}{\longrightarrow} (K\otimes_{\mbb{Z}_p}W(\overline{\mbb{F}_p}))^\times,$$
 with functoriality given by norm maps.  The augmentation to $\mbb{Z}$ is given by taking $p$-adic valuation on the term in degree $0$ (for $K=\mbb{Q}_p$, then in general by composing with the norm functoriality).  The isomorphism with $\mbb{G}_m^{tr}$ in degree $1$ is induced by the obvious inclusion $K^\times\subset(K\otimes_{\mbb{Z}_p}W(\overline{\mbb{F}_p}))^\times$.

 Note that if we fix the refined class formations on finite fields via the above-discussed procedure of first fixing them for $\mbb{F}_p$ then inducing them on any finite field, then fix them on nonarchimedean local fields by compatibility with the finite field case using unramified extensions as above, then we automatically get compatibility under extensions of nonarchimedean local fields; this is again most clear from the Weil groupoid perspective.

Yet another perspective on refinements of local class field theory for non-archimedean local fields (say, of characteristic zero for simplicity) is as follows.  Recall from Theorem \ref{cfvsrcf} that refining $(\mc{C}_K,\mbb{G}_m,\on{inv})$ is equivalent to giving a map in $\on{D}(\mbb{Z})$
 $$\mbb{Q} \to R\Gamma(\mc{C}_K;\mbb{G}_m[2])$$
 which induces $\on{inv}$ in degree $0$.  This data can be reinterpreted as that of an pro-etale $\widetilde{\mbb{G}_m}$-gerbe over $\on{Spec}(K)$ where  $\widetilde{\mbb{G}_m}= \ul{\on{Hom}}(\mbb{Q},\mbb{G}_m)$, or equivalently (see \cite{tannakian}) as that of an $K$-linear $\mbb{Q}$-graded Tannakian category $\mc{T}$, such that $\mc{T}$ is pro-etale locally isomorphic to the $\mbb{Q}$-graded Tannakian category of $\mbb{Q}$-graded vector spaces.  In these terms, if we fix an algebraic closure of $k$ with induced maximal unramified extension $K^{nr}$ of $K$, then the $\mbb{Q}$-graded Tannakian category corresponding to the above-described refined class formation is the usual isocrystal-type category, also known as the local Kottwitz category, c.f.\ \cite{kottwitz}, \cite{iakovenko}: for $K=\mbb{Q}_p$ it is the $\otimes$-category of finite dimensional $W(\overline{\mbb{F}_p})[\frac{1}{p}]$-vector spaces equipped with a $\varphi$-semilinear automorphism $\sigma$, with the $\mbb{Q}$-grading coming from the Dieudonne-Manin classification: for $\lambda=\frac{r}{s}\in\mbb{Q}$ written as a reduced fraction with $s>0$, the $\lambda$-graded piece is the full subcategory of objects with $\sigma^s=p^r$.

\item Now consider archimedean local fields.  Again, it suffices to consider $K=\mbb{R}$, by passing to finite extensions.  The class formation from local class field theory for $K=\mbb{R}$ can be described exactly as for nonarchimedean local fields: $\mc{C}_\mbb{R}$ is the Galois category of finite etale $\mbb{R}$-schemes, and the sheaf on $\mc{C}_\mbb{R}$ is given by $\mbb{G}_m$.  Here there is a unique
$$\on{inv}: H^2(\mc{C}_\mbb{R};\mbb{G}_m)\simeq \left(\frac{1}{2}\mbb{Z}\right)/\mbb{Z},$$
as there is a unique isomorphism between any two groups of order two.  The 2-groupoid of refinements of this class formation has:
\begin{enumerate}
\item $\pi_0=\ast$;
\item $\pi_1=\on{Ext}(\frac{1}{2}\mbb{Z};\mbb{R}^\times)=0$;
\item $\pi_2=\on{Hom}(\frac{1}{2}\mbb{Z};\mbb{R}^\times)=\mbb{R}^\times$.
\end{enumerate}
Again using the trick of passing to $\mbb{G}_m^1$ using $\mbb{R}_{>0}\oplus \mbb{G}_m^1= \mbb{G}_m$ as in the case $F=\mbb{Q}$ above, we can reduce the $\mbb{R}^\times$-ambiguity to a $\pm 1$-ambiguity.  Thus we get a 2-groupoid of refinements of class field theory over $\mbb{R}$ which is equivalent to $B^2\pm 1$: any two objects are isomorphic via an isomorphism which is unique up to homotopy, but that homotopy between isomorphisms is only unique up to sign.

Another way to understand this ambiguity is as follows.  Suppose we fix a choice of nontrivial division algebra $\mbb{H}$ over $\mbb{R}$.  Then we get an induced Weil groupoid (hence refined class formation, as we will see in the next section) as follows: the objects are the pairs $(\mathbf{C},\alpha)$ where $\mathbf{C}$ is an algebraic closure of $\mbb{R}$ and $\alpha$ is a Morita equivalence $\mbb{H}\otimes_{\mbb{R}}\mathbf{C}\sim \mathbf{C}$ (i.e.\ a $\mathbf{C}$-linear equivalence of categories of modules); morphisms are obvious.  Clearly this groupoid only depends on the Morita class of $\mathbb{H}$, and via this we see that giving a refined class formation for $\mbb{R}$ is equivalent to giving a representative for the non-trivial Morita class over $\mbb{R}$.  To account for the reduction to $\pm 1$, we can incorporate Hermitian data in the categories.

Again there is a (closely related) Tannakian perspective: the data of a refinement of class field theory over $\mbb{R}$ is equivalent to the data of a $\frac{1}{2}\mbb{Z}$-graded Tannakian category over $\mbb{R}$ which is locally but not globally isomorphic to $\frac{1}{2}\mbb{Z}$-graded vector spaces.  As in the non-archimedean case, there is an isocrystal-style description of such a Tannakian category, provided we fix an algebraic closure $\mbb{C}$ of $\mbb{R}$.  Namely it is the category of $\frac{1}{2}\mbb{Z}$-graded $\mbb{C}$-vector spaces $V$ equipped with a semi-linear automorphism $\sigma$ such that $\sigma^2 = (-1)^{2\lambda}$ in graded degree $\lambda$.  Again this is also known as the local Kottwitz category for $\mbb{R}$, see \cite{kottwitz}, \cite{iakovenko}.

\item Finally, there is the case of function fields $K$ over finite fields (and their ``restricted ramification" variants).  This is not our focus so we do not go into details, but the situation is basically the same as that of nonarchimedean local fields: the ambiguity in choice of refined class formation for $K$ reduces to that for its finite field of constants $k$.
\end{enumerate}

\end{example}

One thing missing from the above discussion of examples is the quite nontrivial question of local-global compatibility in the number field case (in the function field case it is trivial). To address this we need to explain the notion of a morphism of refined class formations.

\begin{definition}\label{maprcf}
Let $(\mc{C},F,\alpha)$ and $(\mc{C}',F',\alpha')$ be two refined class formations (condensed or not: the discussion is the same in either case).  A morphism 
$$f:(\mc{C},F,\alpha)\to (\mc{C}',F',\alpha')$$
consists of the following data:
\begin{enumerate}
\item An exact functor $f^{-1}:\mc{C}'\to\mc{C}$;
\item A natural transformation $f_\natural: f_\ast F \to F'$;
\end{enumerate}
subject to the condition that $\alpha'\circ f_\natural$ should agree with the map induced on $f_\ast$ by $\alpha:F\to\mbb{Z}$.

Here $f_\ast G := G\circ f^{-1}$ for any functor $G$ from $\mc{C}$ to some target $\infty$-category; note in particular that $f_\ast$ sends (co)sheaves to (co)sheaves, and sends constant functors to constant functors.  (It also commutes with the $(-)^{tr}$ operation from Definition \ref{trfunctor}, since $f^{-1}$, being exact, also induces an additive functor on span categories.)
\end{definition}

\begin{remark}
An obvious elaboration of this definition makes refined class formations into a $(2,1)$-category.
\end{remark}

Recall from Theorem \ref{cfvsrcf} that a refined class formation $(\mc{C},F,\alpha)$ can be equivalently encoded as the data $(\mc{C},A,\kappa)$ of a class formation plus a map
$$\kappa: \frac{1}{\#\mc{C}}\mbb{Z} \to R\Gamma(\mc{C};A[2]),$$
and similarly $(\mc{C}',F',\alpha')$ can be encoded as $(\mc{C}',A',\kappa')$.  We can analyze what it means, in these terms, to give a map of refined class formations, by passing to $(-)^{tr}$ and considering the sheafified Postnikov tower, as in the proof of Theorem \ref{cfvsrcf}.  For simplicity, we will restrict to the situation when the first datum $f^{-1}$ in Definition \ref{maprcf} Galois-corresponds to an \emph{injective} homomorphism $f:\Gamma\to\Gamma'$ of profinite groups.  In that case, we can analyze the possibilities for
$$f_\natural:f_\ast F\to F'$$
as follows.  First, applying $(-)^{tr}$ we see it is equivalent to give
$$f_\natural:f_\ast (F^{tr}) \to (F')^{tr},$$
a map of $\on{D}(\mbb{Z})$-valued sheaves.  Recall that $F^{tr}$ has just two nonvanishing sheafified homology groups: in degree $0$, given by $\frac{1}{\#\mc{C}}\mbb{Z}$, and in degree $1$, given by $A$; and similarly for $(F')^{tr}$.  Because of our simplifying assumption on $f^{-1}$, the functor $f_\ast$ is t-exact with respect to the natural t-structure on $\on{D}(\mbb{Z})$-valued sheaves; thus $f_\ast(F^{tr})$ likewise only has two nonvanishing sheafified homology groups, with $H_0 = f_\ast (\frac{1}{\#\mc{C}}\mbb{Z})$ and $H_1=f_\ast A$.  Thus, giving a map $f_\natural$ as above is equivalent to giving maps of sheaves of abelian groups
$$f_\ast \frac{1}{\#\mc{C}}\mbb{Z}\to \frac{1}{\#\mc{C}'}\mbb{Z}$$
and
$$f_\ast A \to A',$$
together with the secondary data of a homotopy between the two different maps $f_\ast \frac{1}{\#\mc{C}}\mbb{Z} \to A'[2]$, one coming from $f_\ast (\frac{1}{\#\mc{C}}\mbb{Z})\to \frac{1}{\#\mc{C}'}\mbb{Z}$ composed with $\kappa'$ and the other from $f_\ast \kappa$ composed with $f_\ast A \to A'$.

The map $f_\ast\frac{1}{\#\mc{C}}\mbb{Z} \to \frac{1}{\#\mc{C}'}\mbb{Z}$ is actually determined by the condition $\alpha'\circ f_\natural = f_\ast \alpha$ in Definition \ref{maprcf}, so really we only need specify $f_\ast A\to A'$ and the homotopy.

Unwinding, we arrive at the following conclusion: to give a map $f:(\mc{C},A,\kappa) \to (\mc{C}',A',\kappa')$ with underlying exact functor $f^{-1}:\mc{C}'\to \mc{C}$ assumed to Galois-correspond to an injective homomorphism of profinite groups amounts to giving the following:

\begin{enumerate}

\item A map $f_\natural : f_\ast A \to A'$ of sheaves on $\mc{C}'$;
\item A homotopy between two specific maps
$$\varphi,\varphi':\frac{1}{\#\mc{C}}\mbb{Z} \to \varprojlim_{(x,U)} R\Gamma(\mc{C}'_{/U};A'[2]).$$
which are canonically defined in terms of the given data.

 Here $(x,U)\mapsto U$ is the filtered inverse system of connected objects of $\mc{C}'$ parametrized by pairs $(x,U)$ with $U$ a connected object of $\mc{C}'$ and $x:\ast \to f^{-1}U$ is a morphism,  and the transition maps in $\varprojlim_{(x,U)} R\Gamma(\mc{C}'_{/U};A'[2])$ are given by transfer maps.

The map $\varphi'$ is gotten as follows: note that $\kappa':\frac{1}{\#\mc{C}'}\mbb{Z} \to R\Gamma(\mc{C}';A'[2])$ naturally extends to a map of filtered inverse systems parametrized by $(x,U)$, where the source is $\frac{1}{\#\mc{C}'_{/U}}\mbb{Z}$ with transition maps the inclusions and the target is $R\Gamma(\mc{C}'_{/U};A'[2])$ with transition maps the transfer maps; then we pass to the inverse limits and restrict the source to $\frac{1}{\#\mc{C}}\mbb{Z}$ to get
$$\varphi':\frac{1}{\#\mc{C}}\mbb{Z} \to \varprojlim_{(x,U)} R\Gamma(\mc{C}'_{/U};A'[2]).$$

The map $\varphi$ is the composition
$$\frac{1}{\#\mc{C}}\mbb{Z} \to R\Gamma(\mc{C};A[2])\simeq R\Gamma(\mc{C}';f_\ast A[2])\to\varprojlim_{(x,U)}R\Gamma(\mc{C}'_{/U};A'[2]),$$
where the last map is induced by $f_\natural$.
\end{enumerate}

\begin{remark}
Exactly as in \ref{condautomatic}, morphisms of condensed refined class formations are the same as morphisms of underlying abstract refined class formations together with a promotion of $f_\natural$ to a map of sheaves of condensed abelian groups.
\end{remark}

From this discussion we conclude the following.

\begin{corollary}
Suppose given two refined class formations, encoded in data $(\mc{C},A,\kappa)$ and $(\mc{C}',A',\kappa')$ as in Theorem \ref{cfvsrcf}.  Suppose also given the following data:
\begin{enumerate}
\item an exact functor $f^{-1}:\mc{C}'\to\mc{C}$, assumed to Galois-correspond to an injective homomorphism of profinite groups;
\item a map of sheaves $f_\natural:f_\ast A \to A'$.
\end{enumerate}
Finally, suppose the following, with notations as above:
\begin{enumerate}
\item $R^1\varprojlim_{(x,U)} A'(U)=0$;
\item $R^2\varprojlim_{(x,U)} A'(U)=0.$
\end{enumerate}
Then there exists a morphism of refined class formations $(\mc{C},A,\kappa) \to (\mc{C}',A',\kappa')$ inducing the given data if and only if for the associated invariant maps
$$\on{inv}:H^2(\mc{C};A)\hookrightarrow\mbb{Q}/\mbb{Z}, \text{   }\on{inv}':H^2(\mc{C}';A')\hookrightarrow\mbb{Q}/\mbb{Z},$$
we have that $\on{inv}$ is equal to the composition
$$H^2(\mc{C};A) = H^2(\mc{C}';f_\ast A)\overset{f_\natural}{\longrightarrow}  H^2(\mc{C}';A')\overset{\on{inv}'}{\longrightarrow}\mbb{Q}/\mbb{Z}.$$

If this condition is satisfied, then the groupoid of maps of refined class formations inducing the given data has:
\begin{enumerate}
\item $\pi_0\simeq \on{Ext}\left(\frac{1}{\#\mc{C}}\mbb{Z},\varprojlim_{(x,U)}A'(U)\right);$

\item $\pi_1 \simeq \on{Hom}\left(\frac{1}{\#\mc{C}}\mbb{Z}, \varprojlim_{(x,U)} A'(U)\right)$.
\end{enumerate}
\end{corollary}

Let us now apply this to the question of local-global compatibility.  Fix the canonical refined class formation $(\mc{C}_\mbb{Q},F_\mbb{Q},\alpha_\mbb{Q})$ for global class field theory for number fields form Example \ref{exrcf}.  Fix also some choice of refined class formation for $\mbb{F}_p$ (e.g.\ by choosing an $\overline{\mbb{F}_p}$) and use this to get an induced refined class formation $(\mc{C}_{\mbb{Q}_p},F_{\mbb{Q}_p},\alpha_{\mbb{Q}_p})$ for $\mbb{Q}_p$ as in Example \ref{exrcf}.  We have the obvious exact functor
$$f^{-1}:\mc{C}_\mbb{Q} \to \mc{C}_{\mbb{Q}_p}$$
from the finite etale site of $\mbb{Q}$ to that of $\mbb{Q}_p$, induced by base-change, and we have the obvious map of sheaves of condensed abelian groups
$$f_\natural: f_\ast A_{\mbb{Q}_p}\to A_\mbb{Q}$$
given by the inclusion of the $p$-adic components of the ideles followed by projection to the idele class group.  The pro-system of $(x,U)$ as above identifies with the pro-system of pairs $(\nu,\on{Spec}(K))$ where $K$ is a number field and $\nu$ is a place of $K$ such that $K_\nu\simeq\mbb{Q}_p$ (such an isomorphism being unique if it exists).

Classical local-global compatibility in cohomological class field theory shows that the global $\on{inv}$ restricts to the local $\on{inv}$ under this map $f_\natural$.  Moreover we have the required $R^i\on{lim}$ vanishing because idele class groups are almost compact with non-compact part equal to $\mbb{R}_{>0}$, and the vanishing is automatic in the compact case by Lemma \ref{rlim}.  Thus we can apply the above analysis to conclude that the groupoid of maps of condensed refined class formations $(\mc{C}_{\mbb{Q}_p},F_{\mbb{Q}_p},\alpha_{\mbb{Q}_p})\to (\mc{C}_{\mbb{Q}},F_\mbb{Q},\alpha_{\mbb{Q}})$ inducing the above $(f^{-1},f_\natural)$ on underlying class formations has:
\begin{enumerate}
\item $\pi_0 \simeq \on{Ext}(\mbb{Q},\varprojlim_{(\nu,K)} \mbb{A}_K^\times / F^\times) = \ast$;
\item $\pi_1 \simeq \on{Hom}(\mbb{Q},\varprojlim_{(\nu,K)} \mbb{A}_K^\times / K^\times)$.
\end{enumerate}
Thus there is a map of refined class formations witnessing local-global compatibility, and such a map is unique up to homotopy, but such a homotopy is highly non-unique (coming from the connected components of the identity in the idele class groups).

Note that the uniquness up to homotopy of such a map is at least slightly surprising, given that, as we saw, the local refined class formations by themselves have $\wh{\mbb{Z}}/\mbb{Z}$ as automorphisms (inducing the identity on underlying class formations).  But this analysis shows that these automorphisms preserve any choice of local-to-global map, up to homotopy.

Similarly, after fixing a condensed refined class formation $(\mc{C}_{\mbb{R}},F_{\mbb{R}},\alpha_{\mbb{R}})$ we can analyze the groupoid of maps
$$(\mc{C}_{\mbb{R}},F_{\mbb{R}},\alpha_{\mbb{R}}) \to (\mc{C}_{\mbb{Q}},F_{\mbb{Q}},\alpha_{\mbb{Q}}).$$
inducing the obvious map of underlying class formations exactly as in the $p$-adic case.  It has
\begin{enumerate}
\item $\pi_0 \simeq \on{Ext}(\frac{1}{2}\mbb{Z},\varprojlim_{(\nu,K)}\mbb{A}_K^\times/K^\times)=\ast$;
\item $\pi_1\simeq \on{Hom}(\frac{1}{2}\mbb{Z},\varprojlim_{(\nu,K)}\mbb{A}_K^\times/K^\times).$
\end{enumerate}
Here the limit is taken along norm maps and runs over number fields $K$ equipped with a real place $\nu$.  Thus, just as in the case of $p$-adic local-global compatibility, there is a choice which is unique up to a highly non-unique homotopy.
\begin{remark}\label{inducedlocalglobal}
For each prime $p$, let us fix such a map of condensed refined class formations
$$f_p: (\mc{C}_{\mbb{Q}_p},F_{\mbb{Q}_p},\alpha_{\mbb{Q}_p}) \to (\mc{C}_\mbb{Q},F_\mbb{Q},
\alpha_\mbb{Q})$$
given by a map $(f_p)_\natural: (f_p)_\ast F_{\mbb{Q}_p} \to F_{\mbb{Q}}$ in $\on{Fun}(\mc{C}_\mbb{Q},\on{D}(\on{CondAb}))$, and similarly let us fix
$$f_\infty : (\mc{C}_{\mbb{R}},F_{\mbb{R}},\alpha_{\mbb{R}}) \to (\mc{C}_\mbb{Q},F_\mbb{Q},
\alpha_\mbb{Q}).$$

Then for any number field $K$ and any place $\nu$ of $K$, we get an induced map of condensed refined class formations
$$f_\nu: (\mc{C}_{K_\nu},F_{K_\nu},\alpha_{F_\nu}) \to (\mc{C}_K,F_K,
\alpha_K)$$
simply by restriction, and there are the evident compatibilities between these local-global maps as $K$ varies, again just by restriction.  Moreover, there is another compatibility, that we can produce ``for free": if $p$ is a prime, we have on the one hand the refined class formation
$$(\mc{C}_{\mbb{Z}_{(p)}},F_{\mbb{Z}_{(p)}},\alpha_{\mbb{Z}_{(p)}})$$
for number fields unramified over $p$, and on the other hand the refined class formation
$$(\mc{C}_{\mbb{F}_{p}},F_{\mbb{F}_{p}},\alpha_{\mbb{F}_{p}})$$
for $\mbb{F}_p$, both discused in Example \ref{exrcf}.  Then using our fixed $f_p$ we can make a map of refined class formations
$$(\mc{C}_{\mbb{F}_{p}},F_{\mbb{F}_{p}},\alpha_{\mbb{F}_{p}})\to (\mc{C}_{\mbb{Z}_{(p)}},F_{\mbb{Z}_{(p)}},\alpha_{\mbb{Z}_{(p)}}):$$
namely, we restrict $(f_p)_\natural$ from $\mc{C}_{\mbb{Q}}$ to $\mc{C}_{\mbb{Z}_{(p)}}$ then take the cone of the natural map from the norm 1 subgroup of the local unit group shifted by 1.  From this construction we tautologically get a commutative square of (condensed) refined class formations comparing this map to $f_p$.  There is a similar situation for any nonarchimedean place of any number field, again obtained by restriction.
\end{remark}

\begin{remark}\label{globalkottwitz}
We take all the data fixed as above: refined class formations for both local and global fields, together with the local-global maps.   Taking the direct sum over all places of $\mbb{Q}$, we deduce a map of $\on{D}(\on{CondAb})$-valued cosheaves
$$\bigoplus_\nu (f_\nu)_\ast F_{\mbb{Q}_\nu} \to F_\mbb{Q}.$$
Note that the source satisfies that on $H_1$ it gives the co-presheaf sending a number field $K$ to $\oplus_\mu K_\mu^\times$ (with morphisms given by the transfer/norm maps).  On $H_1$, the above map factors as
$$\oplus_\mu K_\mu^\times\hookrightarrow \mbb{A}_K^\times \twoheadrightarrow \mbb{A}_K^\times/K^\times.$$
Pushing out via the first inclusion, we produce from $\bigoplus_\nu (f_\nu)_\ast F_{\mbb{Q}_\nu}$ a new co-presheaf $F_{\mbb{A}}$ with $H_0(K)=\oplus_{\nu}\mbb{Z}$ and $H_1$ given by the idele groups (with norm map functoriality), such that we can uniquely extend the above map of cosheaves to a map
$$F_{\mbb{A}} \to F_{\mbb{Q}}$$
of co-presheaves inducing the above factorization on $H_1$.  In fact $F_{\mbb{A}}$ is also a cosheaf, because the quotient $K\mapsto \mathbb{A}_K^\times/\oplus_\mu K_\mu^\times$ is easily seen to be cohomologically trivial.  Thus the fiber $F_{\on{Kott}}$ of this map $F_\mbb{A}\to F_{\mbb{Q}}$ is a cosheaf, and we can apply $(-)^{tr}$ to it (Defintion \ref{trfunctor}) to get a $\on{D}(\mbb{Z})$-valued sheaf $(F_{\on{Kott}})^{tr}$ on $\mc{C}_\mbb{Q}$ with the following sheafified homology groups:
$$H^{shf}_0(F_{\on{Kott}}^{tr})(K) = \on{ker}\left(\bigoplus_{\mu\in S_f(K)}\mbb{Q}\oplus \bigoplus_{\mu\in S_{\infty}(K)}\frac{1}{2}\mbb{Z}\to \mbb{Q}\right)$$
$$H_1(F_{\on{Kott}}^{tr}) = \mbb{G}_m.$$

Here $K$ is a number field and $S_f(K)$ (resp.\ $S_\infty(K)$) is the set of non-archimedean (resp.\ archimedean) places of $K$.  Note that $H^{shf}_0(F_{\on{Kott}}^{tr})$ is a sheaf of torsionfree abelian groups on $\mc{C}_{\mbb{Q}}$, hence there is a canonically determined pro-torus over $\mbb{Q}$ with character group $H^{shf}_0(F_{\on{Kott}}^{tr})$, call it $\mathbb{D}_\mbb{Q}$.  The datum of a sheaf $F^{tr}_{Kott}$ equipped with the above identification of homology groups is equivalent to the datum of a pro-etale $\mbb{D}_{\mbb{Q}}$-gerbe over $\on{Spec}(\mbb{Q})$.  This matches the gerbe constructed by Kottwitz in \cite{kottwitz}, by direct comparison.

We already saw above that the datum of the local Kottwitz gerbes is equivalent to the datum of refined local class formations (or equivalently local Weil groupoids); now we see that the datum of global Kottwitz gerbes does \emph{not} correspond directly to that of refined global class formations (or global Weil groupoids): rather, the data of global Kottwitz gerbes plus compatibility with the local Kottwitz gerbes is equivalent to the data of global refined class formation plus compatibility with the local refined class formations.  In the above presentation, the local data is contained in $F_{\mbb{A}}$, the global Kottwitz data together with compatibility is encoded in the map $F_{Kott} \to F_{\mbb{A}}$, the global Weil data together with compatibility is encoded in the map $F_{\mbb{A}} \to F_{\mbb{Q}}$, and the equivalence between the Kottwitz and Weil data in total is an instance of the axiom of a stable $\infty$-category that cofiber and fiber sequences are the same thing.  Of course in some sense this all goes back to Nakayama-Tate, \cite{tatetori}.  For a more explicit cocycle-based discussion of the same situation, see also \cite{taylorcocycles}.
\end{remark}

\begin{remark}
It causes some philosophical discomfort that the above data appears to not be canonical.  It's also annoying that the ``construction'' of this data is so indirect.  Here is a clean statement which would automatically give all of the above structure.  Let $\on{NRing}$ denote the full subcategory of locally compact Hausdorff commutative rings spanned by the rings commonly used in number theory: the global and local fields, their rings of ($S$-)integers, and the finite fields.  Then all of the above-discussed structure (refined class field theory for all number-theoretic rings, plus local-global compatibility) amounts to the following data:
\begin{enumerate}
\item A functor $F:\on{NRing}^{op} \to \on{D}(\mbb{Z})$ with values concentrated in degrees $[0,1
]$;
\item An isomorphism $H_0F\simeq\mbb{Z}$ to the constant functor with value $\mbb{Z}$;
\item An isomorphism $H_1F(R)\simeq C_R$ to the usual class group assigned to $R$, such that the functoriality in $R$ gives the ``usual" maps between class groups.
\end{enumerate}
This functor $F$ should satisfy the property that it gives a cosheaf on the finite etale site of any number ring, and that its restriction to such a finite etale site recovers ``the'' refined class formations as discussed above.

Wouldn't it be nice if there were some direct construction of such a functor $F$?  The article \cite{clausenartin} goes part of the way towards this, by showing that $R\mapsto K_1(\on{lc}_R)$ (the algebraic $K_1$-group of the category of locally compact Hausdorff $R$-modules) serves as the appropriate $H_1F$ (see also \cite{braunling}).  However $K_0(\on{lc}_R)$ is far from being constant with value $\mbb{Z}$, so a serious modification of this idea is required to realize the above hope.

Possibly this is too much to expect, even just for finite fields.  But for number fields there \emph{is} a posteriori a canonical such functor by Example \ref{exrcf}, so it makes sense to ask to describe it directly, without resorting to cohomological analysis.  This is surely a difficult problem, but one could start just with cyclotomic fields where special features may be helpful.
\end{remark}

\begin{remark}
For those who prefer sheaves to cosheaves, it may make sense to remember the condensed structure on $F$ and pass to the Pontryagin dual.  Then you're looking for a derived sheaf with $H_0=\mbb{R}/\mbb{Z}$ and $H_{-1}$ the Pontryagin dual of the class group.  Note that this derived sheaf is almost discrete.  Actually, in the number field setting, if we use the compact form of the class groups and take the fiber of the map from $\mbb{R}$ in degree $0$, we see that the data we're looking for is just a derived sheaf with $\mbb{Z}$ in degree $1$ and the (discrete) dual to the norm-1 idele class group in degree $-1$: entirely discrete, but it completely recovers global class field theory.
\end{remark}

\section{From refined class formations to Weil groupoids}

Now we turn to the problem of assigning Weil groupoids to refined class formations, essentially redoing the analysis in \cite{artintate} from a slightly different perspective.  First let's discuss a more basic instance of the same idea: associating Galois groupoids to Galois categories.  We know from Grothendieck Galois theory that that the $2$-category of Galois categories (objects: Galois categories; 1-morphisms: exact functors; 2-morphisms: natural transformations) is a $(2,1)$-category which identifies with the opposite of the $(2,1)$-category whose objects are the profinite groups $G$, whose morphisms are the continuous homomorphisms $G\to G'$, and whose set of $2$-morphisms
$$\varphi_0\simeq \varphi_1: G\to G'$$
between two such homomorphisms is the set of $g'\in G'$ such that $\varphi_1 = g'\varphi_0 (g')^{-1}$.  This same $2$-category can also be realized as the full subcategory of $\on{CondAn}$ spanned by the objects of the form $BG$ where $G$ is a profinite group.

Our first order of business is to give a more direct description of this fully faithful embedding.  This is based on the following standard definition.

\begin{definition}
A map of condensed anima $f:X \to Y$ is \emph{finite etale} if, locally on $Y$ in the condensed Grothendieck topology, it is isomorphic to a map of the form $I\times Y \to Y$ where $I$ is a finite set.
\end{definition}

\begin{remark}
It follows from descent theory that finite etale maps $X\to Y$ in $\on{CondAn}$ are classified by maps
$$Y \to \on{fSet}^\simeq \simeq \bigsqcup_{n\geq 0} \ast/\Sigma_n$$
in $\on{CondAn}$.
\end{remark}

\begin{example}
\begin{enumerate}
\item Suppose $S$ is an extremally disconnected profinite set.  By definition every cover of $S$ in the condensed Grothendieck topology is refined by a finite disjoint union decomposition of $S$.  It follows that the finite etale maps $X\to S$ are exactly the finite disjoint unions of clopen subsets of $S$.
\item Suppose $S$ is a compact Hausdorff space.  By non-abelian cohomological descent applied to a hypercover by extremally disconnected profinite sets, we deduce that maps $S \to BG$ in $\on{CondAn}$ for any group $G$ 
are the same as global sections $\Gamma(S;BG)$ in the sense of non-abelian sheaf cohomology of the topological space $S$.  It follows that every finite etale cover of $S$ is locally trivial in the open cover topology of $S$, and hence that finite etale maps to $S$ correspond to finite-sheeted covering spaces in the usual sense from point-set topology.
\item Suppose $G$ is a locally compact Hausdroff group.  Then finite etale maps to $BG$ identify with finite sets with continuous $G$-action.  This is obvious by descent along $\ast\to BG$.
\item For a general condensed anima $S$, every finite etale map $X\to S$ uniquely descends to the Postnikov truncation $\tau_{\leq 1}S$.
\end{enumerate}
\end{example}

For $S\in\on{CondAn}$, the full sub-$\infty$-category
$$\on{fEt}_S\subset \on{CondAn}_{/S}$$
is a $1$-category which admits all finite limits and colimits.  (The above inclusion preserves finite limits; for colimits in general one has to apply the relative Postnikov truncation $\tau_{\leq 1}$ to the colimit in $\on{CondAn}_{/S}$, but for quotients by finite groups acting faithfully this last truncation is unnecessary.)  If $T\to S$ is a map, then pullback induces a functor
$$\on{fEt}_S\to \on{fEt}_T$$
which is exact, i.e.\ preserves all finite limits and colimits.   This category $\on{fEt}_{S}$ satisfies all the same exactness properties as the category of finite sets, if these are expressed in terms of interactions between finite limits and finite colimits. All of this is general topos-theory nonsense, using descent and the fact that base-change functors preserve finite limits, all colimits, and relative Postnikov truncations (which can be built in terms of finite limits and geometric realizations).

The only thing preventing $\on{fEt}_S$ from being a Galois category is the potential lack of conservative fiber functor.  In other words, $\on{fEt}_S$ is the more general kind of Galois-type category, which corresponds to a profinite groupoid not necessarily having a unique isomorphism class of objects.  We will not need to go into the details of this; we simply make the following defintion.

\begin{definition}\label{galtocond}
Suppose $\mc{C}$ is a small category admitting all finite limits and colimits.  Define a condensed anima $\ul{\mc{C}}$ by sending an extremally disconnected $S$ to the groupoid of exact functors
$$\mc{C} \to \on{fEt}_S.$$
\end{definition}

\begin{remark}
By descent, we actually get a description of the whole functor of points of $\ul{\mc{C}}$: for an arbitrary $S\in\on{CondAn}$ we have
$$\ul{\mc{C}}(S) = \on{Fun}^{ex}(\mc{C},\on{fEt}_S).$$
\end{remark}

\begin{proposition}
The above functor restricted to those $\mc{C}$ which are Galois categories,
$$\mc{C} \mapsto \ul{\mc{C}}: \on{GalCat}^{op} \to \on{CondAn},$$
has the following properties:
\begin{enumerate}
\item It is fully faithful.
\item For a profinite group $\Gamma$, its value on the Galois category of finite continuous $\Gamma$-sets identifies with $B\Gamma$.
\item On the essential image of this functor (which is the full subcategory of $\on{CondAn}$ spanned by the above $B\Gamma$) the inverse functor is given by sending $S\mapsto \on{fEt}_{/S}$.
\end{enumerate}
\end{proposition}
\begin{proof}
It is tautological from descent that the functor $\mc{C} \mapsto \ul{\mc{C}}$ admits the left adjoint
$$S\mapsto \on{fEt}_S.$$
In particular, for a small category with finite limits and colimits $\mc{C}$ we have an obvious exact comparison functor
$$\mc{C} \to \on{fEt}_{\ul{\mc{C}}}$$
and for a condensed anima $X$ we have an obvious map
$$X \to \ul{\on{fEt}_X}.$$

Consider now the case of the Galois cateogry $\mc{C}_\Gamma$ of fintie continuous $\Gamma$-sets for a profinite group $\Gamma$.  As discussed above we have $\on{fEt}_{B\Gamma}\simeq \mc{C}_\Gamma$, which gives a natural comparison functor
$$B\Gamma \to \ul{\mc{C}_\Gamma},$$
and to prove all the claims it suffices to show this is an isomorphism for all $\Gamma$.
Note that both functors $\Gamma \mapsto B\Gamma$ and $\Gamma \mapsto \ul{\mc{C}_\Gamma}$ commute with filtered inverse limits.  For the latter this is clear, and for the former see Lemma \ref{nonabrlim}.  Thus we can reduce to the case where $\Gamma$ is a finite group, where we need to see that giving a $\Gamma$-torsor over a profinite set $S$ is the same as giving an exact functor from finite $\Gamma$-sets to $\on{fEt}_{/S}$.  This is standard Galois theory.
\end{proof}

Thus we have our direct description of the fully faithful contravariant embedding of Galois categories into condensed anima.  Now we add refined class formations into the picture.  To do this, it will be convenient to ``de-linearize" refined class formations using the following.

\begin{lemma}\label{gerbe}
Consider the $\infty$-category $\mc{A}$ of pairs $(M,\alpha)$ where $M\in\on{D}(\on{CondAb})$ lives in degrees $[0,1]$ and $\alpha:H_0(M)\simeq\mbb{Z}$.  Then:
\begin{enumerate}
\item There is a fully faithful embedding $\mc{A} \to \on{CondAn}$, sending
$$(M,\alpha) \mapsto (\Omega^\infty M)\times_{\mbb{Z}} \{1\},$$
where the map $\Omega^\infty M\to\mbb{Z}$ is induced by $\alpha.$
\item The essential image of this functor consists of the \emph{abelian gerbes}: namely, those condensed anima isomorphic to $BA$ for some condensed abelian group $A$.
\item The inverse to this functor sends
$$X\mapsto (\tau_{\leq 1}(\mbb{Z}[X]),\epsilon);$$
for $X$ an abelian gerbe, where $\epsilon:\tau_{\leq 1}\mbb{Z}[X] \to \mbb{Z}$ is induced by $X\to \ast$.  More generally, for $X\in\on{CondAn}$ arbitrary and $(M,\alpha)\in\mc{A}$, we have that maps
$$(\tau_{\leq 1}(\mbb{Z}[X]),\epsilon) \to (M,\alpha)$$
(in the obvious sense) are in natural equivalence with maps
$$X \to (\Omega^\infty M)\times_{\mbb{Z}} \{1\}.$$
\item For any $X\in\on{CondAn}$ with $\pi_0X=\ast$, there is an initial abelian gerbe $X^{ab}$ with a map $X\to X^{ab}$, namely that given by
$$X^{ab} = (\tau_{\leq 1}\mbb{Z}[X])\times_{\mbb{Z}}\{1\}.$$
\end{enumerate}
\end{lemma}
\begin{proof}
In fact this works in any $\infty$-topos: one just has to say ``locally isomorphic to some $BA$'' as the definition of abelian gerbe in 2.  (In $\on{CondAn}$ every cover of $\ast$ splits so locally implies globally.)

Let's first prove the last claim in 3.  By adjunction, maps $X\to \Omega^\infty M$ are the same as maps $\mbb{Z}[X] \to M$.  It is clear that the condition that the first map land in the $\alpha=1$-component is the same as the condition that the second map intertwine $\epsilon$ and $\alpha$.  This gives the last claim in 3.

If $(M,\alpha)\in\mc{A}$, then clearly
$$\pi_0 (\Omega^\infty M)\times_{\mbb{Z}}\{1\} = \ast$$
and after any choice of basepoint $x\in (\Omega^\infty M)\times_{\mbb{Z}}\{1\}$ we get, by subtracting $x$,
$$\pi_1(\Omega^\infty M\times_{\mbb{Z}}\{1\},x)\simeq \pi_1(\Omega^\infty M,0) = H_1(M)\in \on{CondAb},$$
and thus $\Omega^\infty\times_{\mbb{Z}}\{1\}\simeq BH_1(M)$ is an abelian gerbe as claimed.  (The only thing that changes in a general $\infty$-topos is that it's not true that $\pi_0=\ast$ implies the existence of a basepoint: it only gives basepoints locally.)

Conversely,
$$BA \to \ast$$
induces the isomorphism 
$$\epsilon: H_0(\mbb{Z}[BA])\overset{\sim}{
\rightarrow} \mbb{Z}$$
as $\pi_0 BA=\ast$, whereas
$$H_1(\mbb{Z}[BA])\simeq A^{ab}=A$$
so that $(\tau_{\leq 1}\mbb{Z}[BA],\epsilon)\in\mc{A}$.  To finish the proofs of 1 and 2 we just need to see that the unit and counit are isomorphisms, i.e.\ that
$$BA\overset{\sim}{\rightarrow}(\tau_{\leq 1}\mbb{Z}[BA])\times_{\mbb{Z}}\{1\}$$
and
$$\tau_{\leq 1}\mbb{Z}[(\Omega^\infty M)\times_{\mbb{Z}}\{1\})]\overset{\sim}{\rightarrow} M,$$
but this is clear by checking on $\pi_0$ and $\pi_1$ in the first case and $H_0$ and $H_1$ in the second case.  Finally, 4 follows immediately from 1,2,3.\end{proof}

From this we get the following.

\begin{proposition}\label{rcfgerbe}
For a Galois category $\mc{C}$, the data $(F,\alpha)$ of a condensed refined class formation on $\mc{C}$ is equivalent to the datum of a functor
$$\mc{G}:\mc{C}_{conn}\to\on{CondAn}$$
where $\mc{C}_{conn}\subset\mc{C}$ is the full subcategory of connected objects, satisfying the following properties:
\begin{enumerate}
\item $\mc{G}$ lands in the full subcategory of abelian gerbes;
\item For all $X\in\mc{C}_{conn}$ and all subgroups $G\subset\on{Aut}(X)$, the map in $\on{CondAn}$
$$\mc{G}(X)_G \to \mc{G}(X/G)$$
induces an isomorphism on abelianization
$$(\mc{G}(X)_G)^{ab} \overset{\sim}{\rightarrow} \mc{G}(X/G).$$
\item For all $X\in\mc{C}_{conn}$ and all subgroups $G\subset\on{Aut}(X)$, the transfer map of condensed abelian groups
$$H_1\mc{G}(X/G)\overset{\sim}{\rightarrow} \left(H_1\mc{G}(X)\right)^G$$
(defined by passing through Lemma \ref{gerbe} using 2) is an isomorphism.
\end{enumerate}
In other words, $\mc{G}$ is the datum of a cosheaf of abelian gerbes on $\mc{C}_{conn}$ satisfying the extra condition about the transfer map being an isomorphism.

Moreover, given a condensed refined class formation $C$ encoded in such data, and another $C'$ similarly encoded, then giving a map of condensed refined class formations $C \to C'$ (Definition \ref{maprcf}) is the same as giving an exact functor $f^{-1}:\mc{C}'\to\mc{C}$ plus a natural transformation
$$f_\natural: f_\ast\mc{G} \to \mc{G}'.$$
\end{proposition}
\begin{proof}
As we reviewed in Remark \ref{whatisaderivedsheaf}, the datum of a cosheaf $F:\mc{C} \to \on{D}(\on{CondAb})$ is the same as the datum of a functor (its restriction to connected objects)
$$F:\mc{C}_{conn} \to \on{D}(\on{CondAb})$$
such that $F(X)/G\overset{\sim}{\rightarrow} F(X/G)$ for all $(X,G)$ as in 2.  But by definition, when $F$ is a refined class formation this restricted $F$ together with its $\alpha$ lands inside the full subcategory denoted $\mc{A}$ in Lemma \ref{gerbe}, and so from Lemma \ref{gerbe} we get the the data of $(F,\alpha)$ is the same as the data of $\mc{G}$.  But we have to translate the conditions carefully: under Lemma \ref{gerbe}, the cosheaf condition on $\mc{G}$ corresponds to saying that $F$ is a cosheaf with values in $\on{D}(\on{CondAb})_{[0,1]}$, but the actual condition is that $F$ be a cosheaf with values in $\on{D}(\on{CondAb})$ which happens to take values in $\on{D}(\on{CondAb})_{[0,1]}$.  We already know from the transfer formalism that this stronger condition on $F$ implies the condition on the transfer maps being isomorphisms in 2.  What remains is to show the converse, i.e.\ that if
$$F:\mc{C}\to \on{D}(\on{CondAb})$$
is a copresheaf taking values in degrees $[0,1]$ such that for all $(X,G)$ as in 2, the map
$$F(X)_G \to F(X/G)$$
induces
$$\tau_{\leq 1}(F(X)_G)\overset{\sim}{\rightarrow} F(X/G)$$
and that the transfer map
$$H_1(X/G) \to (H_1F(X))^G$$
(defined via the additive category $\on{D}(\on{CondAb})_{[0,1]}$) is an isomorphism, we need to see that $F(X)_G$ already lives in degrees $[0,1]$ for all $(X,G)$.

For this, consider the diagram
$$F(X)_G \to F(X/G) \overset{tr}{\rightarrow} F(X)^G$$
where the composition is the transfer map for the $G$-action on $F(X)$.  The assumptions imply that the composition is injective on $H_0$, an isomorphism on $H_1$, and surjective on $H_2$, and this is also so for any subgroup $H\subset G$.  Thus, by Tate-Nakayama (applied to $S$-valued points, for all extremally disconnected $S$), the composition is an isomorphism.  A fortiori $F(X)_G$ lives in degrees $[0,1]$ as desired.
\end{proof}
 
Note that if $S\in\on{CondAn}$, then we have a tautological cosheaf of condensed anima on $\on{fEt}_S$, namely the one sending
$$(T\to S) \mapsto T.$$
We denote this cosheaf by $\tau_S$.  Then we make the following definition expressing how to ``realize'' refined class formations in condensed anima.

\begin{definition}\label{rcftocond}
Let $C=(\mc{C},F,\alpha)$ be a condensed refined class formation, and suppose $\mc{G}:\mc{C}_{conn}\to\on{CondAn}$ is the corresponding cosheaf of abelian gerbes, as in Proposition \ref{rcfgerbe}.  Define a condensed anima $\mathcal{W}_C$ by setting, for $S$ extremally disconnected profinite, $\mathcal{W}_C(S)$ to be the anima of pairs $(f^{-1},f_\natural)$ where:
\begin{enumerate}
\item $f^{-1}$ is an exact functor $\mc{C} \to \on{fEt}_S$;
\item $f_\natural$ is a natural transformation $f_\ast \tau_S \to \mc{G}$, i.e.\ a family of maps of condensed anima
$$f^{-1}(U)\to \mc{G}(U)$$
natural in $U\in\mc{C}_{conn}$.
\end{enumerate}
\end{definition}

\begin{remark}
Again, just as in the discussion of Galois categories above, it follows from descent that $\mathcal{W}_C(S)$ admits the same description for arbitrary $S\in\on{CondAn}$.
\end{remark}

\begin{remark}
Evidently,  this condensed anima $\mathcal{W}_C$ is 1-truncated (i.e.\ it is a condensed groupoid) and comes with a natural map
$$\mc{W}_C \to \ul{\mc{C}}$$
to the condensed groupoid $\ul{\mc{C}}$ from Definition \ref{galtocond}.
\end{remark}

\begin{remark}
If we wish to stick to the original perspective of $(F,\alpha)$ and avoid abelian gerbes, then via Lemma \ref{gerbe} we can rephrase the datum $f_\natural$ as follows: we give a map in $\on{D}(\on{CondAn})$
$$\mbb{Z}[f^{-1}(U)]\to F(U)$$
natural in $U$ which intertwines the augmentation $\epsilon$ on $\mbb{Z}[f^{-1}U]$ (induced by $f^{-1}U\to \ast$) with $\alpha:F\to\mbb{Z}$.  Here we can take $U\in\mc{C}_{conn}$ or $U$ arbitrary; it doesn't matter by the fact the everything is a cosheaf, sending finite disjoint unions to direct sums.
\end{remark}

To analyze this condensed groupoid $\mc{W}_C$, recall that $\mc{C}$ can canonically be written as a filtered colimit of Galois subcategories, each of which is equivalent to the Galois category associated to a finite group (or, canonically, a connected finite groupoid), namely
$$\mc{C} = \cup_N \mc{C}_N$$
where the indexing poset consists of all isomorphism classes of connected Galois objects of $\mc{C}$.  From Example \ref{newrcf}, we can restrict any refined class formation to these subcategories $\mc{C}_N$, and the datum of a refined class formation on $\mc{C}$ is equivalent to the data of compatible refined class formations on each $\mc{C}_N$.  We deduce the following:

\begin{lemma}\label{rcfinverselimit}
Let $C$ be a condensed refined class formation, and for $N$ indexing a finite ``quotient'' $\mc{C}_N\subset \mc{C}$ as above, let $C_N$ denote the condensed refined class formation obtained by restricting the data to $\mc{C}_N$. Then we have
$$\mc{W}_C = \varprojlim_N \mc{W}_{C_N}$$
in $\on{CondAn}$.
\end{lemma}
\begin{proof}
This is immediate from $\mc{C}=\cup_N \mc{C}_N$.
\end{proof}

In this way the study of $C$ reduces to the case where the Galois category $\mc{C}$ corresponds to a connected finite groupoid, and to the study of filtered inverse limits.  We start with the first.

In fact, the situation is quite clear in the abstract: a condensed refined class formation $C$ on $\mc{C}\simeq \mc{C}_\Gamma$ for a finite group $\Gamma$ is the same as a certain kind of abelian gerbe over $B\Gamma$, or to say things more canonically, an abelian gerbe over $\ul{\mc{C}}$.  Then the definitions are set up so that this equivalence is precisely implemented by $C\mapsto (\mc{W}_C\to \ul{\mc{C}})$.  In other words:

\begin{lemma}
Let $\mc{C}$ be a Galois category equivalent to finite $\Gamma$-sets for some finite group $\Gamma$, and suppose given a refined condensed class formation $C$ on $\mc{C}$.  Then the natural map of condensed groupoids
$$\mc{W}_C \to \ul{\mc{C}}$$
is an abelian gerbe, and the refined class formation is in turn recovered from this gerbe by 
$$F(U) = \tau_{\leq 1}\mbb{Z}[U\times_{\ul{\mc{C}}}\mc{W}_C]$$
for $U\in \on{fEt}_{\ul{\mc{C}}}=\mc{C}$, equipped with the obvious augmentation $\alpha$ induced by $U\times_{\ul{\mc{C}}}\mc{W}_C\to\ast$.

Thus, via this procedure, refined condensed class formations for $\mc{C}$ are equivalent to certain abelian gerbes over $\ul{\mc{C}}$ (namely those satisfying the conditions in Proposition \ref{rcfgerbe}).
\end{lemma}
\begin{proof}
Choose the isomorphism $\mc{C}=\mc{C}_\Gamma$.  Then the cosheaf of abelian gerbes on $\mc{C}$ corresponding to $(F,\alpha)$ via Proposition \ref{rcfgerbe} is left Kan extended from its value on $\Gamma$ as a $\Gamma$-set, and the datum is just an abelian gerbe $X\in \on{CondAb}$ equipped with a $\Gamma$-action.  By the same token an $S$-valued point of $\mc{W}_C$ amounts to a $\Gamma$-torsor $\widetilde{S}\to S$ together with a $\Gamma$-equivariant map $\widetilde{S} \to X$.  This exactly means that $\ul{C} = X/\Gamma$ mapping to $\ul{\mc{C}} = B\Gamma$, giving the claims.
\end{proof}

\begin{remark}\label{groupextension}
In more pedestrian terms, if we choose a basepoint of $\mc{W}_C$, which induces a basepoint of $\ul{\mc{C}}$ hence an isomorphism $\ul{\mc{C}}\simeq B\Gamma$ for a uniquely determined finite group $\Gamma$, then we get $\mc{W}_C\simeq BW$ where $W$ is a condensed group which is an extension of $\Gamma$ with abelian kernel $A$:

$$1 \to A \to W \to \Gamma \to 1.$$
The condensed abelian group $A$ with its induced conjugation action of $\Gamma$ describes the condensed class formation underlying $C$, and the extra datum
$$\kappa: \frac{1}{\#\Gamma}\mbb{Z} \to R\Gamma(B\Gamma;A[2])$$
is described as follows: its value on $\frac{1}{\#\Gamma}$ classifies the given abelian extension $A \to W \to \Gamma$.
\end{remark}

\begin{remark}\label{groupextensiontransition}
In this situation, if $\mc{C}'$ is a Galois full subcategory of $\mc{C}$ and we take the induced refined condensed class formation $C'$ on $\mc{C}'$ as in Example \ref{newrcf} given just by restriction, then we get the induced commutative square of condensed anima
$$\xymatrix{
\mc{W}_C\ar[r]\ar[d] & \mc{W}_{C'}\ar[d] \\
\ul{\mc{C}}\ar[r] & \ul{\mc{C}'}.\\ }$$

\noindent By Proposition \ref{rcfgerbe}, the abelian gerbe $\mc{W}_{C'}\to\ul{\mc{C}'}$ is determined as the abelianization of the composite map $\mc{W}_C \to \ul{\mc{C}'}$.

Note that the topos-theoretic notion of abelianization commutes with base change.  Thus we get the following more concrete description of the situation. If $1 \to A \to W \to \Gamma \to 1$ is the group extension corresponding to a choice of basepoint of $\mc{W}_C$, and $\Gamma\to\Gamma'$ is the quotient of $\Gamma$ describing $\mc{C}'\subset\mc{C}$, then the right-hand map corresponds to
$$W' \to \Gamma',$$
where $W' = W/[N,N]$ with $N=\on{ker}(W\twoheadrightarrow \Gamma')$, and $[N,N]$ denotes the commutator subgroup of $N$.  In particular the kernel of $W'\to \Gamma'$, which describes the band of the gerbe, is $N^{ab}$.  On the other hand from the class formation perspective we see that the band is also $A^K$, where $K=\on{ker}(\Gamma\to\Gamma')$ acts on $A:=\on{ker}(W\twoheadrightarrow\Gamma)$ by restricting the action of $\Gamma$ by conjugation.  The isomorphism $N^{ab}\simeq A^K$ relating these descriptions is induced by the transfer.
\end{remark}

Now we turn to the problem of controlling the inverse limits in 
$$\mc{W}_C = \varprojlim_N \mc{W}_{C_N}$$
from Lemma \ref{rcfinverselimit}.

\begin{proposition}\label{groupoidisgroup}
Let $C$ be a condensed refined class formation with underlying condensed class formation $(\mc{C},A)$.  Suppose that
for all connected objects $U\in\mc{C}$,
$$\on{ker}(A(U) \to A(\ast))$$
is compact Hausdorff, where the map is the transfer map associated to $U\to \ast$ (corresponding to the natural covariant functoriality of $H_1F$ via $A^{tr}=H_1F$).

Then there exists a condensed group $W$ such that
$$\mc{W}_C\simeq BW.$$
Moreover $W$ is an extension of $A(\ast)$ by a compact Hausdorff group, hence (Remark \ref{extlchaus}) is locally compact Hausdorff provided $A(\ast)$ is.
\end{proposition}

\begin{proof}
Let $N$ index the Galois full subcategories $\mc{C}_N\subset\mc{C}$ such that $\mc{C}_N$ is isomorphic to finite $\Gamma_N$-sets for some finite group $\Gamma_N$.  As in Remark \ref{groupextension}, $\mc{W}_{C_N}$ gives us, after a choice of basepoint, a condensed group extension
$$1 \to A_N \to W_N \to \Gamma_N \to 1.$$
By the special case of Remark \ref{groupextensiontransition} where $\Gamma'=1$, we have $W_N^{ab}\simeq A_N^{\Gamma_N}=A(\ast)$, and the homomorphism $W_N^{ab}\to \Gamma_N^{ab}$ is surjective with kernel $\on{tr}(A_N)\subset A(\ast)$ under this isomorphism.  From the hypothesis that $\on{ker}(\on{tr})$ is compact we deduce that $\on{ker}(W_N\to W_N^{ab}\simeq A(\ast))$ is an extension of a finite group by a compact group, hence is compact.

Thus the abelianization of $\mc{W}_{C_N}$ is independent of $N$, namely it is the trivial gerbe with band $A(\ast)$ where here $\ast$ means the terminal object of $\mc{C}$.  Moreover, the inverse system formed by the $\mc{W}_{C_N}$ maps to the constant inverse system with value $B(A(\ast))$, and the fiber of this map is an inverse system of gerbes with \emph{compact Hausdorff} band.  Recalling that $\mc{W}_C=\varprojlim_N\mc{W}_{C_N}$, we deduce from Lemma \ref{nonabrlim} a fiber sequence
$$BK \to \mc{W}_C \to B(A(\ast))$$
for some compact Hausdorff group $K$.  Thus $\mc{W}_C=BW$ with $W$ an extension of $A(\ast)$ by $K$, finishing the proof.
\end{proof}

Let us now refine this to an equivalence of categories, showing that refined class formations (of a certain type) are equivalent to ``Weil groupoids''; compare \cite{tatentb}, \cite{dasgupta} Section 5.

\begin{theorem}
The functor sending a condensed refined class formation $C$ to the map of condensed anima
$$\mc{W}_C \to \ul{\mc{C}}$$
from Definition \ref{rcftocond} induces an equivalence between the following two $(2,1)$-categories:
\begin{enumerate}
\item The full subcategory of condensed refined class formations spanned by those whose underlying condensed class formation $(\mc{C},A)$ satisfies the following: for all connected $U\in\mc{C}$, we have:
\begin{enumerate}
\item $A(U)$ is locally compact Hausdorff;
\item $\on{ker}(A(U)\to A(\ast))$ is compact.
\end{enumerate}

\item The full subcategory of $\on{Fun}(\Delta^1,\on{CondAn})$ spanned by those maps of the form $B\varphi: BW\to B\Gamma$ where $\varphi:W\to \Gamma$ is a homomorphism of locally compact Hausdorff groups such that:
\begin{enumerate}
\item $\Gamma$ is profinite and $\varphi:W\to \Gamma$ has dense image;
\item $W\overset{\sim}{\rightarrow}\varprojlim_N W/W^c_N$ where $N$ runs over all open normal subgroups of $\Gamma$; here, for $H\subset \Gamma$ an open subgroup, we let $W_H = \varphi^{-1}H$ and let $W_H^c = \overline{[W_H,W_H]}$ be the closure of the commutator subgroup of $W_H$.  Moreover we require that the transition maps in this inverse limit have compact kernel.
\item For any open subgroup $H\subset \Gamma$ and any open normal subgroup $N\subset H$, the transfer map induces an isomorphism
$$W_H/W_H^c \overset{\sim}{\rightarrow} (W_N/W_N^c)^{H/N}.$$
of locally compact Hausdorff abelian groups.
\end{enumerate}
\end{enumerate}
The functor backwards sends $X\to Y$ in $\on{Fun}(\Delta^1;\on{CondAn})$ satisfying the above conditions to the refined condensed class formation
$$(\on{fEt}_Y,\left(\tau_{\geq -1}R\Gamma(-\times_YX;\mbb{R}/\mbb{Z})\right)^\vee,\epsilon),$$
where $(-)^\vee$ stands for internal Hom to $\mathbb{R}/\mbb{Z}$ and the augmentation $\epsilon$ to $\mbb{Z}$ is induced by $-\times_YX\to\ast$.  In particular, the underlying class formation assigns $H\mapsto \varphi^{-1}H/\overline{[\varphi^{-1}H,\varphi^{-1}H]}$ equipped with the sheaf structure given by the transfer maps.
\end{theorem}
\begin{proof}
First suppose given a $C=(\mc{C},F,\alpha)$ as in 1.  By Proposition \ref{groupoidisgroup} we have that $\mc{W}_C\to\ul{\mc{C}}$ is of the form $BW\to B\Gamma$ where $W$ is a locally compact Hausdorff group which is an extension of $A(\ast)$ by a compact Hausdorff group and $\Gamma$ is a profinite group.

Suppose $\Gamma'\subset \Gamma$ is an open subgroup, and let $C'$ denote the restriction of $C$ to the corresponding slice category $\mc{C}'$ in $\mc{C}$ as in Example \ref{newrcf}.  It follows from the definitions (specifically the cosheaf property) that $$\mc{W}_{C'} = \mc{W}_C\times_{\ul{\mc{C}}}\ul{\mc{C}'}.$$
On the other hand by Proposition \ref{groupoidisgroup} we have $\mc{W}_{C'} = BW'$, and it follows that $W'= \varphi^{-1}\Gamma'\subset W$ and $W/W'\simeq \Gamma/\Gamma'$, which means that $\varphi$ has dense image.

By Lemma \ref{rcfinverselimit} we have the inverse limit of locally compact Hausdorff groups
$$W = \varprojlim_N W_N$$
where the transition maps are quotient maps by compact normal subgroups (Remark \ref{groupextensiontransition}).  Here $W_N$ is such that $\mc{W}_{C_N} = BW_N$; by Remark \ref{groupextension} it fits into an extension
$$1\to A_N \to W_N \to \Gamma/N \to 1$$
where $A_N = A(\Gamma/N)$.  Since $A_N$ is abelian and $W_N$ is Hausdorff, the kernel of the surjection $$W\to W_N$$ contains $\overline{[\varphi^{-1}N,\varphi^{-1}N]}$.  To see the opposite containment, we can pass to the open subgroup $N$ as in the previous paragraph to reduce to showing $W/\overline{[W,W]}\overset{\sim}{\rightarrow} A(\ast)$. But on the finite levels we know that $W_N^{ab}\overset{\sim}{\rightarrow} A(\ast)$ by Remark \ref{groupextensiontransition} (here without needing to take closure of the commutator subgroup in forming the abelianization), and then we get the desired claim by passing to the limit.  This proves property (b).

Finally, via the above identifications, property (c) is a translation of the fact that $A=H_1F^{tr}$ is a sheaf.

Next suppose given $\varphi:W\to \Gamma$ is as in 2, and let's check that
$$(\on{fEt}_{B\Gamma},\left(\tau_{\geq -1}(R\Gamma((-)\times_YX;\mbb{R}/\mbb{Z})\right)^\vee,\epsilon)$$
is a condensed refined class formation as claimed.  We have that $\on{fEt}_{B\Gamma}$ is a Galois category by \ref{galtocond}.  For $\Gamma'\subset \Gamma$ an open subgroup, corresponding to $B\Gamma' \to B\Gamma$ in $\on{fEt}_{B\Gamma}$, we have
$$H^0(B\Gamma'\times_{B\Gamma}BW;\mbb{R}/\mbb{Z})=\mbb{R}/\mbb{Z}$$
via $\epsilon$: this follows because $B\Gamma'\times_{B\Gamma}BW = BW'$ with $W'=\varphi^{-1}\Gamma'$ due to the hypothesis that $\varphi$ has dense image.  Moreover, for the same reason
$$H^1(B\Gamma'\times_{B\Gamma}BW;\mbb{R}/\mbb{Z})\simeq \on{Hom}_{cont}(W';\mbb{R}/\mbb{Z}).$$
When we apply Pontryagin duality, we therefore get something with $H_0 = \mbb{Z}$ and $$H_{1}(B\Gamma') = W'/\overline{[W',W']}.$$
Then it follows from Propositon \ref{rcfgerbe} that we have a refined class formation.  Unraveling the above, we see that these two procedures (going from 1 to 2 and going from 2 to 1) are mutually inverse via natural transformations, finishing the proof.
\end{proof}

Under stronger hypotheses, which are satisfied in practice for the class formations from number theory (though \emph{not} for their finite level analogs), we have an even nicer situation: $\ul{\mc{C}}$ is determined by $\mc{W}_C$.

\begin{corollary}\label{weildeterminesrcf}
The functor sending a condensed refined class formation $C = (\mc{C},F,\alpha)$ to the condensed anima $\mc{W}_C$
from Definition \ref{rcftocond} induces an equivalence between the following two $(2,1)$-categories:
\begin{enumerate}
\item The full subcategory of condensed refined class formations spanned by those $(\mc{C},F,\alpha)$ such that for all connected $U\in\mc{C}$, we have:
\begin{enumerate}
\item $H_1F(U)$ is locally compact Hausdorff;
\item $\on{ker}(H_1F(U)\to H_1F(\ast))$ is compact;
\item Every open subgroup of $H_1F(U)$ of finite index contains $\on{im}(H_1F(V)\to H_1F(U))$ for some connected $V\in\mc{C}$ mapping to $U$. (This property is the ``existence theorem'' for the underlying class formation.)
\end{enumerate}

\item The full subcategory of $\on{CondAn}$ spanned by those objects of the form $BW$ where $W$ is a locally compact Hausdorff group such that:
\begin{enumerate}
\item $W\overset{\sim}{\rightarrow}\varprojlim_N W/N^c$ where $N$ runs over all open normal subgroups of $W$ of finite index and $N^c$ is the closure of the commutator subgroup of $N$, and the transition maps in this inverse system have compact kernel;
\item For any open subgroup $H\subset W$ of finite index and any open normal subgroup $N\subset H$ of finite index, the transfer map induces an isomorphism of locally compact Hausdorff groups
$$H/H^c\overset{\sim}{\rightarrow} (N/N^c)^{H/N}.$$

\end{enumerate}
\end{enumerate}
\end{corollary}
\begin{proof}
Suppose the condensed refined class formation satisfies conditions 1(a)(b)(c).  Then we claim that the associated homomorphism $\varphi:W\to\Gamma$ from the previous theorem identifies $\Gamma$ with the profinite completion of $W$.  Because $W\to \Gamma$ has dense image, the map $\wh{W}\to\Gamma$ is surjective, so we need only see that the kernel is trivial, which amounts to the claim that every open subgroup of $W$ of finite index contains $\varphi^{-1}\Gamma'$ for some open subgroup $\Gamma'\subset \Gamma$.

However, every open subgroup of $W$ contains $(\varphi^{-1}N)^c$ for some open normal subgroup $N$ of $\Gamma$ by 2(b) in the previous theorem, and thus corresponds to an open subgroup of the quotient $W/(\varphi^{-1}N)^c$.  We can intersect with $\varphi^{-1}N/(\varphi^{-1}N)^c = H_1F(U)$ to get an open subgroup there, of finite index if our original subgroup of $W$ is, which by hypothesis 1(b) contains the image of $H_1F(U')$ for some $U'$ corresponding to some smaller normal subgroup $N'\subset N$.  It follows that our open subgroup contains $\varphi^{-1}N'$, as required.

Conversely, suppose we have $W$ satisfying 2(a) and 2(b).  Setting $\Gamma=\widehat{W}$ we get a homomorphism $\varphi:W \to \Gamma$ which satisfies the properties 2(a),(b),(c) of the previous theorem, and therefore corresponds to a condensed refined class formation, for which we need only check the existence theorem.  But an open subgroup of $H_1F(U) = H/H^c$ can be viewed as an open subgroup of $W/H^c$ when $H\subset W$ is the open subgroup corresponding to $U$, and for this open subgroup the property 1(c) is tautological.
\end{proof}

Thus nice enough refined class formations can equivalently be viewed as certain condensed groupoids.  This applies to all of the number-theoretic examples of class formations discussed above in Example \ref{exrcf}, and the corresponding condensed anima is $BW$ where $W$ is ``the'' Weil group.

In these number theoretic examples, we have a dichotomy: the underlying condensed class formations $(\mc{C},A)$ have either one or the other of the following properties.

\begin{enumerate}
\item For all connected $U\in\mc{C}$, the condensed abelian group $A(U)$ is an extension of $\mbb{Z}$ by a profinite abelian group (called the \emph{$\mbb{Z}$-case} in \cite{tatentb});
\item For all connected $U\in\mc{C}$, the condensed abelian group $A(U)$ is an extension of $\mbb{R}$ by a compact Hausdorff group (called the \emph{$\mbb{R}$-case} in \cite{tatentb}).
\end{enumerate}

The $\mbb{Z}$-case occurs in all non-archimedean settings (finite fields, nonarchimedean local fields, and function fields), while the $\mbb{R}$-case occurs in all archimedean or mixed settings  (archimedean local fields and number fields).  It is formally convenient to also allow a third simpler case, the \emph{compact} case, when $A(U)$ is compact Hausdorff for all $U$.  Thus we end up with the following trichotomy.

\begin{definition}\label{zrcompact}
Let $C$ be a condensed refined class formation which satisfies 1(a)(b)(c) from Corollary \ref{weildeterminesrcf}.  We say that:
\begin{enumerate}
\item $C$ is of $\mbb{R}$-type if $H_1F(U)$ is an extension of $\mbb{R}$ by a compact Hausdorff group for all connected $U\in\mc{C}$;
\item $C$ is of $\mbb{Z}$-type if $H_1F(U)$ is an extension of $\mbb{Z}$ by a profinite group for all connected $U\in\mc{C}$;
\item $C$ is of compact type if $H_1F(U)$ is compact Hausdorff for all connected $U\in\mc{C}$.
\end{enumerate}
For brevity, if we say that $C$ is a \emph{refined class formation of $\mbb{R}$-type, resp.\ $\mbb{Z}$-type, resp.\ compact type}, we mean that $C$ is a condensed refined class formation satisfying 1(a)(b)(c) from Corollary \ref{weildeterminesrcf} as well as 1 resp.\ 2 resp.\ 3 just above.
\end{definition}

To finish this section we analyze the $\mbb{Z}$-case; this amounts to a reprise of a special case of a theorem of Brumer relating class formations to the condition of strict cohomological dimension 2 for a profinite group, \cite{brumer} Theorem 1.

\begin{proposition}\label{ztype}
The equivalence in Corollary \ref{weildeterminesrcf} restricts to an equivalence between the following two $(2,1)$-categories:

\begin{enumerate}
\item The category of refined class formations of $\mbb{Z}$-type;
\item The full subcategory of $\on{CondAn}$ spanned by those condensed anima of the form $BW$, where $W$ is of the following type: take a profinite group $\Gamma$ of \emph{strict cohomological dimension $2$} and a surjective homomorphism $\Gamma\twoheadrightarrow \wh{\mbb{Z}}$ and set
$$W= \Gamma\times_{\wh{\mbb{Z}}}\mbb{Z}.$$
\end{enumerate}
Moreover, in this case we have $H^i(U;\mbb{R}/\mbb{Z})=0$ for all $i\geq 2$ and all $U\in\on{fEt}_{BW}$, so that the associated condensed refined class formation is given by $F=R\Gamma(-;\mbb{R}/\mbb{Z})^\vee$ without any truncation needed.
\end{proposition}
\begin{proof}
Suppose given a $W$ as in 2, corresponding to $\Gamma\twoheadrightarrow\wh{\mbb{Z}}$.  We recall the following equivalent form of the strict cohomological dimension $2$ hypothesis, see \cite{galcoh}: for every open subgroup $H\subset \Gamma$, we have
$$H^i(BH;\mbb{Q}/\mbb{Z})=0$$
for all $i\geq 2$.  (Actually, this characterizes strict cohomological dimension $\leq 2$, but if scd is $\leq 1$ then $\Gamma=\{1\}$ so there are no surjections to $\wh{\mbb{Z}}$.)  Moreover we can replace $\mbb{Q}/\mbb{Z}$ by $\mbb{R}/\mbb{Z}$ here because uniquely divisible groups have no higher cohomology.  This also passes to arbitrary closed subgroups by taking filtered colimits.

Now, consider the extension
$$1\to \Gamma^1\to\Gamma\to\wh{\mbb{Z}}\to 1.$$
In particular this induces an action of $\wh{\mbb{Z}}$ on $B\Gamma^1$ and hence on cohomology, and letting $\gamma\in\wh{\mbb{Z}}$ denote the generator we deduce that:
\begin{enumerate}
\item $H^i(B\Gamma^1;\mbb{R}/\mbb{Z})=0$ for $i\geq 2$
\item $H^1(B\Gamma^1;\mbb{R}/\mbb{Z})_{\gamma-1}=H^1(B\Gamma^1;\mbb{Q}/\mbb{Z})_{\gamma-1}=H^2(B\Gamma;\mbb{Q}/\mbb{Z})=0$;
\end{enumerate}

Thus for $W$, which fits into the extension
$$ 1\to \Gamma^1 \to W \to \mbb{Z} \to 1,$$
we have $H^i(BW;\mbb{R}/\mbb{Z})=0$ for $i\geq 2$.  Every open subgroup of $W$ of finite index is of the same form as $W$ so we can repeat this argument there to get the same vanishing. This already proves the last claim.  Moreover, the corresponding class formation has values on connected $U\in\on{fEt}_{BW}$ equal to the abelianization in locally compact Hausdorff groups of the corresponding open subgroup $H\subset W$ of finite index.  But since $W\twoheadrightarrow\mbb{Z}$ we deduce that $H\twoheadrightarrow N\mbb{Z}\simeq\mbb{Z}$ for some $N>0$ and hence the same for its abelianization, which shows that we are in situation 1.

Now let us go from 1 to 2.  From the fact that $H_1F(U)\to H_1F(\ast)$ has compact kernel, we deduce that the induced map from the $\mbb{Z}$-quotient of the former to the $\mbb{Z}$-quotient of the latter is injective, hence an isomorphism on rationalization.  Thus, if we push out $F$ via $H_1F$ mapping to its profinite completion, we get a new copresheaf $F'$ with $H_0\simeq\mbb{Z}$ and profinite $H_1$ and having a map from $F$ where the cofiber has uniquely divisible values, isomorphic to $\wh{\mbb{Z}}/\mbb{Z}[1]$ on connected objects, and where the copresheaf structure maps associated to maps between connected objects are isomorphisms.  This cofiber is thus a cosheaf, and therefore so is $F'$.  Thus we have produced a new refined class formation $F'$ for which $H_1F'(U)$ is the profinite completion of $H_1F(U)$ for all connected $U$.

Thus the underlying profinite group $\Gamma$ of our Galois category carries a condensed refined class formation with profinite $H_1$, and still satisfying the ``existence theorem'' axiom.  Let us show that this implies that $\Gamma$ has strict cohomological dimension 2 (or rather $\leq 2$).  Since the $H_1$ are profinite, from Theorem \ref{groupoidisgroup} we see that the corresponding Weil group is also profinite hence identifies with $\Gamma$ by Corollary \ref{weildeterminesrcf}.  Thus, by going backwards to refined class formations we see that the $\on{D}(\mbb{Z})$-valued sheaf
$$R\Gamma(-;\mbb{R}/\mbb{Z})$$
on $\on{fEt}_{B\Gamma}$ satisfies that its truncation $\tau_{\geq -1}$ is also a $\on{D}(\mbb{Z})$-valued sheaf.  We deduce that $\tau_{\geq -1}R\Gamma(-;\mbb{Q}/\mbb{Z})$ is likewise a sheaf.  Clearly its sheafified homology groups vanish outside degrees $-1$ and $0$ and gives $\mbb{Q}/\mbb{Z}$ in degree $0$.  Thus if we can show it vanishes in degree $-1$ as well, we will deduce that
$$\tau_{\geq -1}R\Gamma(-;\mbb{Q}/\mbb{Z}) = R\Gamma(-;\mbb{Q}/\mbb{Z})$$
and hence a fortiori we will have the required vanishing to show that $\Gamma$ has strict cohomological dimension 2.

However in degree -1 we get the presheaf sending $H\subset \Gamma$ to $\on{Hom}_{cont}(H;\mbb{Q}/\mbb{Z})$.  Since $H$ is profinite any such homomorphism is killed by passing to an open subgroup, showing we get vanishing in thet colimit as desired.

Thus we have seen that $\Gamma$ has strict cohomological dimension two.  Moreover $\Gamma\twoheadrightarrow \Gamma/\overline{[\Gamma,\Gamma]}\simeq \wh{H_1F(\ast)}\twoheadrightarrow \wh{\mbb{Z}}$ gives a surjective homomorphism.  Thus we get a $G=\Gamma\times_{\wh{\mbb{Z}}}\mbb{Z}$ as in 2 with an obvious map $W\to G$ where $W$ is the Weil group corresponding to our class formation.  Under the equivalence with condensed refined class formations this corresponds to a map of condensed refined class formations which by inspection is an isomorphism, thus completing the argument that 1 gives 2 and finishing the proof.
\end{proof}

\begin{remark}
Contained in the above analysis was the following: refined class formations of \emph{profinite} type (i.e.\ of compact type and such that the class group $A(U)$ is profinite for all $U$) are in 1-1 correspondence with profinite groupoids $B\Gamma$ such that $\Gamma$ is a profinite group of strict cohomological dimension $\leq 2$.  As mentioned, modulo language this is a theorem of Brumer, \cite{brumer}.
\end{remark}

\section{Weil-Moore anima}

We finished the previous section with a result (Proposition \ref{ztype}) showing that in the $\mbb{Z}$-case, refined class formations correspond to profinite groups of strict cohomological dimension 2 equipped with a surjection to $\wh{\mbb{Z}}$.  Part of the reason this is so nice is that we get the refined class formation directly from the $\mbb{R}/\mbb{Z}$-cohomology (of finite etale covers) without any truncation, the key point being that this cohomology automatically vanishes outside cohomological degrees $0$ and $1$.

This fails drastically in the $\mbb{R}$-case.

\begin{example}
Consider the refined class formation for the local field $\mbb{C}$: the Galois category is trivial, and the corresponding Weil groupoid is $B\mbb{C}^\times$.  The $\mbb{R}/\mbb{Z}$-cohomology of $B\mbb{C}^\times$, in degrees $\geq 2$, matches the $\mbb{Z}$-cohomology up to a shift by Lemma \ref{cohrmodz}, which matches the usual $\mbb{Z}$-cohomology of $BS^1\sim \mbb{C}\mbb{P}^\infty$ in algebraic topology by Lemma \ref{cohliegp}.  Thus we have nonvanishing of $\mbb{R}/\mbb{Z}$-cohomology in all odd degrees.
\end{example}

We view this as a defect: our perspective in this paper is that all reasonable class formations should arise via the \emph{untruncated} cohomology of some object; this should be the ultimate explanation for why the (co)sheaf property holds on the full derived level even though the refined class formation only lives in two degrees.  In the $\mbb{Z}$-cases this object is the Weil groupoid, but in the other cases we need to add higher homotopy groups.  We will use the formalism for doing this from Section \ref{mooresection}, leading to the following definition.

\begin{definition}
Let $C$ be a refined class formation which is either of $\mbb{Z}$-type, of $\mbb{R}$-type, or of compact type (see Definition \ref{zrcompact}), with associated Weil groupoid $\mc{W}_C\simeq BW$ from Corollary \ref{weildeterminesrcf}.  A \emph{Weil-Moore} anima for $C$ is a condensed anima $X$ equipped with an isomorphism
$$t:\tau_{\leq 1}X\simeq \mc{W}_C$$
such that $(X,t)$ is a Moore anima in the sense of Definition \ref{moore}, i.e. (to recall):
\begin{enumerate}
\item $\pi_iX$ is compact Hausdorff for all $i\geq 2$;
\item $H^i(X^\circ;\mbb{R}/\mbb{Z})=0$ for all $i\geq 2$, where $X^\circ = X\times_{BW}BW^\circ$ with $W^\circ$ the connected component of the identity in $W$.
\end{enumerate}
\end{definition}

We now show that if $(X,t)$ is a Weil-Moore anima, then the condensed anima $X$ by itself canonically recovers the class formation via untruncated cohomology, as desired.

\begin{lemma}\label{mooreisweilmoore}
Let $C$ be a refined class formation which is either of $\mbb{R}$-type, of $\mbb{Z}$-type, or of compact type, and suppose given a Weil-Moore anima $(X,t)$ for $C$.

Then:
\begin{enumerate}
\item for every finite etale $Y\to X$, we have $H^i(Y;\mbb{R}/\mbb{Z})=0$ for all $i\geq 2$;
\item the refined class formation $C$ is recovered as
$$C=(\on{fEt}_X,R\Gamma(-;\mbb{R}/\mbb{Z})^\vee,\epsilon)$$
where $\epsilon$ is induced by $(-)\to \ast$.
\end{enumerate}
\end{lemma}
\begin{proof}
First we show that 1 implies 2.  By Corollary \ref{weildeterminesrcf} the class formation $C$ is recovered from $\mc{W}_C$ in a similar way but by taking cohomology truncated to degrees $\geq -1$.  This truncated cohomology is the same for $\mc{W}_C$ as for $X$ because of the datum $t$.  On the other hand by 1 the truncation is irrelevant for $X$, so this gives 2.

Now let us prove 1.  First suppose we are in the $\mbb{Z}$-type case.  Then the associated Weil group $W$ is totally disconnected, so from Example \ref{tdmoore} we have that every Moore anima $X$ for $W$ satisfies $X\overset{\sim}{\rightarrow} BW$.  Thus 1 follows from Proposition \ref{ztype}.

Next suppose we are in the compact type case.  Then by the basic theory of compact Hausdorff groups, $W^\circ$ is the intersection of all open subgroups $H$ of $W$.  It follows that
$$X^\circ = \varprojlim_H X_H$$
where $X^\circ=X\times_{BW}BW^\circ$ and $X_H=X\times_{BW}BH$. On the level of $\mbb{R}/\mbb{Z}$ cohomology, this means that the cohomology of $X^\circ$ is the filtered colimit of the cohomologies of the $X_H$ (Lemma \ref{invlimchaus}).  It follows that the stalk of the sheaf
$$(Y\to X)\mapsto R\Gamma(Y;\mbb{R}/\mbb{Z})$$
on the Galois category $\on{fEt}_X$ identifies with $R\Gamma(X^\circ;\mbb{R}/\mbb{Z})$.  (Let us work with the underlying $\on{D}(\mbb{Z})$-valued sheaf here: the cohomology is discrete in positive degrees anyway, by Lemma \ref{cohrmodz}.) Similarly the stalk of the truncation $\tau_{\geq -1}R\Gamma(-;\mbb{R}/\mbb{Z})$ is given by $\tau_{\geq -1}\Gamma(X^\circ;\mbb{R}/\mbb{Z})$.  The Moore anima hypothesis therefore yields that the map from truncated cohomology to untruncated cohomology is an isomorphism on stalks.  However, both of these are (coconnective) sheaves: the untruncated cohomology by definition, and the truncated one because it is Pontryagin dual to the cosheaf which gives our refined class formation $C$, see Corollary \ref{weildeterminesrcf}.  Thus the map from truncated to untruncated cohomology is an isomorphism, proving 1.

Finally, suppose we are in the $\mbb{R}$-type case.  Let $F:\mc{C}\to\on{D}(\on{CondAb})$ be the cosheaf giving the refined class formation $C$.  The axioms imply that for any connected $U\in\mc{C}$, the map
$$H_1F(U) \to H_1F(\ast)$$
induces an isomorphism on the $\mbb{R}$-quotients of these two groups.  Thus we can identify all the $\mbb{R}$-quotients with a single $V\simeq\mbb{R}$.  After a choice of nullhomotopy of an obstruction map
$$\frac{1}{\#\mc{C}}\mbb{Z} \to R\Gamma(\mc{C};V[2])\simeq V[2],$$
we can make a refined class formation of compact type $C'$ on the same Galois category, with $F'\to F$ inducing on $H_1$ the inclusion of the maximal compact subgroup.  On the level of Weil groupoids, this gives a fiber sequence
$$\mc{W}_{C'} \to \mc{W}_C \to B\mbb{R}.$$
Given our Weil-Moore anima $X\to \mc{W}_C$, we can therefore pull it back to $X'\to\mc{W}_{C'}$, and this will be a Weil-Moore anima for $\mc{W}_{C'}$ (indeed, we get a similar fiber sequence for $(X')^\circ$ and $X^\circ$, then we can use that $BV\to\ast$ induces an isomorphism on cohomology with discrete coefficients).  By the compact-type case just established, this shows that for $Y\to X$ finite etale, the pullback $Y'\to X'$ has the required vanishing in $\mbb{R}/\mbb{Z}$-cohomology.  But the $\mbb{R}/\mbb{Z}$-cohomology of $Y$ is the cohomology over $BV$ of the $\mbb{R}/\mbb{Z}$-cohomology of $Y'$, which is discrete in positive degrees by Lemma \ref{cohrmodz}, so we successfully reduce to the compact-type case.
\end{proof}

Now we recall that in Section \ref{mooresection}, we built an obstruction theory for constructing Moore anima.  The takeaway was the following, Corollary \ref{mooreconclusion}: suppose $G$ is an almost compact group satisfying the following two assumptions:
\begin{enumerate}
\item $H^2(BG^\circ;\mbb{R}/\mbb{Z})=0$;
\item For all discrete $G/G^\circ$-modules $M$ and all $i\geq 3$ we have
$$\on{Ext}^i_{G/G^\circ}(M,c_G)=0$$
where
$$c_G = \tau_{\geq -1} R\Gamma(BG^\circ;\mbb{R}/\mbb{Z}).$$
\end{enumerate}
Then $G$ admits a Moore anima, unique up to isomorphism.

This motivates the following, which is our main technical observation.

\begin{proposition}\label{extvanishing}
Let $C$ be a refined class formation which is either of $\mbb{Z}$-type, of $\mbb{R}$-type, or of compact type, with associated Weil groupoid $\mc{W}_C\simeq BW$.  If $C$ is of $\mbb{Z}$-type, assume furthermore the following technical hypothesis: the kernel $W^1$ of the corresponding surjective homomorphism
$$W \twoheadrightarrow \mbb{Z},$$
which is a profinite group (see Proposition \ref{ztype}), has cohomological dimension $\leq 1$ with respect to torsion coefficients.

Then for any discrete $W/W^\circ$-module $M$ and any $i\geq 3$, we have
$$\on{Ext}^i_{W/W^\circ}(M,c_W)=0.$$
\end{proposition}
\begin{proof}
First let's consider the case where $C$ is of compact type, so that $W$ is a compact Hausdorff group.

Note that $c_W$ is quite close to being discrete: it is $\mbb{R}/\mbb{Z}$ in degree $0$ and Pontryagin dual to the (abelianization of the) compact $W^\circ$ in degree $-1$.  It will be technically convenient to instead work with the fully-discrete (Lemma \ref{cohchaus})
$$d_W := \tau_{\geq -2}R\Gamma(BW^\circ;\mbb{Z}).$$
Note by Lemma \ref{cohrmodz} that there is a natural cofiber sequence
$$R\Gamma(BW^\circ;\mbb{Z}) \to \mbb{R} \to R\Gamma(BW^\circ;\mbb{R}/\mbb{Z}),$$
and we can truncate this to get a cofiber sequence
$$d_W \to \mbb{R} \to c_W.$$

Translating from $c_W$ to $d_W$, we claim it suffices to prove
$$\on{Ext}_{W/W^\circ}^i(M,d_W)=0$$
for all $i\geq 4$ and all discrete $M$ instead.  Indeed, we have
$$\on{Ext}^i_{W/W^\circ}(M;\mbb{R})=0$$
for $i>0$ by Lemma \ref{extagainstrvs}.

Note that $W/W^\circ$ is profinite, since $W$ is compact Hausdorff.  Since $R\Gamma(BW^\circ;\mbb{Z})\in\on{D}_{W/W^\circ}(\on{CondAb})$ is discrete, using Lemma \ref{compareexts} we can represent it by a complex
$$[I^0 \overset{d^0}{\rightarrow} I^1 \overset{d^1}{\rightarrow} I^2 \overset{d^2}{\rightarrow} I^3 \to \dots]$$
of injective discrete $W/W^\circ$-modules.  Thus the truncation $d_W$ is represented by
$$[I^0 \overset{d^0}{\rightarrow} I^1 \overset{d^1}{\rightarrow} \on{ker}(d^2)].$$

Now, let us agree to identify every discrete $W/W^\circ$-module $I$ with its corresponding sheaf of abelian groups on $\on{fEt}_{BW}$.  Thus, if $H\subset W$ is an open subgroup, corresponding to the finite etale $BH\to BW$, then
$$I(BH) = I^H,$$
the $H$-fixed point subgroup of $I$.
For $X\to BW$ finite etale, it follows that
$$[I^0(X) \to I^1(X) \to I^2(X) \to \ldots]$$
represents $R\Gamma(X;\mbb{Z})$, and
$$[(I^0)(X) \to (I^1)(X) \to \on{ker}(d^2)(X)]$$
represents $\tau_{\geq -2} R\Gamma(X;\mbb{Z})$.  On the other hand, exactly as above we have the fiber sequence
$$\tau_{\geq -2} R\Gamma(X;\mbb{Z}) \to \mbb{R} \to \tau_{\geq -1} R\Gamma(X;\mbb{R}/\mbb{Z}).$$
The third term $\tau_{\geq -1}R\Gamma(-;\mbb{R}/\mbb{Z})$ gives a sheaf on $\on{fEt}_{BW}$ with coefficients in $\on{D}(\on{CondAb})$ because it is Pontryagin dual to our refined class formation.  The middle term obviously does as well, by vanishing of higher cohomology of profinite groups with $\mbb{R}$-coefficients.  Thus the first term also gives a sheaf.  Since injective abelian group sheaves sitting in degree $0$ are also sheaves with values in $\on{D}(\mbb{Z})$ (classically, this is the fact that injective sheaves have vanishing higher cohomology), we deduce from this the following conclusion: $\on{ker}(d^2)[0]$ is a sheaf with values in $\on{D}(\mbb{Z})$.  In classical terms, this means that $\on{ker}(d^2)$ is \emph{cohomologically trivial} $W/W^\circ$-module.  It therefore follows by the main result of \cite{martinez} that it admits an injective resolution with just two terms:
$$0 \to \on{ker}(d^2) \to J^0 \to J^1 \to 0.$$
We can splice this on to our complex representing $d_W$ to get a complex of injectives
$$[I^0 \to I^1 \to J^0 \to J^1]$$
which represents $d_W$.  The required Ext vanishing follows.

Now let us handle the $\mbb{R}$-case.  As in the proof of Lemma \ref{mooreisweilmoore}, we can make a compact-type refined class formation $C'$ such that the associated Weil group $W'$ satisfies
$$1 \to W' \to W \to \mbb{R} \to 1.$$
It follows that $W/W^\circ \simeq W'/(W')^\circ$ and
$$c_W \simeq c_{W'}\oplus \mbb{R}[-1].$$
As above we already showed Ext vanishing with value $\mbb{R}$ in all positive degrees, this means that the claim for $C$ follows from the claim for $C'$, proven above.

Finally, consider the $\mbb{Z}$-case.  We first note that the compact-type case treated above admits a refinement when applied to the refined class formation corresponding to the profinite group $\Gamma=W^1$. Namely, from the long exact sequence associated to $\cdot p:\mbb{Z} \to \mbb{Z}$ for all prime numbers $p$ we deduce from our extra assumption the following: $H^2(BH;\mbb{Z})$ is a divisible abelian group for all open subgroups $H\subset \Gamma$.   Then $\on{ker}(d^2)^H$ is an extension of divisible groups, hence is also divisible. So not only is $\on{ker}(d^2)$ cohomologically trivial, but its $H$-fixed points are divisible for all $H$.  Thus again by \cite{martinez} we deduce that $\on{ker}(d^2)$ is on the nose injective, with no need of further resolution.  Therefore we do one better in our Ext vanishing result for $\Gamma$.  Returning to our Weil group $W$, note that since $W$ is totally disconnected we have $W^\circ=\{1\}$ and $c_W=\mbb{R}/\mbb{Z}$, and similarly for $W^1$.  On the other hand
$$\on{RHom}_W(M,\mbb{R}/\mbb{Z})\simeq R\Gamma(B\mbb{Z};\on{RHom}_{W^1}(M,\mbb{R}/\mbb{Z})),$$
so as $B\mbb{Z}$-cohomology makes degrees go down at most by one, using that we saved one degree above we get the conclusion.
\end{proof}

\begin{remark}
The extra assumption in the $\mbb{Z}$-type case, that $W^1$ have cohomological dimension 1 with respect to torsion coefficients, is certainly satisfied in all number theoretic examples.  Here is one way of verifying it.  Note that it is equivalent to ask that $H^i(BH;\mbb{F}_p)=0$ for all $i>1$ all open subgroups $H\subset W^1$ and all primes $p$.  We know this vanishing for $i>2$ from Proposition \ref{ztype}, so it's just about showing $H^2(BH;\mbb{F}_p)=0$ for all $p$ (which is exactly what we used in the above proof anyway).  Every open subgroup $H\subset W^1$ arises from some open subgroup of finite index $H'\subset W$ by passing ``up the unramified tower'', formed by the preimages $H'_N$ of $N\mbb{Z}\subset\mbb{Z}$ under $H' \twoheadrightarrow \mbb{Z}$.  On the other hand the $\mbb{F}_p$-cohomology of $BH'$ is calculated by
$$\on{RHom}_{\mbb{Z}}(F(BH'),\mbb{F}_p)$$
where $F$ is the cosheaf giving our refined class formation, by Proposition \ref{ztype}.  We deduce the following: $H^2(BH;\mbb{F}_p)$ is the Pontryagin dual of the inverse limit

$$\varprojlim_N H_1F(BH'_N)[p]$$

\noindent of the $p$-torsion in the corresponding ``class groups'' along the ``norm maps''.  For example, in the case of nonarchimedean local fields, we find that we need to check that the inverse limit of $\mu_p(F_N)$ under the norm maps gives $0$ as $F_N$ runs through the unramified extensions of any nonarchimedean field $F$.  This is obvious as extensions of degree $p$ induce the $0$ map, and there are always more and more of those.
\end{remark}

Finally we can prove the main theorem.

\begin{theorem}
Let $C$ be a condensed refined class formation which is either of $\mbb{Z}$-type, of $\mbb{R}$-type, or of compact type.  Then there exists a Weil-Moore anima $(X,t)$ for $C$, and moreover such a Weil-Moore anima is unique up to isomorphism.

In the $\mbb{Z}$-type case and the profinite case it is completely unique (contractible space of choices).
\end{theorem}
\begin{proof}
Let $\mc{W}_C\simeq BW$ denote the associated Weil groupoid (Corollary \ref{weildeterminesrcf}).  By definition, a Weil-Moore anima for $C$ is the same thing as a Moore anima for $BW$.  In the $\mbb{Z}$-type and profinite cases, $W$ is totally disconnected, so the space of such Moore anima is contractible by Example \ref{tdmoore}.  In the $\mbb{R}$-type and compact-type cases, because of Corollary \ref{mooreconclusion} and Proposition \ref{extvanishing}, we need only show $H^2(BW^\circ;\mbb{R}/\mbb{Z})=0$.  But $W^\circ$ is abelian by construction, so this follows from Lemma \ref{moore1ab}.
\end{proof}

\begin{remark}\label{vanishingchauscoh}
Let $X$ be a Weil-Moore anima as above and $A$ a compact abelian group with $W=\pi_1X$-action.  By Lemma \ref{cohchauscoeff} we have
$$R\Gamma(X;A) \simeq \on{RHom}_{W/W^\circ}(A^\vee;c_W).$$
Thus Proposition \ref{extvanishing} translates to the following fundamental fact: suppose $X$ is a Weil-Moore anima whose underlying class formation is of $\mbb{Z}$-type, $\mbb{R}$-type, or compact-type, and suppose in the $\mbb{Z}$-type case that it satisfies the extra hypothesis in Proposition \ref{extvanishing}.  Then 
$$H^i(X;A)=0$$
for all compact Hausdorff coefficient systems $A$ on $X$ and all $i>2$.
\end{remark}

We get Weil-Moore anima attached to any of the number theoretic class formations discussed in Example \ref{exrcf}.  Most prominent for us are the following.

\begin{example}\label{wmex}
\begin{enumerate}
\item Fix a number field $K$ and a nonempty set of places $S$, and consider the condensed refined class formation $C_{K,S} = (\mc{C}_{K,S},F_{K,S},\alpha_{K,S})$ which was canonically produced in Example \ref{exrcf}, whose underlying Galois category consists of the finite extensions $L/K$ unramified outside $S$ and whose class group takes the value
$$A(L) = L^\times \backslash \mathbb{A}_L^\times / \prod_{\nu\not\in S} (L_\nu)^{\times, 1}.$$
(This can also be interpreted as the ``compactly supported Picard group of $\on{Spec}(\mc{O}_{K,S})$'', at least when $S$ is finite and contains all archimedean places.)  Then associated to $C_{K,S}$ is a Weil-Moore anima $X_{K,S}$, uniquely specified up to isomorphism.
\item For a nonarchimedean local field $K$, after a small choice (of algebraic closure of its residue field), we got a canonical refined class formation $C_K$ where $\mc{C}_K$ is the Galois category corresponding to finite separable extensions of $K$, and where the class group is the multiplicative group.  Then we get a Weil-Moore anima $X_K\simeq BW_K$, this time canonically (given the non-canonical choice of refined class formation).
\item For an archimedean local field $\mbb{R}$ or $\mbb{C}$, we got a pretty much canonical refined class formation of the same sort as in the nonarchimedean case.  Thus we get Weil-Moore anima $X_{\mbb{R}}$ and $X_{\mbb{C}}$, uniquely specified up to isomorphism.  In contrast to the nonarchimdean case, these are different from $\mc{W}_{\mbb{R}}\simeq BW_{\mbb{R}}$ and $\mc{W}_{\mbb{C}}\simeq BW_{\mbb{C}}$ respectively, but they can still be made quite explicit.

For motivation, consider the simpler $X_{\mbb{C}}$.  The Weil groupoid is $\mc{W}_{\mbb{C}}=B\mbb{C}^\times$.  When we follow the obstruction theory for adding higher homotopy groups to build the Weil-Moore anima, we find it mimics the procedure of building $S^2$ from $K(\mbb{Z},2)$ from Serre's thesis.  This suggests the following construction of $X_{\mbb{C}}$, and even $X_{\mbb{R}}$ (which recovers $X_{\mbb{C}}$ by pullback).

Recall that to make our refined class formation for $\mbb{R}$, or equivalently our Weil grouopid $\mc{W}_{\mbb{R}}$, we needed to ``choose $\mbb{H}$'', but only up to Morita equivalence (more precisely, as an object in the Morita 2-groupoid).  Namely, our preferred model for $\mc{W}_{\mbb{R}}$
was as the condensed groupoid of pairs consisting of a choice of algebraic closure $\mathbf{C}$ of $\mbb{R}$ and a Morita equivalence $\mathbb{H}\otimes_{\mathbb{R}}\mathbf{C}\sim \mathbf{C}$.

To proceed towards constructing $X_{\mbb{R}}$, we start by fixing $\mbb{H}$ more honestly, as a division algebra.  A Morita equivalence as above is implemented by an $\mbb{H}\otimes_{\mbb{R}}\mathbf{C}$-module $M$ whose underlying $\mbb{H}$-module will be free of rank one.  This describes a map
$$\mc{W}_{\mbb{R}} \to B\mbb{H}^\times,$$
which after choice of basepoint implements the usual identification of $W_{\mbb{R}}$ with a subgroup of $\mbb{H}^\times$ (say, $W_{\mbb{R}}=\mbb{C}^\times\cup j\cdot \mbb{C}^\times$).

The fiber of the above map identifies with $\mbb{R}\mbb{P}^2$, which more precisely arises as condensed set of closed points in the Brauer-Severi variety $\on{Tw}_{\mbb{R}}$ attached to $\mbb{H}$.  This corresponds to the description
$$\mbb{R}\mbb{P}^2 = \mbb{H}^\times/W_\mbb{R}$$
as a quotient.  If we pass to the norm-1 version of everything, then $\mbb{H}^\times$ gets replaced by the unit quaternion group $S^3$, and $W_\mbb{R}$ gets replaced by $S^1\cup j\cdot S^1$ acting on it by isometries, but $\mbb{R}\mbb{P}^2$ stays the same.  Then we explained in Example \ref{riemannian} how to build a condensed anima $C$ with $\pi_0C=\ast$, $\pi_1C = U(1)\cup j\cdot U(1)$, $\pi_2C=\mbb{R}/\mbb{Z}$, and $\pi_nC$ finite for $n\geq 3$ together with a map
$$\mbb{R}\mbb{P}^2 \to C$$
which induces an equivalence on categories of discrete coefficient systems (and their cohomology).  We get $X_\mbb{R}$ by adding back in the $\mbb{R}_{>0}$-factor:

$$X_\mbb{R} = C\times B\mbb{R}_{>0}.$$

\end{enumerate}
\end{example}

We would also like to have the ``obvious'' maps between these different Weil-Moore anima.  Unfortunately, here the situation is quite a bit less satisfactory than for the Weil groupoids (or refined class formations), which was already slightly unsatisfactory.  Namely we will see that the space of possible maps (inducing a given map of Weil groupoids) is nonempty, but possibly non-connected.  For example, it's quite possible that there are non-homotopic local-to-global maps $X_{\mbb{Q}_p} \to X_{\mathbb{Q}}$.

Of course, in principle this leads to further issues if one ever wants a commutative diagram involving choices of such maps.  Nonetheless, with some extra work we will manage to produce the most relevant commutative squares.  Furthermore, the non-uniqueness up to homotopy always goes away on profinite completion, see Remark \ref{okonprof}.

\begin{proposition}
Let $X$ and $X'$ be two Weil-Moore anima, with underlying condensed refined class formations $C$ and $C'$ respectively.  If $C$ is of $\mbb{Z}$-type, assume it satisfies the same extra technical hypothesis as in Proposition \ref{extvanishing}.  Suppose given a map $C\to C'$, or equivalently (Corollary \ref{weildeterminesrcf}) a map
$$f_1: \tau_{\leq 1}X \to \tau_{\leq 1}X'.$$

Then there exists a map $f:X\to X'$ inducing $f_1$, and more precisely the set of homotopy classes of such $f$ inducing $f_1$ is a torsor for the group  
$$H^2(X;\pi_2X'),$$
where $\pi_2X'$ is the twisted coefficient system over $X$ gotten by pulling the standard $\pi_1X'$ action on $\pi_2X'$ back via $f_1$.
\end{proposition}
\begin{proof}
This is standard obstruction theory.  We interpret $f_1$ equivalently as a map
$$X\to \tau_{\leq 1}X'.$$
The second stage of the Postnikov tower for $X'$ is classified by a k-invariant in $H^3(\tau_{\leq 1}X';\pi_2X')$, or more honestly by a cocyle-level object, a higher gerbe.  Lifting our given map $X\to \tau_{\leq 1}X'$ to $\tau_{\leq 2}X'$ is the same as splitting the pullback of this gerbe, which is the same as trivializing the corresponding cocycle on pullback to $X$.  This can be done if and only if the corresponding class in
$$H^3(X;\pi_2X')$$
vanishes, in which case the choices up to homotopy are a  torsor for $H^3(X;\pi_2X')$.  But the group $H^3(X;\pi_2X')$ vanishes by Remark \ref{vanishingchauscoh}.  By the same token, the obstructions to both existence and uniqueness up to homotopy vanish for higher terms in the Postnikov tower.
\end{proof}

Thus, for example, if $F$ is a number field and $\nu$ is a place of $F$, we can make a local-global map of Weil-Moore anima
$$X_{F_\nu} \to X_F$$
inducing ``the'' (unique up to homotopy) map of refined class formations discussed following Lemma \ref{maprcf}.  Unfortunately it is probably not unique up to homotopy, and because of this, we need some extra work to produce certain relevant commutative diagrams.

\begin{lemma}
Suppose given Weil-Moore anima $X,Y,Y'$, all of which satisfy the conditions of the previous result.  Suppose also given a map $g:Y\to Y'$ and a map $\tau_{\leq 1}X\to\tau_{\leq 1}Y$, giving in particular the sequence of maps
$$\tau_{\leq 1}X \to \tau_{\leq 1}Y \overset{\tau_{\leq 1}g}{\longrightarrow} \tau_{\leq 1} Y'.$$
Assume that $\pi_2g:\pi_2Y\to\pi_2Y'$ is surjective.

Then the map (induced by composing with $g$) from the set of lifts of $\tau_{\leq 1}X \to \tau_{\leq 1}Y$ to a map $X\to Y$, towards the set of lifts of $\tau_{\leq 1}X \to \tau_{\leq 1}Y'$ to a map $X \to Y'$, is surjective.

\end{lemma}
\begin{proof}
By functoriality of the obstruction theory in the previous result, the map in question is induced by a map of torsors compatible with the homomorphism
$$H^2(X;\pi_2Y) \to H^2(X;\pi_2Y')$$
induced by $\pi_2g$. Thus it suffices to see that this homomorphism is surjective.  But the kernel of the surjective map $\pi_2Y\to \pi_2Y'$ is also compact Hausdorff, so this follows form the long exact sequence in cohomology and the same vanishing result (Remark \ref{vanishingchauscoh}) we used above.
\end{proof}

Now is perhaps the right time to discuss the $\pi_2$ of Weil-Moore anima, as an excuse to see how to verify the hypotheses of the previous lemma in practice.

\begin{lemma}\label{pi2}
Let $X$ be a Weil-Moore anima with underlying condensed class formation $(\mc{C},A)$.  Then we have
$$\pi_2X \simeq \left(\on{Sym}^2_{\mbb{Z}}\Lambda\right)^\vee,$$
where $\Lambda$ is the torsionfree discrete $\mbb{Z}$-module gotten as follows: consider the subgroup presheaf
$$K^\circ \subset A$$
given by the maximal compact subgroup of the component group of the class group, and give it a \emph{co-presheaf} structure using the transfer maps on $A$.  We thus get a presheaf $(K^\circ)^\vee$ via Pontryagin duality, and $\Lambda$ is the stalk of this presheaf (on $\mc{C}$).  From this description wee see a $\wh{\pi_1X} = \wh{W}$-action and hence a $W$-action on $\Lambda$, thereby also giving a $W$-action on $\left(\on{Sym}^2_{\mbb{Z}}\Lambda\right)^\vee$, and this agrees with the usual action of $\pi_1X$ on $\pi_2X$ via the above isomorphism.

Moreover, if $X\to X'$ is a map of Weil-Moore anima such that:
\begin{enumerate}
\item for all $U'\in\on{fEt}_{X'}$ connected the pullback $U=U'\times_{X'}X$ is also connected, or equivalently the map $\pi_1X \to \pi_1X'$ is surjective on profinite completion;
\item for all $U'\in \on{fEt}_{X'}$ the induced map on class groups
$$A(U) \to A'(U')$$
is surjective on maximal compact subgroups;
\end{enumerate}
then $\pi_2X \to \pi_2 X'$ is surjective.
\end{lemma}
\begin{proof}
Recall from Lemma \ref{mooreinduct} that if $X$ is the Moore anima of a general locally compact Hausdorff group $G$, then 

$$\pi_2X = \left(H^3(BG^\circ;\mbb{R}/\mbb{Z})\right)^\vee.$$

For the fundamental group $G\simeq \pi_1X$ of a Weil-Moore anima, we have

$$G = \varprojlim_N G_N$$
where $G_N$ are the Weil groups of the ``finite quotients'' of $\mc{C}$; moreover the maps $G \to G_N$ are given as the quotient by a compact normal subgroup, hence are open and proper.  It follows that the induced map
$$G^\circ \to (G_N)^\circ$$
is also such a quotient, and that
$$G^\circ \simeq \varprojlim_N (G_N)^\circ.$$ 

On the other hand $G_N$ is an extension of a finite group by the class group corresponding to $N$, so $(G_N)^\circ$ is the same as the connected component of the class group, with functoriality induced by the transfer map.  We get the same relation

$$K^\circ \simeq \varprojlim_N K_N^\circ$$
on the level of maximal compact subgroups of component groups (indeed, at most we are ignoring an $\mbb{R}$-factor), and hence
$$H^3(BG^\circ;\mbb{R}/\mbb{Z})\simeq H^3(BK^\circ;\mbb{R}/\mbb{Z})\simeq H^4(BK^\circ;\mbb{Z})\simeq \varinjlim_N H^4(BK_N^\circ;\mbb{Z}).$$
But for a real torus $T$ we have $H^4(BT;\mbb{Z})=\on{Sym}^2_{\mbb{Z}}(H^2(BT;\mbb{Z})) = \on{Sym}^2_{\mbb{Z}}(T^\vee),$
and we deduce the same for an arbitrary connected compact abelian group $T$ by passing to filtered colimits.  The desired claim
$$\pi_2 X \simeq H^3(BG^\circ;\mbb{R}/\mbb{Z})^\vee\simeq \left(\on{Sym}^2_{\mbb{Z}}\Lambda\right)^\vee$$
follows.

For the second claim, let us first show that 1 and 2 together imply (actually, are equivalent to) the statement that the induced homomorphism
$$G \to G'$$
on fundamental groups is surjective on maximal compact subgroups.  Again by compactness arguments and reduction to the finite level version of the Weil groups, this follows from the fact that a map of group extensions which is surjective on the quotient terms and on the kernel terms is surjective in the middle.  We deduce that $G\to G'$ is also surjective on connected components of maximal compact subgroups.

By Pontryagin duality we get $\Lambda' \hookrightarrow \Lambda$.  Applying $\on{Sym}^2_{\mbb{Z}}$ preserves inclusions of torsionfree abelian groups (one can use filtered colimits to reduce to the case of finite free abelian groups), so applying $\on{Sym}^2_{\mbb{Z}}$ then Pontryagin duality again we deduce $\pi_2X \twoheadrightarrow \pi_2X'$ as desired.
\end{proof}

\begin{example}\label{compatiblewm}
Let us apply this to the problem of choosing Weil-Moore anima in number-theoretic situations.  First, we fix the condensed refined class formations and maps between them as discussed in Section \ref{cfsection}.  This gives the $\tau_{\leq 1}$-level of the Weil-Moore data we're trying to construct.

We start with $K=\mbb{Q}$ and $S$ a finite set of places of $\mbb{Q}$ containing $\infty$ and such that if it contains a prime then it contains all smaller primes.  For every such $S$, we arbitrarily choose a Weil-Moore anima $X_{\mbb{Q},S}$ lifting the relevant condensed refined class formation (Example \ref{exrcf}).  Moreover, if $S$ is of this form and $p$ is the smallest prime number \emph{not} in $S$, then we arbitrarily choose a map of Weil-Moore anima
$$X_{\mbb{Q},S\cup\{p\}} \to X_{\mbb{Q},S}$$
lifting the given map of condensed refined class formations.  Furthermore, there is a map of Weil-Moore anima
$$X_{\mbb{F}_{p}} \to X_{\mbb{Q},S},$$
this time unique up to homotopy because $X_{\mbb{F}_{p}}=B\mbb{Z}$ has cohomological dimension 1, inducing the map of condensed refined class formations given by the unramified at $p$ condition.  Then, using the above results, we choose a lift of the composite
$$X_{\mbb{Q}_{p}} \to X_{\mbb{F}_{p}}\to X_{\mbb{Q},S}$$
along
$$X_{\mbb{Q},S\cup\{p\}} \to X_{\mbb{Q},S},$$
getting an important commutative diagram
$$\xymatrix{
X_{\mbb{Q}_{p}}\ar[r]\ar[d] & X_{\mbb{Q},S\cup\{p\}}\ar[d] \\
X_{\mbb{F}_{p}}\ar[r] & X_{\mbb{Q},S}.}$$
We can also continue to lift $X_{\mbb{Q}_{p}}\to X_{\mbb{Q},S\cup\{p\}}$ along further $X_{\mbb{Q},S\cup\{p,p'\}}\to X_{\mbb{Q},S\cup\{p\}}$ of the above form, for the same reason.

With all the above data fixed, we set
$$X_\mbb{Q} := \varprojlim_S X_{\mbb{Q},S}$$
as our Weil-Moore anima for $\mbb{Q}$.  It is indeed easy to see that this is a Weil-Moore anima for $\mbb{Q}$, using that:
\begin{enumerate}
\item we have an inverse limit of Galois groups and an inverse limit of class groups, giving an inverse limit of Weil groups;
\item we can use Lemma \ref{invlimchaus} to see that the higher homotopy doesn't mess with the $\pi_1$ or $\pi_0$ of the limit and to control the cohomology of the inverse limit (to verify the Weil-Moore axioms).
\end{enumerate}
The above data in particular gives $X_{\mbb{Q}_p} \to X_{\mbb{Q}}$ for all $p$, and of course we can also fix a Weil-Moore anima $X_{\mbb{R}}$ for $\mbb{R}$ as in Example \ref{wmex} and a local-global map $X_{\mbb{R}}\to X_{\mbb{Q}}$.

For an arbitrary number field $K$, if we choose $S$ as above to contain all primes of $\mbb{Q}$ which ramify in $K$, then by pullback along
$$X_{\mbb{Q},S} \to BG_{\mbb{Q},S} \leftarrow BG_{K,S}$$
(which is canonical) we obtain analogous data for $K$, in a natural way (given the fixed data for $\mbb{Q}$).

It's possible that system of data might be sufficient in practice, but it is nonetheless a bit annoying that we have not allowed arbitrary finite sets of places containing $\infty$, but only those of the form ``$\infty$ together with all primes $\leq N$'' (or in the case of number fields $K$, only those finite sets of places arising as preimages of such constrained finite sets of places of $\mbb{Q}$, and additionally required to contain all ramified primes).   Perhaps a more refined analysis could yield further data.  Of course, if one fixes a number field and a finite set of places and doesn't worry about any compatibility as this data varies, then one can just fix choices of the Weil-Moore anima and local-global maps.
\end{example}

\begin{remark}\label{okonprof}
Although two choices of Weil-Moore maps need not be homotopic, they will always be homotopic on profinite completion (for which see Proposition \ref{profcomp}).  To show this, note that if $X$ is a Weil-Moore anima, then the map

$$\wh{X} \to \wh{\tau_{\leq 1}X}$$
is an isomorphism on $\tau_{\leq 2}$.  Indeed, it is obviously an isomorphism on $\tau_{\leq 1}$, so by the Postnikov tower it suffices to show that if $A$ is a finite abelian group with $\pi_1X$-action, then
$$H^2(\tau_{\leq 1}X;A)\overset{\sim}{\rightarrow} H^2(X;A).$$
Choosing a basepoint of $X$ we have $\tau_{\leq 1}X=BW$ and $\tau_{\leq 2}X = K(\pi_2X,2)/W$, so it suffices to show that $\on{Hom}(\pi_2X,A)=0$, which follows from the fact that $\pi_2X$ is connected (Lemma \ref{pi2}).

Now, given another Weil-Moore anima $Y$ and a map
$$Y\to \tau_{\leq 1}X$$
(or equivalently a map of the underlying condensed refined class formations), by the above this induces
$$Y\to \tau_{\leq 2}\wh{X}$$
Then the obstructions to existence and uniqueness-up-to-homotopy of further lifts along the Postnikov tower of $\wh{X}$ lie in $$H^i(Y;\pi_{\ast}\wh{X})$$
for $i\geq 3$, which vanish by Remark \ref{vanishingchauscoh}.
\end{remark}

\section{The Weil-Moore analogs of some theorems of Tate}

We will now prove the analogs of some fundamental theorems of Tate in Galois cohomology (\cite{tateduality}) for these new objects, the Weil-Moore anima.  Our focus is on the cohomology of Weil-Moore anima with discrete coefficients, often restricted to be $p$-power torsion for some fixed prime $p$.  The first theorem expresses that Weil-Moore anima have strict cohomological dimension 2 with discrete coefficients, modulo some technicalities in the $\mbb{Z}$-cases.

\begin{theorem}\label{strictcohdim}
Let $X$ be a Weil-Moore anima with underlying class formation either of $\mbb{Z}$-type, $\mbb{R}$-type, or compact type, and let $W=\pi_1X$ denote ``the" associated Weil group, with profinite completion $\wh{W}$ (the Galois group).  If $X$ is of $\mbb{Z}$-type, assume the same additional technical assumption from Proposition \ref{extvanishing} is satisfied. Then for any discrete group $M$ with $W$-action, we have:
\begin{enumerate}
\item $H^i(X;M)=0$ for $i>3$;
\item If either $M$ comes from a $\wh{W}$-action or $M$ is torsion, then $H^3(X;M)=0$.
\end{enumerate}
\end{theorem}
\begin{proof}
We recall from Remark \ref{vanishingchauscoh} that the cohomology with \emph{compact Hausdorff} coefficients vanishes in degrees $\geq 3$.  Now we must move to discrete coefficients.

Suppose $M$ comes from a $\wh{W}$-action; thus we need to show cohomology vanishing in degrees $>2$.  By Lemma \ref{cohachaus} cohomology on $X$ commutes with filtered colimits, so we reduce to the case where $M$ is a finitely generated discrete $\wh{W}$-module; in particular $M$ is finitely generated as an abelian group.  Consider the exact sequence
$$0 \to M_{tors} \to M \to M\otimes\mbb{R} \to M\otimes(\mbb{R}/\mbb{Z}) \to 0.$$
We have cohomology vanishing in degrees $>2$ for $M\otimes(\mbb{R}/\mbb{Z})$ and $M_{tors}$ because these are compact, and we have cohomology vanishing in degrees $>1$ for $M\otimes\mbb{R}$ by Lemma \ref{cohachaus}.  We deduce cohomology vanishing in degrees $>3$ for $M$.  Now we need to improve this to get cohomology vanishing for degree $3$ as well.  Since $M$ is finitely generated, it is a quotient of a finite direct sum of copies of $\mbb{Z}[W/H]$ for some open subgroup $H\subset W$ of finite index.  By cohomology vanishing for the kernel in degree 4, we deduce that the vanishing claim for $M$ is implied by that for $\mbb{Z}[W/H]$.  Thus we need
$$H^3(Y;\mbb{Z})=0$$
for all connected finite etale $Y\to X$ (corresponding to $H\subset W$ via Galois theory).  But by Lemma \ref{cohrmodz} this is equivalent to $H^2(Y;\mbb{R}/\mbb{Z})=0$, which holds because that cohomology group is Pontryagin dual to $H_2$ of our refined class formation hence vanishes by definition. 

If $X$ is of compact type or $\mbb{R}$-type, then every discrete abelian group $M$ with continuous $W$-action comes from a $\wh{W}$-action, as $\wh{W}=W/W^\circ$.  Thus we are done in those cases.  In the $\mbb{Z}$-case, we have a short exact sequence
$$1 \to W^1 \to W \to \mbb{Z} \to 1$$
where $W^1$ is a profinite group of cohomological dimension $\leq 1$ with torsion discrete coefficients (hence $\leq 2$ with arbitrary discrete coefficients), so we conclude because in any case $\mbb{Z}$ has cohomological dimension 1.
\end{proof}

\begin{remark}\label{leopoldt}
The analogous result to 2 for Galois cohomology with restricted ramification (with odd-primary coefficients, or for number fields with no real places, these being ``obvious" necessary restrictions)  is conjectured to be true, but is currently unknown, and in fact it turns out to be equivalent to Leopoldt's conjecture.  See \cite{nsw} for a thorough discussion.  In the Weil-Moore setting, we avoid Leopoldt's conjecture.
\end{remark}

Having discussed cohomological dimension properties, we now turn to finiteness properties, in the setting of $p$-primary torsion coefficient systems. 

\begin{lemma}\label{checkperfmodp}
Let $p$ be a prime, $G$ a group object in anima, and $M\in \on{D}(\mbb{S}[G])$ such that $M$ is $p$-complete and bounded below.  The following properties are equivalent:
\begin{enumerate}
\item $M\in\on{Perf}(\mbb{S}[G]_{\wh{p}})$.
\item $M\otimes_{\mbb{S}[G]}\mbb{F}_p[\pi_0G]\in \on{Perf}(\mbb{F}_p[\pi_0 G])$.
\end{enumerate}
\end{lemma}
\begin{proof}
Note that $\mbb{F}_p[\pi_0G] = \pi_0(\mbb{S}[G]_{\wh{p}})/p$.  Thus this is a specialization of a general fact about $p$-complete bounded above modules over $p$-complete connective ring spectra, see \cite{clausen} Lemma 13.16.
\end{proof}

\begin{lemma}\label{reducefiniteness}
Let $p$ be a prime, $G$ a finite group, and $N\in\on{D}(\mbb{F}_p[G])$.  The following properties are equivalent:
\begin{enumerate}
\item $N\in\on{Perf}(\mbb{F}_p[G])$.
\item $N^H \in \on{Perf}(\mbb{F}_p)$ for all subgroups $H\subset G$.
\item $N_H \in \on{Perf}(\mbb{F}_p)$ for all subgroups $H\subset G$.
\end{enumerate}
\end{lemma}
\begin{proof}
It's clear that 1 $\Rightarrow$ 2,3.  On the other hand, 2 $\Leftrightarrow$ 3 because under either 2 or 3, Lemma \ref{uniformbound} implies that $N_H\simeq N^H$ for all $H$, so that 2 and 3 are literally the same.  Thus it suffices to show that 3 implies 1, i.e.\ if $N_H\in \on{Perf}(\mbb{F}_p)$ for all $H\subset G$, then $N\in \on{Perf}(\mbb{F}_p[G])$.  Let $G_p$ be a $p$-Sylow subgroup.  The unit and counit for the adjunctions between restriction and induction, which agrees with co-induction, give maps
$$N \to \on{Ind}\on{Res}N \to N$$
with composition $\cdot [G:G_p]$, hence an isomorphism.  Thus $N$ is a summand of $\on{Ind}\on{Res}N$.  In this manner we reduce to the case when $G$ is a $p$-group.  But in that case the augmentation
$$\mbb{F}_p[G] \to \mbb{F}_p$$
is a nilpotent thickening and hence base-change along it, which corresponds to $(-)_G$, detects perfect complexes (see again \cite{clausen} Lemma 13.16, though this is well-known).\end{proof}

\begin{proposition}\label{pfinite}
Let $X$ be a Weil-Moore anima with underlying class formation either of $\mbb{Z}$-type, $\mbb{R}$-type, or compact type, and set $W=\pi_1X$.  For a prime $p$, the following properties are equivalent:
\begin{enumerate}
\item If $A$ denotes the sheaf of abelian groups on $\on{fEt}_X$ describing the class formation associated to $X$, then for all $Y\to X$ in $\on{fEt}_X$ connected, both $A(Y)[p]$ and $A(Y)/pA(Y)$ are finite abelian groups; or equivalently, the derived $A(Y)/p$ is finite in each degree.
\item For all $Y\to X$ finite etale connected, $R\Gamma(Y;\mbb{F}_p)\in\on{Perf}(\mbb{F}_p)$.
\item Let $R$ be a $p$-power torsion $E_1$-ring and let $M\in \on{Perf}(R)$ with $\Omega X$-action.  If we are in the $\mbb{Z}$-case, assume that $M$ descends to $B\wh{W}$.  Then $R\Gamma(X;M)\in \on{Perf}(R)$.
\end{enumerate}
\end{proposition}
\begin{proof}
Suppose 1 holds, i.e.\ $A(Y)/p$ is finite for all $Y$.  Then if $F$ denotes the cosheaf on $\on{fEt}_X$ giving the condensed refined class formation, then as $F(Y)$ and $A(Y)[1]$ only differ by $\mbb{Z}$ in degree $0$, this implies (indeed is equivalent to) the statement that $F(Y)/p$ is finite.  However by definition
$$F(Y) = R\Gamma(Y;\mbb{R}/\mbb{Z})^\vee$$
so that
$$F(Y)/p = R\Gamma(Y;\mbb{F}_p)^\vee$$
By Pontryagin duality we therefore see that 1 $\Leftrightarrow$ 2.  Evidently, 2 is the special case of 3 where $R=\mbb{F}_p$ and $M=\mbb{F}_p[W/H]$ as $H$ runs through open subgroups of $W=\pi_1X$.  Thus to finish proving 1 $\Leftrightarrow$ 2 $\Leftrightarrow$ 3 it suffices to show that 2 implies 3.

The coefficient system $M$ is encoded by a condensed group homomorphism
$$\Omega X \to \ul{\on{Aut}}_R(M).$$
Since $M$ is a perfect $R$-module, the target is discrete.  Thus by Remark \ref{prodiscrete} there is a discrete anima $A$ with a map $X\to A$ such that $\pi_0A=\ast$ and $\pi_1A=W/N$ is the quotient of $\pi_1X=W$ by an open normal subgroup of finite index, and such that $M$ descends to $A$.  Thus we get
$$\Omega X \to \Omega A \to \on{Aut}_R(M).$$
The second map corresponds to an $E_1$-ring homomorphism
$$\mbb{S}[\Omega A]_{\wh{p}} \to \on{End}_R(M),$$
making $M$ into an $R$-$\mbb{S}[\Omega A]_{\wh{p}}$-bimodule.  By commutation of $R\Gamma(X;-)$ with colimits of discrete coefficient systems (a consequence of the finite cohomological dimension plus the commutation of cohomology with filtered colimits), we deduce
$$R\Gamma(X;M) \simeq M\otimes_{\mbb{S}[\Omega A]_{\wh{p}}}R\Gamma(X;\mbb{S}[\Omega A]_{\wh{p}}).$$
Thus we reduce to showing that $R\Gamma(X;\mbb{S}[\Omega A]_{\wh{p}})\in \on{Perf}(\mbb{S}[\Omega A]_{\wh{p}})$.  By Lemma \ref{checkperfmodp} we can check this after base-change to $\mbb{F}_p[W/N]$, and again by commutation of $R\Gamma(X;-)$ with colimits of discrete coefficient systems this reduces to showing
$$R\Gamma(X;\mbb{F}_p[W/N])\in\on{Perf}(\mbb{F}_p[W/N]).$$
Then by Lemma \ref{reducefiniteness} this amounts to 2 for every finite etale $Y\to X$ covered by the Galois cover corresponding to $N$, as desired. 
\end{proof}

We discuss the examples from number theory, to see when this $p$-finiteness holds.  Note that in verifying the $p$-finiteness for the class groups as in 1, we are free to work modulo finitely generated abelian groups and uniquely $p$-divisible groups, since those clearly have the desired mod $p$ finiteness property.

\begin{example}

\begin{enumerate}
\item Suppose $K$ is a nonarchimedean local field of characteristic $0$.  Modulo finitely generated abelian groups, the class group $K^\times$ is isomorphic (via exp and log) to $\mc{O}_K$, and $\mc{O}_K$ is isomorphic to a finite direct sum of copies of $\mbb{Z}_\ell$ where $\ell$ is the residue characteristic, hence has the required mod $p$ finiteness no matter what $\ell$ is.  More generally we can easily see that the same conclusion holds as long as the characteristic of $K$ is different from $p$.  

Thus, for a nonarchimedean local field $K$, the corresponding Weil-Moore anima $X_K$ has the $p$-finiteness property whenever $K$ has characteristic $\neq p$.  (This is of course just a restatement of the usual finiteness of Galois cohomology in these cases.)
\item Suppose $K$ is an archimedean local field, so $K=\mbb{R}$ or $\mbb{C}$.  In this case the mod $p$ finiteness for $K^\times$ is obvious.  Thus for the corresponding Weil-Moore anima $X_K$, we have the above $p$-finiteness property.
\item Suppose $K$ is a global field and $S$ is a nonempty finite set of places of $F$.  To verify the $p$-finiteness conditions for $X=X_{K,S}$, we have to investigate the mod $p$ finiteness of the class group
$$A_{K,S}:= K^\times\backslash \mathbb{A}_K^\times/\prod_{\nu\not\in S} (K_\nu^\times)^1.$$
Using the exact sequence in Example \ref{exrcf}, we see that modulo finitely generated abelian groups and uniquely divisible groups, $A_{K,S}$ identifies with $\prod_{\nu\in S}K^\times_\nu$. Thus from the local case we conclude the desired mod $p$ finiteness for $A_{K,S}$, and hence the $p$-finiteness property of $X_{K,S}$, as long as $K$ has characteristic $\neq p$.
\end{enumerate}
\end{example}

In applying condition 3, it is a bit annoying to have to assume the extra hypothesis that $M$ comes from a $\wh{W}$-action in the $\mbb{Z}$-type cases.  We can explain exactly when this extra hypothesis is necessary, as follows.

\begin{lemma}\label{whenignorewhat}
Let $X$ be  Weil-Moore anima of $\mbb{Z}$-type, so $X=BW$ with
$$W\simeq G\times_{\wh{\mbb{Z}}}\mbb{Z}$$
with $G$ a profinite group of strict cohomologial dimension 2 equipped with a surjective homomorphism $G\to\wh{\mbb{Z}}$ (see Proposition \ref{ztype}).  Let
$$W^1 = \on{ker}(W\to\mbb{Z}) = \on{ker}(G\to\wh{\mbb{Z}}).$$
Then the following conditions are equivalent.

\begin{enumerate}
\item For every open subgroup $I\subset W^1$, we have
$$R\Gamma(BI;\mbb{F}_p)\in\on{Perf}(\mbb{F}_p).$$
In other words, $BW^1$ satisfies the $p$-finiteness properties of Proposition \ref{pfinite}.
\item For every $p$-power torsion $E_1$-ring $R$ and every $M\in\on{Perf}(R)$ with $W$-action, we have
$$R\Gamma(X;M)\in \on{Perf}(R).$$
That is, 3 in Proposition \ref{pfinite} holds for $X,R,M$ without the assumption that $M$ descends to $\wh{X}$.

\end{enumerate}
\end{lemma}
\begin{proof}

Assume 1.  Then by Proposition \ref{pfinite} applied to $X=BW^1$ we have $R\Gamma(BW^1;M)\in \on{Perf}(R)$ for all $M$ and $R$ as in 2.  On the other hand  $
R\Gamma(X;M)$ is the $\mbb{Z}$-fixed points of $R\Gamma(BW^1;M)$, hence sits in a fiber sequence with two copies of $R\Gamma(BW^1;M)$, whence 2.

Conversely, assume 2.  To show 1 it suffices to show it for $I\subset W^1$ open and normal in $W$, as every open subgroup of $W^1$ contains one of these.  The surjection $W/I \to \mbb{Z}$ splits over $d\mbb{Z}$ for some $d>0$; if we choose such a splitting, then we can view $\mbb{F}_p[W/I] \in\on{Perf}(\mbb{F}_p[d\mbb{Z}])$.  On the other hand it has the $W$-action commuting with this structure, so from 2 we get
$$R\Gamma(X;\mbb{F}_p[W/I])\in\on{Perf}(\mbb{F}_p[d\mbb{Z}]).$$
On the other hand
$$R\Gamma(X;\mbb{F}_p[W/I])\simeq R\Gamma(BI;\mbb{F}_p)[-1],$$
giving $R\Gamma(BI;\mbb{F}_p)\in \on{Perf}(\mbb{F}_p[d\mbb{Z}])$, where the $d\mbb{Z}$-action comes from restricting the canonical $W/I$-action on $BI$ to our chosen splitting.

However, note that $I$ can equally well be viewed as a normal subgroup of $G=\wh{W}$, so the action of $W/I$ on $BI$ factors through the profinite completion.  We deduce that the action of $d\mbb{Z}$ on $BI$ likewise factors through the profinite completion.  In other words, if we view $\on{Spec}(\mbb{F}_p[d\mbb{Z}])$ as $\mbb{G}_m$, then $R\Gamma(I;\mbb{F}_p)$ is supported at the torsion.  Thus, if perfect, it must be perfect over $\mbb{F}_p$, showing that 2 implies 1.
\end{proof}

\begin{example}
Let $p$ be a prime.  If $X$ is the Weil-Moore anima attached to a nonarchimedean local field of residue characteristic $\neq p$, or attached to a function field of characteristic $\neq p$ with restricted ramification at a finite nonempty set of places, then the corresponding $BW^1$ also has $p$-finiteness.  Thus, by the above, Proposition \ref{pfinite} part 3 holds for all $M$ and $R$ without the technical hypothesis that $M$ descend to $\wh{X}$.

On the other hand, if $X$ is the Weil-Moore anima attached to a finite extension of $\mbb{Q}_p$, then although $X$ satisfies the $p$-finiteness properties of Proposition \ref{pfinite}, the corresponding $BW^1$ does not, and therefore the technical hypothesis on $M$ is necessary in 3 of Proposition \ref{pfinite}.

We will see in Remark \ref{globaldescends} that if $M$ comes from a global Weil-Moore repesentation, then $M$ automatically descends to $\wh{X}$, so from a global perspective this technicality is not so important.  From a purely local perspective, however, it indicates that in this context ($p$-adic local fields with $p$-adic coefficients) it is better to replace Weil-Moore representations and Weil-Moore cohomology by their $(\varphi,\Gamma)$-analogs, with cohomology given by the Herr complex, as in \cite{emertongee}.  The same remark will also apply to our discussion of Poitou-Tate duality.
\end{example}

Following Tate, whenever this $p$-finiteness holds we can ask about the Euler characteristic.  More precisely, given a $p$-power torsion $E_1$-ring $R$ and an $M\in\on{Perf}(R)$ with $\Omega X$-action (assumed to descend to $\wh{X}$ if $X$ is $\mbb{Z}$-type), we can ask to describe the class
$$[R\Gamma(X;M)] \in K_0(R).$$

The magic of Tate's formulas is that in number-theoretic situations, this ``Euler characteristic" only depends on the class $[M]\in K_0(R)$, forgetting the $\Omega X$-action.  (Rather, in the original setting of Galois cohomology, this is only true up to archimedean corrections; but passing to Weil-Moore anima fixes this.)  In topological situations, say for a local system on a finite CW-complex, this kind of Euler characteristic formula is simple to explain, because we can break our space up into contractible pieces.  But for, e.g.\ a general profinite group with $p$-finiteness it can easily fail: see the treatment by Serre in his article on Euler-Poincar\'{e} measures, \cite{serreep}, which analyzes this phenomenon.  Here we rephrase Serre's arguments, but avoiding the scaffolding of Euler-Poincar\'{e} measures.

\begin{theorem}\label{chi}
Let $X$ be a Weil-Moore anima with underlying class formation either of $\mbb{Z}$-type, $\mbb{R}$-type, or compact type, and let $W=\pi_1X$ denote ``the" associated Weil group and $A$ the induced class formation (a sheaf on $\on{fEt}_X$).

Let $p$ be a prime number, and assume the $p$-finiteness properties of cohomology of $X$ from Proposition \ref{pfinite} are satisfied.  Then for any integer $\chi \in\mbb{Z}$, the following properties are equivalent:
\begin{enumerate}
\item For every $G$-Galois cover $Y\to X$, the class of
$$A(Y)_{\wh{p}}[\frac{1}{p}]$$
(derived $p$-completion followed by inverting $p$) in $K_0(\on{Perf}(\mbb{Q}_p[G]))$ is equal to
$$[\mbb{Q}_p] - \chi\cdot[\mbb{Q}_p[G]]$$
(trivial representation minus $\chi$ times the regular representation).
\item For every $G$-Galois cover $Y\to X$, the class of $R\Gamma(Y;\mbb{Q}_p)\in \on{Perf}(\mbb{Q}_p[G])$ in $K_0(\mbb{Q}_p[G])$ is equal to $\chi\cdot [\mbb{Q}_p[G]]$.
\item For every $G$-Galois cover $Y\to X$, the class of $R\Gamma(Y;\mbb{F}_p)\in \on{Perf}(\mbb{F}_p[G])$ in $K_0(\mbb{F}_p[G])$ is equal to $\chi\cdot [\mbb{F}_p[G]]$.
\item For every $p$-power torsion $E_1$-ring $R$ and every $M\in\on{Perf}(R)$ with $\Omega X$-action,
$$[R\Gamma(X;M)] = \chi\cdot [M] \in K_0(R),$$
where if $X$ is of $\mbb{Z}$-type we assume $M$ comes from a $\wh{W}$-action.
\end{enumerate}
\end{theorem}

\begin{proof}
First we check that 1 and 2 make sense.  In 1, note that the $p$-finiteness property (the derived $A(Y)/p$ is finite) implies that $A(Y)_{\wh{p}}\in \on{Perf}(\mbb{Z}_p)$ and hence $A(Y)_{\wh{p}}[1/p]\in \on{Perf}(\mbb{Q}_p)$ and hence $A(Y)_{\wh{p}}[1/p]\in \on{Perf}(\mbb{Q}_p[G])$ (as representation theory of finite groups is semisimple in characteristic zero), so the statement makes sense.  Similarly, by commutation of cohomology with colimits we have
$$R\Gamma(Y;\mbb{Q}_p) = R\Gamma(Y;\mbb{Z}_p)[1/p]$$
so that $R\Gamma(Y;\mbb{Q}_p)\in\on{Perf}(\mbb{Q}_p[G
])$ and so the statement of 2 makes sense.

To see 1 $\Leftrightarrow$ 2, note that the regular representation is self-dual, so 2 means that that the class of $F(Y)_{\wh{p}}[1/p]$ is equal to the class of $\chi$ times the regular representation.  As $F(Y)$ and $A(Y)[1]$ differ exactly by $\mbb{Z}$ in degree $0$, this gives 1 $\Leftrightarrow$ 2.

That 2 $\Leftrightarrow$ 3 follows from the following results in modular representation theory, see \cite{serrerepthy}: first,
$$K_0(\mbb{Z}_p[G])\overset{\sim}{\rightarrow} K_0(\mbb{F}_p[G]);$$
and second
$$K_0(\mbb{Z}_p[G])\hookrightarrow K_0(\mbb{Q}_p[G]).$$
(The first fact is rather formal and holds for any $\mbb{Z}_p$-algebra which is finitely generated as a $\mbb{Z}_p$-module, but the second fact is not.  Ofer Gabber communicated the following example showing that the second fact can fail if $\mbb{Z}_p[G]$ is replaced by an arbitrary finite free $\mbb{Z}_p$-algebra $A$: take $A$ to be the subring of 2x2-matrices over $\mbb{Z}_p$ consisting of those matrices whose bottom-left entry is divisible by $p$.  The source is $\mbb{Z}\oplus\mbb{Z}$ and the target is $\mbb{Z}$.)

It remains to show that 3 and 4 are equivalent.  Note that 3 is the special case of 4 where $R=\mbb{F}_p[W/N]$ and $M=R$ for arbitrary open normal subgroups $N\subset W$ of finite index.  For the converse, assume 3 holds.  As in the proof of Proposition \ref{pfinite} we can find $X\to A$ with $A$ discrete, $\pi_0A=\ast$, and $\pi_1A=W/N$ such that $M$ comes from $A$; then $M$ becomes an $R$-$\mbb{S}[\Omega A]_{\wh{p}}$-bimodule and we have
$$R\Gamma(X;M) \simeq M\otimes_{\mbb{S}[\Omega A]_{\wh{p}}}R\Gamma(X;\mbb{S}[\Omega A]_{\wh{p}}).$$
Thus it suffices to show that
$$[R\Gamma(X;\mbb{S}[\Omega A]_{\wh{p}})] = \chi\cdot [\mbb{S}[\Omega A]_{\wh{p}}] \in K_0(\mbb{S}[\Omega A]_{\wh{p}}).$$
But isomorphism classes of finitely generated projective modules over $\mbb{S}[\Omega A]_{\wh{p}}$ and $\mbb{F}_p[W/N]$ (which is its $\pi_0$ mod $p$) are the same for elementary reasons of lifting idempotents, hence they have the same $K_0$, so we can check this on base-change to $\mbb{F}_p[W/N]$ where it amounts to 3, as desired.
\end{proof}

Now we investigate the examples from number theory to see that in each case the Weil-Moore anima has a well-defined Euler characteristic $\chi\in\mbb{Z}$ (in the sense that the above equivalent conditions hold).

\begin{example}
\begin{enumerate}
\item Let $K$ be a nonarchimedean local field of characteristic $0$. To investigate condition 1 for the Weil-Moore anima $X_K$, let $L$ be a finite Galois extension of $K$; we need to understand the class in the Grothendieck group of $\mbb{Q}_p$-linear $\on{Gal}(L/K)$-representations of
$$V:= (L^\times)_{\wh{p}}[1/p]$$
where we recall that $(-)_{\wh{p}}$ denotes derived $p$-completion.  For these purposes we can clearly work modulo finite groups and uniquely $p$-divisible groups, and we can also split short exact sequences.  This gives
$$[L^\times] = [\mbb{Z}] + [\mathcal{O}^\times_L] = [\mbb{Z}] + [\mc{O}_L]$$
where in the last step we use $\on{exp}$.  Applying $p$-completion then inverting $p$, we get
$$[V] = [\mbb{Q}_p] + [L]$$
if $F$ has residue characteristic $p$ and
$$[V] = [\mbb{Q}_p]$$
if $F$ has residue characteristic $\neq p$.  By the normal basis theorem, $L$ gives the regular representation as an $K$-vector space, hence gives $[K:\mbb{Q}_p]$ times the regular representation as a $\mbb{Q}_p$-vector space.  Thus 1 holds with $\chi=-[K:\mbb{Q}_p]$ when the residue characteristic of $K$ is $p$ and 1 holds with $\chi=0$ when the residue characteristic of $K$ is $\neq p$.  For simplicity we assumed $K$ has characteristic zero in the above analysis, but it's obvious that we get the same conclusion also when $K$ has nonzero characteristic different from $p$.

Thus the Euler characteristic of $X_K$ for nonarchimedean local $K$ is $-[K:\mbb{Q}_p]$ if $K$ is a finite extension of $\mbb{Q}_p$, undefined if $K$ has characteristic $p$, and $0$ otherwise.

\item Let $K$ be an archimedean local field.  If $K\simeq\mbb{C}$, then the Galois group is trivial, but the Euler characteristic of $(\mbb{C}^\times)_{\wh{p}}[1/p]\simeq\mbb{Q}_p[1]$ is $-1$.  We get that 1 holds with $\chi=2$ and thus the Euler characteristic of $X_\mbb{C}$ is $2$.  If $K=\mbb{R}$, we have to do the above calculation equivariantly with respect to the Galois group.  Here is one way to express it: for any algebraic closure $\mathbf{C}$ of $\mathbb{R}$, we have the short exact sequence
$$0 \to \mbb{Z}(1)(\mathbf{C}) \to \mathbf{C}\overset{\on{exp}}{\rightarrow} \mathbf{C}^\times \to 0.$$
The middle term is uniquely divisible and the kernel term gives the sign representation.  Thus $\mathbf{C}^\times$ gives negative the sign representation.  This means that 5 holds with $\chi=1$, so the Euler characteristic of $X_{\mbb{R}}$ is $1$.

Geometrically, of course, these are just the usual Euler characteristics of $S^2$ and $\mbb{R}\mbb{P}^2$, and this archimedean version of Tate's Euler characteristic formula follows from the fact that these spaces admit, say, triangulations (which gives a good explanation for why the Euler characteristic of a local system only depends on the dimension).

Uniformly, we can say that if $K$ is an archimedean local field, then the Euler characteristic of $X_K$ is $[K:\mbb{R}]$ (exactly the opposite sign from the case of finite extensions of $\mbb{Q}_p$!).
\item Now suppose $K$ is a number field and $S$ is a nonempty finite set of places of $K$.  To get a nice result we have to assume that $S$ contains all archimedean places and all places above $p$.  Again we have to take a Galois extension $L/K$ unramified outside $S$, say with group $G$, and study the class group
$$A_{L,S} = L^\times\backslash \mathbb{A}_L^\times/\prod_{\mu\not\in S}\mathcal{O}_{L_\mu}^\times$$
after derived $p$-completing then inverting $p$, as a $G$-representation over $\mbb{Q}_p$. We can work up to finite groups and uniquely $p$-divisible groups.  In particular we are free to remove all places from $S_L$ except those above $p$ and $\infty$, and we can ignore Picard groups.  Thus $[A_{L,S}]$ may as well be
$$[(L\otimes\mbb{R})^\times] + [(L\otimes\mbb{Q}_p)^\times] - [\mc{O}_{L,S_\infty\cup S_p}^\times] = [(L\otimes\mbb{R})^\times] + [(\mc{O}_L\otimes\mbb{Z}_p)^\times] - [\mc{O}_L^\times].$$

We analyze $\mathcal{O}_L^\times$ using Dirichlet's unit theorem.  Actually it's more convenient to use the following extended form, described by Tate in \cite{TateStark}.  Let
$$\widetilde{\mathcal{O}_L^\times} = \{(u\in\mc{O}_L^\times,x\in L\otimes\mbb{R}): \on{exp}(x) = u\text{ in } (L\otimes\mbb{R})^\times\}.$$
This related to $\mc{O}_L^\times$ as follows: the map $\widetilde{\mathcal{O}_L^\times}\to\mathcal{O}_L^\times$ has image the group of units which are positive at every real place, hence its cokernel is finite; on the other hand its kernel is
$$\mbb{Z}(1)(L\otimes\mbb{R})$$
where $\mbb{Z}(1)$ is the sheaf on finite etale $\mbb{R}$-algebras given by the kernel of the exponential map.  Thus
$$[\mc{O}_L^\times] = [\widetilde{\mc{O}_L^\times}] - [\mbb{Z}(1)(L\otimes\mbb{R})].$$

On the other hand, there is the following form of Dirichlet's unit theorem.  The tautological map
$$\widetilde{\mc{O}_L^\times} \to L\otimes \mbb{R}$$
sending $(u,x) \mapsto x$ identifies $\widetilde{\mc{O}_L^
\times}$ with a full lattice inside $L^0\otimes\mbb{R}$, where $L^0=\on{ker}(\on{tr}_{L/\mbb{Q}}:L\to\mbb{Q})$.  It follows that the base-change to $\mbb{R}$ of the $\mbb{Q}$-representation $(\widetilde{\mc{O}_L^\times})_{\mbb{Q}}$ identifies with the base-change to $\mbb{R}$ of the $\mbb{Q}$-representation $L^0$.  From general representation theory we deduce that the two $\mbb{Q}$-representations are isomorphic. In particular the associated $\mbb{Q}_p$-representations are also isomorphic.  Thus 
$$[\widetilde{\mc{O}_L^\times}] = [\mc{O}_L]-[\mbb{Z}],$$
so that
$$[\mc{O}_L^\times] = [\mc{O}_L] - [\mbb{Z}] - [\mbb{Z}(1)(L\otimes\mbb{R})],$$
or equivalently
$$[\mc{O}_L^\times] = [\mc{O}_L] - [\mbb{Z}] + [(L\otimes\mbb{R})^\times].$$
On the other hand $[\mc{O}_L] = [\mc{O}_L\otimes\mbb{Z}_p] = [(\mc{O}_L\otimes\mbb{Z}_p)^\times]$ again by $p$-adic logarithm.  Thus we get miraculous cancellation yielding
$$[A_{L,S}] = [\mbb{Z}],$$
so that 5 holds with $\chi=0$.  The conclusion is that $X_{K,S}$ has Euler characteristic $0$.
\end{enumerate}
\end{example}

\begin{remark}
Note the similarity of the above calculations to the calculations of \emph{Herbrand quotients} as in \cite{artintate}, which constitute one of the crucial inputs to the verification of the class formation axioms.

This is not a coincidence, as we have the following general fact.  If $G$ is a cyclic group of order $p$ and $A$ is a $G$-module such that $A[p]$ and $A/pA$ are finite, then $\wh{H}^0(G;A)$ and $\wh{H}^{1}(G;A)$ are also finite, hence so is the Herbrand quotient
$$h_A=\# \wh{H}^0(G;A)/ \# \wh{H}^{1}(G;A);$$
and it it is determined by the class of $V=A_{\wh{p}}[1/p]$ in the Grothendieck group of $\mbb{Q}_p$-representations of $G$. Namely,
$$h_A = p^{\chi_V(g)},$$
where $\chi_V:G\to\mbb{Q}_p$ is the character of $V$ (it will necessarily take values in $\mbb{Z}$) and $g$ is any non-identity element of $G$.

In this way, the calculation $[A_{L,S}]=[\mbb{Z}]$ above implies that the Herbrand quotient of $A_{L,S}$ is $p$.
\end{remark}

\begin{remark}
We have just seem that Tate's global Euler characteristic formula, in our setting of Weil-Mooere anima, becomes the statement that $X_{K,S}$, the Weil-Moore anima with restricted ramification, has Euler characteristic $0$.  This is the same answer as one gets in the function field case, but there we have a much nicer explanation, given by Kato (\cite{kato}): if $C$ denotes the associated affine curve over $\mbb{F}_p$, then given (for example) a lisse $\mbb{Q}_\ell$-sheaf $V$ on $C$, not only is its cohomology on $C$ finite dimensional, but it is also finite dimensional on $\overline{C}=C\times_{\mbb{F}_p}\overline{\mbb{F}_p}$.  Then the cofiber sequence
$$R\Gamma(C;V) \to R\Gamma(\overline{C};V) \overset{1-\varphi}{\longrightarrow} R\Gamma(\overline{C};V)$$
immediately gives vanishing of the Euler characteristic over $C$.

Actually, it gives more, namely it gives a canonical trivialization of $R\Gamma(C;V)$ as a point in $K(\mbb{Q}_\ell)$.  On the other hand, if it happens to be the case that $R\Gamma(C,V)=0$, then we clearly get another completely obvious trivialization.  Comparing the two trivializations gives a well-defined element in $K_1(\mbb{Q}_\ell)=(\mbb{Q}_\ell)^\times$.  Up to sign conventions, this identifies with the value of the associated L-function at $s=0$.  Other L-values can be obtained by twisting by powers of the cyclotomic character.  All but finitely many of such twists will have vanishing $R\Gamma$, and in this way one gets a ``meromorphic function on the weight space'', and this is just the usual L-function of $V$ as defined in this context by Grothendieck.

Conjecturally, according to Kato, there should similarly be a more refined version of the calculations proving Tate's global Euler characteristic formula, yielding not just a proof that it vanishes in K-theory, but a specific trivialization in K-theory (or at least, in the $\tau_{\leq 1}$-truncation of K-theory); moreover, restricting to situations where it is possible to vary in a weight space which includes points where the cohomology vanishes, this should explain the existence of $p$-adic L-functions.  In other words, in the language of Beilinson (\cite{beilinson}), Kato's Iwasawa Main Conjecture predicts an ``animation'' of Tate's Euler characteristic formula, and says that a $p$-adic L-function ``is" such an animation.
\end{remark}

\begin{remark}
In the $\mbb{Z}$-case we again made the annoying assumption that the coefficient system $M$ comes from a $\wh{W}$-action.  In Remark \ref{whenignorewhat} we gave a criterion for when such an assumption can be ignored for the purposes of getting the conclusion $R\Gamma(X;M)\in\on{Perf}(R)$: namely the criterion says that the ``inertia group"
$$W^1 = \on{ker}(W\to\mbb{Z})$$
should also have the $p$-finiteness property.

When this is the case, we get to ask the Euler characteristic question also for such $M$.  But actually, we see directly from the argument in Remark \ref{whenignorewhat} that under such a hypothesis we have
$$[R\Gamma(X;M)] = 0 \in K_0(R),$$
because $R\Gamma(X;M)$ can be written as the fiber of a map between two copies of the same perfect $R$-module.

In particular, the $p$-finiteness property for $W^1$ can only be satisfied in situations of zero Euler characteristic for $W$.  This is in complete accordance with our discussion of number-theoretic examples above.
\end{remark}

If $X$ is a Weil-Moore anima, then by construction, the composite map
$$X \to BW \to BG$$
realizes $G$ (the Galois group of the Galois category defining the class formation underlying $X$) as the profinite completion of $\pi_1X=W$ (the Weil group).  Given a coefficient system on $BG$, we can ask to compare its cohomology on $BG$ (that is, the usual Galois cohomology) with the cohomology of its pullback to $X$ (that is, the Weil-Moore cohomology).  We will now prove results describing such a comparison in number-theoretic examples.

We start with a simple but important lemma expressing the difference between the cohomology of $X_{K,S}$ and $X_{K,S'}$ when $S$ and $S'$ differ by a single place.

\begin{lemma}\label{fillinpuncture}
Let $K$ be a number field, $S$ a nonempty set of places of $K$, and $\nu$ a place of $K$ not in $S$.  Let $S'=S\cup\{\nu\}$.  Fix a commutative diagram of Weil-Moore anima
$$\xymatrix{
X_{K_\nu} \ar[r]\ar[d] & X_{K,S'}\ar[d] \\
X_{\mc{O}_\nu}\ar[r] & X_{K,S}
}$$
inducing the commutative square of refined class formations as discussed in Example \ref{compatiblewm}, where $X_{\mc{O}_\nu}\simeq B\mbb{Z}$ is the Weil-Moore anima for the residue field if $\nu$ is nonarchimedean and $X_{\mc{O}_\nu} = B\mbb{R}_{>0}$ if $\nu$ is archimedean.  (In either case, $X_{\mc{O}_\nu}=B(K_\nu^\times/(K_\nu^\times)^1)$.)

Then for any discrete coefficient system $M$ on $X_{K,S}$, the above diagram gives a pullback on $M$-cohomology.
\end{lemma}
\begin{proof}
By devissage we reduce to the case $M=\mbb{Z}[G_{K,S}/N]$ for an open normal subgroup $N\subset G_{K,S}$ of finite index.  Then replacing $K$ by the Galois cover corresponding to $N$ we can further reduce to the case $M=\mbb{Z}$.  Since the claim with $\mbb{R}$-coefficients is obvious from Proposition \ref{cohachaus}, we are reduced to showing that the $\mbb{R}/\mbb{Z}$-cohomology gives a pullback.  But by definition this is Pontryagin dual to the claim that the class groups sit in a (derived) pushout square, which follows from the definitions.
\end{proof}

\begin{remark}
On the level of fundamental groups, we get a pushout of condensed groups, because $W_{K,S} = W_{K,S'}/I_\nu$ where $I_\nu$ is the inertia group.  Combining with the above cohomology result and an induction on the Postnikov tower, we deduce that the above commutative square behaves like a pushout square with respect to maps to discrete anima.  It also behaves like a pushout square with respect to maps to condensed anima with compact Hausdorff homotopy, by an analogous argument.

It seems that, even though the above square is (presumably?) not literally a pushout of condensed anima, it at least behaves like one with respect to maps to most targets.  Intuitively, $X_{K,S}$ is obtained from $X_{K,S'}$ by ``filling in the puncture at $\nu$".
\end{remark}

\begin{remark}\label{fillinpuncturegal}
The analog of the above result also holds for Galois cohomology, in a suitable context: suppose $p$ is a prime, $K$ is a number field, $S$ a finite set of places of $K$ containing all archimedean places and all places above $p$, and let $M$ be a $p$-primary torsion discrete $G_{K,S}$-module.  If $\nu$ is a place of $K$ not in $S$ and we set $S' = S\cup\{\nu\}$, then we have a pullback diagram

$$
\xymatrix{
R\Gamma(BG_{K,S};M)\ar[r]\ar[d] & R\Gamma(BG_{\kappa_\nu};M)\ar[d] \\
R\Gamma(BG_{K,S'};M)\ar[r] & R\Gamma(BG_{K_\nu};M) 
},$$ 
where $\kappa_\nu$ denotes the residue field.  Indeed, if we take etale cohomology instead of Galois cohomology everywhere, this is a straightforward consequence of standard gluing triangles (plus the fact that the etale site of $K_\nu$ agrees with that of its henselian analog, which is a consequence of Krasner's lemma).  On the other hand etale cohomology and Galois cohomology agree in the given situation by \cite{milneduality} Proposition 2.9.
\end{remark}

Now we can compare Galois cohomology to Weil-Moore cohomology.  We will see that the difference is concentrated at the archimedean places.

The local case is simple.

\begin{lemma}\label{nonarchcomp}
For any $\mbb{Z}$-type Weil-Moore anima $X$ with $G=\wh{\pi_1X}$, the comparison map
$$R\Gamma(BG;M) \to R\Gamma(X;M)$$
is an isomorphism for all discrete torsion $G$-modules $M$.
\end{lemma}
\begin{proof}
Recall that in this case, $X=BW$, and
$$W = G\times_{\wh{\mbb{Z}}}\mbb{Z}$$
for some surjection $G\twoheadrightarrow \wh{\mbb{Z}}$.  Thus we reduce to the familiar claim that the map $B\mbb{Z} \to B\wh{\mbb{Z}}$ induces an isomorphism on cohomology for all discrete torsion coefficient systems.
\end{proof}

\begin{remark}
Note that this also shows that $\wh{BW} = B\wh{W}$ in this setting.
\end{remark}

Now we handle the global case.

\begin{theorem}\label{comparetogal}
Let $p$ be a prime, $K$ a number field, and $S$ a set of places of $K$ containing all archimedean places and all places above $p$.  Fix choices for global Weil-Moore anima $X_{K,S}$ and local Weil-Moore anima $X_{K_\nu}$ for $\nu|\infty$ together with the local-global maps
$$X_{K_\nu} \to X_{K,S}$$
as discussed in Example \ref{compatiblewm}.  Note that by construction this map lives over the analogous natural map

$$BG_{K_\nu} \to BG_{K,S}$$

\noindent of Galois groupoids.

Then for any $p$-primary torsion discrete $G_{K,S}$-module $M$, the resulting commutative square
$$
\xymatrix{
R\Gamma(BG_{K,S};M)\ar[r]\ar[d] & \prod_{\nu\mid\infty} R\Gamma(BG_{K_\nu};M)\ar[d] \\
R\Gamma(X_{K,S};M)\ar[r] & \prod_{\nu\mid\infty} R\Gamma(X_{K_\nu};M) 
}$$ 
is a pullback.

\end{theorem}

\begin{proof}
By devissage and use of filtered colimits we can reduce to the case $M=\mbb{F}_p[G_{K,S}/N]$ for open normal $N\subset G_{K,S}$.  Replacing $K$ by the Galois extension corresponding to $N$ it suffices to treat $M=\mbb{F}_p$.  Moreover, shrinking $N$ if necessary we can always assume that all infinite places of $K$ are complex.  Thus we need to show that if $p$ is a prime and $K$ is a number field with no real places, then

$$R\Gamma(BG_{K,S};\mbb{F}_p)\overset{\sim}{\rightarrow}R\Gamma(X_{K,S};\mbb{F}_p)\times_{\prod_{\nu|\infty}R\Gamma(X_{K_\nu};\mbb{F}_p)}\prod_{\nu|\infty}\mbb{F}_p.$$

\noindent If we compare with the case of Lemma \ref{fillinpuncture} where $\nu$ is complex, we find that this is also equivalent to showing that 
$$R\Gamma(BG_{K,S};\mbb{F}_p)\overset{\sim}{\rightarrow} R\Gamma(X_{K,S_f};\mbb{F}_p),$$
where $S_f\subset S$ denotes the subset of nonarchimedean places.

Next let us reduce to the case where $S$ is the set of all places of $K$.  For this, note that if we compare Lemma \ref{fillinpuncture} with Remark \ref{fillinpuncturegal} and Lemma \ref{nonarchcomp}, we see that for $S$ finite, the statements for $S$ and $S\cup\{\nu\}$ are equivalent for any $\nu$ not in $S$.  It follows readily that the statements for any two finite sets $S$ and $S'$ are equivalent (meaning that the homotopy fibers of the comparison maps are isomorphic via the obvious natural map when $S\subset S'$).  Passing to the colimit we deduce the same also for infinite sets, whence the desired reduction.

Thus we can assume $S$ is the set of all places.  The cohomology $R\Gamma(X_{K,S_f};\mbb{F}_p)$ is given by global sections of the derived sheaf
$$(F/p)^\vee$$
on $\on{Spec}(K)_{et}$, where $F$ is the cosheaf giving the refined class formation, $(-)/p$ means derived modding out by $p$, and $(-)^\vee$ means Pontryagin duality.  This receives a map from
$$\mbb{F}_p = (\mbb{Z}/p)^\vee$$
via the augmentation $\alpha:F\to\mbb{Z}$, and on hypercohomology this map implements the comparison map we're trying to prove is an isomorphism.  Thus, it suffices to show that the sheafification of the fiber of the map of presheaves $\mbb{F}_p\to (F/p)^\vee$ vanishes.  But this fiber is
$$(A/p)^\vee,$$
where $A$ is the copresheaf giving the class groups.  Thus, to prove the result, what we need to show is that when we take the colimit of these values $(A/p)^\vee$ along all extensions $L$ of $K$ inside a fixed algebraic closure, we get $0$.

Recall that these class groups are given by
$$A(L) = (L^\times\backslash\mathbb{A}_L^\times)/\prod_{\mu\mid\infty}(L_\mu^\times)^1.$$
The cokernel of multiplication by $p$ on $A(L)$ is profinite (on the level of $S$-class groups for finite $S$ it is finite as we saw above, and then we take the inverse limit).  Thus, we deduce from the existence theorem that the inverse limit of $A(L)/pA(L)$ gives $0$.  We are therefore reduced to checking that the inverse limit of the $p$-torsion in $A(L)$ gives $0$.  For this, consider the short exact sequence

$$0 \to L^\times \to \oplus_{\mu | \infty} \mbb{R}_{>0} \times \mbb{A}_{L,f}^\times \to A(L) \to 0.$$
By the snake lemma we deduce that suffices to show the following two claims:
\begin{enumerate}
\item $\varprojlim_L\on{ker}\left(L^\times\otimes \mbb{F}_p \to \prod_{\mu\in S_f}L_\mu^\times\otimes\mbb{F}_p\right)=0$.
\item $\varprojlim_L \left(\prod_{\mu\in S_f}\mu_p(L_\mu)\right)=0$.

\end{enumerate}

For 1, in fact the kernel vanishes for fixed $L$, without needing to pass to the inverse limit, by the Grunwald-Wang theorem.  For 2, note that if $E'/E$ is any field extension of degree divisible by $p$, then the norm map $\mu_p(E')\to\mu_p(E)$ is $0$.  Thus it suffices to show (up to replacing $L$ by an extension) that for all $\mu\in S_f$ there is an extension $L'/L$ whose local degree at $\mu$ is divisible by $p$.  This can be arranged by taking a sufficiently large $p$-cyclotomic extension.
\end{proof}

\begin{remark}
In the above pullback square, we took the product only over the archimedean places in $S$.  But actually we can also throw in as many non-archimedean places as we like  and the statement will still hold.  This is immediate from Lemma \ref{nonarchcomp}.  In particular, we can interpret Theorem \ref{comparetogal} as saying that the compactly supported Galois cohomology is the same as the compactly supported Weil-Moore cohomology.

More precisely, if $p$ is a prime, $K$ is a number field, and $S$ is a finite set of places containing all archimedean and $p$-adic places, then
$$R\Gamma_c(BG_{K,S};M)\overset{\sim}{\rightarrow} R\Gamma_c(X_{K,S};M)$$
for all $p$-primary torsion $G_{K,S}$-modules $M$, where

$$R\Gamma_c(BG_{K,S};M) = \on{Fib}\left(R\Gamma(BG_{K,S};M)\to\prod_{\nu\in S}R\Gamma(BG_{K_\nu};M)\right)$$
and
$$R\Gamma_c(X_{K,S};M) = \on{Fib}\left(R\Gamma(X_{K,S};M)\to\prod_{\nu\in S}R\Gamma(X_{K_\nu};M)\right).$$
\end{remark}

\begin{remark}
There is an obvious $\pi_{\leq 1}$-level analog of this statement: for any discrete groupoid $Y$, we have
$$\on{Map}(BG_{K,S},Y) \overset{\sim}{\rightarrow} \on{Map}(X_{K,S},Y)\times_{\prod_\nu\on{Map}(X_{K_\nu},Y)}\prod_\nu\on{Map}(BG_{K_\nu},Y).$$
Inducting on the Postnikov tower, we deduce the same statement for $Y$ any discrete anima whose higher homotopy groups are $p$-primary torsion.  In this sense we can say that $BG_{K,S}$ (which is the usual etale homotopy type of $\on{Spec}(\mc{O}_{K,S})$, at least as far as $p$-primary coefficients are concerned) is obtained from $X_{K,S}$ by ``filling in the archimedean punctures'' (by coning off $S^2$ to a point for complex places, and coning off $\mbb{R}\mbb{P}^2$ to the orbifold point $BC_2$ at real places).  I thank Jacob Lurie for suggesting this interpretation.
\end{remark}

The last result we want to discuss is Poitou-Tate duality.  But it will help to set up a bit of formalism of coefficient objects first.  We will work in the $\mbb{Z}$-linear context, and we take the following definition.

\begin{definition}
Let $X$ be a Weil-Moore anima, or more generally a condensed anima with $\pi_0X=\ast$.
\begin{enumerate}
\item Define a \emph{condensed coefficient system} on $X$ to be an object $M$ in the $\infty$-category $\on{D}(\on{CondAn}_{/X})$ of unbounded derived abelian group objects in the slice topos $\on{CondAn}_{/X}$; that is, $\on{D}(\on{CondAn}_{/X})$ is $\infty$-category of small $\on{D}(\mbb{Z})$-valued sheaves on the site $\on{Extr}_{/X}$ of extremally disconnected profinite sets with a map to $X$.
\item Define a \emph{discrete coefficient system} to be a condensed coefficient system such that the object on $\on{D}(\on{CondAb})$ obtained by pulling back under (any choice of) $\ast \to X$ is discrete, i.e.\ lies in $\on{D}(\mbb{Z})\subset\on{D}(\on{CondAb})$.
\end{enumerate}
Denote the full subcategory of discrete coefficient systems by
$$\on{D}_{et}(X;\mbb{Z}) \subset \on{D}(\on{CondAn}_{/X}).$$
\end{definition}

\begin{remark}
If we fix $\ast \to X$, giving the condensed group $\Omega X$, then $\on{D}(\on{CondAn}_{/X})$ is the same as the $\infty$-category of $\Omega X$-equivariant derived condensed abelian groups
$$\on{D}_{\Omega X}(\on{CondAb})$$
which we considered in Section \ref{condback}.  Thus, informally, $\on{D}_{et}(X;\mbb{Z})$ is the $\infty$-category of discrete derived abelian groups with continuous $\Omega X$-action.
\end{remark}

\begin{remark} By general topos theory, $\on{D}(\on{CondAn}_{/X})$ has a natural t-structure, has all colimits, and has a natural tensor product; moreover all of these features are preserved under pullback, in particular under the map $\ast \to X$.  Note that, over $\ast$, all these operations preserve the full subcategory $\on{D}(\mbb{Z})\subset\on{D}(\on{CondAn})$.  It follows that $\on{D}_{et}(X;\mbb{Z})$ is closed under colimits, tensor products, and t-truncations inside $\on{D}(\on{CondAn}_{/X})$, and that all these operations commute with pullback along $\ast \to X$ viewed as a conservative functor
$$\on{D}_{et}(X;\mbb{Z})\to\on{D}(\mbb{Z}).$$
The heart of $\on{D}_{et}(X;\mbb{Z})$ identifies with the abelian category of discrete abelian groups with continuous $\pi_1X = W$-action.
\end{remark}

\begin{remark}
If we write $X$ as a (small) colimit of extremally disconnected profinite sets, we can view $\on{D}_{et}(X;\mbb{Z})$ as the corresponding limit of the usual (unbounded derived) sheaf categories on profinite sets.  We will not use this remark, except perhaps to note that $\on{D}_{et}(X;\mbb{Z})$ is a presentable $\infty$-category.
\end{remark}

\begin{remark}
Suppose $R$ is an $E_1$-$\mbb{Z}$-algebra and $M\in\on{D}(R)$.  By definition, to give an $N\in\on{Mod}_R(\on{D}_{et}(X;\mbb{Z}))$ whose pullback along $\ast \to X$ gives $M$ is the same as to give a pointed map of condensed anima
$$X \to B\on{Aut}_R(M)$$
where $\on{Aut}_R(M)$ is the \emph{condensed} automorphism object of $M$ as an $R$-module.  If $M\in\on{Perf}(R)$, however, then $\on{Aut}_R(M)$ is discrete and we are in the realm of maps from $X$ to discrete anima.  For example, if $R$ is an animated commutative ring, then local systems of $R$-modules on $X$ of rank $d$ can either be understood as objects in $\on{Mod}_R(\on{D}_{et}(X;\mbb{Z}))$ whose underlying object in $\on{D}(R)$ lies in $\on{Vect}_d(R)$ or as being classified by maps of condensed anima
$$X \to B\on{GL}_d(R)$$
where the target is discrete.
\end{remark}

We have the following basic results allowing to work with these coefficient systems; compare with the set-up in \cite{locsysres} Appendix E.

\begin{lemma}\label{lisse}
Let $X$ be a Weil-Moore anima, with underlying class formation of $\mbb{Z}$-type, $\mbb{R}$-type, or compact type, with $W=\pi_1X$.  Then:

\begin{enumerate}
\item The global sections functor $\on{D}_{et}(X;\mbb{Z}) \to \on{D}(\on{CondAb})$ lands in $\on{D}(\mbb{Z})$ and preserves colimits.
\item For every connected finite etale $Y\to X$, the object $\mbb{Z}[h_Y]\in \on{D}_{et}(X;\mbb{Z})$ corepresenting sections over $Y$ is a compact object.  If $Y\to X$ corresponds to the open subgroup $H\subset W=\pi_1X$ of finite index, then $\mbb{Z}[h_Y]$ lives in degree $0$ and corresponds to $\mbb{Z}[W/H]$ as a $W$-module.
\item Denote by $\on{Lisse}(X)\subset\on{D}_{et}(X;\mbb{Z})$ the thick subcategory generated by the objects $\mbb{Z}[h_Y]$ as above.  Then $M\in\on{Lisse}(X)$ if and only if $H_iM=0$ for all but finitely many $i$, and $H_iM$ corresponds to a finitely generated discrete abelian group with continuous $\wh{W}$-action for all $i$.

\item The induced colimit-preserving functor
$$\on{Ind}(\on{Lisse}(X))\to \on{D}_{et}(X;\mbb{Z})$$
is fully faithful.
\item $\on{Ind}(\on{Lisse}(X))$ inherits the t-structure from $\on{D}_{et}(X;\mbb{Z})$, i.e.\ the essential image of the functor in 4 is closed under t-truncation.
\item Suppose $X$ is either compact-type or $\mbb{R}$-type.  Then for all $n\in\mbb{Z}$, the full subcategories of $t_{\leq n}$-truncated objects in $\on{Ind}(\on{Lisse}(X))$ and $\on{D}_{et}(X;\mbb{Z})$ are identified via the above fully faithful functor, and (hence) $\on{D}_{et}(X;\mbb{Z})$ is the left-completion of $\on{Ind}(\on{Lisse}(X))$.
\item Suppose $X$ is of profinite type (meaning, the class groups in the underlying class formation are profinite abelian groups, or equivalently $X\simeq BW$ where $W$ is profinite).  Then $\on{Ind}(\on{Lisse}(X))=\on{D}_{et}(X;\mbb{Z})$.
\end{enumerate}

\end{lemma}
\begin{proof}
For 1, let us first check the claim that
$$\Gamma(X;M)\in\on{D}(\mbb{Z})$$
for all $M\in\on{D}_{et}(X;\mbb{Z})$.  When $M$ is in degree $0$, this is given by Lemma \ref{cohachaus}.  The claim follows by devissage when $M$ lives in finitely many degrees.  Next suppose $M$ is connective.  We have
$$\Gamma(X;M)\overset{\sim}{\rightarrow} \varprojlim_n \Gamma(X;\tau_{\leq n}M).$$
By the cohomological dimension estimate in Theorem \ref{strictcohdim}, this limit stabilizes in any given degree, so we deduce the claim for such $M$. (If $X$ is $\mbb{Z}$-type but does not satisfy the technical hypothesis in Theorem \ref{strictcohdim}, we still get finite cohomological dimension, just with bound worse by one.)  Finally, if $M$ is arbitrary, then
$$\Gamma(X;\tau_{\geq d}M)\to\Gamma(X;M)$$
is an isomorphism in degrees $\geq d$ so we reduce to the connective case by varying $d$ (and shifting).

Next we similarly show that $\Gamma(X;-)$ preserves colimits on discrete coefficient systems.  It suffices to consider direct sums.   By means of truncation and by the finite cohomological dimension property, we reduce to the case where all the objects are in degree $0$, in which case we get the claim from Lemma \ref{cohachaus}.

For 2, the claim that $\mbb{Z}[h_Y]$ is compact is immediate from 1 (applied to $Y$, which is also a Weil-Moore anima).  For the claim that $\mbb{Z}[h_Y]$ lives in degree $0$ and is described as claimed, we can even show this for the object corepresenting global sections over $Y$ for all condensed coefficient systems on $X$: it will a fortiori live in the discrete full subcategory.  Namely, this $\mbb{Z}[h_Y]$ is given by the value on the constant sheaf $\mbb{Z}$ of the left adjoint to pullback from $X$ to $Y$.  Such left adjoints satisfy base-change by topos nonsense, so the pullback to $\ast \to X$ identifies with the analogous object for the pullback $Y\times_X\ast \to \ast$, which is just the finite set $W/H$.  The claim follows.

For 3, the collection of finitely generated discrete $\wh{W}$-modules in degree $0$ is an abelian subcategory closed under extensions (inside all continuous discrete $W$-modules, which is the heart of $\on{D}_{et}(X;\mbb{Z})$), which shows that the claimed collection of $M$ is a thick subcategory.  As by 2 each $\mbb{Z}[h_Y]$ lies in the claimed collection, we deduce one containement.  For the other containment, by devissage it suffices to show that every finitely generated discrete $\wh{W}$-module $M$ lies in the thick subcategory generated by the $\mbb{Z}[W/H]$ (inside $\on{D}_{et}(X;\mbb{Z})$).  By finite generation, there is an open normal subgroup $N\subset W$ of finite index which acts trivially on $M$.  Then we can make a (possibly infinite-length) free resolution of $M$ by finitely generated free $\mbb{Z}[W/N]$-modules.  In $\on{D}_{et}(X;\mbb{Z})$ this identifies $M$ with the filtered colimit of the stupid truncations of this resolution.  However, we claim that $M$ is a compact object, which will imply that $M$ is a retract of one of these truncations and hence lies in the claimed thick subcategory.  To prove the compactness, note that the object of $\on{D}(\mbb{Z})$ obtained by pulling $M$ back along $\ast \to X$ lies in $\on{Perf}(\mbb{Z})$; it follows that $M$ is dualizable under the tensor product, and therefore compactness of $M$ follows from compactness of the unit, given by 1.

Claim 4 is a formal consequence: compact objects form a thick subcategory, hence every object in $\on{Lisse}(X)$ is compact, giving the fully faithful embedding as claimed.

For 5, we need to see that if $M\in\on{Ind}(\on{Lisse}(X))$, then $\tau_{[n,m]}M\in\on{Ind}(\on{Lisse}(X))$ for all $n\leq m$ in $\mbb{Z}$.  Because the t-structure on $\on{D}_{et}(X;\mbb{Z})$ is compatible with filtered colimits, it suffices to show that $\on{Lisse}(X)$ is closed under truncations.  But this follows from 3.

For 6, by compatibility with t-truncation, it suffices to show that every object in the heart of $\on{D}_{et}(\mbb{Z})$ lies in $\on{Ind}(\on{Lisse}(X))$.  But this follows because the connected component $W^\circ$ acts trivially on every discrete module, and $W/W^\circ\simeq \wh{W}$ in the $\mbb{R}$-typpe or compact-type settings.

For 7, it suffices to show that if $M\in\on{D}_{et}(X;\mbb{Z})$ satisfies that $\Gamma(Y;M)=0$ for all finite etale $Y\to X$, then $M=0$.  But because $X=BW$ with $W$ profinite, the the inverse limit over all such $Y$ with a compatible map from $\ast \to X$ gives $\ast$, so on taking $M$-cohomology we deduce that the pullback of $M$ to $\ast$ is the filtered colimit of the $\Gamma(Y;M)$ hence is $0$ as desired. 
\end{proof}

\begin{remark}
The above result expresses that $\on{Ind}(\on{Lisse}(X))$ is quite close to $\on{D}_{et}(X;\mbb{Z})$.  However, outside the profinite case (or the $\mbb{R}$-type case with profinite maximal compact subgroup, which is equivalent with respect to discrete coefficients) it is unlikely that
$$\on{Ind}(\on{Lisse}(X)) \to \on{D}_{et}(X;\mbb{Z})$$
can be an isomorphism.  For example, if $X$ is the Weil-Moore anima of $\mbb{C}$, then this is equivalent to the question of whether the full subcategory of the derived category of sheaves on $S^2$ generated under colimits and shifts by the constant sheaf $\mbb{Z}$ identifies with the full subcategory spanned by those derived sheaves with constant homology sheaves.  The  answer is negative as is explained by a simple Koszul duality argument in \cite{locsysres} E.2.6.
\end{remark}

Recall from Proposition \ref{profcomp} the notion of profinite completion of a condensed anima.  We now analyze the difference between $D_{et}(-;\mbb{Z})$ for $X$ and for its profinite completion $\wh{X}$.  The relation is particularly simple for torsion coefficients.

\begin{proposition}
Let $X$ be a Weil-Moore anima of $\mbb{Z}$-type, $\mbb{R}$-type, or compact type, and let $\wh{X}$ be its profinite completion and $W=\pi_1X$.  For a presentable stable $\infty$-category $\mc{C}$, let $\mc{C}_{tors}$ denote the torsion full subcategory, consisting of those $X\in\mc{C}$ such that $X\otimes\mbb{Q}=0$.  Then the pullback
$$\on{D}_{et}(\wh{X};\mbb{Z})_{tors} \to \on{D}_{et}(X;\mbb{Z})_{tors}$$
is fully faithful.  Moreover, the essential image consists of those $M\in\on{D}_{et}(X;\mbb{Z})_{tors}$ such that $H_nM$, which is a discrete torsion $W$-module, comes from a $\wh{W}$-module, for all $n\in\mbb{Z}$.  In particular, the above pullback functor is an equivalence if $X$ is compact-type or $\mbb{R}$-type.
\end{proposition}
\begin{proof}
First we show that if $A$ is a torsion abelian group with continuous $\wh{W}=\pi_1\wh{X}$-action, then
$$R\Gamma(\wh{X};A)\overset{\sim}{\rightarrow} R\Gamma(X;A).$$
Cohomology commutes with filtered colimits on both $X$ and $\wh{X}$ by Lemma \ref{cohachaus}, so we can reduce to the case where $A$ is finite and is trivial on some finite Galois cover of $\wh{X}$.  Finite Galois covers of $X$ and $\wh{X}$ correspond via $Y\mapsto \wh{Y}$, and if $G$ is the (finite) Galois group then $Y/G=X$ and $\wh{Y}/G=\wh{X}$.  Thus the claim reduces to the case of constant finite coefficients, which follows from the definition of profinite completion.

From this fact, we deduce in particular that $\wh{X}$ has finite cohomological dimension with torsion coefficients (as we know the same for $X$).  The same holds for any finite etale cover, because it is an instance of the same fact.  Thus by the same argument as in Lemma \ref{lisse}, we get
$$\on{Ind}(\on{Lisse}_{tors}(\wh{X}))\hookrightarrow \on{D}_{et}(\wh{X};\mbb{Z})_{tors}$$
where $\on{Lisse}_{tors}(\wh{X})$ is the thick subcategory generated by the $\mbb{F}_p[h_Y]$ for $Y\to\wh{X}$ finite etale and varying primes $p$.  We also get that this fully faithful inclusion identifies $\on{D}_{et}(\wh{X};
\mbb{Z})_{tors}$ with the left t-completion of $\on{Ind}(\on{Lisse}_{tors}(\wh{X}))$ (which is closed under t-truncation).  Comparing with $\on{Ind}(\on{Lisse}(X))\hookrightarrow\on{D}_{et}(X;\mbb{Z})$, we see that to prove the fully faithfulness it suffices to show that Homs over $\wh{X}$ and $X$ match between objects in $\on{Lisse}_{tors}(\wh{X})$.  By dualizability we reduce to Homs from the unit, where the claim amounts to the cohomology comparison already given above.

For the claim about the essential image, note that the fully faithful functor
$$\on{D}_{et}(\wh{X};\mbb{Z})_{tors}\hookrightarrow\on{D}_{et}(X;\mbb{Z})_{tors},$$
being induced by a pullback map, is t-exact.  It follows that an object lies in the essential image if and only if each of its homology objects does, and the claim then follows readily.
\end{proof}

\begin{corollary}
Suppose $X$ and $Y$ are Weil-Moore anima such that $Y$ is of $\mbb{R}$-type or compact type.  Then for every map $f:X\to Y$ and every $M\in\on{D}_{et}(Y;\mbb{Z})_{tors}$, the pullback $f^\ast M$ uniquely descends to an object in $\on{D}_{et}(\wh{X};\mbb{Z})_{tors}$.
\end{corollary}
\begin{proof}
By the above, $M$ uniquely descends to $\wh{Y}$, and then the claim follows from functoriality.
\end{proof}

\begin{remark}\label{globaldescends}
In particular, this applies when $Y$ is the Weil-Moore anima of a number field and $X$ is the Weil-Moore anima of one of its nonarchimedean local fields.  The conclusion is that a global Weil-Moore representation with discrete torsion (possibly derived) coefficients always naturally restricts to a local \emph{Galois} representation, not just a local Weil group representation.  This claim clearly does not hold with $\mbb{C}$-coefficients.
\end{remark}

\begin{corollary}
Suppose $X$ and $Y$ are Weil-Moore anima such that $Y$ is of compact type or $\mbb{R}$-type.  Then for any map $f:X\to Y$, the pullback functor
$$f^\ast:\on{D}_{et}(Y;\mbb{Z})_{tors} \to \on{D}_{et}(X;\mbb{Z})_{tors},$$
up to homotopy, only depends on the map of condensed refined class formations induced by $f$.
\end{corollary}
\begin{proof}
By the above, $f^\ast$ factors through the profinite completion, so this follows from Remark \ref{okonprof}.
\end{proof}

This means that for the purposes of discrete torsion coefficients, we can ignore the possible non-uniqueness up to homotopy of Weil-Moore maps.  (We recall that in number-theoretic examples, whenever $Y$ is of $\mbb{Z}$-type then so is $X$, and hence we are anyway able to avoid the nonuniqueness of Weil-Moore maps, even before passing to profinite completion.)

Finally, we turn to duality.  We aim to prove a version of Poitou-Tate duality for Weil-Moore anima.  We start with local fields, where in the nonarchimedean case the result is of course just equivalent to usual local Tate duality by Lemma \ref{nonarchcomp}.  We will give a proof from the Weil-Moore perspective, as a warm-up for the global case.  The key thing to understand turns out to be the following: if $F$ is the cosheaf giving the refined class formation, then there are two natural ways to build a sheaf from $F$, namely taking by applying the covariant operation $(-)^{tr}$ or by applying the contravariant operation $(-)^\vee$, and at the level of linear algebra local Tate duality will amount to the assertion that mod $p$, the two resulting sheaves are isomorphic up to a twist and shift.  But we will set things up in a more standard way.

\begin{theorem}\label{localtate}
Let $p$ be a prime and $K$ a local field of characteristic $\neq p$.  Let $X$ be a Weil-Moore anima for $K$, with $W=\pi_1X$ the Weil group.  Then:
\begin{enumerate}
\item There is a natural map
$$[X]:R\Gamma(X;\left(\mbb{Q}_p/\mbb{Z}_p\right)(1)[2])\to \mbb{Q}_p/\mbb{Z}_p$$
which is an isomorphism in degree $0$, the bottom (homological) degree.
\item For every $p$-power torsion $E_1$-$\mbb{Z}$-algebra $R$ and every $M\in\on{Perf}(R)$ with $\Omega X$-action, assumed to come from a $\wh{W}$-action if $K$ is nonarchimedean, the $R$-$R$-bimodule map
$$R\Gamma(X;M(1)[2])\otimes_\mbb{Z} R\Gamma(X;M^\vee)\to R\Gamma(X;R(1)[2])\overset{[X]\otimes_{\mbb{Z}}R[-1]}{\longrightarrow} R$$
induces an isomorphism
$$R\Gamma(X;M^\vee)\simeq R\Gamma(X;M(1)[2])^\vee$$
in $\on{Perf}(R^{op})$, where $(-)^\vee$ is the functor of left $R$-linear maps to $R$, viewed as landing in $\on{Perf}(R^{op})$ via the right $R$-module structure on $R$.

\end{enumerate}
\end{theorem}
\begin{proof}
Consider the sheaf on $\on{fEt}_X$ given by
$$\mathcal{F}:= \left(R\Gamma(-;(\mbb{Q}_p/\mbb{Z}_p)(1)[2])^\vee\right)^{tr},$$
so that, in terms of the fundamental cosheaf $F$ from the refined class formation, we have
$$\mc{F} = (F^{tr})_{\wh{p}}(-1)[-2].$$
By construction of the class formation the (unsheafified) $H_0(F^{tr})$ is $\mbb{Z}^{tr}$ and $H_1(F^{tr})$ is $\mbb{G}_m$ and all the other $H_i$ are $0$.  Thus, $(F^{tr})_{\wh{p}}$ has top nonvanishing homology group given by $H_2 = \mbb{Z}_p(1)$.  Tate untwisting and shifting down by two, we deduce that
$$H_0\mc{F} = \mbb{Z}_p$$
and that this is the top nonvanishing homology group sheaf of $\mc{F}$.  In particular, $1\in\mbb{Z}_p$ corresponds to a canonical map

$$[X]: R\Gamma(X;(\mbb{Q}_p/\mbb{Z}_p)(1)[2])\to\mbb{Q}_p/\mbb{Z}_p$$
which is an isomorphism in bottom degree $0$ by Pontryagin duality.  This proves 1.

For 2, in the non-archimedean case let us replace $X$ by $\wh{X}$, so that in any case we are either $\mbb{R}$-type or compact type. Let
$$\Gamma^!:\on{D}(\mbb{Z})[p^\infty] \to \on{D}_{et}(X;\mbb{Z})[p^\infty]$$
be the right adjoint to the global sections functor in the opposite direction (on $p$-primary torsion discrete derived coefficient systems).  We claim that the map
$$(\mbb{Q}_p/\mbb{Z}_p)(1)[2] \to \Gamma^!(\mbb{Q}_p/\mbb{Z}_p)$$
adjoint to the map of 1 is an isomorphism, or equivalently (checking mod $p$ and twisting) that the induced map
$$\mbb{F}_p \to (\Gamma^!\mbb{F}_p)(-1)[-2]$$
is an isomorphism.

To prove this, note that $\Gamma^!\mbb{F}_p$ is bounded above in the t-structure by the cohomological dimension estimate in Lemma \ref{strictcohdim}.  Thus, by Lemma \ref{lisse} the above map is an isomorphism if and only if induces an isomorphism of sheaves on $\on{fEt}_X$ when we map in from $\mbb{F}_p[h_Y]$ for arbitrary $Y\to X$ finite etale.  When we implement this, by Lemma \ref{alttr} the object $\Gamma^!\mbb{F}_p(-1)[-2]$ gives rise to the sheaf on $\on{fEt}_X$ given by
$$Y\mapsto (R\Gamma(Y;\mbb{F}_p(1)[2])^\vee)^{tr},$$
or in other words, in terms of the fundamental cosheaf $F$ coming from the refined class formation,
$$Y\mapsto F^{tr}(Y)/p(-1)[-2],$$\
which is exactly the mod $p$ version of the $\mc{F}$ we analyzed above.  Moreover, the map $\mbb{F}_p\to (\Gamma^!\mbb{F}_p)(-1)[-2]$ is we are trying to show is an isomorphism is just selecting a global section of this sheaf on $\on{fEt}_X$, and this global section is the one we wrote down above in the proof of 1.

On the other hand, the source $\mbb{F}_p$ gives rise to the sheaf on $\on{fEt}_X$ given by
$$Y\mapsto R\Gamma(Y;\mbb{F}_p),$$
which in terms of $F$ is
$$Y\mapsto (F/p)^\vee.$$
By construction, this map of sheaves on $\on{fEt}_X$
$$R\Gamma(-;\mbb{F}_p) \to (R\Gamma(-;\mbb{F}_p(1)[2])^\vee)^{tr}$$
is an isomorphism in degree $0$ on global sections, and we want to prove it's an isomorphism.  We will do this by seeing that it's an isomorphism on stalks.

Consider the natural map
$$\mbb{F}_p \to (F/p)^\vee$$
induced by $\alpha:F\to\mbb{Z}$, whose fiber is given by the (derived) $p$-torsion in the Pontryagin dual to the class group.  By the existence theorem, when we pass to stalks we can replace the class group by its connected component.  Thus in the nonarchimedean case $\mbb{F}_p \to (F/p)^\vee$ is an isomorphism on stalks, while in the archimedean case we get that the stalks of $(F/p)^\vee$ are given by $\mbb{F}_p$ in degree $0$ and $\mbb{F}_p(-1)$ in degree $-2$.  We can likewise analyze $F^{tr}/p(-1)[-2]$.  When $F$ is nonarchimedean the sheafified $H_0F^{tr}$ is $\mbb{Q}$ so doesn't contribute mod $p$; moreover multiplication by $p$ on $H_1F=\mbb{G}_m$ is surjective on stalks because $\on{char}(F)\neq p$.  Thus we get that for the stalks of $F^{tr}/p(-1)[-2]$ there is only the $\mbb{F}_p$ in degree $0$ we already discussed at the beginning of this proof.  But when $F$ is archimedean we additionally get an $\mbb{F}_p(-1)$ in degree $-2$ (if we normalize so that $H_0F^{tr}=\mbb{Z}$, and not $\frac{1}{2}\mbb{Z}$, on the complex numbers).

In the nonarchimedean case, we are done: a map from $\mbb{F}_p$ to itself is either $0$ or an isomorphism, but our map was nonzero on global sections by construction.  Thus it must be an isomorphism.

In the archimedean case, because we are checking on stalks we reduce to the case $K=\mbb{C}$.  Then the Tate twist and the $(-)^{tr}$ operation can be ignored, and our map is simply a map
$$R\Gamma(X;\mbb{F}_p) \to R\Gamma(X;\mbb{F}_p[2])^\vee$$
induced by cup product followed by a nonzero linear form on $H^2$.  Because of the simple structure of cohomology (nonzero except in degrees $0$ and $2$, where it's one-dimensional) this forces the map to be an isomorphism, as desired.

To sum up, we have shown that the map in 1 induces an isomorphism
$$(\mbb{Q}_p/\mbb{Z}_p)(1)[2] \overset{\sim}{\rightarrow} \Gamma^!(\mbb{Q}_p/\mbb{Z}_p).$$
Next we claim that\
$$\Gamma^!:\on{D}(\mbb{Z})[p^\infty] \to \on{D}_{et}(X;\mbb{Z})[p^\infty]$$
preserves all colimits, or equivalently that the natural map
$$M\otimes_\mbb{Z} \Gamma^!(\mbb{Q}_p/\mbb{Z}_p[-1]) \to \Gamma^!(M)$$
is an isomorphism for all $M\in\on{D}(\mbb{Z})[p^\infty]$.  We can again test mod $p$ and thus reduce to showing
$$M\otimes_{\mbb{F}_p} \Gamma^!(\mbb{F}_p)\overset{\sim}{\rightarrow} \Gamma^!(M)$$
for all $M\in\on{D}(\mbb{F}_p)$.  Because $\Gamma^!(\mbb{F}_p)$ (as calculated above) is just a shift and a twist, the functor on the left commutes with the inverse limit of the Postnikov tower of $M$; the functor on the right evidently does as well because $\Gamma^!$ is a right adjoint.  Thus we can reduce to the case of $M$ bounded above, when the left-hand side is also bounded above.  Then again by Lemma \ref{lisse} we can test on maps from objects of the form $\mbb{F}_p[h_Y]$ for $Y\to X$ finite etale, in which case the claim follows from the mod $p$ finiteness property of $Y$ from Proposition \ref{pfinite}.

To sum up once more, we have shown that\
$$\Gamma^!(M) \simeq M(1)[2]$$
for all $M\in\on{D}(\mbb{Z})[p^\infty]$.  As both $\Gamma$ and $\Gamma^!$ preserve colimits, the adjunction between them passes to $R$-module objects for all $E_1$-algebras $R$ in $\on{D}(\mbb{Z})[p^\infty]$, i.e.\ all $p$-primary torsion $E_1$-$\mbb{Z}$-algebras $R$.  Now, let $M$ be as in 2: an $R$-linear coefficient system on $X$ whose fiber at $\ast \to X$ lies in $\on{Perf}(R)$.  Then we can calculate the $R^{op}$-module of $R$-linear maps
$$ M(1)[2] \to \Gamma^!(R)$$
 in two ways:
\begin{enumerate}
\item By the adjunction between $\Gamma^!$ and $\Gamma$, they are the same as $R$-linear maps
$$ R\Gamma(X;M(1)[2]) \to R,$$
so we get
$$R\Gamma(X;M(1)[2])^\vee.$$
\item By the above calculation $\Gamma^!(R)\simeq R(1)[2]$, we get
$$R\Gamma(X;M^\vee).$$
\end{enumerate}
This gives an isomorphism
$$R\Gamma(X;M^\vee)\simeq R\Gamma(X;M(1)[2])^\vee,$$
which by a simple unwinding is induced by the pairing described in 2.
\end{proof}

\begin{remark}
Of course, in the archimedean cases we have a better explanation for this duality: it follows from Poincare duality for the manifolds $\mbb{R}\mbb{P}^2$ and $S^2$.  (One also easily sees, with either proof, that the duality holds with $\mbb{Z}$-coefficients, not just $p$-torsion coefficients.)  For a similar geometric proof in the $p$-adic case, see \cite{farguescft}, where the geometric model is the Fargues-Fontaine curve.
\end{remark}

Now we go the number field case with restricted ramification, which is different in that the cohomology is not self-dual (up to twist), but rather is dual to a compactly-supported variant.
 
\begin{theorem}\label{globaltate}
Let $p$ be a prime, $K$ be a number field, and $S$ a finite set of places of $K$ containing all $p$-adic and archimedean places.  Fix Weil-Moore data as in Example \ref{compatiblewm}.  For $M$ a $p$-primary torsion discrete coefficient system on $X:=X_{K,S}$, set

$$R\Gamma_c(X;M):=\on{Fib}\left(R\Gamma(X;M)\to\oplus_{\nu\in S}R\Gamma(X_{K_\nu};M)\right).$$

Then:
\begin{enumerate}
\item There is a unique map
$$[X]:R\Gamma_c(X;(\mbb{Q}_p/\mbb{Z}_p)(1)[3]) \to \mbb{Q}_p/\mbb{Z}_p$$
such that for all (or just one) $\nu\in S$, the composition
$$R\Gamma(X_{K_\nu};(\mbb{Q}_p/\mbb{Z}_p)(1)[2])\overset{\partial_\nu} {\longrightarrow} R\Gamma_c(X;(\mbb{Q}_p/\mbb{Z}_p)(1)[3])\overset{[X]}{\longrightarrow} \mbb{Q}_p/\mbb{Z}_p$$
identifies with the local fundamental class $[X_{K_\nu}]$ from Theorem \ref{localtate}, where $\partial_\nu$ is the $\nu$-component of the boundary map arising from the definition of $R\Gamma_c$.
\item For all $p$-power torsion $E_1$-$\mbb{Z}$-algebras $R$ and all perfect $R$-modules $M$ with $\Omega X$-action, the pairing induced by cup product
$$ R\Gamma_c(X;M(1)[3]) \otimes_{\mbb{Z}} R\Gamma(X;M^\vee) \to R\Gamma_c(X;R(1)[3])\overset{[X]}{\longrightarrow} R$$
induces an isomorphism $$R\Gamma(X;M^\vee)\simeq R\Gamma_c(X;M(1)[3])^\vee,$$
in $\on{Perf}(R^{op})$, where $(-)^\vee$ denotes the duality functor $\on{Perf}(R)^{op} \to \on{Perf}(R^{op})$.
\end{enumerate}
\end{theorem}
\begin{proof}
We proceed similarly to the proof of local Poitou-Tate duality.  For a $p$-primary torsion discrete coefficient system $M$ on $X$, we make the following definition of $R\Gamma_c(-;M)$ as a sheaf on $\on{fEt}_X$: for $Y\to X$ finite etale, set
$$R\Gamma_c(Y;M) := \on{Fib}\left(R\Gamma(Y;M) \to \oplus_{\nu\in S} R\Gamma(Y\times_XX_{K_\nu};M)\right).$$

\noindent From this we define a sheaf on $\on{fEt}_X$ by
$$\mc{F}_c := \left(R\Gamma_c(-;(\mbb{Q}_p/\mbb{Z}_p)(1)[3])^\vee\right)^{tr}.$$
If $F$ denotes the cosheaf giving the condensed refined class formation for $(K,S)$, and ditto $F_\nu$ for the local field $K_\nu$, then setting
$$F_{Kott}=\on{fib}(\oplus_{\nu\in S} (f_\nu)_\ast F_\nu \to F)$$
where $f_\nu$ is the local-to-global pushforward map on finite etale sites, we deduce that
$$\mc{F}_c = (F_{Kott})^{tr}_{\wh{p}}(-1)[-2].$$

By definition of the class groups, $(F_{Kott})^{tr}$ has top nonzero homology group given by $H_1 = \mbb{G}_m$, as a sheaf on the finite etale site of $\mc{O}_{K,S}$ (identified with $\on{fEt}_X$).  Thus its $p$-completion has top nonzero homology group given by $H_2
=\mbb{Z}_p(1)$.  We deduce that $\mc{F}_c$ has top nonzero homology group given by
$$H_0 \mc{F}_c = \mbb{Z}_p.$$
In particular $1\in \mbb{Z}_p$ provides a canonical

$$[X]: R\Gamma_c(X;(\mbb{Q}_p/\mbb{Z}_p)(1)[3])\to\mbb{Q}_p/\mbb{Z}_p.$$
To prove 1, it suffices to note that the map $\mbb{G}_m = H_1 F_{Kott} \to H_1((f_\nu)_\ast F_\nu) = (f_\nu)_\ast \mbb{G}_m$ is by definition the natural map of restricting global units to local units.

Now we turn to 2.  As in the proof of \ref{localtate}, let us consider the functor
$$\Gamma_c^!:\on{D}(\mbb{Z})[p^\infty] \to \on{D}_{et}(X;\mbb{Z})[p^\infty]$$
which is right adjoint to $R\Gamma_c$.  Part 1 gives us a natural map
$$(\mbb{Q}_p/\mbb{Z}_p)(1)[3] \to \Gamma_c^!(\mbb{Q}_p/\mbb{Z}_p),$$
and we claim it is an isomorphism.

Again as in \ref{localtate} it suffices to show
$$\mbb{F}_p \to \Gamma_c^!(\mbb{F}_p)(-1)[-3]$$
is an isomorphism, and for this it suffices to show the induced map of sheaves on $\on{fEt}_X$ is an isomorphism.  By Lemma \ref{alttr} this map of sheaves on $\on{fEt}_X$ takes the form
$$R\Gamma(-;\mbb{F}_p) \to \left(R\Gamma_c(-;\mbb{F}_p(1)[3])^\vee\right)^{tr}.$$
We also have the following additional information from the local-global compatibility property in 1: for all places $\nu$ of $K$, the above map sits in a commutative square
$$\xymatrix{
R\Gamma(-;\mbb{F}_p)\ar[r]\ar[d] & \left(R\Gamma_c(-;\mbb{F}_p(1)[3])^\vee\right)^{tr}\ar[d] \\
R\Gamma((-)\times_X X_{K_\nu};\mbb{F}_p)\ar[r] &  \left(R\Gamma((-)\times_X X_{K_\nu};\mbb{F}_p(1)[2])^\vee\right)^{tr}}$$
where the left arrow is the restriction map, the right arrow is dual to $\partial_v$, and the bottom map is the comparison map from local Poitou-Tate duality.  Translating over to the $F_{Kott}$ and $F_\nu$ as above, we can rewrite this as

$$\xymatrix{
R\Gamma(-;\mbb{F}_p)\ar[r]\ar[d] & F_{Kott}^{tr}/p(-1)[-2]\ar[d] \\
R\Gamma((-)\times_X X_{K_\nu};\mbb{F}_p)\ar[r] &  (f_\nu)_\ast (F_\nu)^{tr}/p(-1)[-2]}.$$

We will also also bring the comparison map of presheaves
$$(f_\nu)_\ast \mbb{F}_p \to R\Gamma((-)\times_X X_{K_\nu};\mbb{F}_p),$$
induced by $\mbb{F}_p\to R\Gamma(-;\mbb{F}_p)$, into play.  Let us restrict only to archimedean $\nu$, and take the direct sum over all $\nu|\infty$ in the bottom row of this commutative square. By Theorem \ref{comparetogal}, the natural map from $\mbb{F}_p$ to the fiber product

$$\oplus_{\nu|\infty}(f_\nu)_\ast\mbb{F}_p\times_{\oplus_{\nu|\infty} R\Gamma( (-)\times_XX_{K_\nu};\mbb{F}_p)} R\Gamma(-;\mbb{F}_p)$$
is an isomorphism on derived sheafification.  On the other hand, the other fiber product
$$\oplus_{\nu|\infty}(f_\nu)_\ast\mbb{F}_p\times_{\oplus_{\nu|\infty}(f_\nu)_\ast (F_\nu)^{tr}/p(-1)[-2]} F_{Kott}^{tr}/p(-1)[-2]$$
also admits a natural map from $\mbb{F}_p$, namely corresponding to the canonical class we found above in $F_{Kott}^{tr}/p(-1)[-2]$ on the one hand, and to the obvious diagonal map $\mbb{F}_p\to \oplus_{\nu|\infty}(f_\nu)_\ast\mbb{F}_p$ on the other hand,  where these agree in $(f_\nu)_\ast (F_\nu)^{tr}/p(-1)[-2]$ because the global fundamental class restricts to the local one at all $\nu$.  Then we claim that this map from $\mbb{F}_p$ to the above fiber product is also an isomorphism on derived sheafification.  

Assuming this claim for now, we know from local Poitou-Tate duality that our archimedean comparison map is an isomorphism.  Thus our global comparison map is an isomorphism if and only if this induced map on fiber product sheaves
$$\mbb{F}_p \to\mbb{F}_p$$
is an isomorphism.  However by construction the source and target compatibly map to $\oplus_{\nu|\infty} (f_\nu)_\ast\mbb{F}_p$ by diagonal embedding.  Thus the map is the identity, finishing the proof.

Thus we are left with showing that

$$\mbb{F}_p \to \oplus_{\nu|\infty}(f_\nu)_\ast\mbb{F}_p\times_{\oplus_{\nu|\infty}(f_\nu)_\ast (F_\nu)^{tr}/p(-1)[-2]} F_{Kott}^{tr}/p(-1)[-2]$$
is an isomorphism on stalks, or equivalently, that the following null-composite sequence is a fiber sequence on sheafification:
$$\mbb{F}_p(1)[2]\to F^{tr}_{Kott}/p \to \oplus_{\nu|\infty}(f_\nu)_\ast\mbb{F}_p^{tr},$$
where the first map corresponds to our fundamental class, coming as above from the identification $H_1F^{tr}_{Kott}=\mbb{G}_m$, and the second map comes from the tautological
$$F_{Kott}\to \oplus_\nu (f_\nu)_\ast F_\nu\to\oplus_\nu (f_\nu)_\ast H_0F_\nu=\oplus_\nu (f_\nu)_\ast\mbb{Z}$$
by applying $(-)^{tr}/p$ (which is exact and commutes with pushforwards).

Multiplication by $p$ on $H_1F^{tr}_{Kott}=\mbb{G}_m$ is surjective on the sheaf level because $p$-Kummer extensions are unramified outside $p$, and obviously the kernel of multiplication by $p$ is $\mbb{F}_p(1)$.  Thus it suffices to show that 
$$H_0F^{tr}_{Kott}\to\oplus_{\nu|\infty}(f_\nu)_\ast\mbb{Z}^{tr}$$
is an isomorphism on stalks modulo uniquely $p$-divisible groups.

From the definitions we have an exact sequence of presheaves on the finite etale site of $\mc{O}_{K,S}$
$$0\to\on{Pic}(-)\to H_0F^{tr}_{Kott}\to \oplus_{\nu\in S}(f_\nu)_\ast\mbb{Z}^{tr}\to \mbb{Z}^{tr} \to 0.$$
Now, $\on{Pic}$ gives $0$ on stalks by the principal ideal theorem, and $\mbb{Z}^{tr}$ is uniquely $p$-divisible on stalks because there are always more and more extensions unramified outside $p$ of degree divisible by $p$, say using the $p$-cyclotomic tower.  Thus it suffices to show that for nonarchimedean $\nu\in S$, the stalks of $(f_\nu)_\ast\mbb{Z}^{tr}$ are uniquely $p$-divisible.  For that, up to replacing $K$ by an aribrary finite extension (unramified outside $S$), we need to show that for any $K,S,\nu$, there is an extension $L$ of $K$ unramified outside $S$ whose local degree at $\nu$ is divisible by $p$.  This can also be arranged using the $p$-cyclotomic tower.

Thus we have proven that
$$(\mbb{Q}_p/\mbb{Z}_p)(1)[3] \overset{\sim}{\rightarrow} \Gamma_c^!(\mbb{Q}_p/\mbb{Z}_p).$$
Next we claim that
$$M\otimes_{\mbb{Z}} \Gamma_c^!(\mbb{Q}_p/\mbb{Z}_p[-1])\overset{\sim}{\rightarrow}\Gamma_c^!(M)$$
for all $M\in\on{D}(\mbb{Z})[p^\infty]$.  Exactly as in the proof of \ref{localtate}, this reduces to bounded above $M$, then to the $p$-finiteness property of \ref{pfinite}.  Again as in \ref{localtate} we deduce that the adjunction between $R\Gamma_c$ and $\Gamma_c^!$ passes to $R$-module coefficients, and then 2 follows by calculating $R$-linear maps
$$M(1)[3] \to \Gamma_c^!(R)\simeq R(1)[3]$$
in two different ways.
\end{proof}

\begin{remark}
We chose the above formulation of Poitou-Tate duality, in the setting of a finite set $S$ of places, to highlight the analogy with 3-manifolds.  But the duality becomes picturesque in a slightly different manner if we pass to the situation where $S$ consists of all places, so we are interested in the Weil-Moore anima $X_K$ of the number field itself.  In this situation, given $R$ and $M$ as above, the analog of the ``boundary cohomology" $\oplus_{\nu\in S}R\Gamma(X_{K_\nu};M)$ is given by the adelic construction
$$R\Gamma(X_{\mbb{A}_K};M) := \varinjlim_{S_0} \prod_{\nu\in S_0}R\Gamma(X_{K_\nu};M)\times \prod_{\nu\not\in S_0} R\Gamma(X_{\mc{O}_\nu};M),$$
where $S_0$ runs over finite subsets of places (of the form considered in Example \ref{compatiblewm}, say).  The object on the right can be canonically viewed as a filtered system of Pro-objects in $\on{Perf}(R)$, such that each transition map in the filtered colimit has cofiber lying in $\on{Perf}(R)\subset \on{Pro}(\on{Perf}(R))$; as such it has the canonical structure of a \emph{Tate $R$-module}, or rather a Tate object in $\on{Perf}(R)$ in the sense of \cite{hennion}.  We recall that:

\begin{enumerate}
\item $\on{Tate}(R)$ contains both $\on{Ind}(\on{Perf}(R))$ and $\on{Pro}(\on{Perf}(R))$ as full subcategories, and is generated by these;
\item Maps $X\to Y$ where $X$ and $Y$ each live in one of these full subcategories (but not necessarily the same one) are the ``obvious'' formal ones;
\item The intersection of these two subcategories identifies with $\on{Perf}(R)$ (sitting inside each in the obvious natural way);
\item The $R$-linear duality $\on{Perf}(R)^{op}\simeq \on{Perf}(R^{op})$ naturally extends to $\on{Tate}(R)^{op}\simeq \on{Tate}(R^{op})$ and interchanges $\on{Ind}(\on{Perf}(-))$ and $\on{Pro}(\on{Perf}(-))$.
\end{enumerate}

On the other hand, $R\Gamma(X_K;M)$ can be naturally viewed as lying in $\on{Ind}(\on{Perf}(R))$ hence also in $\on{Tate}(R)$, by writing $X_K=\varprojlim_{S_0} X_{K,S_0}$ which gives a colimit on $R\Gamma$.  By the same token there is a natural restriction map $$R\Gamma(X_K;M) \to R\Gamma(X_{\mbb{A}_K};M)$$
in $\on{Tate}(R)$.

Then the statement of Poitou-Tate duality, at this level, is that the fiber $R\Gamma_c(X_K;M)$ of above map lies in $\on{Pro}(\on{Perf}(R))\subset \on{Tate}_R$, and the resulting cofiber sequence

$$R\Gamma(X_K;M) \to R\Gamma(X_{\mbb{A}_K};M) \to R\Gamma_c(X_K;M)[1]$$
has the following self-duality: the $R$-linear dual of this cofiber sequence  for $M$ identifies with the same cofiber sequence but for $M^\vee(1)[2]$.  On the middle term this follows from local Tate duality, and for the outer terms it follows from global Tate duality (in the form proved above).

From this statement one can, following Nekovar (\cite{nekovar}) extract dualities in $\on{Perf}(R)$ by a formal procedure.  Suppose given a ``Selmer structure'' $\mc{S}$ for $M$, meaning a cofiber sequence
$$ C \to R\Gamma(X_{\mbb{A}_K};M) \to D$$
where $C\in\on{Pro}(\on{Perf}(R))$ and $D\in\on{Ind}(\on{Perf}(R))$ (so, opposite to the situation of $R\Gamma(X_K;M)$ mapping to $R\Gamma(X_{\mbb{A}_K};M)$, where something discrete maps with compact quotient).  Then applying $R$-linear duality and local Tate duality we get a ``dual Selmer structure'' $\mc{S}^\ast$ on $M^\vee (1)[2]$ given by

$$ D^\vee \to R\Gamma(X_{\mbb{A}_K};M^\vee (1)[2]) \to C^\vee.$$
If we define the Selmer cohomology to be
$$R\Gamma_{\mc{S}}(X_K;M) := C\times_{R\Gamma(X_{\mbb{A}_K};M)}R\Gamma(X_K;M),$$
then it formally follows that $R\Gamma_{\mc{S}}(X_K;M)\in \on{Perf}(R)$, and we have the following duality:
$$R\Gamma_{\mc{S}}(X_K;M)^\vee = R\Gamma_{\mc{S}^\ast}(X_K;M^\vee(1)[2]).$$
\end{remark}

In the above proofs of Poitou-Tate duality we used a bit of abstract formalism that we now should justify.  The situation is the following: let $\mc{C}$ be a Galois category and $\mc{F}$ a $\on{D}(\mbb{Z})$-valued sheaf on $\mc{C}$.  Consider the functor
$$M\mapsto R\Gamma(\mc{C};M\otimes\mc{F})$$
from $\on{Sh}(\mc{C};\on{D}(\mbb{Z}))$ to $\on{D}(\mbb{Z})$, where $\otimes$  means sheafified derived tensor product over $\mbb{Z}$.  Composing with the Yoneda embedding, we get a co-presheaf
$$U \mapsto R\Gamma(\mc{C};\mbb{Z}[h_U]\otimes\mc{F})$$
on $\mc{C}$ with values in $\on{D}(\mbb{Z})$.  On the other hand, as we described in Definition \ref{trfunctor}, the formalism of transfer maps also naturally produces a copresheaf $\mc{F}^{tr}$ from $\mc{F}$, taking the same values as $\mc{F}$.  The claim is that these two copresheaves are naturally isomorphic.

\begin{lemma}\label{alttr}
Let $\mc{C}$ be a Galois category, $\mc{D}$ a $\mbb{Z}$-linear presentable $\infty$-category, and $\mc{F}\in\on{Sh}(\mc{C};\mc{D})$.  Then there is a natural isomorphism between the following two functors $\mc{C}\to\mc{D}$:
\begin{enumerate}
\item The functor $U\mapsto \Gamma(\mc{C};\mbb{Z}[h_U]\otimes\mc{F})$;
\item The functor $\mc{F}^{tr}$ from Definition \ref{trfunctor}.
\end{enumerate}
\end{lemma}
\begin{proof}
By definition of $(-)^{tr}$, it suffices to show that the first functor extends to a functor
$$\on{Span}(\mc{C})\to\mc{D}$$
whose restriction along $\mc{C}^{op} \to \on{Span}(\mc{C})$ identifies with $\mc{F}$.

Such an extension comes as follows: we apply the transfers construction to the cosheaf $U\mapsto \mbb{Z}[h_U]$ to get a functor $\on{Span}(\mc{C}) \to \on{Sh}(\mc{C};\on{D}(\mbb{Z}))$, which one can then compose with $R\Gamma(\mc{C};-\otimes\mc{F})$.  To identify the restiction along $\mc{C}^{op} \to \on{Span}(\mc{C})$ with $\mc{F}$, it suffices to show that that the presheaf given by $(U\mapsto \mbb{Z}[h_U])^{tr}$ identifes with the presheaf given by $U\mapsto \mbb{Z}[h_U]^\vee$ (meaning the $\mbb{Z}$-linear dual).  However both these objects live in degree $0$ so this is elementary to check: one shows that the canonical self-duality of the finite free $\mbb{Z}$-module on a finite set implements the desired identification.
\end{proof}

\begin{remark}
Although $\on{D}(\on{CondAb})$ is not presentable for technical reasons (there is no small set of generators), the conclusion still holds for $\mc{D}=\on{D}(\on{CondAb})$: for example one can reduce to the case $\mc{D}=\on{D}(\mbb{Z})$ by checking on $S$-valued points for all extremally disconnected $S$.  (Recall that taking $S$-valued points preserves all limits and colimits.)
\end{remark}

\begin{remark}
We restricted to the $\mbb{Z}$-linear context for familiarity, but the same claim holds in the more general semi-additive context.  One can still check by hand because the base-case involves free $E_\infty$-spaces on finite sets, which are still just 1-truncated.
\end{remark}

\begin{remark}\label{poitoulemma}
Suppose, in the situation of Lemma \ref{alttr}, that $\mc{F}^{tr}$ is actually a cosheaf.  For example, this happens if $\mc{D}=\on{D}(\mbb{Z})$ and $\mc{F}$ takes values in some bounded range of degrees, by Proposition \ref{trissheaf}.  It follows from left Kan extension that $\mc{F}^{tr}$ extends uniquely along Yoneda to a colimit-preserving $\mbb{Z}$-linear functor
$$\on{Sh}(\mc{C};\on{D}(\mbb{Z}))\to \mc{D},$$
which we will continue to denote
$$\mc{G} \mapsto \mc{F}^{tr}(\mc{G}).$$
Left Kan extension from Lemma \ref{alttr} then provides a natural transformation
$$\mc{F}^{tr}(-) \to R\Gamma(\mc{C};-\otimes\mc{F})$$
which is an isomorphism on the $\mbb{Z}[h_U]$.

We claim that, at least in the situation where $\mc{D}=\on{D}(\mbb{Z})$ and $\mc{F}$ takes values in a bounded range of degrees, the above natural transformation is actually an isomorphism on \emph{all} $\mc{G}\in\on{Sh}(\mc{C};\on{D}(\mbb{Z}))$.  As in the previous remark, we can deduce the analogous claim for $\mc{D}=\on{D}(\on{CondAb})$ by evaluating on extremally disconnecteds.

To prove this, by the above it suffices to show that $R\Gamma(\mc{C};-\otimes\mc{F})$ preserves colimits.  By the criterion as in \cite{cm1} Prop.\ 4.12, for that it suffices to show that for all $X\in\mc{C}$ connected, all subgroups $G\subset \on{Aut}(X)$, and all $M\in \on{D}(\mbb{Z})$ with $G$-action, the transfer map
$$(M\otimes \mc{F}(X))_G\to (M\otimes \mc{F}(X))^G$$
is an isomorphism.  (For then it follows that one iteration of the Cech construction suffices to sheafify the presheaf tensor product, and moreover it also follows that the Cech construction preserves colimits sectionwise, and combing we get the desired claim.)  By Lemma \ref{transferlemma} we have that $\mc{F}(X)$ has vanishing $H$-Tate cohomology for any subgroup $H\subset G$; let us show more generally that the desired conclusion holds with $\mc{F}(X)$ replaced by any $N\in\on{D}(\mbb{Z}[G])$ living in bounded degrees satisfying this vanishing hypothesis in Tate cohomology.

By the usual transfer tricks we reduce to the case of $G$ a $p$-group for some prime $p$.   Resolving $N$ by injective $\mbb{Z}[G]$-modules, we can reduce to the case where $N$ lives in degree $0$.  Then by the Nakayama-Rim theorem (\cite{corpslocaux} Chapter IX), $N$ has a two-term projective resolution, so we reduce to the case when $N$ is a free $\mbb{Z}[G]$-module, which is trivial.
\end{remark}

To finish, let us use this to give a somewhat different perspective on Poitou-Tate duality, by explaining how Poitou's general duality theorem for class formations (\cite{poitou}) looks from the perspective of Weil-Moore anima.

\begin{theorem}
Let $X$ be a Weil-Moore anima with underlying class formation either of $\mbb{Z}$-type, $\mbb{R}$-type, or compact type.  Let $F$ denote the cosheaf on $\mc{C}=\on{fEt}_X$ giving the condensed refined class formation, and let $F^{tr}$ denote the sheaf obtained by passing to transfer maps.

Then for a discrete abelian group $M$ with $\wh{\pi_1X}$-action, corresponding on the one hand to a coefficient system on $X$ and on the other hand to a sheaf of abelian groups on $\mc{C}$, there is a natural map
$$R\Gamma(\mc{C};M\otimes F^{tr})\to R\Gamma(X;M^\vee)^\vee.$$
Here $(-)^\vee$ means Pontryagin duality, and on the left we have derived sheafified tensor product over $\mbb{Z}$.

Moreover, this map is an isomorphism whenever $M$ is finitely generated (as an abelian group).
\end{theorem}
\begin{proof}
We have a tautological pullback functor $\pi^{-1}$ from $\on{fEt_X}$ to condensed anima over $X$, which therefore induces a pullback on $\on{D}(\mbb{Z})$-valued sheaves.  We follow this with Pontryagin duality, then global sections, then internal Hom to $\mbb{R}/\mbb{Z}$ again, getting the $\mbb{Z}$-linear functor
$$\mc{F} \mapsto R\Gamma(X;(\pi^\ast\mc{F})^\vee)^\vee$$
from $\on{Sh}(\mc{C};\on{D}(\mbb{Z}))$ to $\on{D}(\on{CondAb})$.  When we specialize to $\mc{F}=\mbb{Z}[h_U]$ for $U\in\on{fEt}_X=\mc{C}$, we get the value
$$R\Gamma(U;(\mbb{Z})^\vee)^\vee = R\Gamma(U;\mbb{R}/\mbb{Z})^\vee = F(U)$$
Thus, the co-presheaf underlying this functor is $F$, which happens to be a cosheaf.  We deduce by left Kan extension a natural transformation
$$F(-) \to R\Gamma(X;(\pi^\ast-)^\vee)^\vee$$
of $\mbb{Z}$-linear functors $\on{Sh}(\mc{C};\on{D}(\mbb{Z}))\to\on{D}(\on{CondAb})$ which is an isomorphism on all $\mbb{Z}[h_U]$.  Applying Remark \ref{poitoulemma}, we can rewrite this as a natural transformation
$$R\Gamma(\mc{C};(-)\otimes F^{tr}) \to R\Gamma(X;(\pi^\ast-)^\vee)^\vee$$
which is an isomorphism on all $\mbb{Z}[h_U]$.  An arbitrary finitely generated $M$ admits a co-free resolution where each term is finite direct sum of copies of $\mbb{Z}[h_U]$ or $\mbb{Z}/N\mbb{Z}[h_U]$ (for $N\in\mbb{Z}$), so we deduce the natural transformation is an isomorphism on all finitely generated $M$ as well.
\end{proof}

For example, when we take $X=X_{K,S}$ for a number field $K$ and a finite set of places $S$ containing all archimedean places and all places above some fixed prime $p$, then from this it is simple to deduce the following version of Poitou-Tate duality
$$R\Gamma(X_{K,S};M^\vee)\overset{\sim}{\rightarrow} R\Gamma_c(BG_{K,S};M(1)[3])^\vee$$
relating Weil-Moore cohomology to Kato-style compactly supported Galois cohomology, for $p$-primary torsion coefficient systems on $BG_{K,S}$.  This is of course also a consequence of Theorem \ref{globaltate} and Lemma \ref{comparetogal}; Theorem \ref{globaltate} is stronger because it applies to all coefficient systems on $X_{K,S}$, not just those coming via pullback from $BG_{K,S}$.

\printbibliography

@article{weil,
  title={Sur la th{\'e}orie du corps de classes},
  author={Weil, Andr{\'e}},
  journal={Journal of the Mathematical Society of Japan},
  volume={3},
  number={1},
  pages={1--35},
  year={1951},
  publisher={The Mathematical Society of Japan}
}

@inproceedings{tatentb,
  title={Number theoretic background},
  author={Tate, John},
  booktitle={Proc. Sympos. Pure Math.},
  volume={33},
  pages={3--26},
  year={1979}
}

@book{artintate,
  title={Class field theory},
  author={Artin, Emil and Tate, John Torrence},
  volume={366},
  year={1968},
  publisher={American Mathematical Soc.}
}

@article{serrethesis,
  title={Homologie singuli{\`e}re des espaces fibr{\'e}s},
  author={Serre, Jean-Pierre},
  journal={Annals of Mathematics},
  volume={54},
  number={3},
  pages={425--505},
  year={1951},
  publisher={JSTOR}
}

@article{farguesscholze,
  title={Geometrization of the local Langlands correspondence},
  author={Fargues, Laurent and Scholze, Peter},
  journal={arXiv preprint arXiv:2102.13459},
  year={2021}
}

@book{morishita,
  title={Knots and primes},
  author={Morishita, Masanori},
  year={2012},
  publisher={Springer}
}

@inproceedings{tateduality,
  title={Duality theorems in Galois cohomology over number fields},
  author={Tate, John},
  booktitle={Proc. Internat. Congr. Mathematicians (Stockholm, 1962)},
  pages={288--295},
  year={1962}
}

@book{milneduality,
  author    = {J. S. Milne},
  title     = {Arithmetic Duality Theorems},
  edition   = {2},
  publisher = {BookSurge Publishers},
  year      = {2006},
}

@article{deligne,
  title={Les constantes locales de l'{\'e}quation fonctionnelle de la fonction L d'Artin d'une repr{\'e}sentation orthogonale},
  author={Deligne, Pierre},
  journal={Inventiones mathematicae},
  volume={35},
  number={1},
  pages={299--316},
  year={1976},
  publisher={Springer-Verlag Berlin/Heidelberg}
}

@article{simpson,
  title={Mixed twistor structures},
  author={Simpson, Carlos},
  journal={arXiv preprint alg-geom/9705006},
  year={1997}
}

@inproceedings{farguescft,
  title={From local class field to the curve and vice versa},
  author={Fargues, Laurent},
  booktitle={AMS Summer Institute in Algebraic Geometry},
  year={2015}
}

@article{lichtenbaum,
  title={The Weil-{\'e}tale topology for number rings},
  author={Lichtenbaum, Stephen},
  journal={Annals of mathematics},
  pages={657--683},
  year={2009},
  publisher={JSTOR}
}

@article{flach,
  title={Cohomology of topological groups with applications to the Weil group},
  author={Flach, Mathias},
  journal={Compositio Mathematica},
  volume={144},
  number={3},
  pages={633--656},
  year={2008},
  publisher={London Mathematical Society}
}

@incollection{kato,
  title={Lectures on the approach to Iwasawa theory for Hasse-Weil L-functions via BdR. Part I},
  author={Kato, Kazuya},
  booktitle={Arithmetic Algebraic Geometry: Lectures given at the 2nd Session of the Centro Internazionale Matematico Estivo (CIME) held in Trento, Italy, June 24--July 2, 1991},
  pages={50--163},
  year={2006},
  publisher={Springer}
}

@inproceedings{langlands,
  title={Automorphic representations, Shimura varieties, and motives. Ein M{\"a}rchen},
  author={Langlands, Robert P},
  booktitle={Automorphic forms, representations and L-functions (Proc. Sympos. Pure Math., Oregon State Univ., Corvallis, Ore., 1977), Part},
  volume={2},
  pages={205--246},
  year={1979}
}

@article{arthur,
  title={A note on the automorphic Langlands group},
  author={Arthur, James},
  journal={Canadian Mathematical Bulletin},
  volume={45},
  number={4},
  pages={466--482},
  year={2002},
  publisher={Cambridge University Press}
}

@article{carlwe,
  title={Algebraic families of Galois representations and potentially semi-stable pseudodeformation rings},
  author={Wang-Erickson, Carl},
  journal={Mathematische Annalen},
  volume={371},
  number={3},
  pages={1615--1681},
  year={2018},
  publisher={Springer}
}

@article{egh,
  title={An introduction to the categorical p-adic Langlands program},
  author={Emerton, Matthew and Gee, Toby and Hellmann, Eugen},
  journal={arXiv preprint arXiv:2210.01404},
  year={2022}
}

@article{zhu,
  title={Coherent sheaves on the stack of Langlands parameters},
  author={Zhu, Xinwen},
  journal={arXiv preprint arXiv:2008.02998},
  year={2020}
}

@article{gaitsgoryshtukas,
  title={From geometric to function-theoretic Langlands (or how to invent shtukas)},
  author={Gaitsgory, Dennis},
  journal={arXiv preprint arXiv:1606.09608},
  year={2016}
}

@misc{scholzerll,
  title={Geometrization of the real local Langlands correspondence (draft version, used for ARGOS seminar)},
  author={Scholze, Peter},
  year={2024}
}

@article{bznbetti,
  title={Betti geometric Langlands},
  author={Ben-Zvi, David and Nadler, David},
  journal={arXiv preprint arXiv:1606.08523},
  year={2016}
}

@article{buzzardgee,
  title={The conjectural connections between automorphic representations and Galois representations},
  author={Buzzard, Kevin and Gee, Toby},
  journal={Automorphic forms and Galois representations},
  volume={1},
  pages={135--187},
  year={2014}
}

@article{bernsteinsign,
  title={Hidden sign in Langlands' correspondence},
  author={Bernstein, Joseph},
  journal={arXiv preprint arXiv:2004.10487},
  year={2020}
}

@article{bzsv,
  title={Relative Langlands duality},
  author={Ben-Zvi, David and Sakellaridis, Yiannis and Venkatesh, Akshay},
  journal={arXiv preprint arXiv:2409.04677},
  year={2024}
}

@article{gaitsgoryraskin,
  title={Proof of the geometric Langlands conjecture V: The multiplicity one theorem},
  author={Gaitsgory, Dennis and Raskin, Sam},
  journal={arXiv preprint arXiv:2409.09856},
  year={2024}
}

@article{locsysres,
  title={The stack of local systems with restricted variation and geometric Langlands theory with nilpotent singular support},
  author={Arinkin, Dima and Gaitsgory, Dennis and Kazhdan, David and Raskin, Sam and Rozenblyum, Nick and Varshavsky, Yasha},
  journal={arXiv preprint arXiv:2010.01906},
  year={2020}
}

@article{beilinson,
  title={Topological $\epsilon$-factors},
  author={Beilinson, Alexander},
  journal={Pure and Applied Mathematics Quarterly},
  volume={3},
  number={1},
  pages={357--391},
  year={2007},
  publisher={International Press of Boston, Inc. Somerville, MA 02143, USA}
}

@article{slh,
  title={Courbes et fibr{\'e}s vectoriels en th{\'e}orie de Hodge $ z $-adique globale},
  author={Li-Huerta, Siyan Daniel},
  journal={arXiv preprint arXiv:2602.04978},
  year={2026}
}

@inproceedings{scholzeicm,
  title={p-adic Geometry},
  author={Scholze, Peter},
  booktitle={Proceedings of the International Congress of Mathematicians: Rio de Janeiro 2018},
  pages={899--933},
  year={2018},
  organization={World Scientific}
}

@incollection{faltings,
  author    = {Gerd Faltings},
  title     = {Does There Exist an Arithmetic Kodaira--Spencer Class?},
  booktitle = {Algebraic Geometry: Hirzebruch 70},
  series    = {Contemporary Mathematics},
  volume    = {241},
  pages     = {141--146},
  publisher = {American Mathematical Society},
  year      = {1999},
}

@article{fontaine,
  title={Formes diff{\'e}rentielles et modules de Tate des vari{\'e}t{\'e}s ab{\'e}liennes sur les corps locaux},
  author={Fontaine, Jean-Marc},
  journal={Inventiones mathematicae},
  volume={65},
  number={3},
  pages={379--409},
  year={1982},
  publisher={Springer-Verlag Berlin/Heidelberg}
}

@article{scholzeberkeley,
  title={Berkeley Lectures on p-adic Geometry:(AMS-207)},
  author={Scholze, Peter and Weinstein, Jared},
  year={2020},
  publisher={Princeton University Press}
}

@book{htt,
  title={Higher topos theory},
  author={Lurie, Jacob},
  year={2009},
  publisher={Princeton University Press}
}

@article{condensed,
      title={Lectures on Condensed Mathematics}, 
      author={Peter Scholze},
      year={2026},
      journal={arXiv preprint 2605.03658},
}

@inproceedings{hofmann,
  title={Equidimensional immersions of locally compact groups},
  author={Hofmann, Karl H and Wu, Ta-Sun and Yang, Jeoung S},
  booktitle={Mathematical Proceedings of the Cambridge Philosophical Society},
  volume={105},
  number={2},
  pages={253--261},
  year={1989},
  organization={Cambridge University Press}
}

@article{iwasawa,
  title={On some types of topological groups},
  author={Iwasawa, Kenkichi},
  journal={Annals of mathematics},
  volume={50},
  number={3},
  pages={507--558},
  year={1949},
  publisher={JSTOR}
}

@article{hochmost,
  title={Cohomology of Lie groups},
  author={Hochschild, Gerhard and Mostow, George Daniel},
  journal={Illinois Journal of Mathematics},
  volume={6},
  number={3},
  pages={367--401},
  year={1962},
  publisher={Duke University Press}
}

@article{kyed,
  title={Topologizing Lie algebra cohomology},
  author={Kyed, David},
  journal={Differential Geometry and its Applications},
  volume={49},
  pages={208--226},
  year={2016},
  publisher={Elsevier}
}

@article{clausen,
  title={Duality and linearization for p-adic lie groups},
  author={Clausen, Dustin},
  journal={arXiv preprint arXiv:2506.18174},
  year={2025}
}

@article{scholzeanalytic,
      title={Lectures on Analytic Geometry}, 
      author={Peter Scholze},
     
      journal={arXiv preprint 2605.03655},
       year={2026},
}

@article{cscx,
      title={Condensed Mathematics and Complex Geometry}, 
      author={Dustin Clausen and Peter Scholze},
      year={2026},
      journal={arXiv preprint 2605.11731}
}

@article{cautis,
      title={Ind-geometric stacks}, 
      author={Sabin Cautis and Harold Williams},
      year={2024},
      journal={arXiv preprint 2306.03043}
}

@article{dasgupta,
      title={The Brumer--Stark Conjecture over Z}, 
      author={Samit Dasgupta and Mahesh Kakde and Jesse Silliman and Jiuya Wang},
    journal={arXiv preprint arXiv:2409.09856},
      year={2023}, 
}

@inproceedings{smith,
  title={Equivariant Moore spaces},
  author={Smith, Justin R},
  booktitle={Algebraic and Geometric Topology: Proceedings of a Conference held at Rutgers University, New Brunswick, USA July 6--13, 1983},
  pages={238--270},
  year={2006},
  organization={Springer}
}

@book{corpslocaux,
  title={Local fields},
  author={Serre, Jean-Pierre},
  year={2013},
  publisher={Springer Science \& Business Media}
}

@book{SGA1,
  title={Rev{\^e}tements {\'e}tales et groupe fondamental: S{\'e}minaire de G{\'e}ometrie Alg{\'e}brique du Bois Marie 1960-61, SGA1},
  author={Grothendieck, Alexander and Raynaud, Michele},
  year={2003},
  publisher={Springer}
}

@article{cm1,
  title={Hyperdescent and {\'e}tale K-theory},
  author={Clausen, Dustin and Mathew, Akhil},
  journal={Inventiones mathematicae},
  volume={225},
  number={3},
  pages={981--1076},
  year={2021},
  publisher={Springer}
}

@article{cm2,
  title={Descent and vanishing in chromatic algebraic $ K $-theory via group actions},
  author={Clausen, Dustin and Mathew, Akhil and Naumann, Niko and Noel, Justin},
  journal={arXiv preprint arXiv:2011.08233},
  year={2020}
}

@article{gepner,
  title={Universality of multiplicative infinite loop space machines},
  author={Gepner, David and Groth, Moritz and Nikolaus, Thomas},
  journal={Algebraic \& Geometric Topology},
  volume={15},
  number={6},
  pages={3107--3153},
  year={2016},
  publisher={Mathematical Sciences Publishers}
}

@book{galcoh,
  title={Cohomologie galoisienne},
  author={Serre, Jean-Pierre},
  volume={5},
  year={1994},
  publisher={Springer Science \& Business Media}
}

@book {TateStark,
    AUTHOR = {Tate, John},
     TITLE = {Les conjectures de {S}tark sur les fonctions {$L$} d'{A}rtin
              en {$s=0$}},
    SERIES = {Progress in Mathematics},
    VOLUME = {47},
    EDITOR = {Bernardi, Dominique and Schappacher, Norbert},
 PUBLISHER = {Birkh\"auser Boston, Inc., Boston, MA},
      YEAR = {1984}
}

@inproceedings{mazur,
  title={Deforming Galois Representations},
  author={Mazur, Barry},
  booktitle={Galois Groups over $\mbb{Q}$: Proceedings of a Workshop Held March 23--27, 1987},
  pages={385--437},
  year={1989},
  organization={Springer}
}

@book{tannakian,
  title={Cat{\'e}gories tannakiennes},
  author={Rivano, N Saavedra},
  year={2006},
  publisher={Springer}
}

@article{kottwitz,
  title={B(G) for all local and global fields},
  author={Kottwitz, Robert},
  journal={arXiv preprint arXiv:1401.5728},
  year={2014}
}

@phdthesis{iakovenko,
  title={Representations of the Kottwitz gerbes},
  author={Iakovenko, Sergei},
  year={2021},
  school={Rheinische Friedrich-Wilhelms-Universitaet Bonn (Germany)}
}

@article{tatetori,
  title={The cohomology groups of tori in finite Galois extensions of number fields},
  author={Tate, John},
  journal={Nagoya Mathematical Journal},
  volume={27},
  number={2},
  pages={709--719},
  year={1966},
  publisher={Cambridge University Press}
}

@article{taylorcocycles,
  title={Cocycles for Kottwitz cohomology},
  author={Sempliner, Jack and Taylor, Richard},
  journal={arXiv preprint arXiv:2407.06031},
  year={2024}
}

@article{clausenartin,
  title={A K-theoretic approach to Artin maps},
  author={Clausen, Dustin},
  journal={arXiv preprint arXiv:1703.07842},
  year={2017}
}

@article{braunling,
  title={Local compactness as the K(1)-local dual of finite generation},
  author={Braunling, Oliver},
  journal={arXiv preprint arXiv:2301.05943},
  year={2023}
}

@article{brumer,
     author = {Brumer, Armand},
     title = {Pseudocompact algebras, profinite groups and class formations},
     journal = {Bull. Amer. Math. Soc.},
     volume = {72},
     number = {4},
     year = {1966},
     pages = { 321-324},
     language = {en},
}

@article{martinez,
  title={Cohomological dimension of discrete modules over profinite groups},
  author={Mart{\'\i}nez, Juan Jos{\'e}},
  journal={Pacific Journal of Mathematics},
  volume={49},
  number={1},
  pages={185--189},
  year={1973},
  publisher={Mathematical Sciences Publishers}
}

@book{nsw,
  title={Cohomology of number fields},
  author={Neukirch, J{\"u}rgen and Schmidt, Alexander and Wingberg, Kay},
  volume={323},
  year={2013},
  publisher={Springer Science \& Business Media}
}

@article{serreep,
  title={La distribution d'Euler-Poincar{\'e} d'un groupe profini},
  author={Serre, Jean-Pierre},
  journal={Galois representations in arithmetic algebraic geometry},
  volume={254},
  pages={461-495},
  year={1998},
  publisher={Cambridge University Press}
}

@book{serrerepthy,
  title={Linear representations of finite groups},
  author={Serre, Jean-Pierre},
  volume={42},
  year={1977},
  publisher={Springer}
}

@book{poitou,
  title={Cohomologie galoisienne des modules finis: S{\'e}minaire de l'Institut de math{\'e}matiques de Lille [1962-1963]},
  author={Poitou, Georges},
  number={13},
  year={1967},
  publisher={Dunod}
}

@book{emertongee,
  title={Moduli stacks of {\'e}tale ($\varphi$, $\Gamma$)-modules and the existence of crystalline lifts},
  author={Emerton, Matthew and Gee, Toby},
  number={215},
  year={2022},
  publisher={Princeton University Press}
}

@article{hennion,
  title={Tate objects in stable $(\infty, 1) $-categories},
  author={Hennion, Benjamin},
  journal={arXiv preprint arXiv:1606.05527},
  year={2016}
}

@book{nekovar,
  title={Selmer complexes},
  author={Nekov{\'a}r, Jan},
  volume={310},
  year={2006},
  publisher={Soci{\'e}t{\'e} math{\'e}matique de France}
}

@article{serrewt1,
  title={Modular forms of weight one and Galois representations},
  author={Serre, Jean-Pierre},
  journal={Springer Collected Works in Mathematics},
  pages={292--367},
  year={2003},
  publisher={Springer Berlin Heidelberg}
}

@article{galatius,
   title={Derived Galois deformation rings},
   volume={327},
   ISSN={0001-8708},
   journal={Advances in Mathematics},
   publisher={Elsevier BV},
   author={Galatius, S. and Venkatesh, A.},
   year={2018},
   month=mar, pages={470–623} }

@article{lafforgue,
  title={Chtoucas pour les groupes r\'eductifs et param\'etrisation de Langlands globale},
  author={Lafforgue, Vincent},
  journal={arXiv preprint arXiv:1209.5352},
  year={2012}
}

@article{scholzebourbaki,
  title={GEOMETRIC LANGLANDS [after Gaitsgory, Raskin,...]},
  author={Scholze, Peter},
  journal={S{\'e}minaire BOURBAKI},
  pages={78e},
  year={2026}
}

\end{document}